\documentclass[12pt,english]{amsart}
\usepackage[margin=1in]{geometry}
\usepackage[T1]{fontenc}
\usepackage[latin9]{inputenc}
\usepackage{amstext}
\usepackage{amsthm}
\usepackage{amssymb}
\usepackage{stmaryrd}
\usepackage{setspace}
\usepackage{graphicx}
\usepackage[export]{adjustbox}
\allowdisplaybreaks

\makeatletter
\numberwithin{equation}{section}
\numberwithin{figure}{section}
\theoremstyle{plain}
\newtheorem{thm}{\protect\theoremname}[section]
\theoremstyle{definition}
\newtheorem{defn}[thm]{\protect\definitionname}
\theoremstyle{plain}
\newtheorem{cor}[thm]{\protect\corollaryname}
\newtheorem{prop}[thm]{\protect\propositionname}
\newtheorem{lem}[thm]{\protect\lemmaname}
\newtheorem{conjecture}[thm]{\protect\conjecturename}
\theoremstyle{remark}
\newtheorem{rem}[thm]{\protect\remarkname}

\usepackage{tikz-cd}
\usepackage{graphicx}
\usepackage{enumerate} 
\usepackage{enumitem} 
\usepackage{romannum}

\makeatother

\usepackage{babel}
\providecommand{\corollaryname}{Corollary}
\providecommand{\definitionname}{Definition}
\providecommand{\remarkname}{Remark}
\providecommand{\theoremname}{Theorem}
\providecommand{\propositionname}{Proposition}
\providecommand{\lemmaname}{Lemma}
\providecommand{\conjecturename}{Conjecture}

\begin{document}
\pagenumbering{arabic}

\title{On 3d Quantum Trace Maps}

\author[Qingjing Chen]{Qingjing Chen}

\author[Andrew Kricker]{Andrew Kricker}

\begin{abstract}
A 
3d quantum trace map is a homomorphism from the skein module of an
ideally triangulated 3-manifold to its quantum gluing module that quantizes the classical trace map. There
are two constructions of such maps, one by Garoufalidis and Yu in
\cite{GY1}, and the other by Panitch and Park in \cite{PP1}. However,
the relationship between these two constructions was 
unknown. We propose
a third construction of the 3d quantum trace map which agrees with
the one given by Garoufalidis and Yu, and extends to certain types
of manifolds with ideally triangulated boundaries. Our 3d quantum
trace map can be compared with that of \cite{PP1} relatively easily 
by subdividing the face suspensions of \cite{PP1} and the ideal tetrahedra of our definition into a common subdivision based on face cones. This allows us to give an exact relation between the definitions, which partially addresses the
equivalence between the constructions of \cite{GY1} and \cite{PP1}.
\end{abstract}
\maketitle

\section{Introduction\label{sec:Introduction}}

For a 3-manifold $Y$ equipped with an ideal triangulation $\mathcal{T}$,
the 3d quantum trace map
\[
Tr_{\mathcal{T}}\colon\mathrm{Sk}(Y)\rightarrow\hat{\mathcal{G}}_{\mathcal{T}}
\]
is a map from the skein module $\mathrm{Sk}(Y)$ of $Y$ to the 
quantum gluing module $\hat{\mathcal{G}}_{\mathcal{T}}$ of $Y$.
The significance of this map is that it gives an exact relation between
two different ``quantizations'' of the $\text{SL}_{2}$-character
variety
\[
\mathfrak{X}(Y):=\text{Hom}(\pi_{1}(Y),\text{SL}_{2})//\text{SL}_{2}
\]
 of the 3-manifold $Y$:
\begin{enumerate}
\item Intrinsically, we can consider the skein module of the manifold $Y$:
fix a commutative ring $R$ containing a distinguished invertible
element $A$, the \emph{Kauffman bracket skein module} $\mathrm{Sk}(Y)$ \cite{Tur,Prz}
is defined to be the $R$-module spanned by isotopy classes of framed
links in $Y$, modulo Kauffman bracket relations. The module $\mathrm{Sk}(Y)$
is a quantization of the character variety $\mathfrak{X}(Y)$ in the
sense that if one specialize the module at $A=1$, $\mathrm{Sk}_{A=1}(Y)$
is a commutative algebra which, up to nilpotent elements, agrees with
the ring of regular functions on $\mathfrak{X}(Y)$ \cite{BW,Prz}.
\item Since the manifold $Y$ has an ideal triangulation $\mathcal{T}$,
we can consider the so called \emph{quantum gluing variety}
$\mathcal{G}_{\mathcal{T}}$ \cite{Dim,NZ}. Roughly speaking, we
assign the so-called \emph{shape parameters} to each edge of ideal tetrahedra
in $\mathcal{T}$ in certain consistent way, these shape-parameters
then satisfy a system of algebraic equations including the \emph{vertex
equations}, the\emph{ lagrangian equations} and the \emph{edge equations}.
The solution set of these equations is an algebraic variety $\mathcal{G}_{\mathcal{T}}$,
which can be viewed as a parametrization of $\mathfrak{X}(Y)$ as
points of $\mathcal{G}_{\mathcal{T}}$ give $\text{SL}_{2}$-representations
of $\pi_{1}(Y)$ up to conjugacy. The quantum gluing module $\hat{\mathcal{G}}_{\mathcal{T}}$
is by construction a deformation of the coordinate ring of the variety
$\mathcal{G}_{\mathcal{T}}$.  
\end{enumerate}

The existence of the 3d quantum trace map was conjectured in \cite{AGLR}
and since then there are two independent constructions of this map,
one by Garoufalidis and Yu \cite{GY1}, and the other by Panitch and
Park \cite{PP1}. We offer a third construction of the quantum trace
map that extends to certain types of 3-manifold with boundaries and explain how it compares with the constructions of Garoufalidis \& Yu and Panitch \& Park for a manifold without boundary (with cusps). 

\subsection{Our construction of a 3d quantum trace map.}
We give an overview of our construction. In what follows, we assume that
we have fixed a commutative ring $R$ containing distinguished invertible
elements $A^{\frac{1}{2}}$ and $(-A^{2})^{\frac{1}{2}}$. The notion of skein modules
of 3-manifolds can be extended to manifolds with boundaries if one considers
framed tangles with endpoints suitably localized on the boundary,
for instance see \cite{CL2} and \cite{PP1}. In our work, we will
take the approach in \cite{PP1}. This requires an extra decoration
on the boundary $\partial Y$ of the 3-manifold called a \emph{boundary
marking}, which is an oriented bipartite graph $\Gamma$ embedded
in $\partial Y$ such that each vertex is either a sink or a source,
and it serves to ``localize'' the endpoints of the framed tangles.
The result is the stated skein module $\mathrm{Sk}(Y,\Gamma)$ (see Definition \ref{def: stated skein modules} or \cite[Definition 3.3]{PP1}) and its
reduced version $\overline{\mathrm{Sk}}(Y,\Gamma)$ (see \cite[Definition 3.33]{PP1}). As we will discuss
later, if $v\in V(\Gamma)$ is a vertex of the boundary marking of
degree $n$, then $\overline{\mathrm{Sk}}(Y,\Gamma)$ is naturally a
module over the reduced skein algebra $\overline{\mathrm{SkAlg}}(D_{n})$
of the $n$-gon $D_{n}$. If $v$ is a sink, denoted $v\in V^{-}(\Gamma)$, then $\overline{\mathrm{Sk}}(Y,\Gamma)$
is a left $\overline{\mathrm{SkAlg}}(D_{n})$-module; if $v$ is a source, denoted $v\in V^{+}(\Gamma)$, 
then $\overline{\mathrm{Sk}}(Y,\Gamma)$ is a right $\overline{\mathrm{SkAlg}}(D_{n})$-module.
Thus if we denote by $D_{v}$ the $n$-gon associated to the vertex
$v$ (of degree $n$), then $\overline{\mathrm{Sk}}(Y,\Gamma)$ is a
\[
\underset{v\in V^{-}(\Gamma)}{\bigotimes}\overline{\mathrm{SkAlg}}(D_{v})\text{-}\underset{v\in V^{+}(\Gamma)}{\bigotimes}\overline{\mathrm{SkAlg}}(D_{v})
\]
-bimodule.

We consider an oriented 
ideally triangulated 3-manifold with boundary
$(Y,\mathcal{T})$. That is, $Y$ can be decomposed into a finite
collection $\mathcal{T}$ of ideal tetrahedra glued along their faces;
in this case, the boundary $\partial Y$ admits an ideal triangulation
$\mathcal{F}$ by some faces of the ideal tetrahedra in $\mathcal{T}$.
(By boundary of the manifold, we mean the actual manifold boundary,
not the cusp.) We fix once and for all a labeling of every edge of every
ideal tetrahedron in $\mathcal{T}$ by the shape parameters. More specifically, every ideal tetrahedron $T\in\mathcal{T}$ inherits
an orientation from the 3-manifold $Y$. We put a tetrahedron $T$
in $\mathbb{R}^{3}$ so that the orientation induced from $Y$ agrees
with the natural orientation of $\mathbb{R}^{3}$. Choose an ideal
vertex $v$ of $T$ and label the three edges adjacent to $v$ by
$z_{T},z^{\prime}_{T},z^{\prime\prime}_{T}$ in clockwise order 
with respect to an outward pointing normal 
and label the
edge opposite to edge labeled $z^{\boxempty}_{T}$ by $y^{\boxempty}_{T}$,
see Figure \ref{fig:ideal T wth shape parameters}.
\begin{figure}[h]
  \includegraphics[scale=0.15]{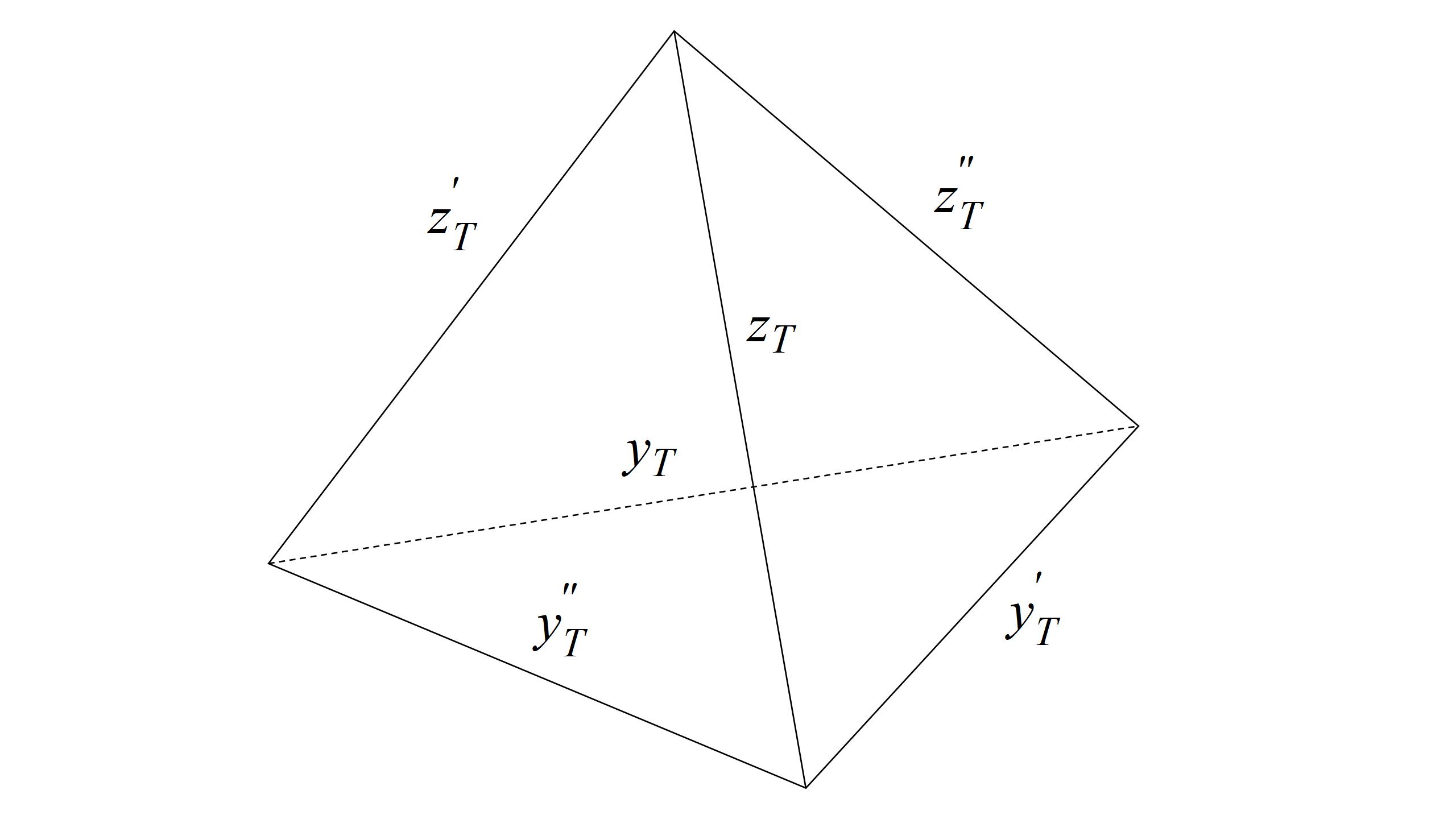} 
  \caption{}
  \label{fig:ideal T wth shape parameters}
\end{figure}
This is done for every ideal tetrahedron $T\in\mathcal{T}$. These labels are
shape parameters, meaning $z^{\boxempty}_{T}$ and $y^{\boxempty}_{T}$ are in fact variables with values in $\mathbb{C}$
that satisfy the following equations:
\begin{itemize}
\item In each ideal tetrahedron we have the \emph{vertex equation}
\begin{equation}
z_{T}z^{\prime}_{T}z^{\prime\prime}_{T}=z_{T}y^{\prime}_{T}y^{\prime\prime}_{T}=y_{T}z^{\prime}_{T}y^{\prime\prime}_{T}=y_{T}y^{\prime}_{T}z^{\prime\prime}_{T}=(-1)^{\frac{1}{2}}\label{eq:classical vertex equation}
\end{equation}
and the \emph{lagrangian equation}
\begin{equation}
z^{2}_{T}+(z^{\prime\prime}_{T})^{-2}=1.
\label{eq:classical lagrangian equation}
\end{equation}
\item We have the \emph{edge equations}. For every internal edge $e$ of the ideal triangulation
$\mathcal{T}$, let $e_{1},e_{2},\dots,e_{k}$ be shape parameters
of the edges of ideal tetrahedra that are identified to $e$. Then the edge equation associated to $e$ is
\begin{equation}
e_{1}e_{2}\dots e_{k}=-1.\label{eq:classical edge equation}
\end{equation}
\end{itemize}
To be more precise, $z^{\boxempty}_{T}$ and $y^{\boxempty}_{T}$ should be
called \emph{square-root} shape parameters as their squares play the
roles of honest shape parameters. In fact, by (\ref{eq:classical vertex equation})
and (\ref{eq:classical lagrangian equation}) we have $\left(z^{\boxempty}_{T}\right)^{2}=\left(y^{\boxempty}_{T}\right)^{2}$
and that the value of $Z^{\boxempty}_{T}:=\left(z^{\boxempty}_{T}\right)^{2}$ in $\mathbb{C}$ determines
a hyperbolic structure on each ideal tetrahedron (therefore $Z^{\boxempty}_{T}$'s are the honest
shape parameters) which, by (\ref{eq:classical edge equation}) together with suitable requirements on the arguments of the $Z^{\boxempty}_{T}$'s,
can be consistently glued by face-pairing isometries and give a hyperbolic
structure on $Y$.

The subscript $T$ on the shape parameters is to indicate the dependence on the ideal tetrahedra. When there is only a single ideal tetrahedron at hand, we will drop the subcript $T$.

Here is our key construction. For such a 3-manifold
$(Y,\mathcal{T})$ we can consider the \emph{canonical boundary marking}
$\Gamma_{0}$ on $\partial Y$, which is an oriented bipartite graph in $\partial Y$ whose vertices are either sources or sinks (see Definition \ref{def:ideally triangulated boundary marked 3-manifolds}).
The sources of $\Gamma_{0}$ are bivalent and are located at the barycentres of the edges of the boundary triangulation. 
The sinks of $\Gamma_{0}$ are trivalent and are located at the barycenters
of the triangles of the boundary triangulation. The (reduced) skein
algebra $\overline{\mathrm{SkAlg}}(D_{3})$ associated with such a trivalent
vertex is given by a certain quantum torus $\mathbb{T}$. We can twist
the usual stacking product structure on $\mathbb{T}$ by certain powers of
$A$ using a grading on the skein module via the boundary states of skeins.
Denoting the resulting product by $\cdot$, we will find that
$\left(\mathbb{T},\cdot\right)$ is commutative and is isomorphic to a Laurent polynomial algebra
$R[\alpha^{\pm1},\beta^{\pm1},\gamma^{\pm1}]$. We can also twist
the natural left action of $\mathbb{T}$ on $\overline{\mathrm{Sk}}(Y,\Gamma_{0})$,
and the result is a left $\left(\mathbb{T},\cdot\right)$-module structure
on $\overline{\mathrm{Sk}}(Y,\Gamma_{0})$. Again, we denote the resulting
left action of $\left(\mathbb{T},\cdot\right)$ on $\overline{\mathrm{Sk}}(Y,\Gamma_{0})$
by $\cdot$. 

We define the \emph{corner-reduced} skein module $\overline{\mathrm{Sk}}^{c}(Y,\Gamma_{0})$
to be the quotient of $\overline{\mathrm{Sk}}(Y,\Gamma_{0})$ that identifies
the left $\cdot$-action of the generators $\alpha$, $\beta$ and $\gamma$ of $\left(\mathbb{T},\cdot \right)$ with
the left multiplication by the scalar $(-A^{2})^{-\frac{1}{2}}$, and this identification is performed for
each trivalent sink of the canonical marking $\Gamma_{0}$ (see Definition
\ref{def:corner reduced module}). This construction
is largely motivated by the corner-reduction of Garoufalidis \& Yu
in \cite{GY1} for certain types of surfaces. A prototypical example
of such a 3-manifold with boundary is a single ideal tetrahedron $T$
(see Figure \ref{fig:ideal T with bm only} for the canonical boundary
marking on $T$), and we have the corresponding corner-reduced skein module $\overline{\mathrm{Sk}}^{c}(T)$. (We will drop $\Gamma_{0}$ from the notation whenever we are refering
to the canonical boundary markings coming from the ideal triangulation.)

Our construction of a quantum trace map may be summarized by
the following three steps.
\begin{enumerate}
\item First, we study the map of skein modules induced by splitting the
3-manifold $Y$ into ideal tetrahedra. In \cite{PP1}, there is a
splitting homomorphism $\overline{\mathrm{Sk}}(Y)\rightarrow\underset{T\in\mathcal{T}}{\overline{\bigotimes}}\overline{\mathrm{Sk}}(T)$
of reduced skein modules induced by the decomposition of $Y$ into
ideal tetrahedra, where the reduced tensor product $\underset{T\in\mathcal{T}}{\overline{\bigotimes}}\overline{\mathrm{Sk}}(T)$
is the quotient of the usual tensor product (of $R$-modules) $\underset{T\in\mathcal{T}}{\bigotimes}\overline{\mathrm{Sk}}(T)$
by two types of relations: (i) the relations coming from gluing ideal
tetrahedra around internal edges of the triangulation $\mathcal{T}$
and (ii) the relations coming from gluing faces of the ideal tetrahedra.
We adapt the argument in \cite{PP1} to the setting of corner-reduced
module and obtain our splitting homomorphism (Theorem \ref{thm:splitting homomorphism})
\[
\sigma\colon\overline{\mathrm{Sk}}^{c}(Y)\rightarrow\underset{T\in\mathcal{T}}{\overline{\bigotimes}}\overline{\mathrm{Sk}}^{c}(T).
\]
What is different from its counterpart in \cite{PP1} is that the\emph{ }reduced tensor
product $\underset{T\in\mathcal{T}}{\overline{\bigotimes}}\overline{\mathrm{Sk}}^{c}(T)$
of the corner-reduced modules is the quotient of the usual tensor
product $\underset{T\in\mathcal{T}}{\bigotimes}\overline{\mathrm{Sk}}^{c}(T)$
by relations coming \emph{only} from gluing of ideal tetrahedra around
internal edges of the triangulation. In a sense, the relations coming
from identifying pairs of faces of tetrahedra become trivial in the corner-reduced
setting.
\item Next, we construct the quantum trace map for a single tetrahdron $T$.
We obtain a map (Definition \ref{def:Quantum trace map on T})
\[
Tr_{T}^{c}\colon\overline{\mathrm{Sk}}^{c}(T)\rightarrow\hat{\mathcal{G}}(T)
\]
which turns out to be an {\bf isomorphism} 
of $R$-modules (Theorem \ref{thm:Tr induce isomorphism between corner reduced module and quantum module}). The codomain $\hat{\mathcal{G}}(T)$,
which we refer to as the \emph{quantum module} of the ideal tetrahedron
$T$, is given by the quotient
\[
\frac{\mathbb{T}\langle T\rangle}{\langle\text{vertex}\rangle_{\text{R}}+\langle\text{lagrangian}\rangle_{\text{R}}}.
\]
Here $\mathbb{T}\langle T\rangle$ is the skein algebra associated
to the six vertices of the boundary marking at the barycenters of
the six edges of $T$, equipped with $\cdot$-product (we can also
twist the reduced skein algebras associated with the sources of $\Gamma_{0}$).
It turns out to be a quantum torus, generated by six elements $\hat{z},\hat{z}^{\prime},\hat{z}^{\prime\prime},\hat{y},\hat{y}^{\prime},\hat{y}^{\prime\prime}$
corresponding to the six classical
shape parameters $z,z^{\prime},z^{\prime\prime},y,y^{\prime},y^{\prime\prime}$ we associated with the six edges of $T$. See Lemma \ref{lem:presntation of (T^tensor4,cdot)  and (B^tensor6, cdot)}
(ii) for the precise presentation of $\mathbb{T}\langle T\rangle$. This quantum torus can be thought
of as a non-commutative deformation of the Laurent polynomial algebra
in the variables corresponding to the classical shape
parameters. Therefore the generators of $\mathbb{T}\langle T\rangle$
are referred to as the \emph{quantized shape parameters}. The submodules
$\langle\text{vertex}\rangle_{\text{R}}$ and $\langle\text{lagrangian}\rangle_{\text{R}}$
are right ideals of $\mathbb{T}\langle T\rangle$ generated by the
quantum versions of vertex and lagrangian equations respectively (see
Definition \ref{def: quantum module of T}).
\item Lastly, to obtain the quantum trace map for $(Y,\mathcal{T})$, we
glue the individual quantum trace maps $Tr_{T}^{c}$ for each ideal
tetrahedron $T\in\mathcal{T}$. First, we must glue the quantum modules
$\hat{\mathcal{G}}(T)$ for each ideal tetrahedron. We define the
\emph{quantum gluing module} $\hat{\mathcal{G}}_{\mathcal{T}}$ to
be the quotient
\[
\hat{\mathcal{G}}_{\mathcal{T}}=\frac{\underset{T\in\mathcal{T}}{\bigotimes}\hat{\mathcal{G}}(T)}{\langle\text{edge}\rangle_{\text{L}}},
\]
where $\langle\text{edge}\rangle_{\text{L}}$ is the (image in $\underset{T\in\mathcal{T}}{\bigotimes}\hat{\mathcal{G}}(T)$
under quotient map of) left ideal of $\underset{T\in\mathcal{T}}{\bigotimes}\mathbb{T}\langle T\rangle$
generated by the quantum version of edge equations (see Definition
\ref{def: quantum gluing module}). This definition
of the quantum gluing module agrees with the one given in \cite{GY1}.
We will show that the tensor product of maps
\[
\underset{T\in\mathcal{T}}{\bigotimes}Tr_{T}^{c}\colon\underset{T\in\mathcal{T}}{\bigotimes}\overline{\mathrm{Sk}}^{c}(T)\rightarrow\underset{T\in\mathcal{T}}{\bigotimes}\hat{\mathcal{G}}(T)
\]
 descends to a well-defined map (Proposition \ref{prop:quantum trace for T glue})
\[
\underset{T\in\mathcal{T}}{\overline{\bigotimes}}Tr_{T}^{c}\colon\underset{T\in\mathcal{T}}{\overline{\bigotimes}}\overline{\mathrm{Sk}}^{c}(T)\rightarrow\hat{\mathcal{G}}_{\mathcal{T}}.
\]
Composing with the splitting homomorphism $\sigma\colon\overline{\mathrm{Sk}}^{c}(Y)\rightarrow\underset{T\in\mathcal{T}}{\overline{\bigotimes}}\overline{\mathrm{Sk}}^{c}(T)$,
we obtain our quantum trace map
\[
Tr_{\mathcal{T}}:=\left(\underset{T\in\mathcal{T}}{\overline{\bigotimes}}Tr_{T}^{c}\right)\circ\sigma\colon\overline{\mathrm{Sk}}^{c}(Y)\rightarrow\underset{T\in\mathcal{T}}{\overline{\bigotimes}}\overline{\mathrm{Sk}}^{c}(T)\rightarrow\hat{\mathcal{G}}_{\mathcal{T}}.
\]
\end{enumerate}

We have arrived at our main result.
\begin{thm}
\label{thm:Main theorem 0, existence of 3d quantum trace map} Let
$(Y,\mathcal{T})$ be an ideally triangulated 3-manifold with boundary,
equipped with the canonical boundary marking $\Gamma_{0}$. 
Our definitions give a well-defined $R$-module homomorphism
\[
Tr_{\mathcal{T}}\colon\overline{\mathrm{Sk}}^{c}(Y)\rightarrow\hat{\mathcal{G}}_{\mathcal{T}}.
\]
\end{thm}
Because $\overline{\mathrm{Sk}}^{c}(Y)$ is an $R$-module quotient
of $\overline{\mathrm{Sk}}(Y)$, which is in turn an $R$-module quotient
of $\mathrm{Sk}(Y)$, therefore composing with the quotient maps we
can also think of $Tr_{\mathcal{T}}$ as defined on $\overline{\mathrm{Sk}}(Y)$
or $\mathrm{Sk}(Y)$.

One reason this definition of quantum trace is significant is because it can be compared naturally to the definitions given in both \cite{GY1} and \cite{PP1}, in the case $Y$ is a cusped
3-manifold without boundary. (In this case, we have $\overline{\mathrm{Sk}}^{c}(Y)=\overline{\mathrm{Sk}}(Y)=\mathrm{Sk}(Y)$.)
We next give an overview of these relationships.


\subsection{Our quantum trace map agrees with that of Garoufalidis \& Yu.}

\begin{thm}
\label{thm: Main theorem 1, our Tr agrees with Tr of GY}Our quantum
trace map $Tr_{\mathcal{T}}$ agrees with the quantum trace map $Tr_{\mathcal{T}}^{[\text{GY}]}$
of Garoufalidis \& Yu when $Y$ is a manifold without boundary (with
cusps).
\end{thm}

We'll now give an overview of Garoufalidis \& Yu's construction of their quantum trace map $Tr_{\mathcal{T}}^{[\text{GY}]}$ and a quick sketch of how Theorem \ref{thm: Main theorem 1, our Tr agrees with Tr of GY} is proved. The construction of Garoufalidis and Yu makes heavy use of the theory
of skein algebra of surfaces developed, for instance, in \cite{CL}.
To start, one considers the dual surface $\Sigma_{\mathcal{T}}$ associated
to the ideal triangulation $\mathcal{T}$ constructed in the following
way: one places a lantern $\mathbb{L}_{T}$ (a 2-sphere with four discs
removed) in each ideal tetrahedron $T\in\mathcal{T}$, see Figure
\ref{fig:ideal T with an embedded lantern(bare)}; then when we glue
the ideal tetrahedra together, these lanterns glue along their boundary
circles into an embedded surface $\Sigma_{\mathcal{T}}$ in the 3--manifold
$Y$.\begin{figure}[h]
  \includegraphics[scale=0.17]{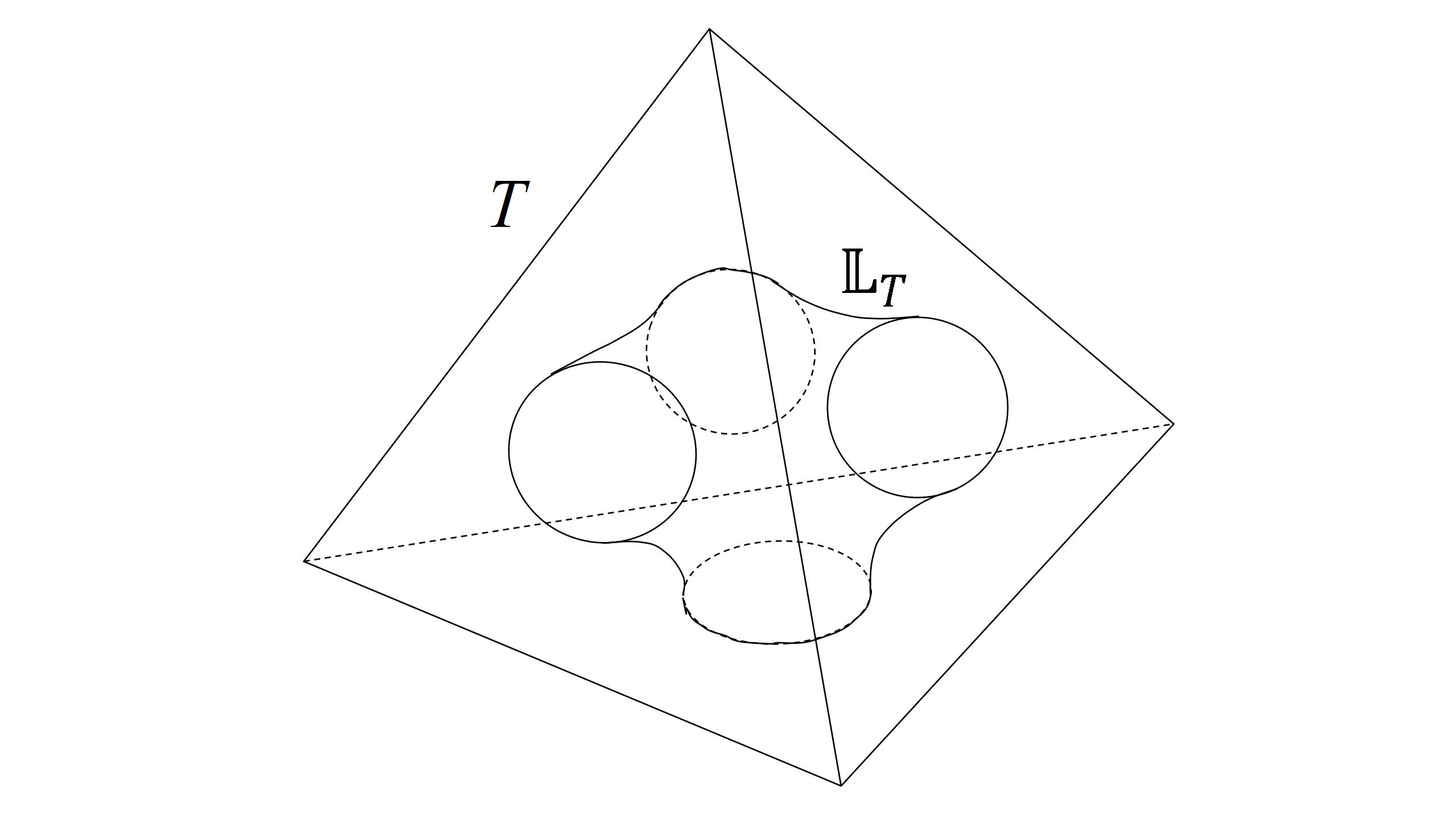} 
  \caption{ideal tetrahedron with an embedded lantern surface}
  \label{fig:ideal T with an embedded lantern(bare)}
\end{figure}
For each lantern $\mathbb{L}_{T}$ in the ideal tetrahedron $T$,
one has the so called corner-reduced skein module $\overline{\mathrm{Sk}}^{c}(\mathbb{L}_{T})$,
obtained in the following way: we first twist the product structure
on the skein algebra $\overline{\mathrm{SkAlg}}(\mathbb{L}_{T})$ in
a way similar to how we twist the product structure on $\mathbb{T}=\overline{\mathrm{SkAlg}}(D_{3})$
mentioned ealier, then $\overline{\mathrm{Sk}}^{c}(\mathbb{L}_{T})$
is defined to be the $R$-module quotient of $\overline{\mathrm{SkAlg}}(\mathbb{L}_{T})$
by a certain right ideal $J^{c}$ of the algebra $\left(\overline{\mathrm{SkAlg}}(\mathbb{L}_{T}),\cdot\right)$
(see Definition \ref{def:corner-reduced skein module of lantern surfaces}).
It was shown in \cite{GY1} that there is a well-defined surjective
$R$-module homomorphism
\[
\overline{\mathrm{Sk}}^{c}(\mathbb{L}_{T})\twoheadrightarrow\hat{\mathcal{G}}(T),
\]
where $\hat{\mathcal{G}}(T)$ is the quantum module of the ideal tetrahedron
$T$ we mentioned earlier. We will show that the natual inclusion $\mathbb{L}_{T}\hookrightarrow T$
induces a well-defined $R$-module homomorphism $\Phi_{T}\colon\overline{\mathrm{Sk}}^{c}(\mathbb{L}_{T})\rightarrow\overline{\mathrm{Sk}}^{c}(T)$
fitting into the following commutative diagram (Lemma \ref{lem:compatibility of GY Tr and our Tr for a single T}):
\begin{equation}
\label{lem:compatibility of GY Tr and our Tr for a single T in intro}
\begin{tikzcd} \overline{\mathrm{Sk}}^{c}(\mathbb{L}_{T}) \arrow[d, "\Phi_{T}"'] \arrow[rd, two heads] &                      \\ \overline{\mathrm{Sk}}^{c}(T) \arrow[r, "Tr_{T}^{c}"']                                  & \hat{\mathcal{G}}(T) \end{tikzcd}
\end{equation}
On the other hand, in \cite{GY1} it was shown that the decomposition
$\Sigma_{\mathcal{T}}=\bigcup_{T\in\mathcal{T}}\mathbb{L}_{T}$ induces
a well-defined splitting homomorphism 
\[
\Theta\colon\mathrm{Sk}(\Sigma_{\mathcal{T}})\rightarrow\bigotimes_{T\in\mathcal{T}}\overline{\mathrm{Sk}}^{c}(\mathbb{L}_{T}),
\]
which, we show, satisfies the following commutative diagram (Lemma
\ref{lem:lantern embeds in T induces map between corner reduced modules})
\begin{equation}
\label{eq:compatability of GY splitting with our splitting in intro} 
\begin{tikzcd} \mathrm{Sk}(\Sigma_{\mathcal{T}}) \arrow[r, "\Theta"] \arrow[d, two heads, "\Phi"] & \bigotimes_{T\in\mathcal{T}}\overline{\mathrm{Sk}}^{c}(\mathbb{L}_{T}) \arrow[r, "\bigotimes\Phi_{T} "] & \bigotimes_{T\in\mathcal{T}}\overline{\mathrm{Sk}}^{c}(T) \arrow[d, two heads] \\ \mathrm{Sk}(Y) \arrow[rr, "\sigma"]                                     &                                                                                                       & \overline{\bigotimes}_{T\in\mathcal{T}}\overline{\mathrm{Sk}}^{c}(T)           \end{tikzcd}.
\end{equation}
The quantum trace map $Tr_{\mathcal{T}}^{[\text{GY}]}$
of Garoufalidis \& Yu is then constructed in the following way. One
first considers the composition of $R$-module homomorphisms:
\[
\mathrm{Sk}(\Sigma_{\mathcal{T}})\stackrel{\Theta}{\rightarrow}\bigotimes_{T\in\mathcal{T}}\overline{\mathrm{Sk}}^{c}(\mathbb{L}_{T})\twoheadrightarrow\bigotimes_{T\in\mathcal{T}}\hat{\mathcal{G}}(T)\twoheadrightarrow\hat{\mathcal{G}}_{\mathcal{T}}.
\]
It was shown in \cite{GY1} that this map descends to $\mathrm{Sk}(Y)$
via the map $\Phi\colon\mathrm{Sk}(\Sigma_{\mathcal{T}})\twoheadrightarrow\mathrm{Sk}(Y)$
induced by the inclusion $\Sigma_{\mathcal{T}}\hookrightarrow Y$;
it is a quotient map with kernel generated by handle slides. The resulting
map defined on $\mathrm{Sk}(Y)$ is the quantum trace map $Tr_{\mathcal{T}}{}^{[\text{GY}]}$.
Therefore, in order to show that $Tr_{\mathcal{T}}^{[\text{GY}]}=Tr_{\mathcal{T}}$,
we only need to show that the following diagram commutes 
\begin{equation}
\begin{tikzcd} \mathrm{Sk}(\Sigma_{\mathcal{T}}) \arrow[r, "\Theta"] \arrow[d, "\Phi"', two heads] & \bigotimes_{T\in\mathcal{T}}\overline{\mathrm{Sk}}^{c}(\mathbb{L}_{T}) \arrow[r, two heads] & \bigotimes_{T\in\mathcal{T}}\hat{\mathcal{G}}(T) \arrow[r, two heads] & \hat{\mathcal{G}}_{\mathcal{T}} \\ \mathrm{Sk}(Y) \arrow[rrru, "Tr_{\mathcal{T}}"']                                    &                                                                                                              &                                                                       &                                 \end{tikzcd}
\end{equation}
But this is boundary of the following commutative diagram
\[ 
\begin{tikzcd} \mathrm{Sk}(\Sigma_{\mathcal{T}}) \arrow[d, "\Phi"', two heads] \arrow[r, "\Theta"] & \bigotimes_{T\in\mathcal{T}}\overline{\mathrm{Sk}}^{c}(\mathbb{L}_{T}) \arrow[d, "\bigotimes \Phi_{T}"'] \arrow[rd, "\bigotimes Tr_T^{c} "] &                                                                       \\ \mathrm{Sk}(Y) \arrow[rd, "\sigma"']                                                & \bigotimes_{T\in\mathcal{T}}\overline{\mathrm{Sk}}^{c}(T) \arrow[d, two heads] \arrow[r, "\bigotimes Tr_{T}^{c}"']                          & \bigotimes_{T\in\mathcal{T}}\hat{\mathcal{G}}(T) \arrow[d, two heads] \\                                                                                   & \overline{\bigotimes}_{T\in\mathcal{T}}\overline{\mathrm{Sk}}^{c}(T) \arrow[r, "\overline{\bigotimes}Tr_{T}^{c}"']                          & \hat{\mathcal{G}}_{\mathcal{T}}                                       \end{tikzcd} 
\]
which is given by combining (\ref{lem:compatibility of GY Tr and our Tr for a single T in intro}),
(\ref{eq:compatability of GY splitting with our splitting in intro})
and the fact that 
\[
\underset{T\in\mathcal{T}}{\bigotimes}Tr_{T}^{c}\colon\underset{T\in\mathcal{T}}{\bigotimes}\overline{\mathrm{Sk}}^{c}(T)\rightarrow\underset{T\in\mathcal{T}}{\bigotimes}\hat{\mathcal{G}}(T)
\]
 descends to the map 
\[
\underset{T\in\mathcal{T}}{\overline{\bigotimes}}Tr_{T}^{c}\colon\underset{T\in\mathcal{T}}{\overline{\bigotimes}}\overline{\mathrm{Sk}}^{c}(T)\rightarrow\hat{\mathcal{G}}_{\mathcal{T}}.
\]

Strictly speaking, our quantum trace map is a lift of the quantum
trace map defined in \cite{GY1}, but this is only a minor technical
issue. To be precise, in \cite{GY1}, one further quotients out the quantum
gluing module $\hat{\mathcal{G}}_{\mathcal{T}}$ by the relation which
identifies the quantized shape parameter associated to opposite edges
in each ideal tetrahedron. However, it is also remarked in \cite{GY1}
that their construction went through without imposing such further
identification on $\hat{\mathcal{G}}_{\mathcal{T}}$.

Another point worth mentioning is that the labeling and orientation
conventions in our work follows those of \cite{PP1} and are different
from those of \cite{GY1}, so the equivalence in the last theorem
is also up to these differences.

\subsection{Our quantum trace map agrees with that of Panitch \& Park up to a natural
map of quantum gluing modules.}

Again in this section we will assume $Y$ is a cusped 3-manifold without boundary and $\mathcal{T}$
is an ideal triangulation of $Y$.

The quantum gluing module $\hat{\mathcal{G}}_{\mathcal{T}}^{[\text{PP}]}$
in \cite{PP1} is constructed slightly differently from ours. Without
getting into the precise definition of $\hat{\mathcal{G}}_{\mathcal{T}}^{[\text{PP}]}$
for the moment,
(see Definition \ref{def:quantized shape parameters and quantum gluing module as in PP}
and \cite[Definition 5.13]{PP1}\footnote{For comparison, the notation for $\hat{\mathcal{G}}_{\mathcal{T}}^{[\text{PP}]}$
is $\text{SQGM}_{\mathcal{T}}(Y)$ in \cite{PP1} and is referred to as the \emph{square-root quantum gluing module}, which is a more accurate terminology because we are indeed in the square-root setting.}), 
as a consequence of the definition we know that, similar to our quantum
gluing module $\hat{\mathcal{G}}_{\mathcal{T}}$, $\hat{\mathcal{G}}_{\mathcal{T}}^{[\text{PP}]}$
is a quotient of a quantum torus $\underset{T\in\mathcal{T}}{\bigotimes}\tilde{\mathbb{T}}\langle T\rangle$,
where each $\tilde{\mathbb{T}}\langle T\rangle$ is a quantum torus
generated by six elements $\tilde{z}_{T},\tilde{z}^{\prime}_{T},\tilde{z}^{\prime\prime}_{T},\tilde{y}_{T},\tilde{y}^{\prime}_{T},\tilde{y}^{\prime\prime}_{T}$
corresponding to the classical shape parameters $z_{T},z^{\prime}_{T},z^{\prime\prime}_{T},y_{T},y^{\prime}_{T},y^{\prime\prime}_{T}$.
We will also refer to the elements $\tilde{z}^{\boxempty}_{T}$ and $\tilde{y}^{\boxempty}_{T}$
as \emph{quantized shape parameters} (in the sense of Panitch \& Park).
It turns out that $\tilde{\mathbb{T}}\langle T\rangle$ and $\mathbb{T}\langle T\rangle$
are isomorphic quantum tori through the isomorphism that matches
quantized shape parameters $\hat{z}^{\boxempty}_{T}\mapsto\tilde{z}^{\boxempty}_{T}$
and $\hat{y}^{\boxempty}_{T}\mapsto\tilde{y}^{\boxempty}_{T}$. It also turns out that these isomorphisms
induce a surjective map 
\[
\mu\colon\hat{\mathcal{G}}_{\mathcal{T}}\twoheadrightarrow\hat{\mathcal{G}}_{\mathcal{T}}^{[\text{PP}]}
\]
 of the quantum gluing modules. This needs the fact that 
 the (quantum
version of) vertex relations, lagrangian relations, and edge relations
all hold in $\hat{\mathcal{G}}_{\mathcal{T}}^{[\text{PP}]}$.

\begin{thm}
\label{thm: Main theorem 2, exact relation between our Tr and PP's Tr}There
is a surjective $R$-module homomorphism $\mu\colon\hat{\mathcal{G}}_{\mathcal{T}}\rightarrow\hat{\mathcal{G}}_{\mathcal{T}}^{[PP]}$
, which sends quantized shape parameters in $\hat{\mathcal{G}}_{\mathcal{T}}$ to the corresponding quantized shape parameters in 
$\hat{\mathcal{G}}_{\mathcal{T}}^{[PP]}$,
such that
\[
\mu\circ Tr_{\mathcal{T}}=Tr_{\mathcal{T}}^{[\text{PP}]}.
\]
\end{thm}

We will now sketch out some of the ideas behind this theorem. The main geometric idea is to further split both the ideal tetrahedra used by $Tr_{\mathcal{T}}$,  and the face suspensions used by $Tr_{\mathcal{T}}^{[\text{PP}]}$, into a common subdivision based on face cones. We define these terms next.

The construction of Panitch \& Park
of their quantum trace map starts with decomposing the manifold $Y$
into so-called \emph{face suspensions}. If $f$ is a face of
the ideal triangulation $\mathcal{T}$, its face suspension $Sf$
is built in the following way. If $f_{i}$ and $f_{j}$ are the bare faces
of ideal tetrahedra that are identified to the face $f$ of the
ideal triangulation, then the face suspension $Sf$ is given by gluing
the corresponding \emph{face cones} $Cf_{i}$ and $Cf_{j}$ along the map that identifies
$f_{i}$ and $f_{j}$; here a face cone of a face $f_{i}$ of a tetrahedron
$T$ is defined to be the union of line segments between the barycenter
of $T$ and points of $f_{i}$. Both face cones and face suspensions
come with canonical 
boundary markings (see Figure \ref{fig:face cone with bm}
and \ref{fig:face suspension with bm}) and therefore we can talk
about their (reduced) skein modules $\overline{\mathrm{Sk}}(Sf)$ and
$\overline{\mathrm{Sk}}(Cf_{i})$. 

For face cones, we will also use a further
reduction, called the \emph{partially corner-reduced} skein module $\overline{\mathrm{Sk}}^{pc}(Cf_{i})$,
whose construction is essentially the same as the corner-reduced modules discussed earlier,
but in this case we perform corner-reduction \emph{only} on the face $f_{i}$ 
and
disregard the other parts of the boundary  
(see Definition \ref{def: partial corner reduced module of face cone}
for the precise description). 

We can decompose the 3-manifold $Y$ into face cones in two ways.
One way is to first decompose $Y$ into face suspensions and then decompose
each of those face suspensions into two face cones. Alternatively, we can first decompose
$Y$ into ideal tetrahedra and then decompose each of those tetrahedra into
four face cones. This results in the following commutative diagram of splitting
homomorphisms,
which is discussed in detail in Section \ref{subsec:splitting homomorphisms for face cones and face suspensions}.
\begin{equation}
\label{eq:commutatve diagram of splitting homomorphisms in intro}
\begin{tikzcd}                                                                 & \underset{f\in\mathcal{T}^{(2)}}{\overline{\bigotimes}}\overline{\mathrm{Sk}}(Sf)  \arrow[rd, "\underset{f\in\mathcal{T}^{(2)}}{\overline{\bigotimes}}\sigma_{Sf}"] &                                                                                                                \\ \mathrm{Sk}(Y) \arrow[ru, "\tilde{\sigma}"] \arrow[rd, "\sigma"'] &                                                                                                                                                                   & {{{\underset{f_{i}\in\mathbf{f}(\mathcal{T})}{\overline{\bigotimes}}\overline{\mathrm{Sk}}^{pc}(Cf_{i})}}} \\                                                                 & \underset{T\in\mathcal{T}}{\overline{\bigotimes}}\overline{\mathrm{Sk}}^{c}(T) \arrow[ru, "\underset{T\in\mathcal{T}}{\overline{\bigotimes}}\sigma_{T}"']           &                                                                                                                \end{tikzcd}. \end{equation}
In this diagram:
\begin{itemize}
\item The map $\sigma$ is the splitting homomorphism we discussed earlier induced by decomposing $Y$
into ideal tetrahedra. Recall that the module $\underset{T\in\mathcal{T}}{\overline{\bigotimes}}\overline{\mathrm{Sk}}^{c}(T)$
is the quotient of the usual tensor product 
of corner-reduced, reduced skein modules of tetrahedra $\underset{T\in\mathcal{T}}{\bigotimes}\overline{\mathrm{Sk}}^{c}(T)$ by relations coming from gluing around edges of 
$\mathcal{T}$. (See Theorem \ref{thm:splitting homomorphism}.)
\smallskip
\item The map $\tilde{\sigma}$ is the splitting homomorphism of Panitch \& Park
induced by decomposing $Y$ into face suspensions (\cite[Corollary 3.37]{PP1}),
where $\underset{f\in\mathcal{T}^{(2)}}{\overline{\bigotimes}}\overline{\mathrm{Sk}}(Sf)$
is the quotient of the usual tensor product $\underset{f\in\mathcal{T}^{(2)}}{\bigotimes}\overline{\mathrm{Sk}}(Sf)$
by two types of relations: (i) the relations coming from gluing around
edges of $\mathcal{T}$ and (ii) the relations coming from gluing
face suspensions around the vertex cones of ideal tetrahedra. (See
Definition \ref{def:reduced tensor product of skein modules of face suspensions}.)\smallskip
\item $\underset{f\in\mathcal{T}^{(2)}}{\overline{\bigotimes}}\sigma_{Sf}$
is the map induced by decomposing each face suspension into 2 face cones
(see Proposition \ref{prop:gluing splitting homomorphism coming from splitting every single face suspension}). To build the co-domain we first consider the tensor product 
$\underset{f_{i}\in\mathbf{f}(\mathcal{T})}{\bigotimes}\overline{\mathrm{Sk}}^{pc}(Cf_{i})$,
which is over $\mathbf{f}(\mathcal{T})$, the set of faces of tetrahedra of $\mathcal{T}$, (the so-called {\em bare faces}), which indexes the face cones of $\mathcal{T}$. 
Then $\underset{f_{i}\in\mathbf{f}(\mathcal{T})}{\overline{\bigotimes}}\overline{\mathrm{Sk}}^{pc}(Cf_{i})$
is the quotient of that tensor product by the same two types
of relations considered in the last point (see Definition \ref{def:reduced tensor peoduct of partially corner-reduced modules of face cones}).
Note that we corner-reduce at the new faces $f_{i}$ appearing in this splitting, and only those faces,  
and the fact that there is no relation coming from gluing up these faces $f_{i}$ is 
because it becomes trivial in the corner-reduced
setting.\smallskip
\item $\underset{T\in\mathcal{T}}{\overline{\bigotimes}}\sigma_{T}$ is
the map induced by decomposing each ideal tetrahedron into 4 face cones.
(See Proposition \ref{prop:gluing splitting homomorphism coming from splitting every ideal tetrahedron}.)
\end{itemize}

To construct $Tr_{\mathcal{T}}^{[\text{PP}]}$, one first constructs
a quantum trace map $Tr_{Sf}\colon\overline{\mathrm{Sk}}(Sf)\rightarrow\mathbf{S}f$
for each face suspension $Sf$. Here $\mathbf{S}f$ is called the
\emph{face suspension module} of $Sf$ (Definition \ref{def:face cone module and face suspension module}
(ii)). It is a quantum torus whose generators are formal variables
corresponding to the marking edges of the boundary marking of $Sf$.
$\mathbf{S}f$ can be expressed as the tensor product $\mathbf{S}f=\mathbf{C}f_{i}\otimes\mathbf{C}f_{j}$
where $\mathbf{C}f_{i}$ and $\mathbf{C}f_{j}$ are the so called
\emph{face cone modules} associated with the face cones $Cf_{i}$
and $Cf_{j}$ that are glued to form $Sf$  (Definition \ref{def:face cone module and face suspension module}
(i)). Next, one consider the reduced
tensor product $\underset{f\in\mathcal{T}^{(2)}}{\overline{\bigotimes}}\mathbf{S}f$,
which is essentially defined in a way so that the usual tensor product
\[
\underset{f\in\mathcal{T}^{(2)}}{\bigotimes}Tr_{Sf}\colon\underset{f\in\mathcal{T}^{(2)}}{\bigotimes}\overline{\mathrm{Sk}}(Sf)\rightarrow\underset{f\in\mathcal{T}^{(2)}}{\bigotimes}\mathbf{S}f
\]
descends to a well-defined map
\[
\underset{f\in\mathcal{T}^{(2)}}{\overline{\bigotimes}}Tr_{Sf}\colon\underset{f\in\mathcal{T}^{(2)}}{\overline{\bigotimes}}\overline{\mathrm{Sk}}(Sf)\rightarrow\underset{f\in\mathcal{T}^{(2)}}{\overline{\bigotimes}}\mathbf{S}f
\]
(see (\ref{eq:reduced tensor product of face suspension modules})).
Then the quantum trace map $Tr_{\mathcal{T}}^{[\text{PP}]}$ of Panitch
\& Park is defined to be the composition
\[
Tr_{\mathcal{T}}^{[\text{PP}]}=\left(\underset{f\in\mathcal{T}^{(2)}}{\overline{\bigotimes}}Tr_{Sf}\right)\circ\tilde{\sigma}\colon\mathrm{Sk}(Y)\rightarrow\underset{f\in\mathcal{T}^{(2)}}{\overline{\bigotimes}}\overline{\mathrm{Sk}}(Sf)\rightarrow\underset{f\in\mathcal{T}^{(2)}}{\overline{\bigotimes}}\mathbf{S}f.
\]
By definition, the quantum gluing module $\hat{\mathcal{G}}_{\mathcal{T}}^{[\text{PP}]}$
is a certain submodule of $\underset{f\in\mathcal{T}^{(2)}}{\overline{\bigotimes}}\mathbf{S}f$
and it is shown that the image of $Tr_{\mathcal{T}}^{[\text{PP}]}$
is contained in $\hat{\mathcal{G}}_{\mathcal{T}}^{[\text{PP}]}$.

To connect $Tr_{\mathcal{T}}^{[\text{PP}]}$ to our quantum trace
map $Tr_{\mathcal{T}}$ we first construct a quantum trace map for
each face cone
\[
Tr_{Cf_{i}}^{pc}\colon\overline{\mathrm{Sk}}^{pc}(Cf_{i})\rightarrow\mathbf{C}f_{i}
\]
in a way similar to the way the quantum trace map $Tr_{T}^{c}$ for an ideal tetrahedron was constructed earlier (see Proposition \ref{prop:P^pc descends}).
Together, they give us a map (Lemma \ref{lem:quantum trace map for face cones glue})
\[
\underset{f_{i}\in\mathbf{f}(\mathcal{T})}{\overline{\bigotimes}}Tr_{Cf_{i}}^{pc}\colon\underset{f_{i}\in\mathbf{f}(\mathcal{T})}{\overline{\bigotimes}}\overline{\mathrm{Sk}}^{pc}(Cf_{i})\rightarrow\underset{f_{i}\in\mathbf{f}(\mathcal{T})}{\overline{\bigotimes}}\mathbf{C}f_{i}=\underset{f\in\mathcal{T}^{(2)}}{\overline{\bigotimes}}\mathbf{S}f
\]
 making the following two diagrams commute: 
 
 First, we have the commutative diagram which essentially says that ``the quantum trace maps on the face cones glue to the quantum trace maps on the face suspensions via the splitting homomorphisms induced by decomposing each face suspension into two face cones'':
\begin{equation}
\label{lem:comparing 2,3 in intro}
\begin{tikzcd}
\underset{f\in\mathcal{T}^{(2)}}{\overline{\bigotimes}}\overline{\mathrm{Sk}}(Sf) \arrow[rrd, "\underset{f\in\mathcal{T}^{(2)}}{\overline{\bigotimes}}Tr_{Sf}"] \arrow[dd, "\underset{f\in\mathcal{T}^{(2)}}{\overline{\bigotimes}}\sigma_{Sf}"'] &  &                                                                    \\
&  & \underset{f\in\mathcal{T}^{(2)}}{\overline{\bigotimes}}\mathbf{S}f \\
\underset{f_{i}\in\mathbf{f}(\mathcal{T})}{\overline{\bigotimes}}\overline{\mathrm{Sk}}^{pc}(Cf_{i}) \arrow[rru, "\underset{f_{i}\in\mathbf{f}(\mathcal{T})}{\overline{\bigotimes}}Tr_{Cf_{i}}^{pc}"']        &  &                                                                   
\end{tikzcd}.
\end{equation}
Second, we have the commutative diagram which essentially says that ``the quantum trace map on the face cones glue to the quantum trace map on the ideal tetrahedra, up to the map $\mu$, via the splitting homomorphism induced by decomposing each ideal tetrahedron into four face cones'':
\begin{equation}
\label{lem:comparing 1,3 in intro}
\begin{tikzcd} \underset{f_{i}\in\mathbf{f}(\mathcal{T})}{\overline{\bigotimes}}\overline{\mathrm{Sk}}^{pc}(Cf_{i}) \arrow[rr, "\underset{f_{i}\in\mathbf{f}(\mathcal{T})}{\overline{\bigotimes}}Tr_{Cf_{i}}^{pc}"] &  & \underset{f\in\mathcal{T}^{(2)}}{\overline{\bigotimes}}\mathbf{S}f \\ &  & {\hat{\mathcal{G}}_{\mathcal{T}}^{[\text{PP}]}} \arrow[u, hook]    \\ \underset{T\in\mathcal{T}}{\overline{\bigotimes}}\overline{\mathrm{Sk}}^{c}(T) \arrow[uu, "\underset{T\in\mathcal{T}}{\overline{\bigotimes}}\sigma_{T}"] \arrow[rr, "\underset{T\in\mathcal{T}}{\overline{\bigotimes}}Tr_{T}^{c}"']      &  & \hat{\mathcal{G}}_{\mathcal{T}} \arrow[u, "\mu"']                  \end{tikzcd}.
\end{equation}
(See Lemma \ref{lem:comparing 2,3} and \ref{lem:comparing 1,3}.) Now combine the commutative diagrams (\ref{eq:commutatve diagram of splitting homomorphisms in intro}),
(\ref{lem:comparing 2,3 in intro}) and (\ref{lem:comparing 1,3 in intro})
to get the following commutative diagram whose boundary proves Theorem \ref{thm: Main theorem 2, exact relation between our Tr and PP's Tr},
that is that $Tr_{\mathcal{T}}^{[\text{PP}]}=\mu\circ Tr_{\mathcal{T}}$.

\[
\begin{tikzcd}  &  & \underset{f\in\mathcal{T}^{(2)}}{\overline{\bigotimes}}\overline{\mathrm{Sk}}(Sf) \arrow[dd] \arrow[rrdd, "\underset{f\in\mathcal{T}^{(2)}}{\overline{\bigotimes}}Tr_{Sf}"] &  &                                                                    \\  &  &                                                                                &  &                                                                    \\ \mathrm{Sk}(Y) \arrow[rruu, "\tilde{\sigma}"] \arrow[rrdd, "\sigma"'] &  & \underset{f_{i}\in\mathbf{f}(\mathcal{T})}{\overline{\bigotimes}}\overline{\mathrm{Sk}}^{pc}(Cf_{i}) \arrow[rr]                                           &  & \underset{f\in\mathcal{T}^{(2)}}{\overline{\bigotimes}}\mathbf{S}f \\  &  &                                                                                &  & {\hat{\mathcal{G}}_{\mathcal{T}}^{[\text{PP}]}} \arrow[u, hook]    \\ &  & \underset{T\in\mathcal{T}}{\overline{\bigotimes}}\overline{\mathrm{Sk}}^{c}(T) \arrow[uu] \arrow[rr, "\underset{T\in\mathcal{T}}{\overline{\bigotimes}}Tr_{T}^{c}"]         &  & \hat{\mathcal{G}}_{\mathcal{T}} \arrow[u, "\mu"']                  \end{tikzcd}
\]

\subsection*{Acknowledgement}
We are indebted to Hiroaki Karuo, Yi Khing Law and Daniel Mathews for
many enlightening discussions. The project also received support through
the AcRF Tier 1 grants RG 17/23 and RG 106/25 from the Singapore Ministry
of Education.

\section{Skein modules of boundary marked 3-manifolds\label{sec:Skein module}}

\subsection{Skein modules of boundary marked 3-manifolds\label{subsec:Skein module basics}}

We recall the basic definitions and properties of skein modules of
3-manifolds. We follow Panitch \& Park \cite{PP1} closely. In particular, we will adopt their approach 
to stated skein modules of 3-manifolds. Throughout, $R$ will be a commutative
ring containing distinguished invertible elements $A^{\frac{1}{2}}$
and $(-A^{2})^{\frac{1}{2}}$.

\begin{defn}
\label{def: boundary marked 3-manifolds}Let $Y$ be a 3-manifold.
A \emph{boundary marking} of $Y$ is a smoothly embedded oriented
graph $\Gamma\subset\partial Y$, such that each vertex is either
a sink or source. 
\end{defn}
In the last definition, we do not require the graph $\Gamma$ to be
connected, and we allow some (or even all) connected components of
$\Gamma$ to be open intervals. 

\begin{defn}
\label{def: stated skein modules}Let $(Y,\Gamma)$ be a boundary
marked 3-manifold. The corresponding \emph{stated skein module} $\mathrm{Sk}(Y,\Gamma)$
is the free $R$-module spanned by isotopy classes of unoriented
ribbon tangles in $Y$, each of whose boundary components lies 
in $\Gamma\backslash V(\Gamma)$ and carries a \emph{state}, which is a sign $\epsilon\in\{\pm1\}$,
modulo the following skein relations\footnote{The red oriented edges in the diagrams are the boundary marking edges}:
\begin{equation}
\tag{S1}
\includegraphics[valign=c , scale=0.2]{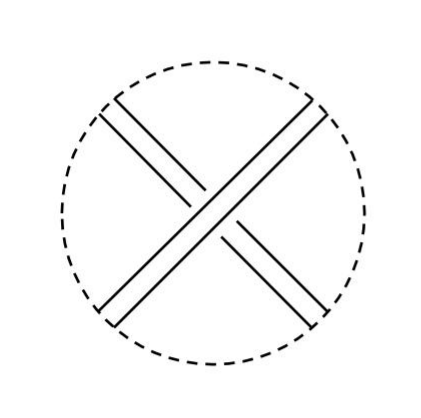}=A\includegraphics[valign=c , scale=0.2]{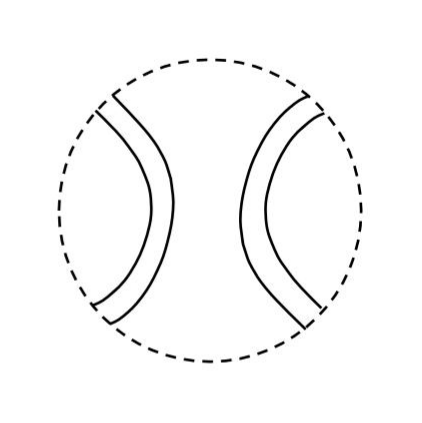}=A^{-1}\includegraphics[valign=c , scale=0.18]{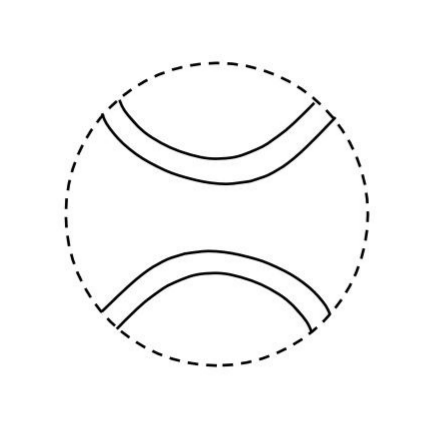}\label{eq:S1}
\end{equation}
\begin{equation}
\tag{S2}
\includegraphics[valign=c , scale=0.2]{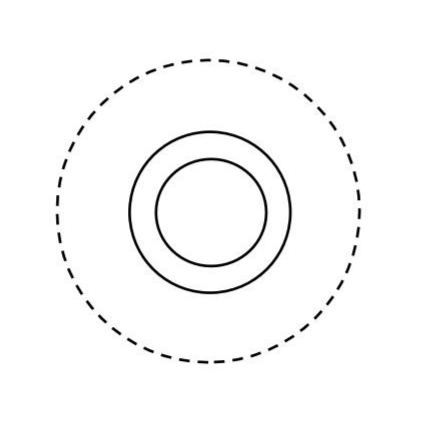}=-A^{2}-A^{-2}\label{eq:S2}
\end{equation}
\begin{equation}
\tag{S3}
\includegraphics[valign=c , scale=0.2]{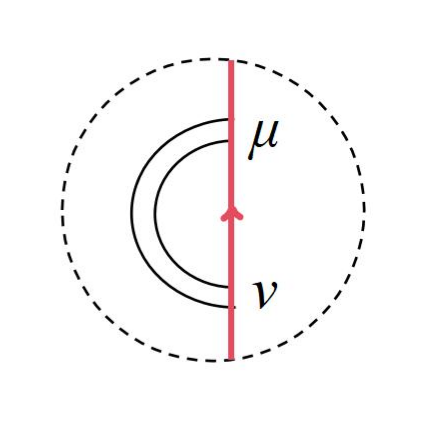}=\delta_{\mu,-\nu}(-A^{2})^{\frac{\mu}{2}},\ \ \mu,\nu\in\{\pm1\}\label{eq:S3}
\end{equation}
\begin{equation}
\tag{S4}
\includegraphics[valign=c , scale=0.2]{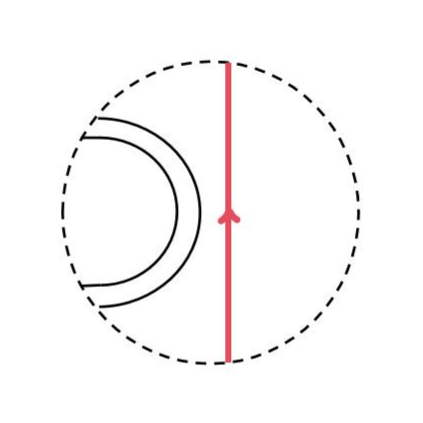}=\sum_{\mu\in\{\pm1\}}(-A^{2})^{\frac{\mu}{2}}\includegraphics[valign=c , scale=0.2]{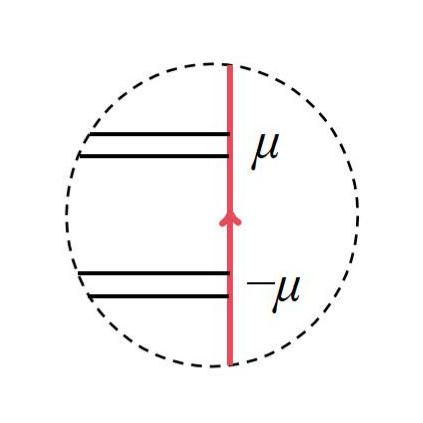}\label{eq:S4}
\end{equation}
\begin{equation}
\tag{S5}
\includegraphics[valign=c , scale=0.2]{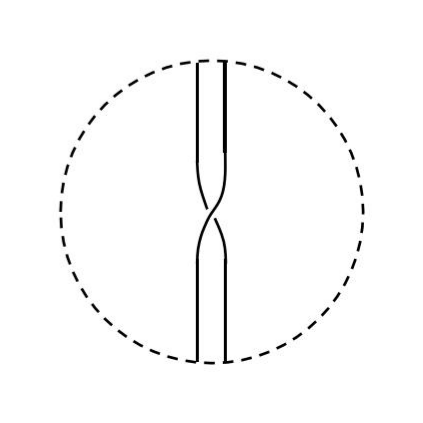}=(-A^{3})^{\frac{1}{2}}\includegraphics[valign=c , scale=0.18]{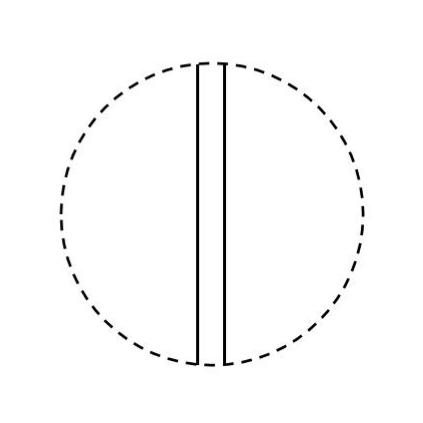}\label{eq:S5}
\end{equation}

\end{defn}

When the marking $\Gamma$ is understood from the context, we will
often write $\mathrm{Sk}(Y)$ instead of $\mathrm{Sk}(Y,\Gamma)$.

Let $\Sigma$ be a \emph{punctured bordered surface} (see \cite{CL}).
We can give the 3-manifold $\Sigma\times I$ a natural boundary marking.
Fix a marked point for each boundary edge of $\Sigma$ and let $M$
be the set of marked points, then $M\times I$ is a boundary marking
on $\Sigma\times I$, oriented by the positive direction of the interval
$I$. Note that there is exactly one marking edge in each component
of $\partial\Sigma\times I$. The \emph{stated skein algebra} $\mathrm{SkAlg}(\Sigma)$
is the $R$-module $\mathrm{Sk}(\Sigma\times I,M\times I)$ together with
a mulplication $\cup$ given by stacking in the $I$-direction, so $\alpha\cup\beta$ means that $\alpha$ has larger $t$ coordinates than $\beta$, where $t\in I$. We
shall often abbreviate $\alpha\cup\beta$ as $\alpha\beta$. It is
not hard to see that the notion of stated skein algebra of punctured
bordered surface given above agrees with the one defined in \cite{CL}
and \cite{GY1}.

An important class of punctured bordered surface is given by the $n$\emph{-gon}
$D_{n}$, which is a closed disc with $n$ puncture points on its boundary. Its skein algebra $\mathrm{SkAlg}(D_{n})$ gives a local
model for the skein module $\mathrm{Sk}(Y,\Gamma)$ of the 3-manifold
$Y$ near a vertex of $\Gamma$, and the latter can be equipped with
a module structure over $\mathrm{SkAlg}(D_{n})$. To describe this module
structure, let $v\in V(\Gamma)$ be a sink of degree $n$ 
and let $\alpha\in\mathrm{SkAlg}(D_{n})$
and $w\in\mathrm{Sk}(Y,\Gamma)$. We define $\alpha\cup w\in\mathrm{Sk}(Y,\Gamma)$,
abbreviated as $\alpha w$, in the following way: first we ``blow
up'' the vertex to an $n$-gon $D_{v}$, then we stack a $D_{n}\times I$
on top of it by identifying $D_{n}\times\{0\}$ with $D_{v}$ (in
a way so that markings ``connect''), this allows us to form the
union $\alpha\cup w$, then we ``blow down'' $D_{n}\times\{1\}$
to a vertex. We illustrate the construction in Figure \ref{fig:Illustration of module sturcture}. 
\begin{figure}[h]
  \centering
  \includegraphics[scale=0.2]{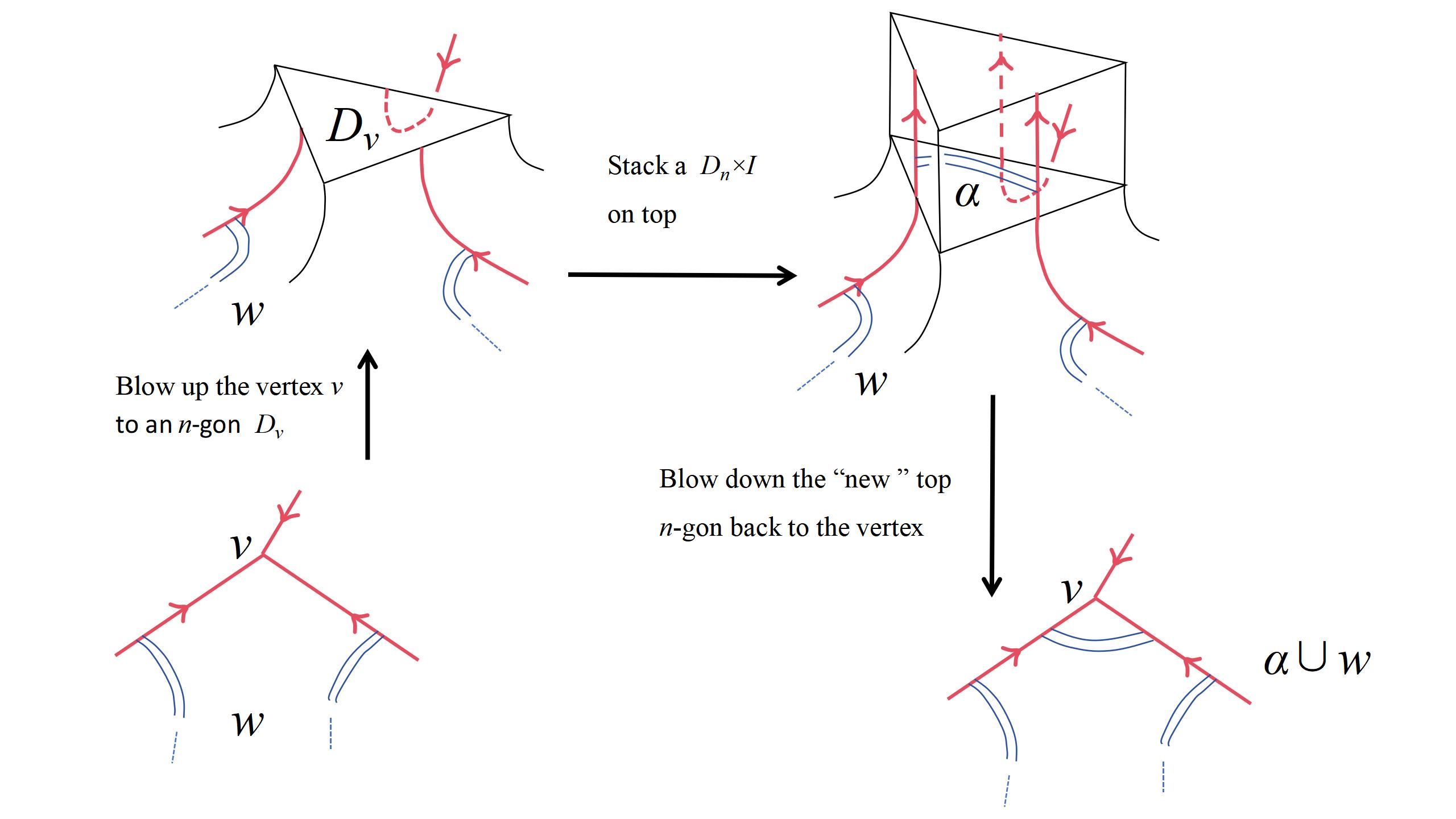} 
  \caption{An illustration of the left $\mathrm{SkAlg}(D_{n})$-module structure on $\mathrm{Sk}(Y,\Gamma)$ given by a sink of degree $n$}
  \label{fig:Illustration of module sturcture}
\end{figure}
This gives us a left $\mathrm{SkAlg}(D_{n})$-module structure on
$\mathrm{Sk}(Y,\Gamma)$. In a similar fashion, if $v\in V(\Gamma)$
is a source of degree $n$, then $\mathrm{Sk}(Y,\Gamma)$ can be equipped
with a right $\mathrm{SkAlg}(D_{n})$-module structure. Let $V^{-}(\Gamma)$
and $V^{+}(\Gamma)$ be the sets of sinks and sources respectively,
and let $D_{v}$ be the $n$-gon associated to the vertex $v$, then
we have equipped the skein module $\mathrm{Sk}(Y,\Gamma)$ with a 
\[
\bigotimes_{v\in V^{-}(\Gamma)}\mathrm{SkAlg}(D_{v})\text{-}\bigotimes_{v\in V^{+}(\Gamma)}\mathrm{SkAlg}(D_{v})
\]
bimodule structure.

\begin{rem}
\label{rem:remark on module structure notation}Let $\Sigma_{1}$
and $\Sigma_{2}$ be punctured bordered surfaces, there is a canonical
isomorphism of algebras $\mathrm{SkAlg}(\Sigma_{1})\otimes\mathrm{SkAlg}(\Sigma_{2})\cong\mathrm{SkAlg}(\Sigma_{1}\coprod\Sigma_{2})$
given by $\alpha_{1}\otimes\alpha_{2}\mapsto\alpha_{1}\cup\alpha_{2}$,
where $\alpha_{i}\in\mathrm{SkAlg}(\Sigma_{i})$. Moreover, in this
case we have $\alpha_{1}\cup\alpha_{2}=\alpha_{2}\cup\alpha_{1}$.
In light of this, whenever we have skeins coming from different copies
of punctured bordered surfaces, we will use $\otimes$-product between
them and $\cup$-product between them interchangeably, and both
are abbreviated as concatenation. In particular, this is applied to
the $n$-gons associated to vertices. For example, if $v,v^{\prime}\in V^{-}(\Gamma)$,
$\alpha,\beta\in\mathrm{SkAlg}(D_{v})$, $\gamma\in\mathrm{SkAlg}(D_{v^{\prime}})$
and $w\in\mathrm{Sk}(Y,\Gamma)$, then we can write either $((\alpha\cup\beta)\otimes\gamma)\cup w$
or $\alpha\cup\beta\cup\gamma\cup w$, and both are abbreviated as
$a\beta\gamma w$, which is also equal to $\gamma\alpha\beta w$. 
\end{rem}
Later we will introduce a different product structure on some skein algebras,
therefore to prevent confusion we fix a convention here: \underline{\emph{in expressions, concatenation
always means} $\cup$}.

In what follows, we will need a reduced version of skein modules (algebras).
Let $\Sigma$ be a punctured bordered surface. An element of $\mathrm{SkAlg}(\Sigma)$
is called a \emph{bad arc} if it is represented by an ``untwisted''
trivial ribbon arc connecting two adjacent marking edges with mixed
states on endpoints such that when the two markings edges are pointing
towards the readers, the endpoint with $-$ state is clockwise from the endpoint with $+$ state, see Figure \ref{fig:bad arc}.
\begin{figure}[h]
  \centering
  \includegraphics[scale=0.2]{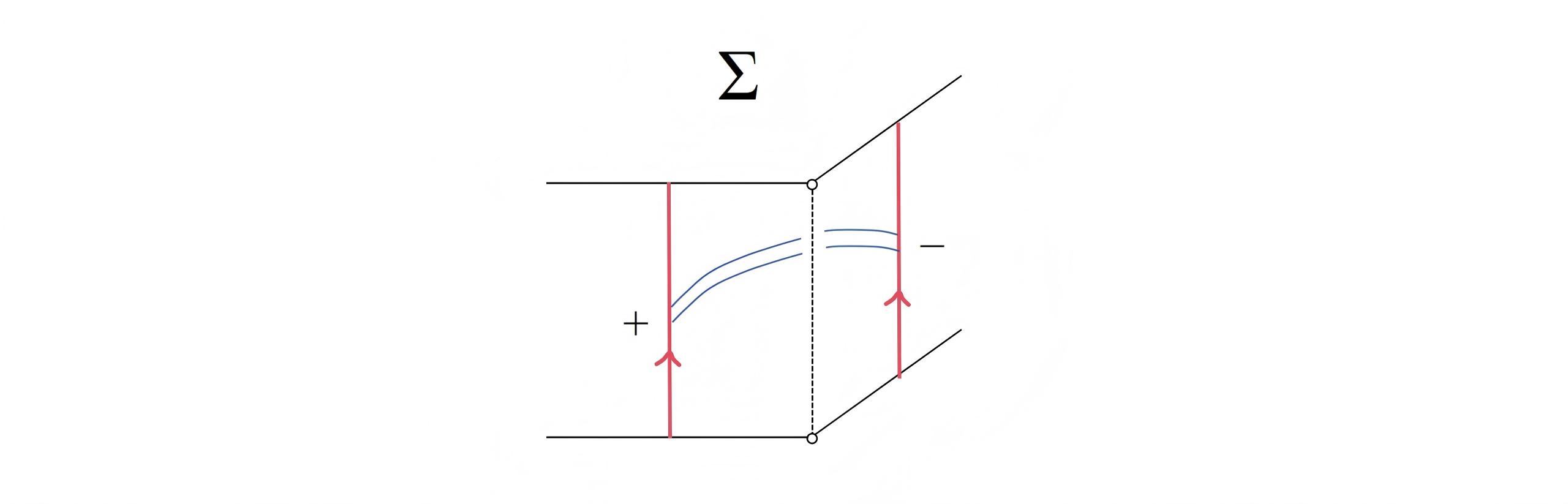} 
  \caption{}
  \label{fig:bad arc}
\end{figure}

It turns out that the left, the right and the two-sided ideal of $\mathrm{SkAlg}(\Sigma)$
generated by bad arcs in $\Sigma$ agree with each other (see \cite[Corollary 4.6]{LY}),
and we denote this ideal by $I^{b}$. The \emph{reduced skein algebra}
$\overline{\mathrm{SkAlg}}(\Sigma)$ is defined to be the quotient algebra
\[
\overline{\mathrm{SkAlg}}(\Sigma):=\mathrm{SkAlg}(\Sigma)
/I^{b}.
\]
If $(Y,\Gamma)$ is a boundary marked 3-manifold, we can also define
its reduced skein module $\overline{\mathrm{Sk}}(Y,\Gamma)$ (\cite[Definition 3.33]{PP1}).
Let $v\in V(\Gamma)$ be a vertex of degree $n$, let $D_{v}$ be
the $n$-gon associated to $v$. Let $I_{v}^{b}$ be the ideal of
$\mathrm{SkAlg}(D_{v})$ generated by the bad arcs, then $\overline{\mathrm{Sk}}(Y,\Gamma)$
is defined to be the $R$-module quotient
\[
\overline{\mathrm{Sk}}(Y,\Gamma):=\frac{\mathrm{Sk}(Y,\Gamma)}{\sum_{v\in V^{-}(\Gamma)}I_{v}^{b}\mathrm{Sk}(Y,\Gamma)+\sum_{v\in V^{+}(\Gamma)}\mathrm{Sk}(Y,\Gamma)I_{v}^{b}}.
\]
That is, $\overline{\mathrm{Sk}}(Y,\Gamma)$ is defined by setting elements
of the form $bw$ to $0$, where $b$ is a bad arc near a sink and
$w\in\mathrm{Sk}(Y,\Gamma)$ is arbitrary; similarly we set elements
of the form $wb$ to $0$ where $b$ is a bad arc near a source. Note
that the 
\[
\bigotimes_{v\in V^{-}(\Gamma)}\mathrm{SkAlg}(D_{v})\text{-}\bigotimes_{v\in V^{+}(\Gamma)}\mathrm{SkAlg}(D_{v})
\]
bimodule structure on $\mathrm{Sk}(Y,\Gamma)$ naturally induces a
\[
\bigotimes_{v\in V^{-}(\Gamma)}\overline{\mathrm{SkAlg}}(D_{v})\text{-}\bigotimes_{v\in V^{+}(\Gamma)}\overline{\mathrm{SkAlg}}(D_{v})
\]
bimodule structure on the reduced skein module $\overline{\mathrm{Sk}}(Y,\Gamma)$.

The next result we need from \cite{PP1} is the key structural theorem giving an algebraic presentation of
the skein module of a boundary marked 3-ball, which serves as the
foundation of our work. Before stating the theorem, we need to place
some restriction on the boundary marking. 

\begin{defn}
\label{def:combinatorial folation}
(i) An \emph{elementary quadrilateral} $Q$ is
a quadrilateral with two opposite vertices punctured and an oriented edge connecting the two unpunctured vertices, called the \emph{marking
edge} of $Q$.
\\
(ii) Let $\Sigma$ be a punctured bordered
surface (possibly with empty boundary). A \emph{combinatorial foliation}
of $\Sigma$ is a decomposition of $\Sigma$ into elementary
quadrilaterals.
\\
(iii) The boundary marking $\Gamma$ of the 3-manifold $Y$ is said to be
\emph{admissible} if there is a combinatorial foliation of $\partial Y$
such that the set of marking edges of elementary quadrilaterals is exactly
the set $E(\Gamma)$ of edges of the boundary marking.
\end{defn}
The reason such a structure is called a foliation is that in
each elementary quadrilateral there is a natural canonical 
foliation, which is shown
on the left of Figure \ref{fig:foliations on elementary Q and face}. The \emph{generic} leaves are given by ideal arcs connecting the two
punctured vertices, and there are also four \emph{singular} leaves corresponding
to the four edges of $Q$.
\begin{figure}[!ht]
  \includegraphics[scale=0.15]{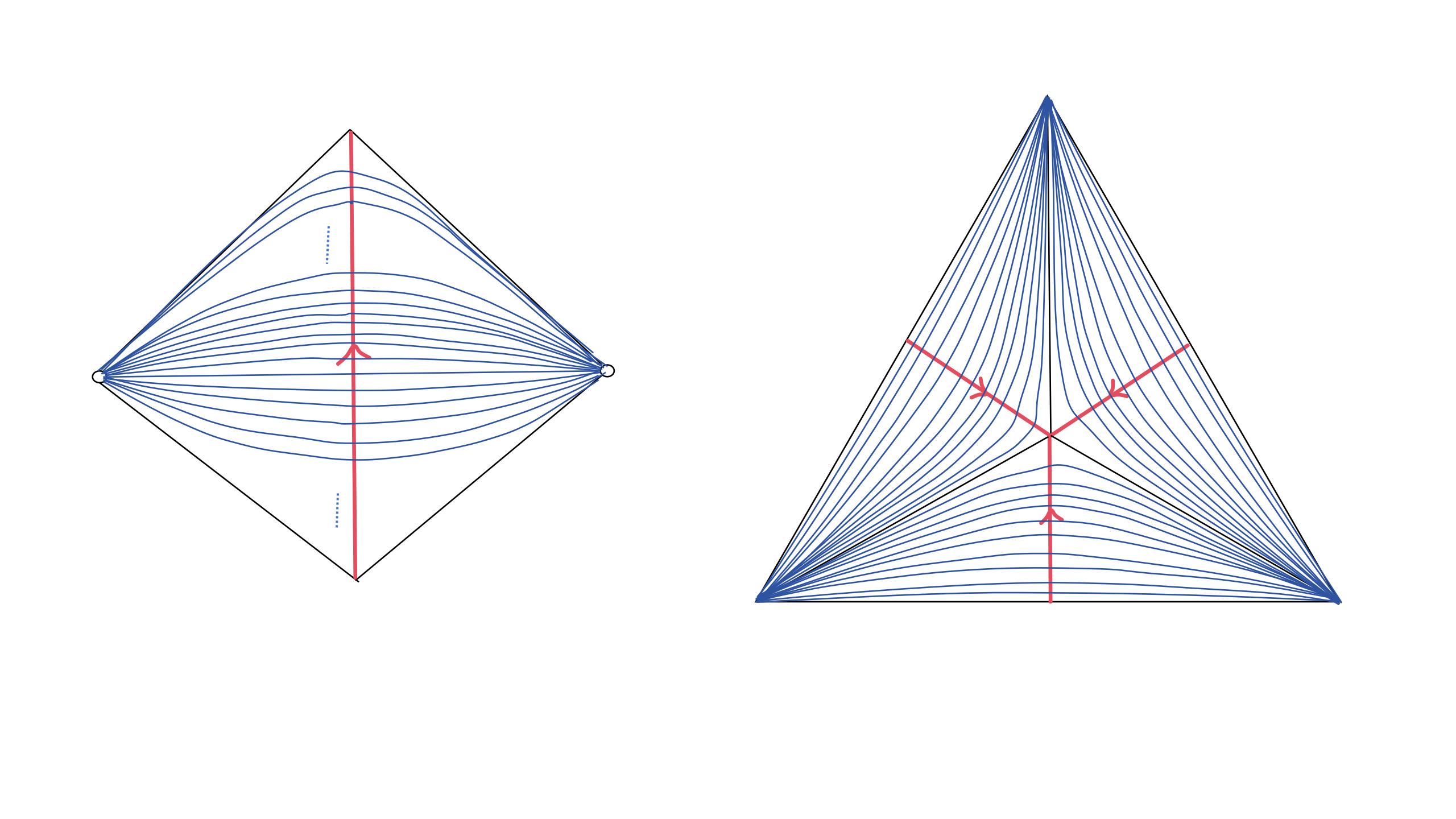} 
  \caption{Left: the canonical foliation on an elementary quadrilateral. Right: an ideal triangle decomposed into three elementary quadrilaterals and the induced foliation}
  \label{fig:foliations on elementary Q and face}
\end{figure}

The prototypical example of a 3-manifold with admissible boundary marking
is given by an ideal tetrahedron $T$ with the canonical
boundary marking shown in Figure \ref{fig:ideal T with bm only}. Note that
each face of $T$ can be decomposed into three elementary quadrilaterals,
with marking edges of the quadrilateral being the canonical marking edges in the boundary marking of the tetrahedron (see the picture on the
right of Figure \ref{fig:foliations on elementary Q and face}). Other
examples of 3-manifold with admissible boundary marking that will be
important to us later are the face cones and the face suspensions along with their canonical boundary markings
(see Figure \ref{fig:face cone with bm} and \ref{fig:face suspension with bm}).
\begin{figure}[!ht]
  \centering
  \includegraphics[scale=0.15]{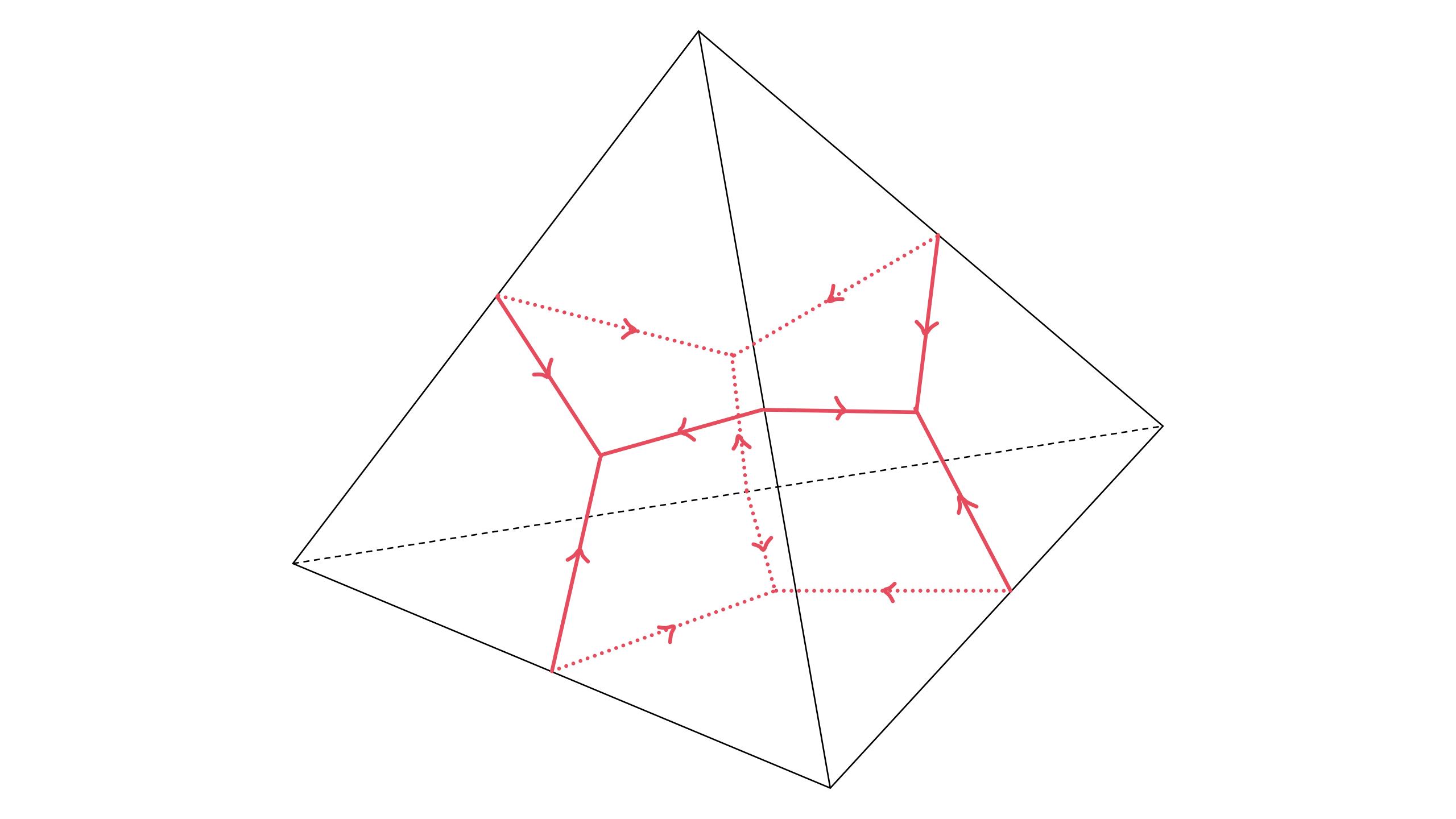} 
  \caption{Ideal tetrahedron with the canonical boundary marking}
  \label{fig:ideal T with bm only}
\end{figure}

Now we can state the desired structural theorem describing the skein modules of
$3$-balls. We present here only the version for reduced skein modules, which is sufficient for our purposes.

\begingroup
\allowdisplaybreaks
\begin{thm}
\label{thm:skein module of 3-ball}
\emph{(\cite[Corollary4.6]{PP1})
}Let $B$ be a 3-ball with an admissible boundary marking $\Gamma$.
As a 
\[
\bigotimes_{v\in V^{-}(\Gamma)}\overline{\mathrm{SkAlg}}(D_{v})\text{-}\bigotimes_{w\in V^{+}(\Gamma)}\overline{\mathrm{SkAlg}}(D_{w})
\]
-bimodule, $\overline{\mathrm{Sk}}(B,\Gamma)$ is cyclic\footnote{Let $R$ and $S$ be rings, an $R$-$S$-bimodule
$M$ is said to be \emph{cyclic} and generated by $m\in M$ if the association $r\otimes s\mapsto rms$ defines a surjective $R$-$S$-bimodule homomorphism $R\otimes S\twoheadrightarrow M$. In this case, the kernel of this surjective homomorphism is the $R$-$S$-sub-bimodule of $R\otimes S$ given explicitly by $\mathrm{Ann}(m)=\{\sum r_{i}\otimes s_{i}\bigm| \sum r_{i}ms_{i}=0 \}$; moreover, $M$ is canonically isomorphic to the $R$-$S$-bimodule quotient $R\otimes S\bigm/\mathrm{Ann}(m)$.} and generated by
the empty skein $[\emptyset]$. Therefore $\overline{\mathrm{Sk}}(B,\Gamma)$
has the presentation
\[
\bigotimes_{v\in V^{-}(\Gamma)}\overline{\mathrm{SkAlg}}(D_{v})\otimes\bigotimes_{w\in V^{+}(\Gamma)}\overline{\mathrm{SkAlg}}(D_{w})\biggm/\mathrm{Ann}(\emptyset),
\]
where $\mathrm{Ann}([\emptyset])$, as a $\bigotimes_{v\in V^{-}(\Gamma)}\overline{\mathrm{SkAlg}}(D_{v})\text{-}\bigotimes_{w\in V^{+}(\Gamma)}\overline{\mathrm{SkAlg}}(D_{w})$-sub-bimodule of
\[
\bigotimes_{v\in V^{-}(\Gamma)}\overline{\mathrm{SkAlg}}(D_{v})\otimes\bigotimes_{w\in V^{+}(\Gamma)}\overline{\mathrm{SkAlg}}(D_{w}),
\]
is generated by the following elements. Each puncture of $\partial B$
gives us a pair of elements:\footnote{Note that because the boundary marking is admissible, each puncture
of $\partial B$ is the center of a $2n$-gon whose boundary is formed
by marking edges, with alternating orientions. The pictures only show the case
$n=3$.}

\begin{equation*}
\includegraphics[valign=c , scale=0.22]{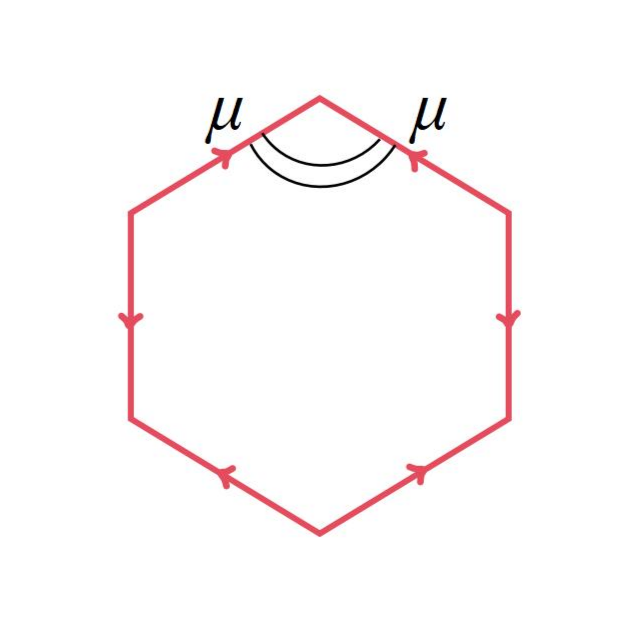}-(-A^{2})^{-\mu(n-1)}\includegraphics[valign=c , scale=0.22]{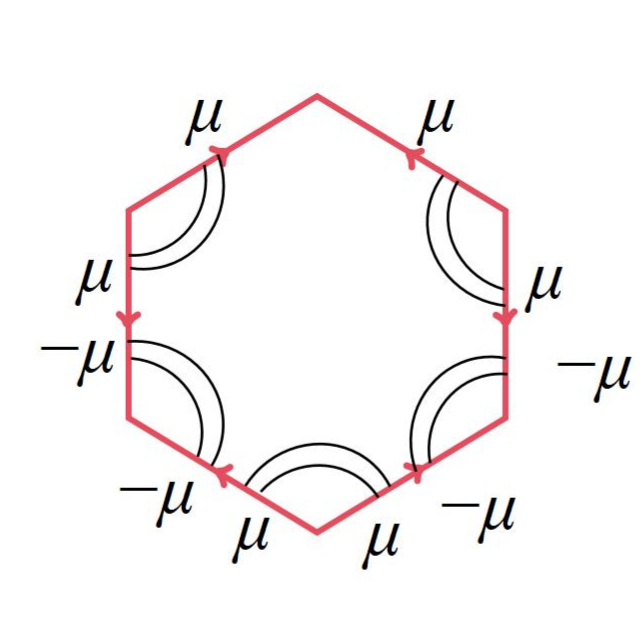}
\end{equation*}
and
\begin{multline*}
\includegraphics[valign=c , scale=0.22]{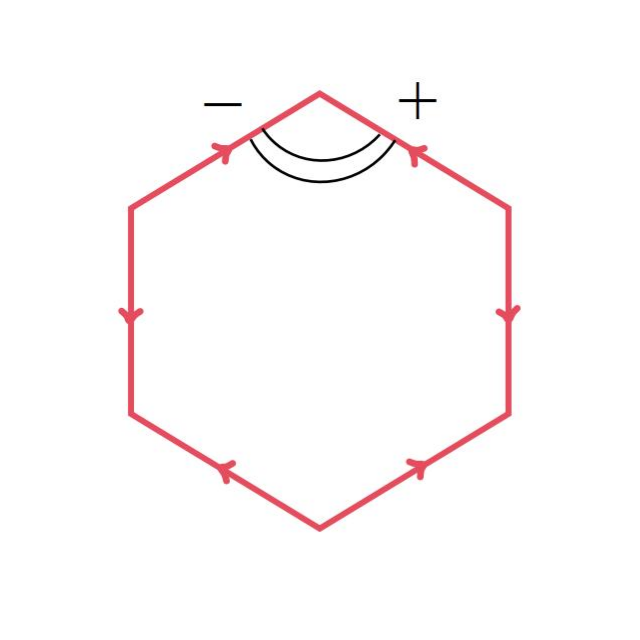}-(-A^{2})^{-n+1}\includegraphics[valign=c , scale=0.22]{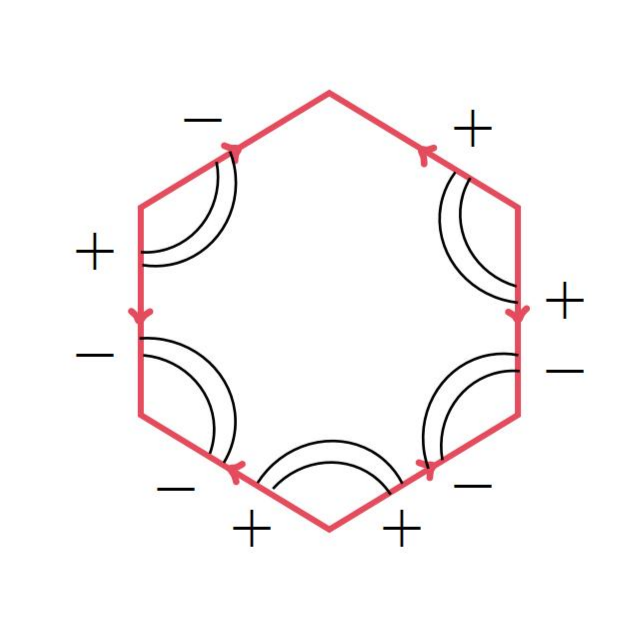}\\
+(-A^{2})^{-n+2}\includegraphics[valign=c , scale=0.22]{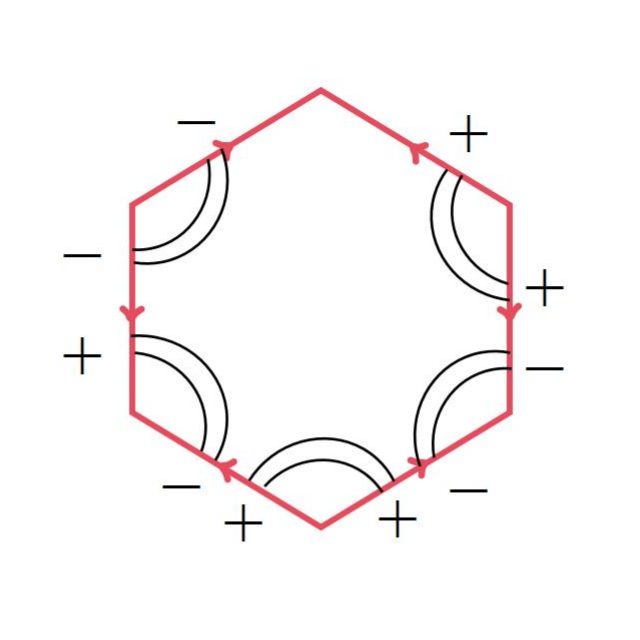}
+\dots +(-A^{2})^{n-1}\includegraphics[valign=c , scale=0.22]{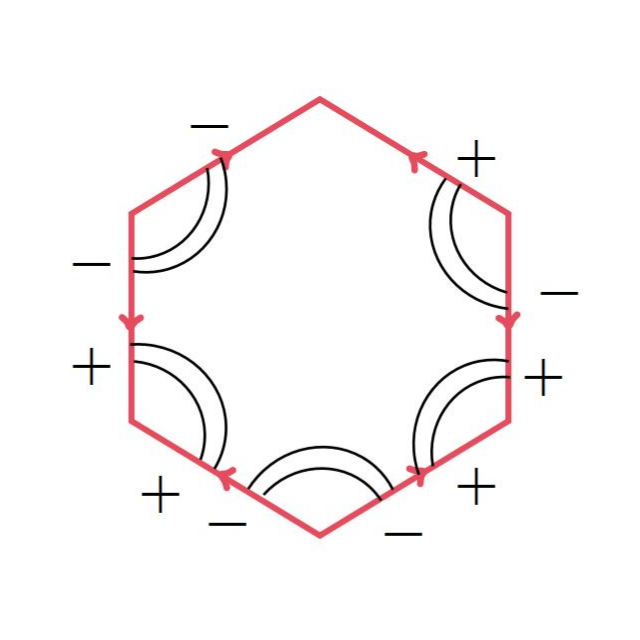}\ .
\end{multline*}
\end{thm}
\endgroup
The set of generators of $\mathrm{Ann}([\emptyset])$ desribed above
is not unique. Indeed we can replace any of the elements by a new element obtained by a rotation of all the diagrams
in the old element by an integer multiple of $2\pi/n$.

Throughout we will also need the notion of \emph{quantum torus} and \emph{Weyl-ordered product}. 

\begin{defn}
\label{def: quantum torus and Weyl-ordering}Let $x_{1},\dots,x_{n}$
be a collection of formal variables and $\Lambda$ an $n\times n$
\emph{skew-symmetric} matrix with integer coefficents. The quantum
torus $\mathbb{T}_{\Lambda}(x_{1},\dots,x_{n})$ is the noncommutative
$R$-algebra defined as the quotient
\[
\mathbb{T}_{\Lambda}(x_{1},\dots,x_{n})=\frac{R\langle x_{1}^{\pm1},\dots,x_{n}^{\pm1}\rangle}{\langle x_{i}x_{j}=A^{\Lambda_{ij}}x_{j}x_{i}\rangle},
\]
where $R\langle x_{1}^{\pm1},\dots,x_{n}^{\pm1}\rangle$ is the \emph{free}
laurent $R$-algebra generated by 
the formal variables and $\langle x_{i}x_{j}=A^{\Lambda_{ij}}x_{j}x_{i}\rangle$
is the two-sided ideal generated by $x_{i}x_{j}-A^{\Lambda_{ij}}x_{j}x_{i}$.
The \emph{Weyl-ordered} product $\left[x_{i_{1}}x_{i_{2}}\dots x_{i_{k}}\right]$
is defined to be
\[
\left[x_{i_{1}}x_{i_{2}}\dots x_{i_{k}}\right]=A^{-\frac{1}{2}\sum_{1\leq p<q\leq k}\Lambda_{i_{p}i_{q}}}x_{i_{1}}x_{i_{2}}\dots x_{i_{k}}.
\]
Note that the Weyl-ordered product is independent of the ordering of the
variables inside the square brackets.

\end{defn}

\subsection{Ideally triangulated boundary marked 3-manifolds\label{subsec:Ideally triangulated bm 3-manifolds}}

In this work, we shall focus on a special type of boundary marked
3-manifold. Let's fix some notations first:
\begin{itemize}
\item Let $T$ be a single ideal tetrahedron, we denote by ${\bf v}(T)$,
$\mathbf{e}(T)$ and $\mathbf{f}(T)$ be the set of (ideal) vertices,
edges and faces of $T$ respectively.
\item Let $\mathcal{T}=\{T_{1},\dots,T_{n}\}$ be an ideal triangulation
of the $3$-manifold $Y$. It gives us a set $\mathcal{T}^{(1)}$
of \emph{edges} and a set $\mathcal{T}^{(2)}$ of \emph{faces} in
$(Y,\mathcal{T})$: an edge (resp. a face) of $(Y,\mathcal{T})$ is
the image under the quotient map $T_{1}\coprod\dots\coprod T_{n}\twoheadrightarrow Y$
of an edge (resp. a face) of some ideal tetrahedron $T_{i}$.
\item In the setting of an ideal triangulation $\mathcal{T}=\{T_{1},\dots,T_{n}\}$, to
avoid confusion, elements of $\mathbf{v}(T_{i})$, $\mathbf{e}(T_{i})$
and $\mathbf{f}(T_{i})$ will then be called \emph{bare vertices},
\emph{bare edges} and\emph{ bare faces}, respectively, of $T_{i}$, following the language of \cite{PP1}.
If the discussion is only about a single ideal tetrahedron $T$, we
can safely drop the adjective ``bare'' and just say vertices, edges
and faces of $T$.
\end{itemize}

\begin{defn}
\label{def:ideally triangulated boundary marked 3-manifolds}
By an \emph{ideally triangulated boundary marked 3-manifold} we mean
a pair $(Y,\mathcal{T})$ where $Y$ is a boundary marked 3-manifold and
$\mathcal{T}$ is an ideal triangulation of $Y$ satisfying the following
conditions:

\begin{enumerate}
\item Each component of the boundary $\partial Y$ is equipped with an ideal
triangulation induced from $\mathcal{T}$; namely, we require that
$\partial Y$ be triangulated by ideal triangles which are faces of
ideal tetrahedra in $\mathcal{T}$. Triangles or edges in this induced
triangulation on $\partial Y$ will be referred to as \emph{boundary
faces }or \emph{boundary edge}s respectively; whereas triangles and
edges of $\mathcal{T}$ which are not in the boundary will be referred
to as \emph{internal faces} and \emph{internal edges} respectively.
\item The boundary marking on $\partial Y$ is \emph{canonical}, meaning
that in
every boundary face $f$ there are exactly three marking edges. These edges have
starting-points the barycenters of the three edges of $f$, which are degree 2 sources, and end-points
the barycenter of $f$, which is a degree 3 sink.
\end{enumerate}
\end{defn}

A single ideal tetrahedron $T$ together with its canonical boundary
marking (Figure \ref{fig:ideal T with bm only}) is an example of an ideally triangulated
boundary marked 3-manifold.

Now let $(Y,\mathcal{T})$ be an ideally triangulated boundary marked
3-manifold, and let $\mathcal{F}$ be the induced triangulation on $\partial Y$.
Let $\mathcal{F}^{(1)}$ be the set of boundary edges and $\mathcal{F}^{(2)}$
be the set of boundary faces. By the earlier discussion we know that
the skein module $\overline{\mathrm{Sk}}(Y)$ is a 
\[
\underset{f\in\mathcal{F}^{(2)}}{\bigotimes}\overline{\mathrm{SkAlg}}(D_{3})\text{-}\underset{e\in\mathcal{F}^{(1)}}{\bigotimes}\overline{\mathrm{SkAlg}}(D_{2})
\]
bimodule. It turns out that both $\overline{\mathrm{SkAlg}}(D_{3})$
and $\overline{\mathrm{SkAlg}}(D_{2})$ have simple explicit algebraic
descriptions. Let us label the three vertices of the triangle $D_{3}$
by $\alpha$, $\beta$ and $\gamma$ and the three opposite edges
by $a$, $b$ and $c$ respectively as in Figure \ref{fig:triangle and bigon}(1).
By an abuse of notation we also denote by $\alpha$ the element of $\overline{\mathrm{SkAlg}}(D_{3})$
represented by an ``untwisted'' trivial ribbon arc connecting the
two marking edges inside the boundary faces $b\times I$ and $c\times I$
of $D_{3}\times I$ with both endpoints in $+$ states, and by $\alpha^{-1}$
the same ribbon arc with both endpoints in $-$ states; in the same
fashion we also have elements $\beta^{\pm1}$ and $\gamma^{\pm1}$
in $\overline{\mathrm{SkAlg}}(D_{3})$.

For the bigon $D_{2}$ we denote by $x$ the ``untwisted'' trivial
ribbon arc connecting the only two marking edges with both endpoints
in $+$ states, and by $x^{-1}$ the same arc with both endpoints
in $-$ states, as in Figure \ref{fig:triangle and bigon}(2).

\begin{figure}[h]
  \includegraphics[scale=0.2]{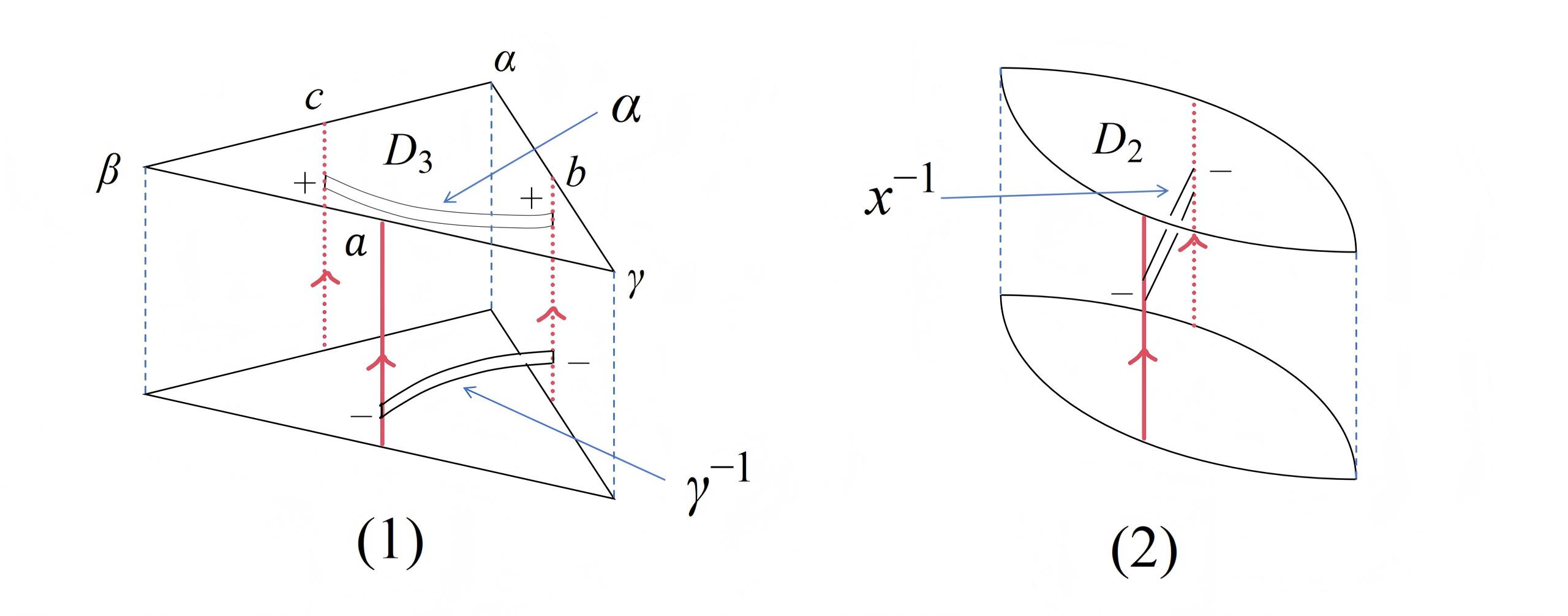} 
  \caption{}
  \label{fig:triangle and bigon}
\end{figure}

We have the following
\begin{thm}
\emph{\label{thm:skein algebras of triangles and bigons}(}\cite{BW},
\cite{CL}\emph{)} With the above notations, the elements $\alpha^{-1}$,
$\beta^{-1}$ and $\gamma^{-1}$ are multiplicative inverses of $\alpha$,
$\beta$ and $\gamma$ respectively in the algebra $\overline{\mathrm{SkAlg}}(D_{3})$,
which is given by a quantum torus: 
\begin{equation}
\overline{\mathrm{SkAlg}}(D_{3})\cong\mathbb{T}:=\frac{R\langle\alpha^{\pm1},\beta^{\pm1},\gamma^{\pm1}\rangle}{\langle\beta\alpha=A\alpha\beta,\ \gamma\beta=A\beta\gamma,\ \alpha\gamma=A\gamma\alpha\rangle}.\label{eq:SkAld(D_3)}
\end{equation}
The element $x^{-1}$ is the multiplicative inverse of $x$ in the
algebra $\overline{\mathrm{SkAlg}}(D_{2})$, which is given by the Laurent
polynomial ring:
\begin{equation}
\overline{\mathrm{SkAlg}}(D_{2})\cong\mathbb{B}:=R[x,x^{-1}].\label{eq:SkAlg(D_2)}
\end{equation}
\end{thm}

\begin{rem}
\label{rem:mixed state element in terms of generators}
In $D_{3}\times I$, one of the two ``untwisted'' trivial ribbon arcs  with mixed states on its endpoints connecting two marking edges is a bad arc and thus becomes $0$ in the reduced skein algebra $\overline{\mathrm{SkAlg}}(D_{3})$ ; the other is a nonzero element in $\overline{\mathrm{SkAlg}}(D_{3})$. Taking the ribbon arc connecting the marking edges in $b\times I$ and $c\times I$ for instance, with $+$ state on the endpoint on the marking edge in $c\times I$ and $-$ state on the endpoint on marking edge in $b\times I$, a simple calculation using the skein relations shows that this element is given by $\left[ \beta \gamma ^{-1}\right]$.
\end{rem}

We will frequently refer to these skein algebras associated to
various boundary faces or triangles, so let us fix some notations.
\begin{defn}
\label{def:boundary face algebra, boundary edge algebra}
Let $(Y,\mathcal{T})$ be an ideally triangulated boundary marked
3-manifold and let $\mathcal{F}$ be the induced ideal triangulation on $\partial Y$.
\begin{enumerate}
    \item The triangle algebra asscociated to the sink at the barycenter of
a boundary face $f\in\mathcal{F}^{(2)}$ will be denoted by $\mathbb{T}_{f}$ and will be called the \emph{boundary face algebra} associated to $f$.
    \item The bigon algebra associated to the bivalent source at the barycenter
of a boundary edge $e\in\mathcal{F}^{(1)}$ will be denoted by $\mathbb{B}_{e}$ and will be called the \emph{boundary edge algebra} associated to $e$. 
\end{enumerate}
\end{defn}

In these notations, $\overline{\mathrm{Sk}}(Y)$ is a 
$
\underset{f\in\mathcal{F}^{(2)}}{\bigotimes}\mathbb{T}_{f}\text{-}\underset{e\in\mathcal{F}^{(1)}}{\bigotimes}\mathbb{B}_{e}$
bimodule.

\subsection{Reduced skein module of an ideal tetrahedron\label{subsec:Reduced skein module of T}}

We consider the ideal tetrahedron $T$ equipped with the canonical
boundary marking shown in Figure \ref{fig:ideal T with bm only}.
In this subsection we will give an explicit algebraic presentation
of the skein module $\overline{\mathrm{Sk}}(T)$. By Theorem \ref{thm:skein module of 3-ball}
and Theorem \ref{thm:skein algebras of triangles and bigons} we know
that $\overline{\mathrm{Sk}}(T)$ is a cyclic $\underset{f\in\mathbf{f}(T)}{\bigotimes}\mathbb{T}_{f}\text{-}\underset{e\in\mathbf{e}(T)}{\bigotimes}\mathbb{B}_{e}=\mathbb{T}^{\otimes4}\text{-}\mathbb{B}^{\otimes6}$-bimodule
generated by the empty skein $[\emptyset]$. We will write the generators
of $\mathrm{Ann}([\emptyset])$ provided by Theorem \ref{thm:skein module of 3-ball}
as algebraic expressions instead of skein diagrams. For this purpose
we need to introduce notation for the generators of the various skein algebras
involved. Recall that the edges of the tetrahedron $T$ are labeled
by the shape parameters\footnote{We are dealing with a single ideal tetrahedron, so we can drop the subcript $T$ from the shape parameters} $z,z^{\prime},z^{\prime\prime},y,y^{\prime},y^{\prime\prime}$.
We will also label a marking edge by $z^{\boxempty}$ (or $y^{\boxempty}$)
if its starting-point is in the edge labeled $z^{\boxempty}$ (resp.
$y^{\boxempty}$). Note that in this way two marking edges with the
same starting-points will have the same labeling but they have different
endpoints and are in different faces of $T$, see Figure \ref{fig:ideal T with bm and labeling of generators}.

\begin{figure}[h]
  \includegraphics[scale=0.2]{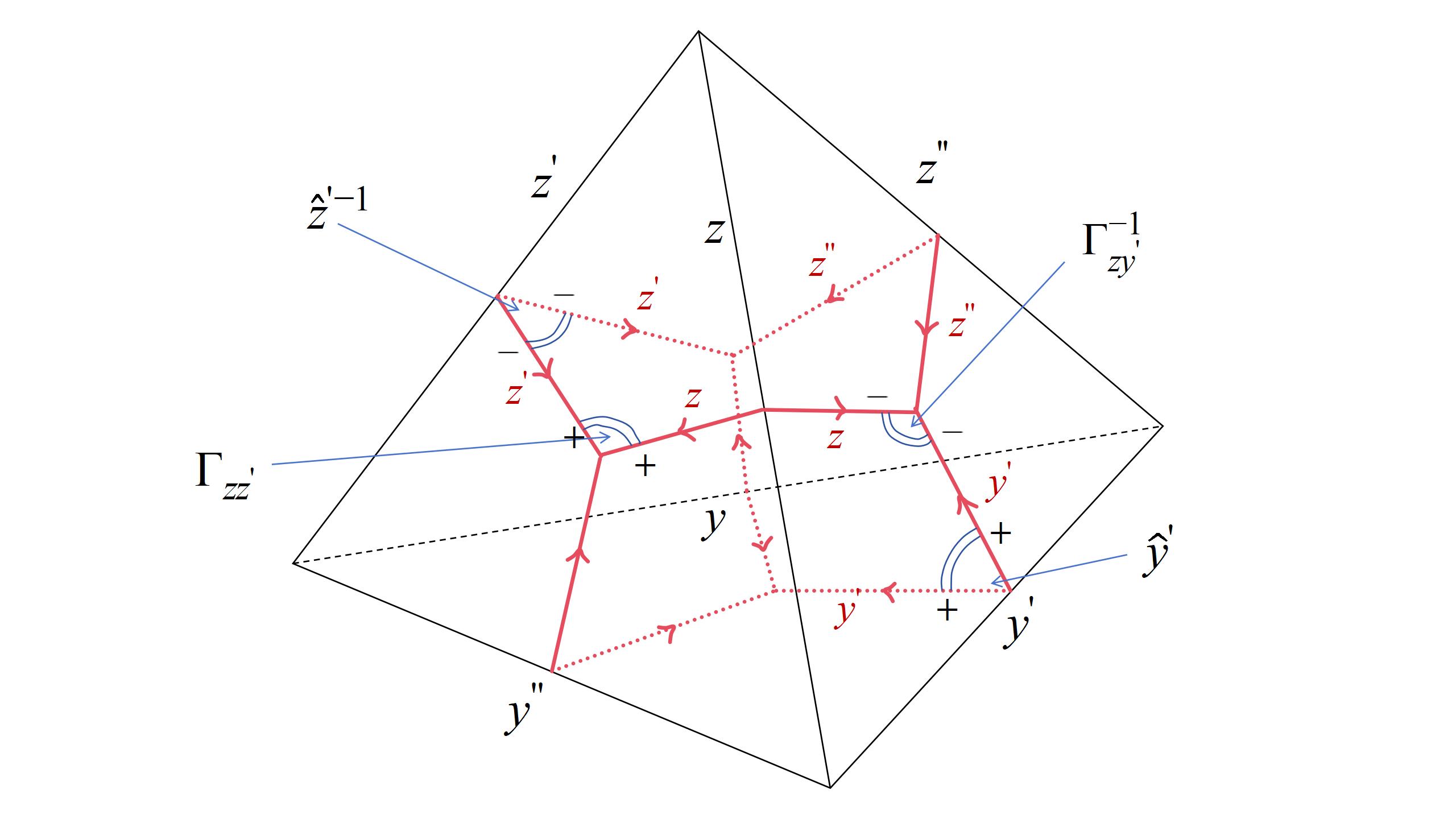} 
  \caption{Labeling of generators of skein algebras associated to vertices}
  \label{fig:ideal T with bm and labeling of generators}
\end{figure}

We adopt the following naming convention on the generators of the various boundary face and boundary edge algebras:

By Theorem \ref{thm:skein algebras of triangles and bigons}, each
boundary face algebra $\mathbb{T}_f$ is generated by
``untwisted'' trivial ribbon arc near the sink at the barycenter of 
$f$ connecting two neighboring marking edges converging at the sink.
We denote the arc which connects the two marking edges labeled $a$
and $b$ and are in $++$ states by $\Gamma_{ab}$, $a,b\in\{z,z^{\prime},z^{\prime\prime},y,y^{\prime},y^{\prime\prime}\}$.
For example, with such labeling of generators, the triangle algebra
$\mathbb{T}_f$ where $f$ is the face of $T$ bounded by edges
labeled $z,z^{\prime},y^{\prime\prime}$ is given by
\[
\frac{R\langle\Gamma_{zz^{\prime}}^{\pm1},\Gamma_{z^{\prime}y^{\prime\prime}}^{\pm1},\Gamma_{y^{\prime\prime}z}^{\pm1}\rangle}{\langle\Gamma_{z^{\prime}y^{\prime\prime}}\Gamma_{zz^{\prime}}=A\Gamma_{zz^{\prime}}\Gamma_{z^{\prime}y^{\prime\prime}},\ \Gamma_{y^{\prime\prime}z}\Gamma_{z^{\prime}y^{\prime\prime}}=A\Gamma_{z^{\prime}y^{\prime\prime}}\Gamma_{y^{\prime\prime}z},\ \Gamma_{zz^{\prime}}\Gamma_{y^{\prime\prime}z}=A\Gamma_{y^{\prime\prime}z}\Gamma_{zz^{\prime}}\rangle}.
\]

Similarly, each boundary edge algebra $\mathbb{B}_e$ 
is generated by ``untwisted'' trivial ribbon arc near the source at the barycenter of $e$ connecting the two marking edges emanating from the source.
If the boundary edge $e$ is labeled $a$, $a\in\{z,z^{\prime},z^{\prime\prime},y,y^{\prime},y^{\prime\prime}\}$, then the generator of $\mathbb{B}_e$ in $++$ states will be denoted
$\hat{a}$ and thus $\mathbb{B}_e=R[\hat{a},\hat{a}^{-1}]$ and
\[
\underset{e\in\mathbf{e}(T)}{\bigotimes}\mathbb{B}_{e}=R[{\hat{z}}^{\pm1},\hat{z}^{\prime\pm1},\hat{z}^{\prime\prime\pm1},\hat{y}^{\pm1},\hat{y}^{\prime\pm1},\hat{y}^{\prime\prime\pm1}].
\]
\\This labeling of generators is illustrated in Figure
\ref{fig:ideal T with bm and labeling of generators}.

\begin{rem}
In the notation $\Gamma_{ab}$, the ordering of $a$ and $b$ is immaterial. $\Gamma_{ab}$ is the same as $\Gamma_{ba}$ and both denote the ``untwisted'' trivial ribbon arc near the sink at the barycenter of a bare face connecting two neighboring marking egdes labeled $a$ and $b$, with both endpoints in $+$ states.
\end{rem}

Now we can write down the generators of $\mathrm{Ann}([\emptyset])$
in terms of the generators $\Gamma_{ab}$, $\hat{z}^{\boxempty}$ and $\hat{y}^{\boxempty}$.
\begin{cor}
\label{cor:reduced skein module of T}
The reduced skein module $\overline{\mathrm{Sk}}(T)$ of a single tetrahedron
$T$ is given by 
\[
\frac{\mathbb{T}^{\otimes4}\otimes\mathbb{B}^{\otimes6}}{\mathrm{Ann}(\emptyset)},
\]
where $\mathrm{Ann}(\emptyset)$ is generated as a $\mathbb{T}^{\otimes4}\otimes\mathbb{B}^{\otimes6}$-sub-bimodule of $\mathbb{T}^{\otimes4}\otimes\mathbb{B}^{\otimes6}$ by
four pairs of elements, one for each vertex of $T$. For simplicity
we only write down the pair of generators corresponding to the vertex
incident to edges labeled $z,z^{\prime},z^{\prime\prime}$:
\begin{eqnarray*}
v_{zz^{\prime}z^{\prime\prime}}& :=& (-A^{2})^{2}\Gamma_{zz^{\prime\prime}}\Gamma_{z^{\prime\prime}z^{\prime}}\Gamma_{z^{\prime}z}-\hat{z}\hat{z}^{\prime}\hat{z}^{\prime\prime} \\
\ell_{zz^{\prime\prime}} &:=&[\Gamma_{zy^{\prime}}\Gamma_{y^{\prime}z^{\prime\prime}}^{-1}]-(-A^{2})^{-1}[\Gamma_{z^{\prime\prime}y}\Gamma_{yz^{\prime}}^{-1}]\Gamma_{zz^{\prime}}^{-1}\hat{z}^{\prime\prime-1}\hat{z}^{\prime}\hat{z}-(-A^{2})\Gamma_{z^{\prime\prime}z^{\prime}}[\Gamma_{z^{\prime}y^{\prime\prime}}\Gamma_{y^{\prime\prime}z}^{-1}]\hat{z}^{\prime\prime-1}\hat{z}^{\prime-1}\hat{z}
\end{eqnarray*}

In a similar fashion one can write down the other three pairs of generators.
\end{cor}

\begin{rem}
\label{rem:freedom in the choice of ell}
Instead of $\ell_{zz^{\prime\prime}}$, we can also use one of the
two other elements $\ell_{z^{\prime\prime}z^{\prime}}$ and $\ell_{z^{\prime}z}$
in the form of $\ell_{zz^{\prime\prime}}$. They are the elements
obtained by rotate the diagrams in $\ell_{zz^{\prime\prime}}$ by
$120^{\circ}$ or $240^{\circ}$. As we shall see, this eventually corresponds to the freedom of choosing (quantum) lagrangian relations in the quantum gluing module.
\end{rem}

\subsection{Corner-reduction\label{subsec:Corner reduction}}

Motivated by the corner-reduction construction of Garoufalidis \& Yu,
we consider a further reduction of the reduced skein module $\overline{\mathrm{Sk}}(Y)$
of an ideally triangulated boundary marked 3-manifold $(Y,\mathcal{T})$.
For this purpose we need to introduce a twist of the product structures
on the various boundary face or edge algebras. Let
$\mathcal{F}$ be the induced ideal triangulation of the boundary and $f\in\mathcal{F}^{(2)}$
a boundary face of $Y$, and denote by $M_{f}=\{e_{1},e_{2},e_{3}\}$
the set of three marking edges in $f$. 
We recall that the skein module
$\overline{\mathrm{Sk}}(Y)$ is $\mathbb{Z}^{M_{f}}$-graded (that is,
if $t\in\overline{\mathrm{Sk}}(Y)$ is a stated tangle, $d_{f}(t)\in\mathbb{Z}^{M_{f}}$
assigns the marking edge $e_{i}$ the sum of states of endpoints of
$t$ on $e_{i}$). This will induce a $\mathbb{Z}^{M_{f}}$-grading
on the algebra $\mathbb{T}_{f}$ via the natural maps 
\[
\mathbb{T}_{f}\rightarrow\overline{\mathrm{Sk}}(Y),\ \kappa\mapsto\kappa[\emptyset];
\]
similarly if $e\in\mathcal{F}^{(1)}$ is a boundary edge, then via
the natural map 
\[
\mathbb{B}_{e}\rightarrow\overline{\mathrm{Sk}}(Y),\ x\mapsto[\emptyset]x,
\]
we can equip $\mathbb{B}_{e}$ with a $\mathbb{Z}^{M_{f}}$-grading.
\begin{rem}
In fact we can give both $\mathbb{T}_{f}$ and $\mathbb{B}_{e}$ a
$\prod_{f\in\mathcal{F}^{(2)}}\mathbb{Z}^{M_{f}}$-grading in an obvious
fashion: for $\mathbb{T}_{f}$, we define its $\mathbb{Z}^{M_{f^{\prime}}}$-degree
to be zero if $f^{\prime}\neq f$; for $\mathbb{B}_{e}$, we define
its $\mathbb{Z}^{M_{f^{\prime}}}$-degree to be $0$ if $f^{\prime}$
is a boundary face not adjacent to the boundary edge $e$.
\end{rem}

We now put a skew-symmetric bilinear form on $\mathbb{Z}^{M_{f}}$.
Suppose $(e_{1},e_{2},e_{3})$ is a \emph{clockwise} ordering of the
three marking edges in $M_{f}$, clockwise with respect to an outward pointing normal vector. More specifically, it means that the outward point normal vector is determined by the ``left hand rule'' by the ordered triple $(e_{1},e_{2},e_{3})$. This clockwise labeling is illustrated in Figure \ref{fig:cl labeling of marking edges}
when $Y=T$ is an ideal tetrahedron. This gives us an identification
$\mathbb{Z}^{M_{f}}\cong\mathbb{Z}^{3}$ by $\mathbb{Z}^{M_{f}}\ni d\mapsto(d_{1},d_{2},d_{3})\in\mathbb{Z}^{3}$,
where $d_{i}$ is the integer assigned to the edge $e_{i}$ by $d$. 

\begin{figure}[h]
  \includegraphics[scale=0.15]{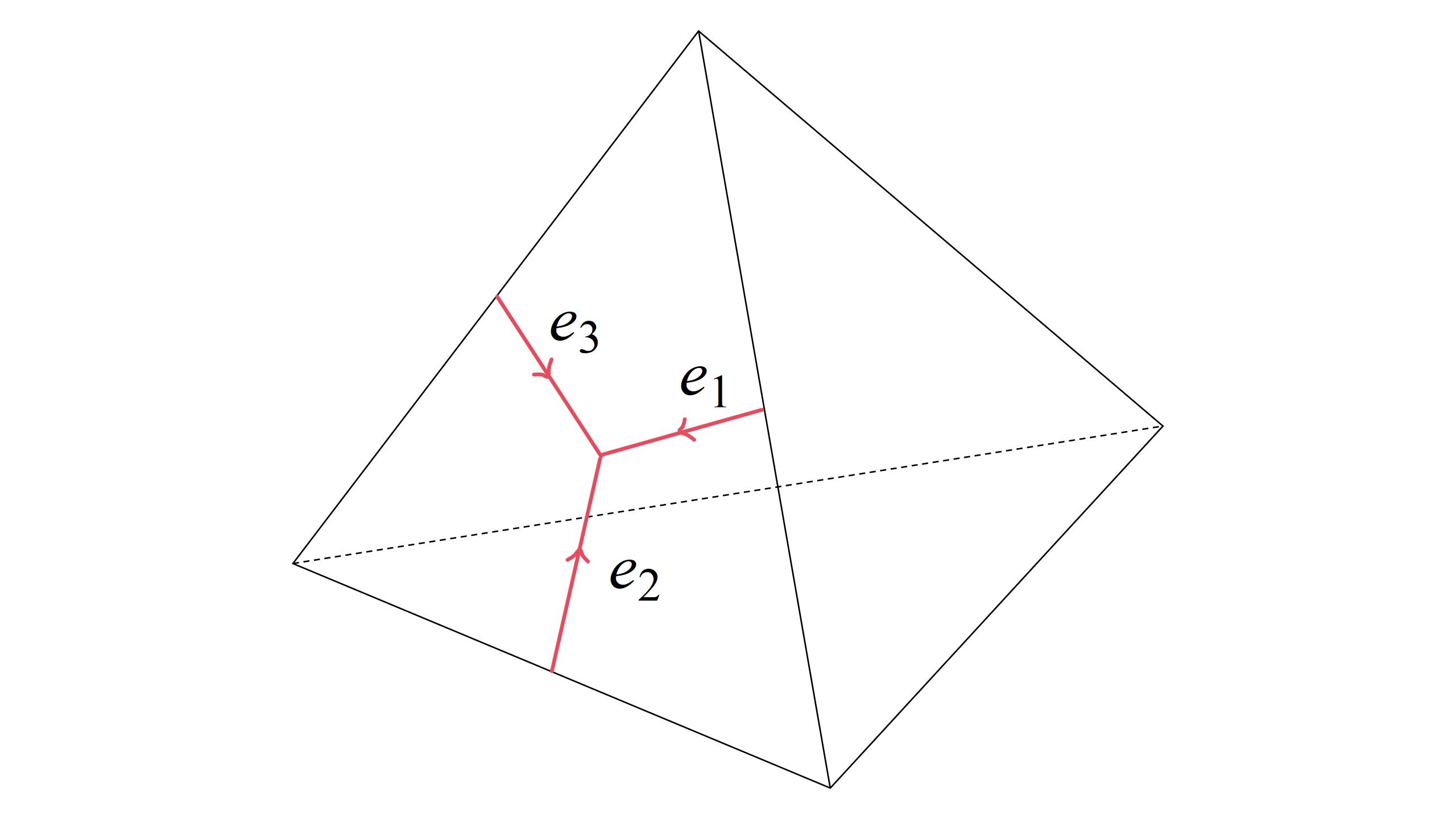} 
  \caption{Clockwise labeling of marking edges in a face of an ideal tetrahedron.}
  \label{fig:cl labeling of marking edges}
\end{figure}

Now the skew-symmetric bilinear form $\langle\ ,\ \rangle_{f}$ on $\mathbb{Z}^{M_{f}}$
is defined to be the pullback by the above identification of the skew-symmetric
bilnear form on $\mathbb{Z}^{3}$ defined by the matrix 
\[
\left(\begin{array}{ccc}
0 & 1 & -1\\
-1 & 0 & 1\\
1 & -1 & 0
\end{array}\right).
\]
Note that $\langle\ ,\ \rangle_{f}$ is independent of the choice of clockwise
ordering on $M_{f}$ used to make identification $\mathbb{Z}^{M_{f}}\cong\mathbb{Z}^{3}$.

Now we are ready to describe the twisted product, which we call the $\cdot$-product. In what follows, if an element $x\in\overline{\mathrm{Sk}}(T)$ or $\bigotimes_{f\in\mathcal{F}^{(2)}}\mathbb{T}_{f}$ or $\bigotimes_{e\in\ \mathcal{F}^{(1)}}\mathbb{B}_{e}$ is $\prod_{f\in\mathcal{F}^{(2)}}\mathbb{Z}^{M_{f}}$-homogeneous,
we write $d(x)$ for the total degree 
\[
\left(d_{f}(x)\right)_{f\in\mathcal{F}^{(2)}}\in\prod_{f\in\mathcal{F}^{(2)}}\mathbb{Z}^{M_{f}};
\]
if $x$ and $x^{\prime}$ are both $\prod_{f\in\mathcal{F}^{(2)}}\mathbb{Z}^{M_{f}}$-homogeneous,
we write $\langle d(x),d(x^{\prime})\rangle$ for $\underset{f\in\mathcal{F}^{(2)}} {\sum}\langle d_{f}(x),d_{f}(x^{\prime})\rangle_{f}$.

\begin{defn}
\label{def: defn of cdot-product}

(i)  Let $\kappa,\kappa^{\prime}\in\bigotimes_{f\in\mathcal{F}^{(2)}}\mathbb{T}_{f}$
be $\prod_{f\in\mathcal{F}^{(2)}}\mathbb{Z}^{M_{f}}$-homogeneous
elements. We let
\[
\kappa\cdot\kappa^{\prime}=A^{-\frac{1}{2}\langle d(\kappa),d(\kappa^{\prime})\rangle}\kappa\cup\kappa^{\prime},
\]
and extend to all elements of $\bigotimes_{f\in\mathcal{F}^{(2)}}\mathbb{T}_{f}$
by bilinearity.

(ii) Let $x,x^{\prime}\in\underset{e\in\mathcal{F}^{(1)}}{\bigotimes}\mathbb{B}_{e}$
be $\prod_{f\in\mathcal{F}^{(2)}}\mathbb{Z}^{M_{f}}$-homogeneous. We define
\[
x\cdot x^{\prime}=A^{-\frac{1}{2}\langle d(x),d(x^{\prime})\rangle}x\cup x^{\prime},
\]
and extend to all elements of $\underset{e\in\mathcal{F}^{(1)}}{\bigotimes}\mathbb{B}_{e}$
by bilinearity.

(iii) We can also twist the natural action of $\mathbb{T}_{f}$
and $\mathbb{B}_{e}$ on $\overline{\mathrm{Sk}}(Y)$. Let $\kappa\in\bigotimes_{f\in\mathcal{F}^{(2)}}\mathbb{T}_{f},\ x\in\bigotimes_{e\in\mathcal{F}^{(1)}}\mathbb{B}_{e}$
and $w\in\overline{\mathrm{Sk}}(Y)$ be $\prod_{f\in\mathcal{F}^{(2)}}\mathbb{Z}^{M_{f}}$-homogeneous.
We define
\[
\kappa\cdot w=A^{-\frac{1}{2}\langle d(\kappa),d(w)\rangle}\kappa\cup w,
\]
and
\[
w\cdot x=A^{-\frac{1}{2}\langle d(w),d(x)\rangle}w\cup x
\]
and extend by bilinearity. Under such ``$\cdot$-action'', $\overline{\mathrm{Sk}}(Y)$
becomes a $(\bigotimes_{f\in\mathcal{F}^{(2)}}\mathbb{T}_{f},\cdot)\text{-}(\bigotimes_{e\in\mathcal{F}^{(1)}}\mathbb{B}_{e},\cdot)$-bimodule,
namely 
\[
(\kappa\cdot w)\cdot x=\kappa\cdot(w\cdot x),
\]
as can be easily checked. 

\end{defn}

 Let us take a closer look at these product structures.
 \begin{itemize}
 \item The $\cdot$-product on  $\bigotimes_{f\in\mathcal{F}^{(2)}}\mathbb{T}_{f}$ restricts to each copy of $\mathbb{T}_{f}$
in the obvious way. Observe that if $\kappa,\kappa^{\prime}\in\mathbb{T}_{f}$
are monomials in the quantum torus $\mathbb{T}_{f}$, then
\[
\kappa\cdot\kappa^{\prime}=A^{-\frac{1}{2}\langle d_{f}(\kappa),d_{f}(\kappa^{\prime})\rangle_{f}}\kappa\cup\kappa^{\prime}=\left[\kappa\cup\kappa^{\prime}\right].
\]
This shows that the algebra $(\mathbb{T}_{f},\cdot)$ can be identified
with the Laurent polynomial algebra $R[\alpha^{\pm1},\beta^{\pm1},\gamma^{\pm1}]$
(when $\mathbb{T}_{f}$ is presented as in (\ref{eq:SkAld(D_3)})).
On the other hand, if $\kappa\in\mathbb{T}_{f}$ and $\kappa^{\prime}\in\mathbb{T}_{f^{\prime}}$
where $f$ and $f^{\prime}$ are different boundary faces, then we
have
\[
\kappa\cdot\kappa^{\prime}=\kappa\cup\kappa^{\prime}=\kappa^{\prime}\cup\kappa=\kappa^{\prime}\cdot\kappa.
\]
In particular, we have $(\bigotimes_{f\in\mathcal{F}^{(2)}}\mathbb{T}_{f},\cdot)=\bigotimes_{f\in\mathcal{F}^{(2)}}(\mathbb{T}_{f},\cdot)$.
 \item For the $\cdot$-product structure on $\underset{e\in\mathcal{F}^{(1)}}{\bigotimes}\mathbb{B}_{e}$, if $x\in\mathbb{B}_{e}$ and $x^{\prime}\in\mathbb{B}_{e^{\prime}}$
are monomials, and $e$ and $e^{\prime}$ are edges of the same boundary face $f$, then we have 
\[
x\cup x^{\prime}=A^{\frac{1}{2}\langle d_{f}(x),d_{f}(x^{\prime})\rangle_{f}}x\cdot x^{\prime}=[x\cdot x^{\prime}];
\]
whereas
if $e$ and $e^{\prime}$ are not edges of the same boundary face, then
\[
x\cup x^{\prime}=x\cdot x^{\prime}=x^{\prime}\cdot x=x^{\prime}\cup x.
\]
\end{itemize}

We are now in a position to define the desired further reduction
of the skein modules. 
\begin{defn}
\label{def:corner reduced module}
(i) For any boundary face $f$ of $Y$,
let us consider the ideal $I_{f}^{c}$ of the commutative algebra
$(\mathbb{T}_{f},\cdot)$ generated, in the presentation (\ref{eq:SkAld(D_3)})
of $\mathbb{T}_{f}$, by the elements
\[
\alpha-(-A^{2})^{-\frac{1}{2}},\ \beta-(-A^{2})^{-\frac{1}{2}},\ \gamma-(-A^{2})^{-\frac{1}{2}}.
\]
By abuse of notation, we also denote by $I_{f}^{c}$ the natural extension
of $I_{f}^{c}$ in the tensor product algebra $\bigotimes_{f\in\mathcal{F}^{(2)}}(\mathbb{T}_{f},\cdot)$,
let $I^{c}$ be the ideal of $\bigotimes_{f\in\mathcal{F}^{(2)}}(\mathbb{T}_{f},\cdot)$
given by the sum
\[
I^{c}=\sum_{f\in\mathcal{F}^{(2)}}I_{f}^{c}.
\]
(ii) The \emph{Corner-reduced skein module} $\overline{\mathrm{Sk}}^{c}(Y)$
of an ideally triangulated boundary marked 3-manifold $(Y,\mathcal{T})$
is defined to be the $R$-module quotient 
\[
\overline{\mathrm{Sk}}^{c}(Y)=\frac{\overline{\mathrm{Sk}}(Y)}{I^{c}\cdot\overline{\mathrm{Sk}}(Y)}.
\]
\end{defn}

\subsection{The case of a single tetrahedron\label{subsec:corner reduction in the case of a singe tetrahedron}}

For a single ideal tetrahedon $T$ with canonical boundary marking,
recall that we have the algebra $\mathbb{T}^{\otimes4}=\bigotimes_{f\in\mathbf{f}(T)}\mathbb{T}_{f}$
associated to the faces of $T$ and the algebra $\mathbb{B}^{\otimes6}=\bigotimes_{e\in\mathbf{e}(T)}\mathbb{B}_{e}$ associated to the edges of $T$.
The $\cdot$-product structure on these algebras and corner-reduction can be made more explicit in this case. In the notation of subsection
\ref{subsec:Reduced skein module of T}, we have the following lemma, whose proof is a straightforward calculation.
\begin{lem}
\label{lem:presntation of (T^tensor4,cdot)  and (B^tensor6, cdot)}
(i) The
algebra $(\mathbb{T}^{\otimes4},\cdot)$ is the Laurent polynomial
algebra 
\[
R[T]:=R\left[\Gamma_{ab}^{\pm1}\biggm|\begin{array}{c}
a,b\in\{z,z^{\prime},z^{\prime\prime},y,y^{\prime},y^{\prime\prime}\}\\
\text{\ensuremath{a} and \ensuremath{b} are not labels of opposite edges}
\end{array}\right].
\]
A monomial in $(\mathbb{T}^{\otimes4},\cdot)=R[T]$ equals the Weyl-ordering
of the same monomial under the usual $\cup$-product in the common underlying set $\mathbb{T}^{\otimes4}$. For example
\[
\Gamma_{zy^{\prime\prime}}\cdot\Gamma_{y^{\prime\prime}y}^{-1}=[\Gamma_{zy^{\prime\prime}}\Gamma_{y^{\prime\prime}y}^{-1}].
\]

(ii) The algebra $(\mathbb{B}^{\otimes6},\cdot)$ 
is isomorphic to the quantum torus
\begin{equation}
\mathbb{T}\left\langle T\right\rangle:=\frac{R\langle\hat{z}^{\pm1},\hat{z}^{\prime\pm1},\hat{z}^{\prime\prime\pm1},\hat{y}^{\pm1},\hat{y}^{\prime\pm1},\hat{y}^{\prime\prime\pm1}\rangle}{\left\langle \begin{array}{c}
\hat{a}\cdot\hat{b}^{\prime}=A\hat{b}^{\prime}\cdot\hat{a}\\
\hat{a}^{\prime}\cdot\hat{b}^{\prime\prime}=A\hat{b}^{\prime\prime}\cdot\hat{a}^{\prime}\\
\hat{a}^{\prime\prime}\cdot\hat{b}=A\hat{b}\cdot\hat{a}^{\prime\prime}
\end{array}\Bigg|a,b\in\{z,y\}\right\rangle }.\label{eq:quantum torus T(T)}
\end{equation}
In this case a monomial in $\mathbb{B}^{\otimes6}=R[\hat{z}^{\pm1},\hat{z}^{\prime\pm1},\hat{z}^{\prime\prime\pm1},\hat{y}^{\pm1},\hat{y}^{\prime\pm1},\hat{y}^{\prime\prime\pm1}],$
i.e. with the respect to the $\cup$-product, equals the Weyl-ordering of the same monomial under the $\cdot$-product.
For example 
\[
\hat{z}^{\prime}\hat{y}^{\prime\prime3}\hat{z}^{-2}=[\hat{z}^{\prime}\cdot\hat{y}^{\prime\prime3}\cdot\hat{z}^{-2}].
\]
\end{lem}

For consistency, the multiplications in the algebras $R[T]$ and $\mathbb{T}\langle T\rangle$ will always be written as $\cdot$.

By construction, the skein module $\overline{\mathrm{Sk}}(T)$, which
is a $\mathbb{T}^{\otimes4}\text{-}\mathbb{B}^{\otimes6}$-bimodule under
$\cup$-action, becomes a $(\mathbb{T}^{\otimes4},\cdot)\text{-}(\mathbb{B}^{\otimes6},\cdot)=R\left[T\right]\text{-}\mathbb{T}\left\langle T\right\rangle$-bimodule
under $\cdot$-action. Let us realize the $R\left[T\right]\text{-}\mathbb{T}\left\langle T\right\rangle$-bimodule
structure on $\overline{\mathrm{Sk}}(T)$ in an alternative way.
Let's first introduce a $R\left[T\right]\text{-}\mathbb{T}\left\langle T\right\rangle$-bimodule
structure on the $R$-module $\mathbb{T}^{\otimes4}\otimes\mathbb{B}^{\otimes6}=R\left[T\right]\otimes\mathbb{T}\left\langle T\right\rangle$
as follows:

\begin{enumerate}
\item Let $\Gamma^{\prime},\Gamma\in R\left[T\right]$
and $x\in\mathbb{T}\left\langle T\right\rangle$ be $\prod_{f\in\mathbf{f}(T)}\mathbb{Z}^{M_{f}}$-homogeneous
(for example when they are monomials), define
\begin{align*}
\Gamma^{\prime}\cdot\left(\Gamma\otimes x\right) & =A^{-\frac{1}{2}\langle d(\Gamma^{\prime}),d(x)\rangle}(\Gamma^{\prime}\cdot\Gamma)\otimes x\\
 & =A^{-\frac{1}{2}\langle d(\Gamma^{\prime}),d(\Gamma)\rangle-\frac{1}{2}\langle d(\Gamma^{\prime}),d(x)\rangle}\Gamma^{\prime}\Gamma\otimes x
\end{align*}
and extend by bilinearity. This defines a left $R\left[T\right]$-module
structure on $R\left[T\right]\otimes\mathbb{T}\left\langle T\right\rangle$.
\item Let $\Gamma\in R\left[T\right]$ and $x,x^{\prime}\in\mathbb{T}\left\langle T\right\rangle$
be $\prod_{f\in\mathbf{f}(T)}\mathbb{Z}^{M_{f}}$-homogeneous, define
\begin{align*}
\left(\Gamma\otimes x\right)\cdot x^{\prime} & =A^{-\frac{1}{2}\langle d(\Gamma),d(x^{\prime})\rangle}\Gamma\otimes(x\cdot x^{\prime})\\
 & =A^{-\frac{1}{2}\langle d(\Gamma),d(x^{\prime})\rangle-\frac{1}{2}\langle d(x),d(x^{\prime})\rangle}\Gamma\otimes xx^{\prime}
\end{align*}
and extend by bilinearity. This defines a right $\mathbb{T}\left\langle T\right\rangle$-module
structure on $R\left[T\right]\otimes\mathbb{T}\left\langle T\right\rangle$.
\item A straightforward calculation shows that if $\Gamma^{\prime},\Gamma\in R\left[T\right]$
and $x,x^{\prime}\in\mathbb{T}\left\langle T\right\rangle$ are $\prod_{f\in\mathbf{f}(T)}\mathbb{Z}^{M_{f}}$-homogeneous,
we have
\begin{multline*}
\left(\Gamma^{\prime}\cdot\left(\Gamma\otimes x\right)\right)\cdot x^{\prime}=\Gamma^{\prime}\cdot\left(\left(\Gamma\otimes x\right)\cdot x^{\prime}\right)\\
\\
=A^{-\frac{1}{2}\left(\langle d(\Gamma^{\prime}),d(x)\rangle+\langle d(\Gamma^{\prime}),d(x^{\prime})\rangle+\langle d(\Gamma),d(x^{\prime})\rangle\right)}\left(\Gamma^{\prime}\cdot\Gamma\right)\otimes\left(x\cdot x^{\prime}\right)\\
\\
=A^{-\frac{1}{2}\left(d(\Gamma^{\prime}),d(\Gamma)\rangle+\langle d(\Gamma^{\prime}),d(x)\rangle+\langle d(\Gamma^{\prime}),d(x^{\prime})\rangle+\langle d(\Gamma),d(x^{\prime})\rangle+\langle d(x),d(x^{\prime})\rangle\right)}\\
\times\Gamma^{\prime}\Gamma\otimes xx^{\prime}.
\end{multline*}
Hence $R\left[T\right]\otimes\mathbb{T}\left\langle T\right\rangle$ is a
$R\left[T\right]\text{-}\mathbb{T}\left\langle T\right\rangle$-bimodule in such
$\cdot$-actions.
\footnote{Note that this $R\left[T \right]\text{-}\mathbb{T}\left\langle T\right\rangle$-bimodule structure on $R\left[T \right]\otimes\mathbb{T}\left\langle T\right\rangle$ is not the natural $R\left[T \right]\text{-}\mathbb{T}\left\langle T\right\rangle$-bimodule structure on $R\left[T \right]\otimes\mathbb{T}\left\langle T\right\rangle$ coming from the tensor product. }
\end{enumerate}

Let's make an important observation here. Because the $\cdot$-product
and $\cup$-product differ by a multiplication of a scalar on homogeneous
elements, if $\{w_{i}\}$ is a collection of $\prod_{f\in\mathbf{f}(T)}\mathbb{Z}^{M_{f}}$-homogeneous
elements of the $R$-module $\mathbb{T}^{\otimes4}\otimes\mathbb{B}^{\otimes6}=R\left[T\right]\otimes\mathbb{T}\left\langle T\right\rangle$,
then the $\mathbb{T}^{\otimes4}\text{-}\mathbb{B}^{\otimes6}$-sub-bimodule of 
$\mathbb{T}^{\otimes4}\otimes \mathbb{B}^{\otimes6}$ generated by $\{w_{i}\}$ in $\cup$-product is the same as
the $R\left[T \right]\text{-}\mathbb{T}\left\langle T\right\rangle$-sub-bimodule of $R\left[T \right]\otimes\mathbb{T}\left\langle T\right\rangle$\footnote{In the $R\left[T \right]\text{-}\mathbb{T}\left\langle T\right\rangle$-bimodule structure of $R\left[T \right]\otimes\mathbb{T}\left\langle T\right\rangle$ introduced by Point (3) above.
}
generated by $\{w_{i}\}$ in $\cdot$-product; they are the 
same as $R$-submodule of the $R$-module $\mathbb{T}^{\otimes4}\otimes\mathbb{B}^{\otimes6}=R\left[T \right]\otimes\mathbb{T}\left\langle T\right\rangle$.

\begin{lem}
\label{lem:generators of Ann(emptyset) are homogeneous, so it is also an ideal in dot-product}The
elements $v_{zz^{\prime}z^{\prime\prime}}$, $\ell_{zz^{\prime\prime}}$ and
similarly the other generators of $\mathrm{Ann}([\emptyset])$ we gave
in Corollary \ref{cor:reduced skein module of T} are $\prod_{f\in\mathbf{f}(T)}\mathbb{Z}^{M_{f}}$-homogeneous,
thus $\mathrm{Ann}([\emptyset])$ can be viewed as the $R\left[T \right]\text{-}\mathbb{T}\left\langle T\right\rangle$-sub-bimodule of $R\left[T \right]\otimes\mathbb{T}\left\langle T\right\rangle$
generated by these elements. 
\end{lem}

\begin{proof}
Recall that 
\[
v_{zz^{\prime}z^{\prime\prime}}:=(-A^{2})^{2}\Gamma_{zz^{\prime\prime}}\Gamma_{z^{\prime\prime}z^{\prime}}\Gamma_{z^{\prime}z}-\hat{z}\hat{z}^{\prime}\hat{z}^{\prime\prime}
\]
\begin{align*}
\ell_{zz^{\prime\prime}}\colon & =[\Gamma_{zy^{\prime}}\Gamma_{y^{\prime}z^{\prime\prime}}^{-1}]-(-A^{2})^{-1}[\Gamma_{z^{\prime\prime}y}\Gamma_{yz^{\prime}}^{-1}]\Gamma_{zz^{\prime}}^{-1}\hat{z}^{\prime\prime-1}\hat{z}^{\prime}\hat{z}\\
 & -(-A^{2})\Gamma_{z^{\prime\prime}z^{\prime}}[\Gamma_{z^{\prime}y^{\prime\prime}}\Gamma_{y^{\prime\prime}z}^{-1}]\hat{z}^{\prime\prime-1}\hat{z}^{\prime-1}\hat{z},
\end{align*}
so we only need to check that
\[
d(\Gamma_{zz^{\prime\prime}}\Gamma_{z^{\prime\prime}z^{\prime}}\Gamma_{z^{\prime}z})=d(\hat{z}\hat{z}^{\prime}\hat{z}^{\prime\prime})
\]
and that
\begin{align*}
d(\Gamma_{zy^{\prime}}\Gamma_{y^{\prime}z^{\prime\prime}}^{-1}) & =d(\Gamma_{z^{\prime\prime}y}\Gamma_{yz^{\prime}}^{-1}\Gamma_{zz^{\prime}}^{-1})+d(\hat{z}^{\prime\prime-1}\hat{z}^{\prime}\hat{z})\\
 & =d(\Gamma_{z^{\prime\prime}z^{\prime}}\Gamma_{z^{\prime}y^{\prime\prime}}\Gamma_{y^{\prime\prime}z}^{-1})+d(\hat{z}^{\prime\prime-1}\hat{z}^{\prime-1}\hat{z}),
\end{align*}
which is a straightforward computation. (In fact, these degree identity
can be easily seen from the diagrams of $v_{zz^{\prime}z^{\prime\prime}}$
and $\ell_{zz^{\prime\prime}}$)
\end{proof}

Therefore $\overline{\mathrm{Sk}}(T)$ is the quotient of $R\left[T \right]\otimes\mathbb{T}\left\langle T\right\rangle$
by a $R\left[T \right]\text{-}\mathbb{T}\left\langle T\right\rangle$-sub-bimodule.
This gives the $R\left[T \right]\text{-}\mathbb{T}\left\langle T\right\rangle$-bimodule
structure on $\overline{\mathrm{Sk}}(T)$ in an alternative manner,
as a quotient of the $R\left[T \right]\text{-}\mathbb{T}\left\langle T\right\rangle$-bimodule
$R\left[T \right]\otimes\mathbb{T}\left\langle T\right\rangle$.

The ideal $I^{c}$ of $\left(\mathbb{T}^{\otimes4},\cdot\right)$
which we considered when forming the corner-reduced module $\overline{\mathrm{Sk}}^{c}(T)$
(see Definition \ref{def:corner reduced module} (i)) is the ideal 
\begin{equation}
\label{eq:I^c for a single ideal T}
\left\langle \Gamma_{ab}-(-A^{2})^{-\frac{1}{2}}\biggm|\begin{array}{c}
a,b\in\{z,z^{\prime},z^{\prime\prime},y,y^{\prime},y^{\prime\prime}\}\\
\text{\ensuremath{a} and \ensuremath{b} are not labels of opposite edges}
\end{array}\right\rangle 
\end{equation}
of the polynomial algebra $\left(\mathbb{T}^{\otimes4},\cdot\right)=R\left[T \right]$.
By definition, $\overline{\mathrm{Sk}}^{c}(T)$ is the quotient of $\overline{\mathrm{Sk}}(T)$
by the $R$-submodule $I^{c}\cdot\overline{\mathrm{Sk}}(T)$, which
is also easily seen to be a $R\left[T \right]\text{-}\mathbb{T}\left\langle T\right\rangle$-sub-bimodule
of $\overline{\mathrm{Sk}}(T)$, therefore $\overline{\mathrm{Sk}}^{c}(T)$
is a quotient $R\left[T \right]\text{-}\mathbb{T}\left\langle T\right\rangle$-bimodule
of $\overline{\mathrm{Sk}}(T)$ and thus of $R\left[T \right]\otimes\mathbb{T}\left\langle T\right\rangle$.
Moreover, quotient by $I^{c}\cdot\overline{\mathrm{Sk}}(T)$ identifies
the left $\cdot$-multiplication of $\Gamma_{ab}$ with scalar multiplication
by $(-A^{2})^{-\frac{1}{2}}$, this in some sense ``kills'' the
$R\left[T \right]$-factor and thus $\overline{\mathrm{Sk}}^{c}(T)$
is a quotient of $\mathbb{T}\left\langle T\right\rangle$ only. This intuition will be made
precise later in Section \ref{sec:Quantum Trace Map} when we construct our
quantum trace map on $\overline{\mathrm{Sk}}^{c}(T)$.

\subsection{Splitting of the 3-manifolds into ideal tetrahedra\label{subsec:Splitting homomorphism}}
In this subsection, we introduce a version of splitting homomorphism, induced by splitting of the 3-manifold $Y$ into ideal tetrahedra,
tailored to corner-reduced skein modules, which is a modification
of the construction in \cite{PP1}. Recall that we have the notion of combinatorial foliation introduced
earlier. If $\mathcal{T}$ is an ideal triangulation of the 3-manifold
$Y$, any \emph{internal} face $f\in\mathcal{T}^{(2)}$ is decomposed
into three elementary quadrilaterals, and the foliations on each of
them give rise to a foliation on $f$, see the picture on the right
of Figure \ref{fig:foliations on elementary Q and face}. (For our
purpose, we only need to consider this foliation on internal faces.)

A ribbon tangle $\ell$ in the boundary marked 3-manifold $Y$ is said to be in \emph{general
position} if it satisfies the following:
\begin{enumerate}
\item Its intersection with every internal face $f\in\mathcal{T}^{(2)}$
is transverse,
\item at the intersection point, the foliation is transverse to the 2-d
tangent plane to the ribbon tangle,
\item no singular leaf meets $\ell$, and
\item no leaf of the foliation meets $\ell$ at two or more points.
\end{enumerate}
Now, given a ribbon link $\ell$ in $Y$ in general position, and an assignment of a state $\pm$ to every intersection of the link with a face, we get an
element
\[
\underset{T\in\mathcal{T}}{\otimes}\ell_{T}^{\overrightarrow{\epsilon}}\in\bigotimes_{T\in\mathcal{T}}\overline{\mathrm{Sk}}(T)
\]
 in the following manner: since $\ell$ is in general position, one
can isotope $\ell$ by moving its part near a face along the leaves
of the foliation so that the intersection of $\ell$ with every internal
face $f$ falls on a marking edge, in a way so that the ribbon surface
of $\ell$ is tangent to the marking edges. Let $\ell_{T}$ be the
part of $\ell$ in $T$, which is now a ribbon tangle in $T$ with endpoints
on the canonical boundary marking of $T$; $\overrightarrow{\epsilon}$
is the assignment of ``compatible states'' for every $\ell_{T},\ T\in\mathcal{T}$
(meaning if $p_{i}$ is endpoint of $\ell_{T_{i}}$ and $p_{j}$ is
an endpoint of $\ell_{T_{j}}$ which are identified upon gluing $T_{i}$
and $T_{j}$, then we have $\overrightarrow{\epsilon}(p_{i})=\overrightarrow{\epsilon}(p_{j})$).
This construction gives us an element $\ell_{T}^{\overrightarrow{\epsilon}}\in\overline{\mathrm{Sk}}(T)$
for each $T\in\mathcal{T}$. We wish to use this construction to define
a splitting homomorphism on the skein module $\mathrm{Sk}(Y)$ of the
3-manifold $Y$. For this we need to understand how an isotopy of
$\ell$ affects the element $\underset{T\in\mathcal{T}}{\otimes}\ell_{T}^{\overrightarrow{\epsilon}}$.
Let's start with some necessary definitions.
\begin{defn}
\label{def:balanced elements}
Let $(Y,\mathcal{T})$ be an ideally triangulated 3-manifold, an element
of the form
\[
\underset{T\in\mathcal{T}}{\otimes}x_{T}\in\bigotimes_{T\in\mathcal{T}}\overline{\mathrm{Sk}}(T)
\]
 is \emph{balanced} if given bare faces $f_{i}\in\mathbf{f}(T_{i})$
and $f_{j}\in\mathbf{f}(T_{j})$ of some tetrahedra $T_{i}$ and $T_{j}$
($T_{i}$ and $T_{j}$ are not necessarily distinct) such that $f_{i}$
and $f_{j}$ are identified to a face $f\in\mathcal{T}^{(2)}$, we
have $d_{f_{i}}(x_{T_{i}})=d_{f_{j}}(x_{T_{j}})$ in the following
sense:

If we label the marking edges in $f_{i}$ and $f_{j}$ by $e_{1},e_{2},e_{3}$
in the way so that the same label $e_{k}$ $(k=1,2,3)$ is given to
the marking edge in $f_{i}$ and the marking edge in $f_{j}$ that
are identified after we glue $f_{i}$ and $f_{j}$ (see Figure \ref{fig:matching labeling of marking edges}),
we have $d_{f_{i}}(x_{T_{i}})=d_{f_{j}}(x_{T_{j}})\in\mathbb{Z}^{\{e_{1},e_{2},e_{3}\}}$. 

In other word, we require $x_{T_{i}}$ and $x_{T_{j}}$ to have the same
sum of states of endpoints on marking edges that are identified upon
gluing. 
\end{defn}

\begin{figure}[h]
  \includegraphics[scale=0.2]{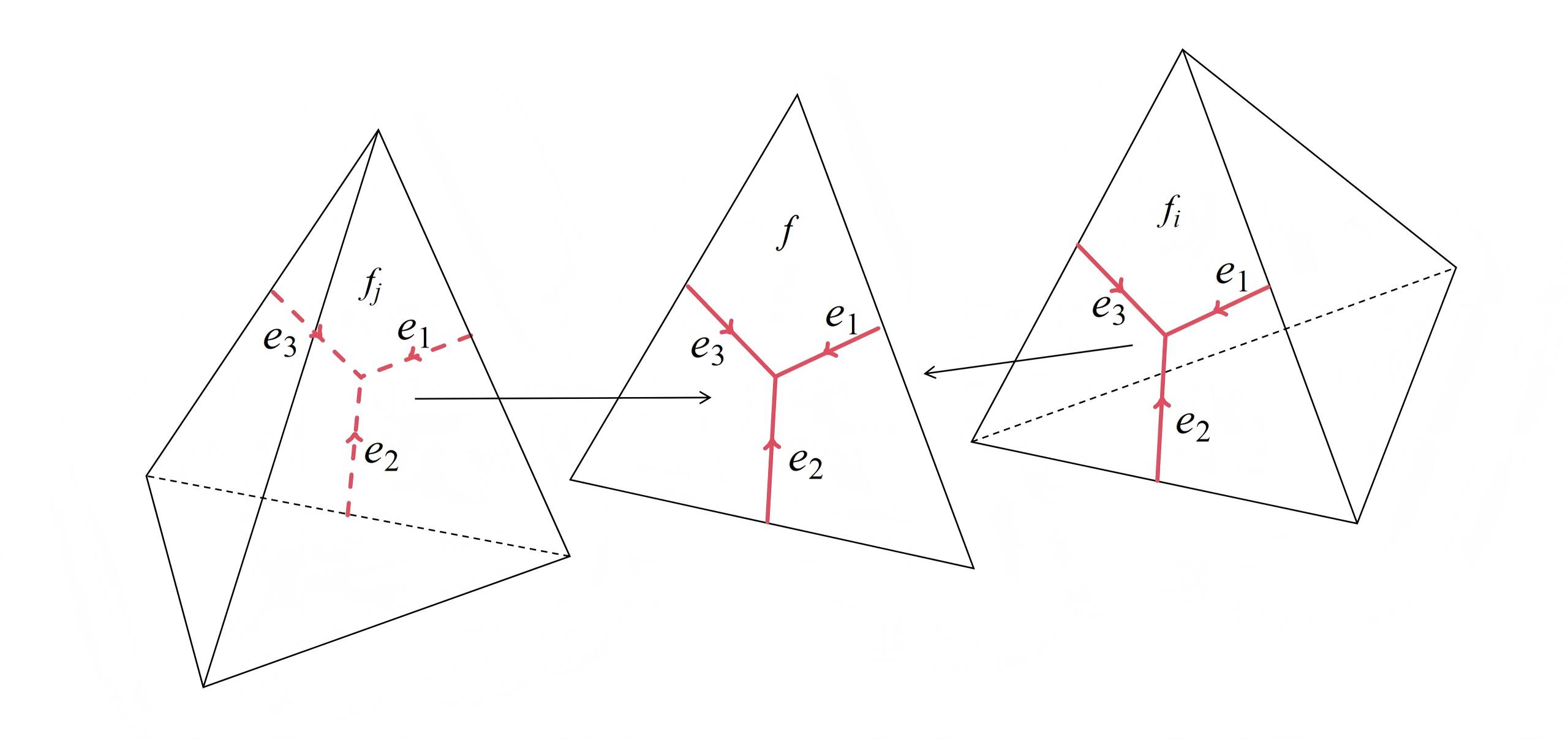} 
  \caption{}
  \label{fig:matching labeling of marking edges}
\end{figure}

For example, given any stated tangle $\ell$ in general position, the element$\underset{T\in\mathcal{T}}{\otimes}\ell_{T}^{\overrightarrow{\epsilon}}\in\bigotimes_{T\in\mathcal{T}}\overline{\mathrm{Sk}}(T)$ is balanced.
\begin{rem}
\label{rem:matching face, opposite skew-symmetric form}
Observe that if $x=\underset{T\in\mathcal{T}}{\otimes}x_{T}$ and
$y=\underset{T\in\mathcal{T}}{\otimes}y_{T}$ are balanced elements
of $\bigotimes_{T\in\mathcal{T}}\overline{\mathrm{Sk}}(T)$, then for
any pair of bare faces $f_{i}$ and $f_{j}$ that are identified to
a face $f\in\mathcal{T}^{(2)}$ we have
\[
\langle d_{f_{i}}(x),d_{f_{i}}(y)\rangle_{f_i}=-\langle d_{f_{j}}(x),d_{f_{j}}(y)\rangle_{f_j},
\]
where $\langle\ ,\ \rangle_{f_i}$ and $\langle\ ,\ \rangle_{f_J}$ are the skew-symmetric bilinear form we
associate to $\mathbb{Z}^{E_{f_{i}}}$ and $\mathbb{Z}^{E_{f_{j}}}$, respectively.
This is because if we have labeled the markings edges in $f_{i}$ by $e_{1},e_{2},e_{3}$
in a clockwise manner with respect to an outward pointing normal vector of $f_{i}$, then a matching
labeling of the marking edges in $f_{j}$ by $e_{1},e_{2},e_{3}$ would
be \emph{counterclockwise} with respect to an outward pointing normal vector of $f_{j}$. Thus we have
$d_{f_{i}}(x)=d_{f_{j}}(x)$ and $d_{f_{i}}(y)=d_{f_{j}}(y)$ as elements
of $\mathbb{Z}^{3}$ but the matrix of the skew-symmetric form with resect to the basis $e_{1},e_{2},e_{3}$ for
$\langle\ ,\ \rangle_{f_i}$ is $\left(\begin{array}{ccc}
0 & 1 & -1\\
-1 & 0 & 1\\
1 & -1 & 0
\end{array}\right)$ while the one for $\langle\ ,\ \rangle_{f_j}$ is $\left(\begin{array}{ccc}
0 & -1 & 1\\
1 & 0 & -1\\
-1 & 1 & 0
\end{array}\right)$.
\end{rem}

\begin{defn}
\label{def:reduced tensor product}

(i) Let $e\in\mathcal{T}^{(1)}$ be an \emph{internal} edge, $T_{1},T_{2},\dots,T_{k}$
be the sequence of ideal tetrahedra glued around $e$ (there could
be repetition in this sequence). Suppose the bare edge of $T_{i}$
identified to $e$ is labeled $e_{i}$ (by the labeling convention
we had in \ref{subsec:Reduced skein module of T}, $e_{i}$ is one of the
$z_{T_{i}}^{\boxempty}$ or $y_{T_{i}}^{\boxempty}$) and recall that
we have the generator $x_{e_{i}}$ of the bigon algebra $\mathbb{B}_{e_{i}}$
associated to the bivalent source on the edge $e_{i}$. We define
an element (see Figure \ref{fig:skein diagram of r(e)})
\[
\hat{e}=\hat{e}_{1}\hat{e}_{2}\dots \hat{e}_{k}\in\bigotimes_{T\in\mathcal{T}}\left(\bigotimes_{e\in\mathbf{e}(T)}\mathbb{B}_{e}\right)
\]

(ii) Let $\mathfrak{R}_{E}$ be the $R$-submodule of $\bigotimes_{T\in\mathcal{T}}\overline{\mathrm{Sk}}(T)$
spanned by elements of the form 
\[
\left(\underset{T\in\mathcal{T}}{\otimes}x_{T}\right)\cup\left(\hat{e}^{\epsilon}-(-A^{2})^{\epsilon}\right),\ \epsilon\in\{\pm1\}
\]
where $\underset{T\in\mathcal{T}}{\otimes}x_{T}$ is balanced, $e\in\mathcal{T}^{(1)}$
an \emph{internal} edge and $\hat{e}$ is the element associated with $e$ defined above.
(Here we are using that for each $T\in\mathcal{T}$, $\overline{\mathrm{Sk}}(T)$
is a right $\bigotimes_{e\in\mathbf{e}(T)}\mathbb{B}_{e}$-module
under $\cup$)

(iii) Let $\mathfrak{R}_{E}^{c}$ be the $R$-submodule of the tensor
product of corner-reduced modules $\bigotimes_{T\in\mathcal{T}}\overline{\mathrm{Sk}}^{c}(T)$
given by the image of $\mathfrak{R}_{E}$ under the natural map
\[
\bigotimes_{T\in\mathcal{T}}\overline{\mathrm{Sk}}(T)\rightarrow\bigotimes_{T\in\mathcal{T}}\overline{\mathrm{Sk}}^{c}(T)
\]

(iv) The \emph{reduced tensor product} $\overline{\bigotimes}_{T\in\mathcal{T}}\overline{\mathrm{Sk}}^{c}(T)$
is defined to be the $R$-module quotient
\[
\overline{\bigotimes_{T\in\mathcal{T}}}\overline{\mathrm{Sk}}^{c}(T)=\bigotimes_{T\in\mathcal{T}}\overline{\mathrm{Sk}}^{c}(T)\biggm/\mathfrak{R}_{E}^{c}
\]

\end{defn}

\begin{rem}
Alternatively, we can think of $\overline{\bigotimes}_{T\in\mathcal{T}}\overline{\mathrm{Sk}}^{c}(T)$
as the $R$-module quotient
\[
\overline{\bigotimes_{T\in\mathcal{T}}}\overline{\mathrm{Sk}}^{c}(T)=\frac{\bigotimes_{T\in\mathcal{T}}\overline{\mathrm{Sk}}(T)}{\mathfrak{R}_{E}+I^{c}\cdot\bigotimes_{T\in\mathcal{T}}\overline{\mathrm{Sk}}(T)},
\]
where $I^{c}$ is the ideal  of the commutative algebra $R[T]$
given by (\ref{eq:I^c for a single ideal T}). This is convenient when one want to utilize the gradings provided
by the boundary marking, which makes sense in $\bigotimes_{T\in\mathcal{T}}\overline{\mathrm{Sk}}(T)$ but not in $\bigotimes_{T\in\mathcal{T}}\overline{\mathrm{Sk}}^{c}(T)$
as the ideal $I^{c}$ is not homogeneous. 
\end{rem}

\begin{figure}[h]
  \includegraphics[scale=0.2]{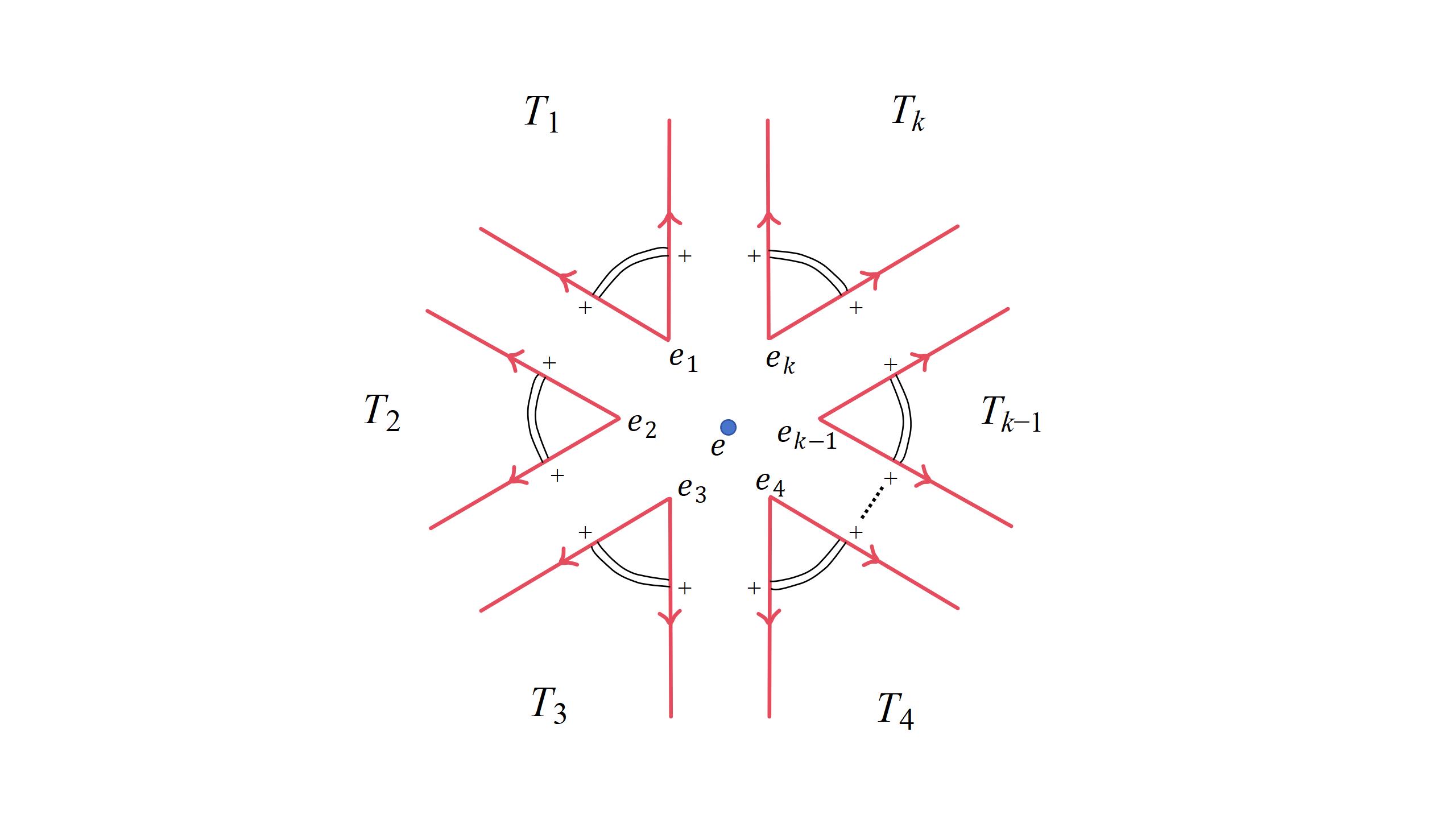} 
  \caption{Skein diagram of the element $\hat{e}$, the edge class $e$ is represented as a blue dot in the center}
  \label{fig:skein diagram of r(e)}
\end{figure}
\begin{thm}
\label{thm:splitting homomorphism}
Let $(Y,\mathcal{T})$ be an ideally triangulated boundary marked
3-manifold. The assignment to any stated ribbon tangle $\ell$ in
general position of the element
\[
\sigma(\ell):=\sum_{\overrightarrow{\epsilon}}\left(\underset{T\in\mathcal{T}}{\otimes}\ell_{T}^{\overrightarrow{\epsilon}}\right)\in\overline{\bigotimes}_{T\in\mathcal{T}}\overline{\mathrm{Sk}}^{c}(T)
\]
where the sum is over all compatible states $\overrightarrow{\epsilon}$
gives a well-defined $R$-module homomorphism
\[
\sigma\colon\overline{\mathrm{Sk}}^{c}(Y)\rightarrow\overline{\bigotimes_{T\in\mathcal{T}}}\overline{\mathrm{Sk}}^{c}(T).
\]
\end{thm}

Note that by composing $\sigma$ with the natural quotient maps $\mathrm{Sk}(Y)\rightarrow\overline{\mathrm{Sk}}(Y)\rightarrow\overline{\mathrm{Sk}}^{c}(Y)$,
we also have splitting homomorphisms defined on $\overline{\mathrm{Sk}}(Y)$
and on $\mathrm{Sk}(Y)$. 

\begingroup
\allowdisplaybreaks
\begin{proof}
The proof is similar to that of Theorem 3.20 in \cite{PP1}, therefore
we will skip some details and only focus on the part that differs with \cite{PP1}. First we prove that the assignment $\ell\mapsto\sigma(\ell)$
above gives a well-defined map on the level of $\mathrm{Sk}(Y)$, we
then show that it descends to $\overline{\mathrm{Sk}}(Y)$ and $\overline{\mathrm{Sk}}^{c}(Y)$
via the quotient maps. 

To show that $\sigma$ is well-defined on $\mathrm{Sk}(Y)$, it is sufficient
to show that $\sigma(\ell)$ is invariant under an isotopy of $\ell$,
after that we can always ``move'' the 3-ball, in which a skein relation
is performed, into the interior of some ideal tetrahedron by a suitable
ambient isotopy of $Y$, ensuring that the map $\sigma$ respect skein
relations and thus gives us a well-defined map $\mathrm{Sk}(Y)$.

So we let $\ell$ and $\ell^{\prime}$ be isotopic tangles in $Y$,
both in general position. The isotopy between $\ell$ and $\ell^{\prime}$
can be decomposed into a finite sequence of the following ``elementary
moves'': 

(\Romannum{1}) An isotopy in the class of (ribbon) tangles in general
positions.

(\Romannum{2}) Half twist of the ribbon tangle near some internal
face $f$. 

(\Romannum{3}) Height exchange.

(\Romannum{4}) Birth or annihilation of a pair of intersection points
with some internal face $f$.

(\Romannum{5}) An isotopy passing the tangle through a singular
leaf. 

See proof of Theorem 3.20 in \cite{PP1} for pictures illustrating
these moves, note that moves of type (\Romannum{2}), (\Romannum{3}),
(\Romannum{4}) and (\Romannum{5}) break down conditions (2),
(4), (1) and (3), respectively, of being in general positions. 

Just as in \cite{PP1}, if $\ell$ and $\ell^{\prime}$ are related
by moves of type (\Romannum{1})-(\Romannum{4}), then $\sigma(\ell)=\sigma(\ell^{\prime})$.
In fact, they are equal even as elements of $\bigotimes_{T\in\mathcal{T}}\mathrm{Sk}(T)$.
For type (\Romannum{1}) move this is obvious; for type (\Romannum{2}),
(\Romannum{3}) and (\Romannum{4}) moves this follows from skein
relations.

Now we focus on type (\Romannum{5}) move. There are two cases, depending
on whether the singular leaf is in the interior of some internal face
$f$ or it is half of some internal edge $e$. 

\begin{figure}[h]
  \includegraphics[scale=0.2]{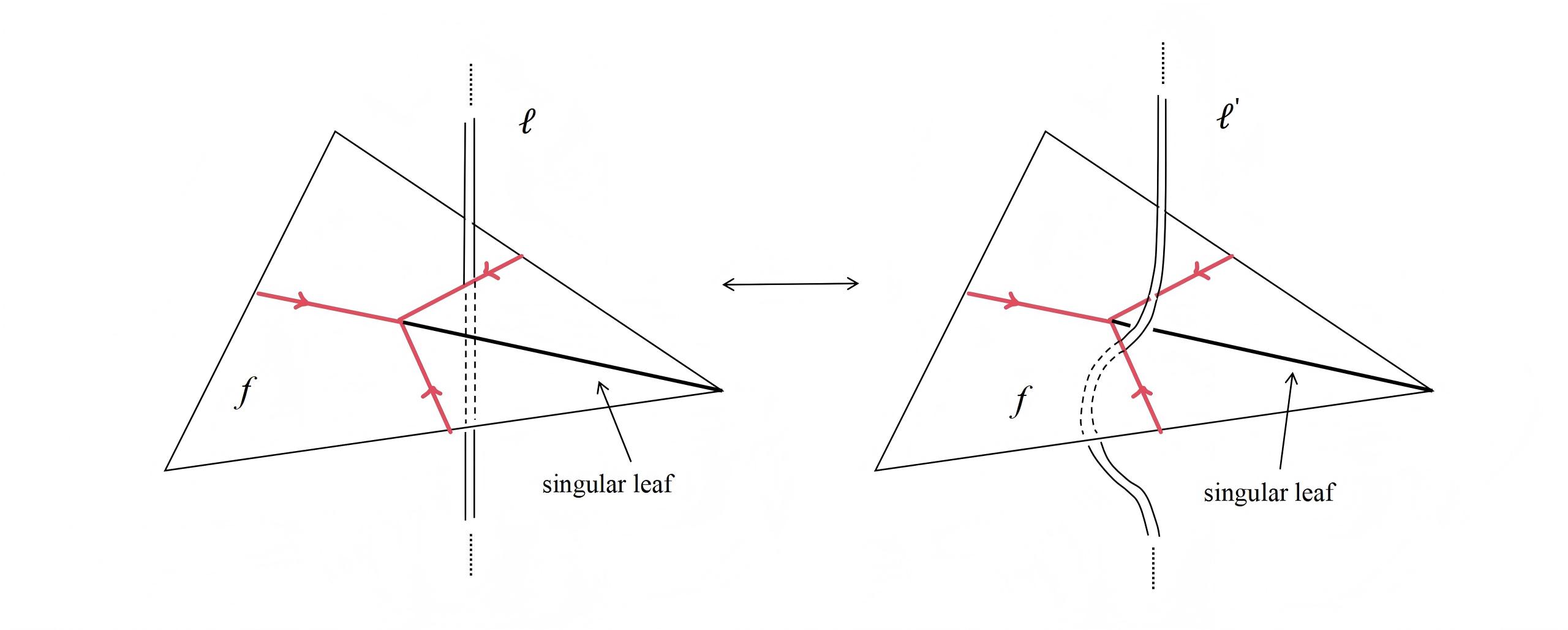} 
  \caption{$\ell$ and $\ell^{\prime}$ are related by an isotopy across a singular leaf which is in the interior of an internal face $f$}
  \label{fig:Move l across a singular leaf (1)}
\end{figure}

\begin{figure}[h]
  \includegraphics[scale=0.2]{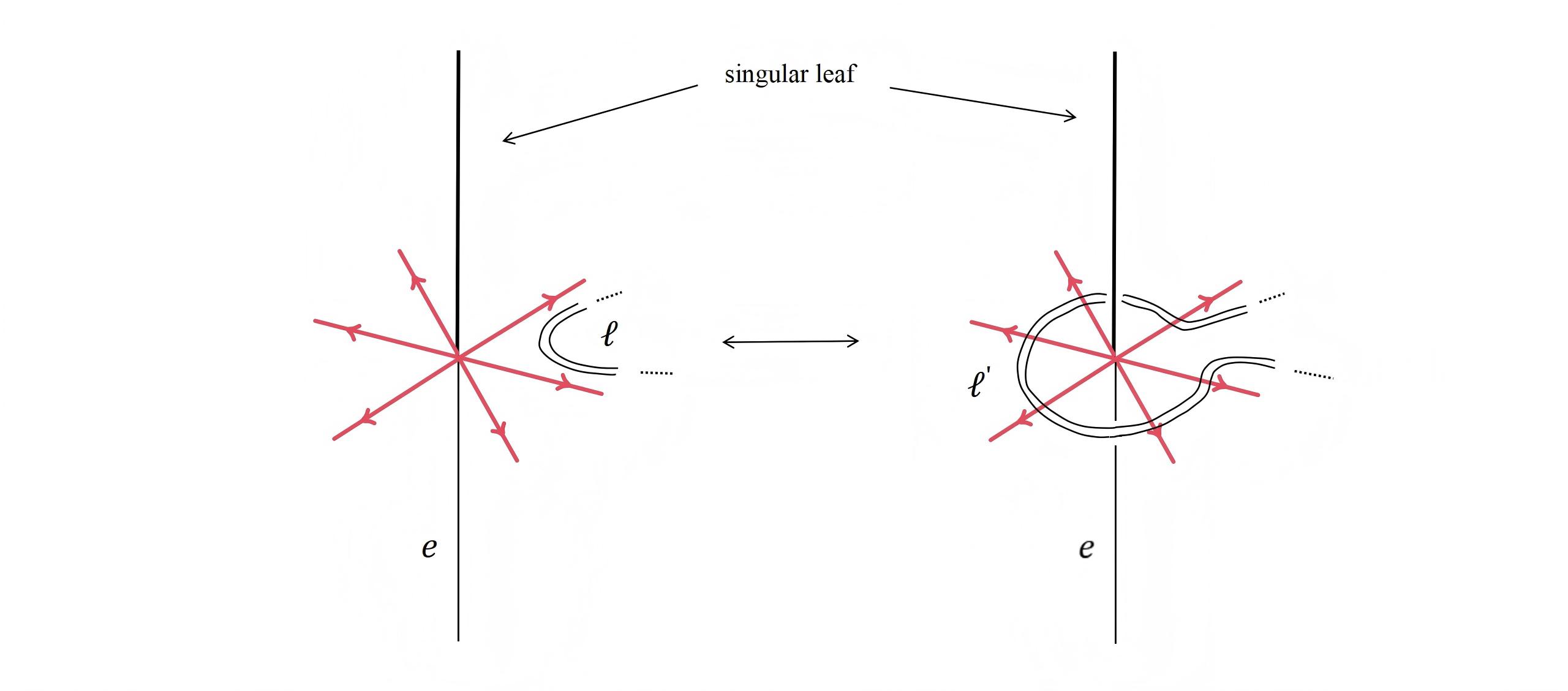} 
  \caption{$\ell$ and $\ell^{\prime}$ are related by an isotopy across a singular leaf which is half of an internal edge $e$}
  \label{fig:Move l across a singular leaf (2)}
\end{figure}
In Figure \ref{fig:Move l across a singular leaf (1)} and \ref{fig:Move l across a singular leaf (2)},
we show the movement of $\ell$ under a type (\Romannum{5}) isotopy.
By virtue of invariance under type (\Romannum{1})-(\Romannum{4})
moves, we may assume that the part of $\ell$ moved by this type (\Romannum{5})
isotopy is close to a vertex of the marking edges. More precisely:
we assume there is a 3-ball centered at the vertex of the marking
so that $\ell$ intersects this 3-ball only in the part moved by the
type (\Romannum{5}) isotopy. (If $\ell$ moves across a singular
leaf in the interior of an internal face $f$, then the vertex is
the sink at the barycenter of $f$; if $\ell$ moves across a singular
leaf which is half of an internal edge $e$, then the vertex is the
source at the barycenter of $e$). Moreover, we can also assume that
the tangles before and after the isotopy are in general positions.
We analyze the two cases separately. 

Case 1: If $\ell$ and $\ell^{\prime}$ are related by a type (\Romannum{5})
isotopy as in the left and right hand side of Figure \ref{fig:Move l across a singular leaf (1)}.
We compute, using skein relations:
\vspace{-1cm}
\begin{eqnarray*}
\sigma(\ell) & =&(-A^{3})^{-\frac{1}{2}}\Biggl[(-A^{2})^{\frac{1}{2}}\includegraphics[valign=c , scale=0.15]{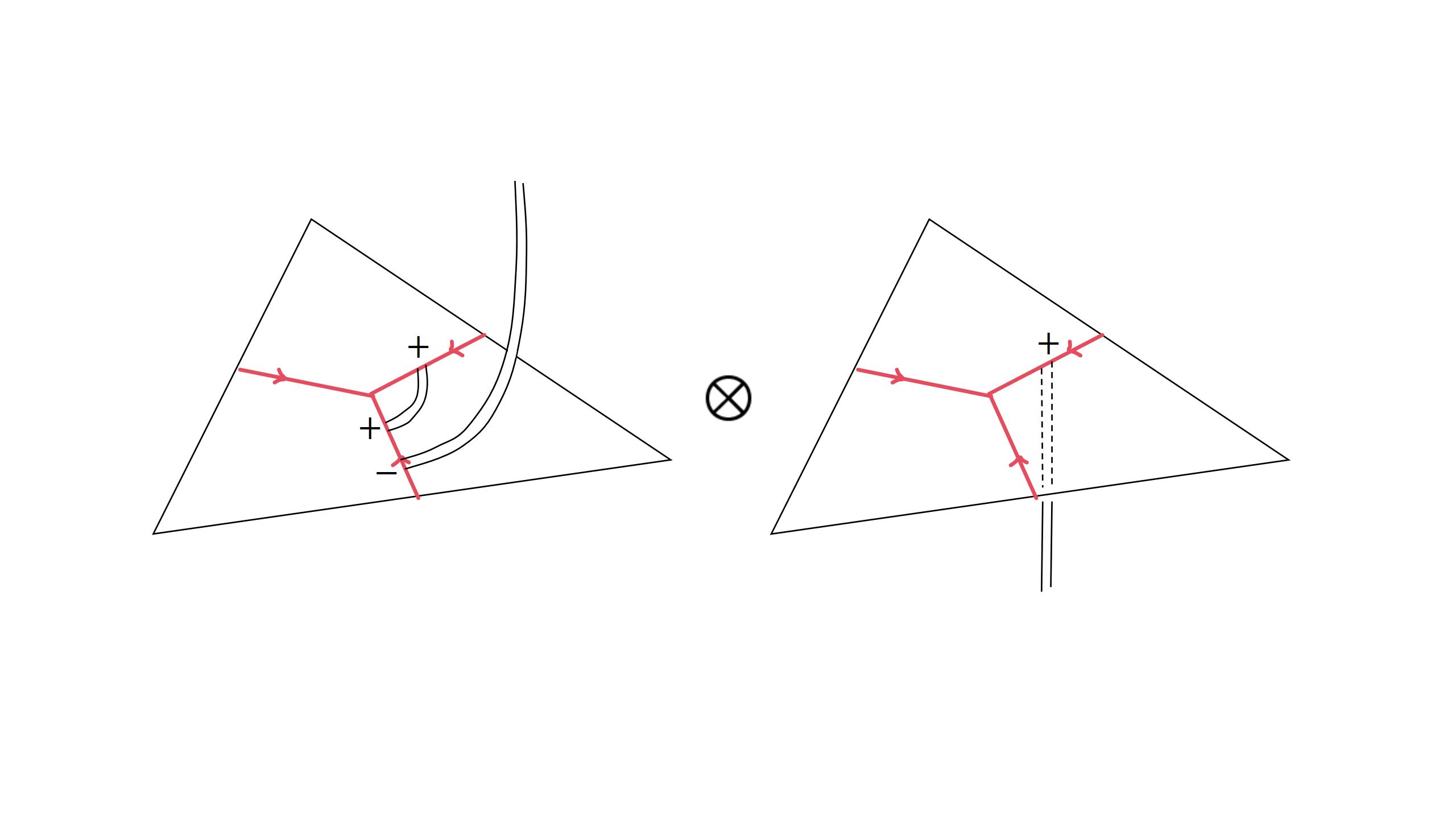}\\[-2cm]
& &+(-A^{2})^{-\frac{1}{2}}\includegraphics[valign=c , scale=0.15]{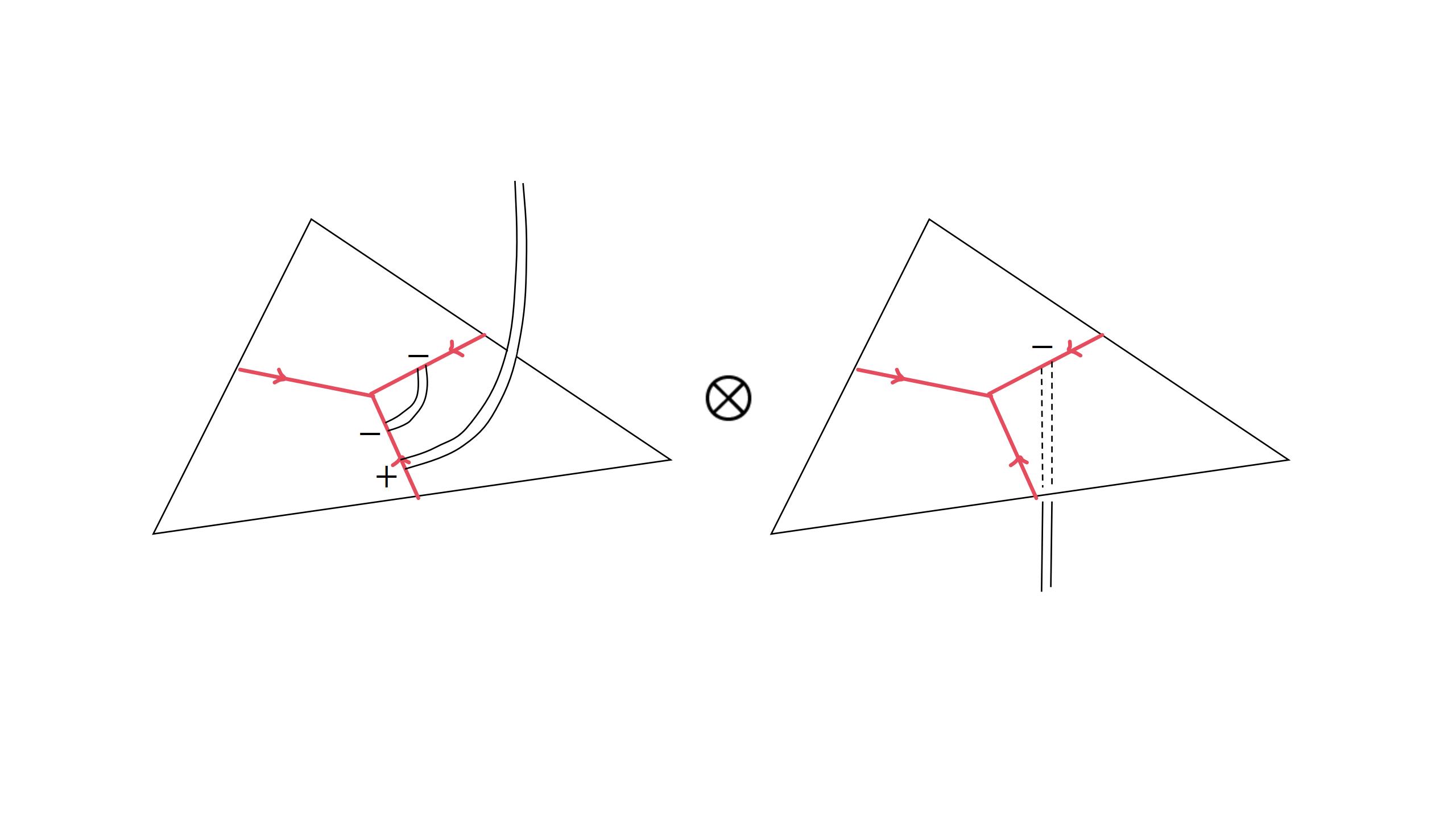}\\[-2cm]
& &+(-A^{2})^{\frac{1}{2}}\includegraphics[valign=c , scale=0.15]{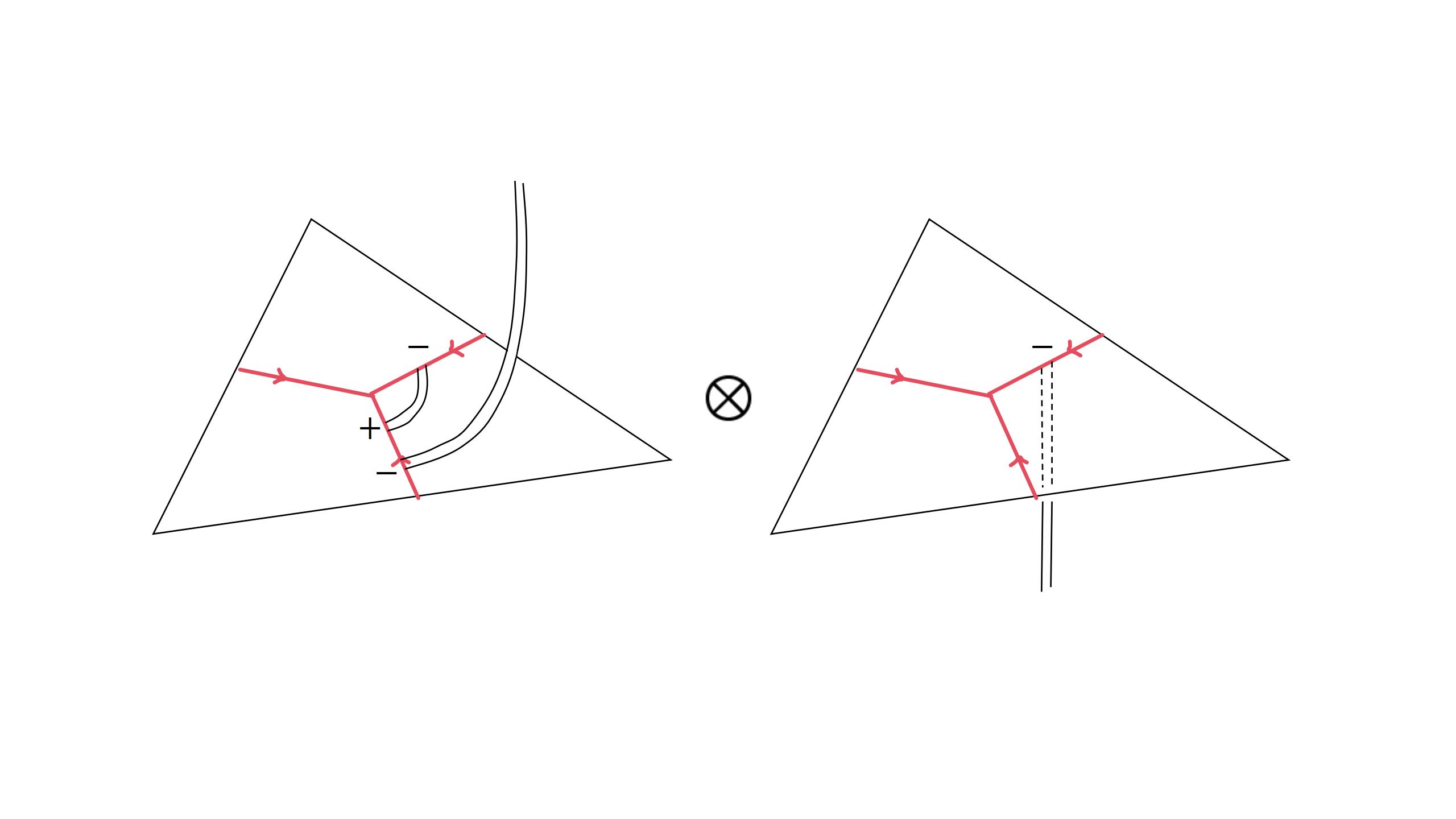}\Biggr]
\end{eqnarray*}
\vspace{-2cm}
and
\begin{eqnarray*}
\sigma(\ell^\prime) & =&(-A^{3})^{-\frac{1}{2}}\Biggl[(-A^{2})^{-\frac{1}{2}}\includegraphics[valign=c , scale=0.15]{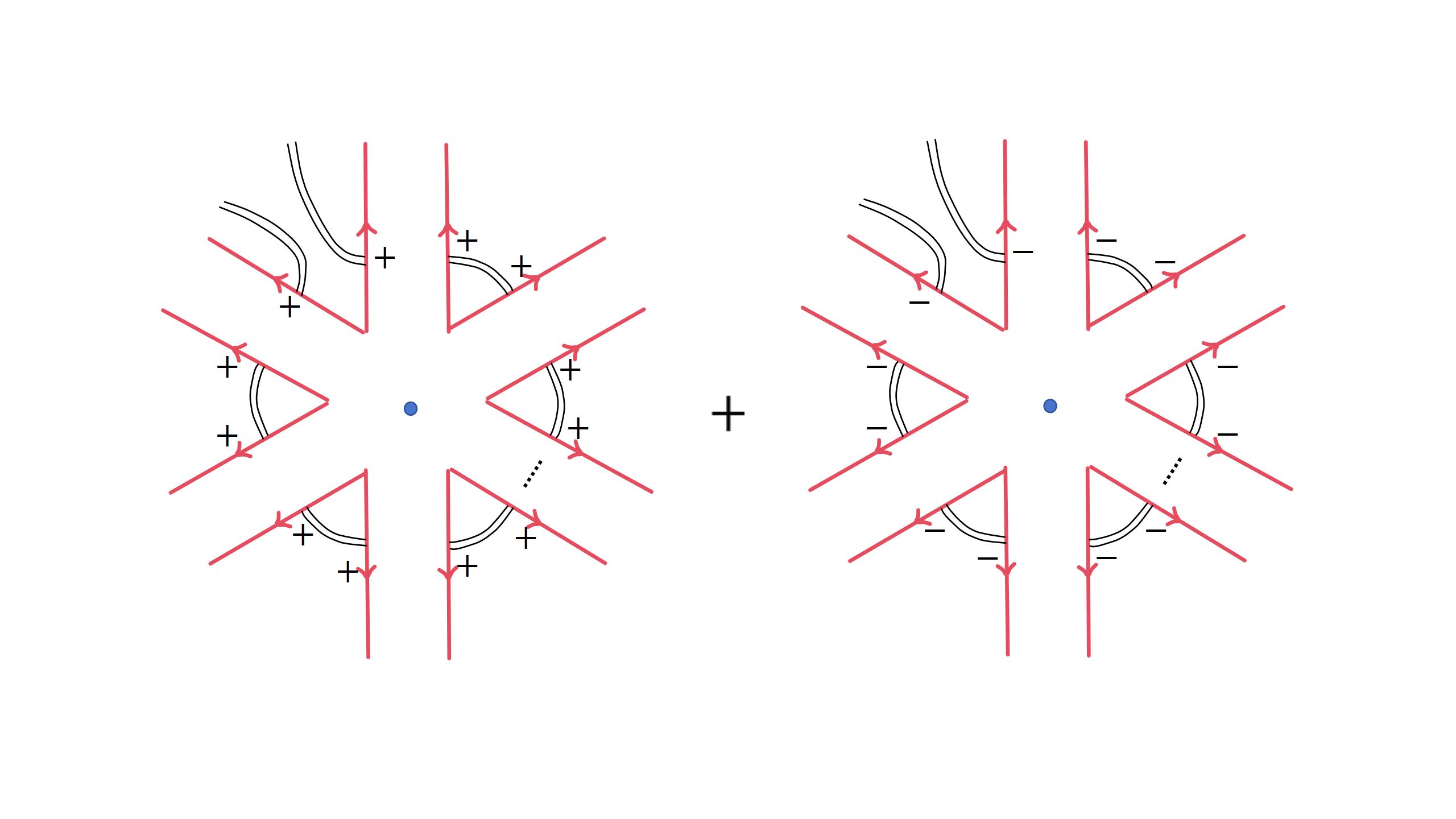}\\[-1cm]
& &+(-A^{2})^{\frac{1}{2}}\includegraphics[valign=c , scale=0.15]{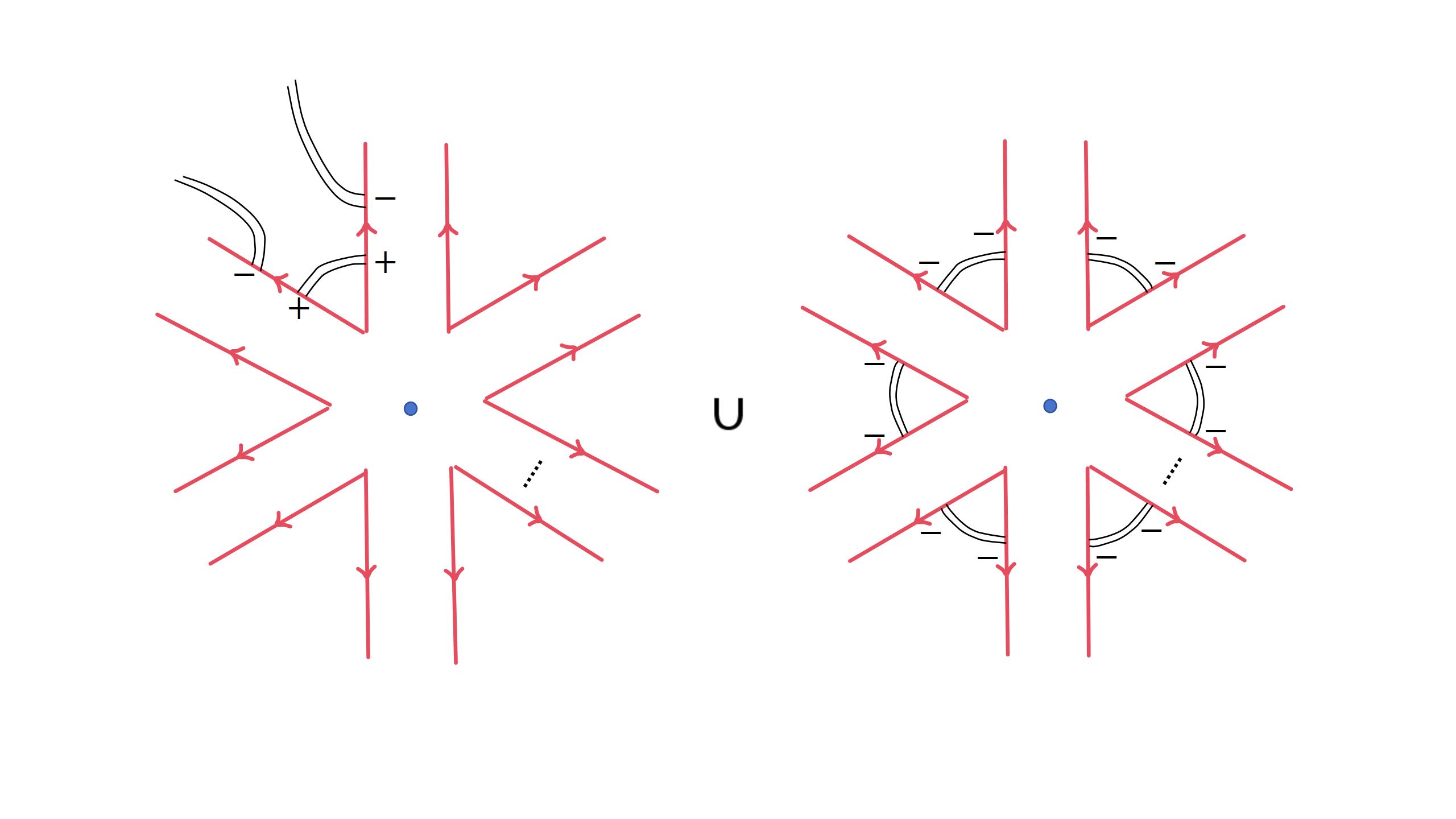}\\[-2cm]
& &+(-A^{2})^{\frac{1}{2}}\includegraphics[valign=c , scale=0.15]{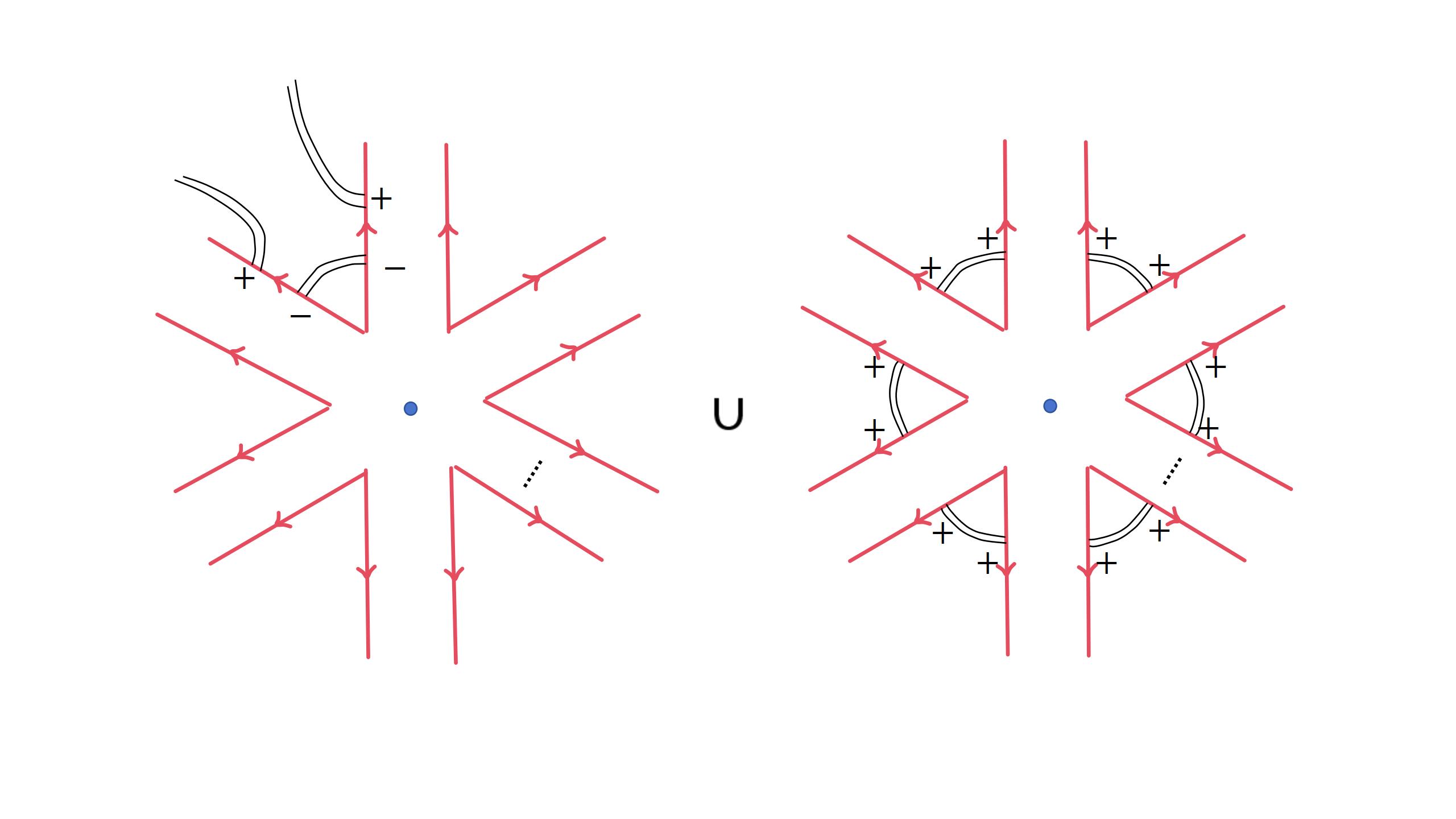}\Biggr]
\end{eqnarray*}
\vspace{-1cm}

We show that $\sigma(\ell)=\sigma(\ell^{\prime})$ as elements of
$\bigotimes_{T\in\mathcal{T}}\overline{\mathrm{Sk}}^{c}(T)$ (we are
using the usual tensor product here, which means that in this case we do not
need to pass to the reduced tensor product). In fact, in $\bigotimes_{T\in\mathcal{T}}\overline{\mathrm{Sk}}^{c}(T)$,
we have
\begin{align*}
 & (-A^{2})^{\frac{1}{2}}\includegraphics[valign=c , scale=0.15]{l1.jpg}\\[-2cm]
= & (-A^{2})^{-\frac{1}{2}}\includegraphics[valign=c , scale=0.15]{L-prime-1.jpg}
\end{align*}
and similarly the identities between the other two pairs of terms on both sides.
To prove this, let $T_{1}$ be the tetrahedron ``above'' the face
$f$ and $T_{2}$ the tetrahedron ``below'' $f$. Let $f_{1}$ and $f_{2}$ be bare faces of $T_{1}$ and $T_{2}$, respectively, that are identified to the face $f$. We further
assume that the isotopy from $\ell$ to $\ell^{\prime}$ happens inside
the union of $T_{1}$ and $T_{2}$. Therefore $\sigma(\ell)$ and
$\sigma(\ell^{\prime})$ differ at most in $\overline{\mathrm{Sk}}^{c}(T_{1})\otimes\overline{\mathrm{Sk}}^{c}(T_{2})$.
So for simplicity we denote
\[
\includegraphics[valign=c , scale=0.32]{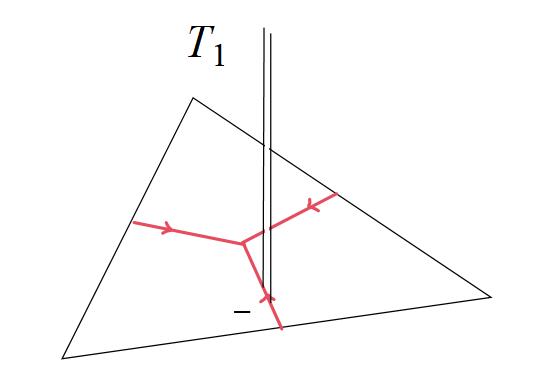}\in\overline{\mathrm{Sk}}(T_{1})\text{ by }\ell_{1},\ \includegraphics[valign=c , scale=0.32]{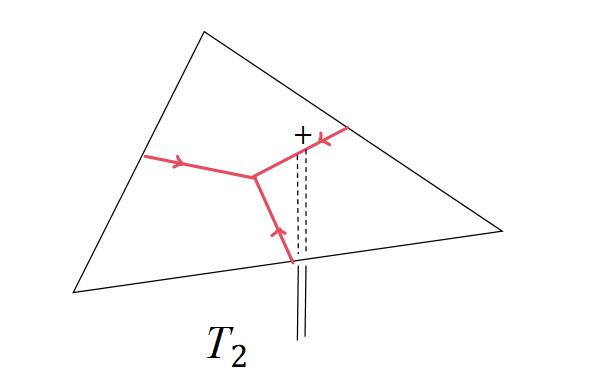}\in\overline{\mathrm{Sk}}(T_{2})\text{ by }\ell_{2},
\]
\[
\includegraphics[valign=c , scale=0.32]{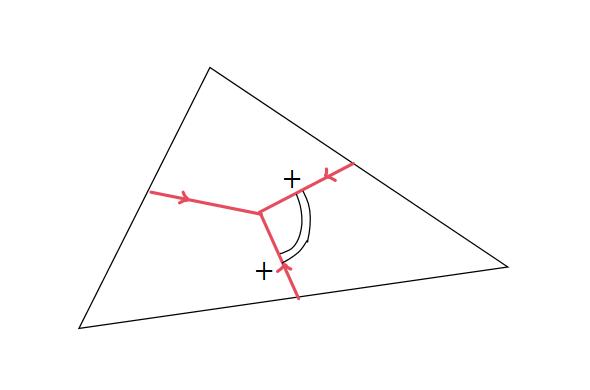}\in\mathbb{T}_{f_{1}}\text{ by }\Gamma_{++},\ \includegraphics[valign=c , scale=0.32]{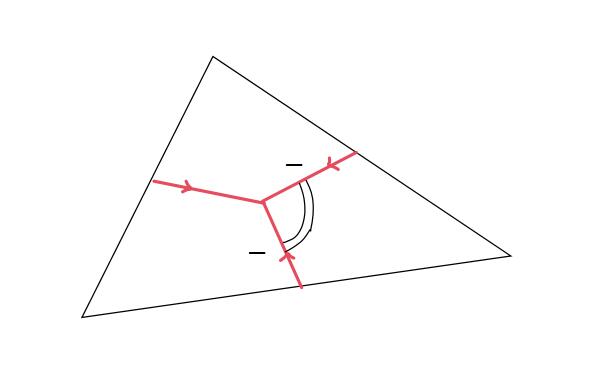}\in\mathbb{T}_{f_{2}}\text{ by }\Gamma_{--}
\]
and we need to show that 
\[
\left((-A^{2})^{\frac{1}{2}}\Gamma_{++}\cup\ell_{1}\right)\otimes\ell_{2}=\ell_{1}\otimes\left((-A^{2})^{-\frac{1}{2}}\Gamma_{--}\cup\ell_{2}\right)
\]
holds in $\overline{\mathrm{Sk}}^{c}(T_{1})\otimes\overline{\mathrm{Sk}}^{c}(T_{2})$.
Inside the corner-reduced module $\overline{\mathrm{Sk}}^{c}(T_{1})$,
we have
\begin{align*}
(-A^{2})^{\frac{1}{2}}\Gamma_{++}\cup\ell_{1} & =(-A^{2})^{\frac{1}{2}}A^{\frac{1}{2}\langle d_{f_{1}}(\Gamma_{++}),d_{f_{1}}(\ell_{1})\rangle_{f_{1}}}\Gamma_{++}\cdot\ell_{1}\\
 & =A^{\frac{1}{2}\langle d_{f_{1}}(\Gamma_{++}),d_{f_{1}}(\ell_{1})\rangle_{f_{1}}}\ell_{1}\\
 & =A^{\frac{1}{2}\langle d_{f_{1}}(\Gamma_{++})+d_{f_{1}}(\ell_{1}),d_{f_{1}}(\ell_{1})\rangle_{f_{1}}}\ell_{1}\\
 & =A^{\frac{1}{2}\langle d_{f_{2}}(\ell_{2}),d_{f_{1}}(\ell_{1})\rangle_{f_{1}}}\ell_{1}.
\end{align*}
In the above calculation, we have used the fact that in $\overline{\mathrm{Sk}}^{c}(T_{1})$,
the left $\cdot$-multiplication of $\Gamma_{++}$ is identified with
multiplication of the scalar $(-A^{2})^{-\frac{1}{2}}$. We have also implicitly used a "matching labeling" of the marking edges in $f_{1}$ and $f_{2}$, see Figure \ref{fig:matching labeling of marking edges}, therefore we have common identifications $\mathbb{Z}^{M_{f_{1}}}=\mathbb{Z}^{3}=\mathbb{Z}^{M_{f_{2}}}$ and therefore the identities $d_{f_{1}}(\Gamma_{++})+d_{f_{1}}(\ell_{1})=d_{f_{2}}(\ell_{2})$ and $\langle d_{f_{2}}(\ell_{2}),d_{f_{1}}(\ell_{1})\rangle_{f_{1}}$. Similar computation shows that in $\overline{\mathrm{Sk}}^{c}(T_{2})$ we have
\[
(-A^{2})^{-\frac{1}{2}}\Gamma_{--}\cup\ell_{1}=A^{\frac{1}{2}\langle d_{f_{1}}(\ell_{1}),d_{f_{2}}(\ell_{2})\rangle_{f_{2}}}\ell_{2}
\]
Thus the desired identity is reduced to
\[
A^{\frac{1}{2}\langle d_{f_{2}}(\ell_{2}),d_{f_{1}}(\ell_{1})\rangle_{f_{1}}}\ell_{1}\otimes\ell_{2}=A^{\frac{1}{2}\langle d_{f_{1}}(\ell_{1}),d_{f_{2}}(\ell_{2})\rangle_{f_{2}}}\ell_{1}\otimes\ell_{2}.
\]
However, we have
\begin{align*}
\langle d_{f_{2}}(\ell_{2}),d_{f_{1}}(\ell_{1})\rangle_{f_{1}} & =-\langle d_{f_{1}}(\ell_{1}),d_{f_{2}}(\ell_{2})\rangle_{f_{2}}\\
 & =\langle d_{f_{1}}(\ell_{1}),d_{f_{2}}(\ell_{2})\rangle_{f_{1}},
\end{align*}
where, in the last identity, we used the fact that when using matching labelings of marking edges, the matrix of the skew-bilinear form associated to the two matching bare faces have opposite signs (see the discussion in Remark \ref{rem:matching face, opposite skew-symmetric form}). This handles Case 1.

Case 2: If $\ell$ and $\ell^{\prime}$ are related by a type (\Romannum{5})
isotopy as in the left and right hand side of Figure \ref{fig:Move l across a singular leaf (2)}.
We have
\begin{align*}
\sigma(\ell^{\prime}) & =\includegraphics[valign=c , scale=0.15]{L-prime-1.jpg}\\[-1.2cm]
 & =\includegraphics[valign=c , scale=0.15]{L-prime-2.jpg}\\[-1.2cm]
 & +\includegraphics[valign=c , scale=0.15]{L-prime-3.jpg}.
\end{align*}
Now the elements 
\[
\includegraphics[valign=c , scale=0.31]{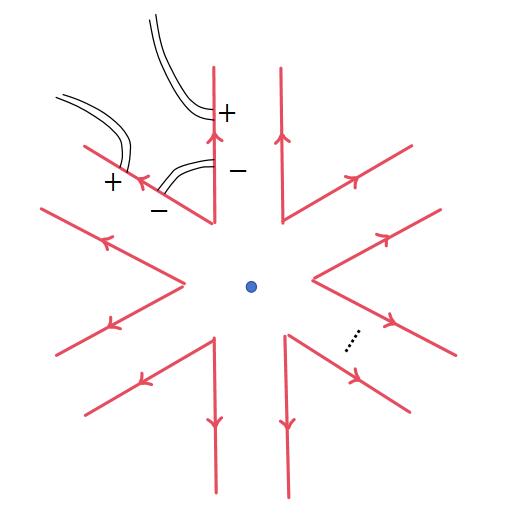}
\text{ and }
\includegraphics[valign=c , scale=0.31]{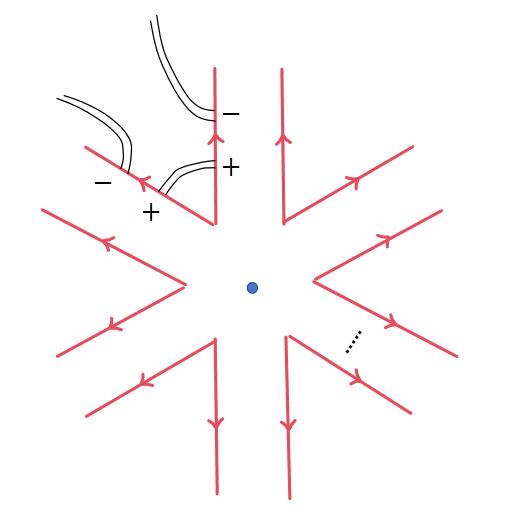}
\]
have the same degree on every marking edges as the element $\sigma(\ell)$,
therefore they are both balanced (because $\sigma(\ell)$ is balanced),
thus modulo elements of the $R$-submodule $\mathfrak{R}_{E}$ (see
Definition \ref{def:reduced tensor product}) we find that 
\begin{align*}
\sigma(\ell^{\prime}) & =(-A^2)\includegraphics[valign=c , scale=0.31]{L-prime-4.jpg}+(-A^2)^{-1}\includegraphics[valign=c , scale=0.31]{L-prime-5.jpg}\\
 & =\includegraphics[valign=c , scale=0.31]{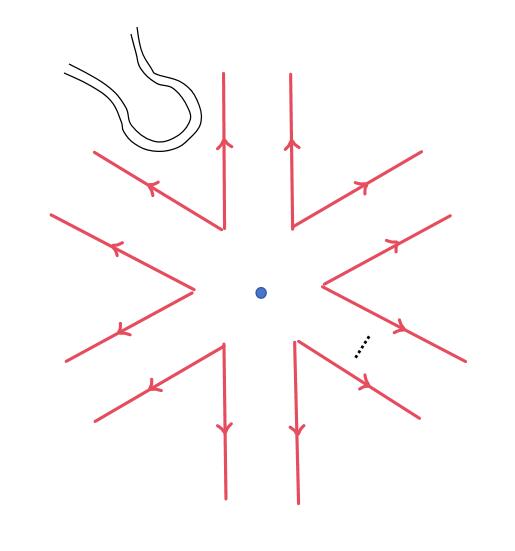}=\sigma(\ell).
\end{align*}

Now we have proved that the map $\sigma$ is well-defined on isotopy
classes of stated tangles in $Y$ and thus gives us a well-defined
splitting homomorphism
\[
\sigma\colon\mathrm{Sk}(Y)\rightarrow\overline{\bigotimes_{T\in\mathcal{T}}}\overline{\mathrm{Sk}}^{c}(T).
\]

To see that it descends to the reduced skein module $\overline{\mathrm{Sk}}(Y)$,
we let $b$ be a bad arc associated to a vertex of the boundary marking,
say a source; then by definition of $\sigma$ we see that for any
stated tangle $\ell$ in $Y$, we have $\sigma(\ell\cup b)=\sigma(\ell)\cup\sigma(b)$.
Now one of the tensor components of $\sigma(b)$ must be a bad arc, therefore
the element $\sigma(\ell)\cup\sigma(b)=0\in\overline{\bigotimes}_{T\in\mathcal{T}}\overline{\mathrm{Sk}}^{c}(T)$. Hence, $\sigma$ descends to $\overline{\mathrm{Sk}}(Y)$.

To see that $\sigma$ further descends to the corner-reduced skein
module $\overline{\mathrm{Sk}}^{c}(Y)$, we let $\Gamma$ be any generator
of the algebra $\mathbb{T}_{f}$ associated to a boundary face $f$. We have
\[
\sigma\left(\Gamma\cdot\ell-(-A^{2})^{-\frac{1}{2}}\ell\right)=\Gamma\cdot\sigma(\ell)-(-A^{2})^{-\frac{1}{2}}\sigma(\ell)=0.
\]
Therefore, $\sigma$ respects the defining relation of the corner-reduced
module and thus descends to $\overline{\mathrm{Sk}}^{c}(Y)$
\end{proof}
\endgroup

\begin{rem}
By the general machinery of splitting homomorphism developed in \cite{PP1},
the decomposition of $Y$ into ideal tetrahedra induces a splitting
homomorphism (see \cite[Corollary 3.29]{PP1})
\[
\sigma\colon\overline{\mathrm{Sk}}(Y)\rightarrow\underset{T\in\mathcal{T}}{\overline{\bigotimes}}\overline{\mathrm{Sk}}(T)
\]
in the reduced skein module setting. The reduced tensor product $\underset{T\in\mathcal{T}}{\overline{\bigotimes}}\overline{\mathrm{Sk}}(T)$
in this setting is the quotient of the usual tensor product $\underset{T\in\mathcal{T}}{\bigotimes}\overline{\mathrm{Sk}}(T)$
by the relations coming from both gluing around the internal edges of the triangulation $\mathcal{T}$ and the relations coming from pairing faces of ideal tetrahedra. However,
in the corner-reduced setting, the reduced tensor product $\underset{T\in\mathcal{T}}{\overline{\bigotimes}}\overline{\mathrm{Sk}}^{c}(T)$
only involves relations coming from gluing around the internal edges.
\end{rem}

\section{Quantum Trace Map\label{sec:Quantum Trace Map}}

\subsection{Quantum trace map for a single tetrahedron\label{subsec:quantum trace map on T}}
We construct the quantum trace map 
\[
Tr_{T}\colon\overline{\mathrm{Sk}}(T)\rightarrow\hat{\mathcal{G}}(T)
\]
for a single tetrahedron $T$ as promised in the introduction. We
then show that it induces an isomorphism from the corner-reduced module
$\overline{\mathrm{Sk}}^{c}(T)$ onto $\mathcal{\hat{\mathcal{G}}}(T)$.

First we carefully 
define the codomain $\hat{\mathcal{G}}(T)$.
\begin{defn}
\label{def: quantum module of T}
Let $T$ be an ideal tetrahedron, with its edges labeled by shape
parameters $z^{\boxempty},y^{\boxempty}$ as in \ref{subsec:Reduced skein module of T}.
The \emph{quantum module} $\hat{\mathcal{G}}(T)$ of $T$ is defined to
be the $R$-module quotient
\[
\hat{\mathcal{G}}(T):=\frac{\mathbb{T}\left\langle T\right\rangle}{V_{T}+L_{T}}
\]
where
\begin{itemize}
\item $\mathbb{T}\left\langle T\right\rangle$ is the quantum torus $(\mathbb{B}^{\otimes6},\cdot)=\left(\bigotimes_{e\in\mathbf{e}(T)}\mathbb{B}_{e},\cdot\right)$,
which we have seen in Lemma \ref{lem:presntation of (T^tensor4,cdot)  and (B^tensor6, cdot)}
with presentation
\[
\mathbb{T}\left\langle T\right\rangle:=\frac{R\langle\hat{z}^{\pm1},\hat{z}^{\prime\pm1},\hat{z}^{\prime\prime\pm1},\hat{y}^{\pm1},\hat{y}^{\prime\pm1},\hat{y}^{\prime\prime\pm1}\rangle}{\left\langle \begin{array}{c}
\hat{a}\cdot\hat{b}^{\prime}=A\hat{b}^{\prime}\cdot\hat{a}\\
\hat{a}^{\prime}\cdot\hat{b}^{\prime\prime}=A\hat{b}^{\prime\prime}\cdot\hat{a}^{\prime}\\
\hat{a}^{\prime\prime}\cdot\hat{b}=A\hat{b}\cdot\hat{a}^{\prime\prime}
\end{array}\Bigg|a,b\in\{z,y\}\right\rangle }.
\]
\item $V_{T}$ is the right ideal of $\mathbb{T}\left\langle T\right\rangle$ generated by \emph{vertex
relations}:
\begin{align*}
&[\hat{z}\cdot\hat{z}^{\prime}\cdot\hat{z}^{\prime\prime}]-(-A^{2})^{\frac{1}{2}}, \ \ [\hat{z}\cdot\hat{y}^{\prime}\cdot\hat{y}^{\prime\prime}]-(-A^{2})^{\frac{1}{2}},\\
&[\hat{y}\cdot\hat{z}^{\prime}\cdot\hat{y}^{\prime\prime}]-(-A^{2})^{\frac{1}{2}},\ \ [\hat{y}\cdot\hat{y}^{\prime}\cdot\hat{z}^{\prime\prime}]-(-A^{2})^{\frac{1}{2}}.
\end{align*}
(Recall that the Weyl-ordering $[\hat{z}\cdot\hat{z}^{\prime}\cdot\hat{z}^{\prime\prime}]$,
for instance, is nothing but the $\cup$-product of skeins $\hat{z}\hat{z}^{\prime}\hat{z}^{\prime\prime}=\hat{z}\cup\hat{z}^{\prime}\cup\hat{z}^{\prime\prime}$.)
\item $L_{T}$ is the right ideal generated by \emph{lagrangian relation}:
\[
\hat{z}^{2}+\hat{z}^{\prime\prime-2}-1.
\]
\end{itemize}
\end{defn}
By definition, elements of $\hat{\mathcal{G}}(T)$ are given by $R$-linear
combinations of monomials (under $\cdot$-product) in the variables
$\hat{z}^{\boxempty}$ and $\hat{y}^{\boxempty}$. Note that if one
takes $R=\mathbb{C}$ and specialize at $A=1$, then the algebra $\hat{\mathcal{G}}_{A=1}(T)$
is the quotient of the Laurent polynomial algebras whose variables are
the shape parameters $z,z^{\prime},z^{\prime\prime},y,y^{\prime},y^{\prime\prime}$
satisfying the classical vertex and lagrangian equations ((\ref{eq:classical vertex equation})
and (\ref{eq:classical lagrangian equation})). For this reason, the
elements $\hat{z}^{\boxempty}$ and $\hat{y}^{\boxempty}$ (which
are skeins by definition) shall be referred to as the \emph{quantized
shape parameters} of the ideal tetrahedron $T$.

Some comments about the quantum module $\hat{\mathcal{G}}(T)$:
\begin{enumerate}
\item Although $\hat{\mathcal{G}}(T)$ is defined as an $R$-module, it is also a right
$\mathbb{T}\left\langle T\right\rangle$-module as it is defined as a quotient of $\mathbb{T}\left\langle T\right\rangle$
by a right $\mathbb{T}\left\langle T\right\rangle$ ideal. 
\item The vertex relations are central in the algebra $\mathbb{T}\left\langle T\right\rangle$,
for example one can check that $[\hat{z}\cdot\hat{z}^{\prime}\cdot\hat{z}^{\prime\prime}]$
commutes with every generator of $\mathbb{T}\left\langle T\right\rangle$. 
\item It is possible to use a different lagrangian relation in the definition.
Each vertex contributes three lagrangian relations. However, they are
all equivalent to the one given above by manipulating vertex relations. 
\end{enumerate}

Recall that $\overline{\mathrm{Sk}}(T)$ is a cyclic $\mathbb{T}^{\otimes4}\text{-}\mathbb{B}^{\otimes6}$-bimodule
generated by the empty skein $[\emptyset]$. We start by describing
a map on the level of the $R$-module $\mathbb{T}^{\otimes4}\otimes\mathbb{B}^{\otimes6}$.
Recall that as $R$-modules we have 
\[
\mathbb{T}^{\otimes4}\otimes\mathbb{B}^{\otimes6}=\left(\mathbb{T}^{\otimes4},\cdot\right)\otimes\left(\mathbb{B}^{\otimes6},\cdot\right)=R\left[T \right]\otimes\mathbb{T}\left\langle T\right\rangle.
\]
\begin{defn}
\label{def: definition of formula P}
We define the following maps:
\begin{enumerate}
    \item Let 
$Q\colon R\left[T \right]\rightarrow R$
be the algebra homomorphism which maps each generator $\Gamma_{ab}$
to the scalar $(-A^{2})^{-\frac{1}{2}}$. 
    \item Let $
P\colon R\left[T \right]\otimes\mathbb{T}\left\langle T\right\rangle\rightarrow\mathbb{T}\left\langle T\right\rangle
$
be the $R$-module homomorphism given by the formula
\begin{equation}
P(\Gamma\otimes x)=A^{\frac{1}{2}\langle d(\Gamma),d(x)\rangle}Q(\Gamma)x\label{eq:formula of P}
\end{equation}
when $\Gamma$ and $x$ are monomials (hence they are homogeneous
and have well-defined $\prod_{f\in\mathbf{f}(T)}\mathbb{Z}^{M_{f}}$-degrees)
and extend by bilinearity. 

\end{enumerate}

\end{defn}

The significance of the factor $A^{\frac{1}{2}\langle d(\Gamma),d(x)\rangle}$
in the above definition is explained by the following lemma. Recall
that $R\left[T \right]\otimes\mathbb{T}\left\langle T\right\rangle$ is a
$R\left[T \right]\text{-}\mathbb{T}\left\langle T\right\rangle$-bimodule with
the bimodule structure given in \ref{subsec:corner reduction in the case of a singe tetrahedron},
while
$\mathbb{T}\left\langle T\right\rangle$ is also a $R\left[T \right]\text{-}\mathbb{T}\left\langle T\right\rangle$-bimodule
with the left $R\left[T \right]$-module structure
given by the algebra homomorphism $Q\colon R\left[T \right]\rightarrow R$.

\begin{lem}
\label{lem:P is a bimodule homomorphism}
P is a $R\left[T \right]\text{-}\mathbb{T}\left\langle T\right\rangle$-bimodule
homomorphism.
\end{lem}

\begingroup
\allowdisplaybreaks
\begin{proof}
Let $\Gamma^{\prime},\Gamma\in R\left[T \right]$
and $x,x^{\prime}\in\mathbb{T}\left\langle T\right\rangle$ be $\prod_{f\in\mathbf{f}(T)}\mathbb{Z}^{M_{f}}$-homogeneous.
By the definition of the $R\left[T \right]\text{-}\mathbb{T}\left\langle T\right\rangle$-bimodule
structure on $R\left[T \right]\otimes\mathbb{T}\left\langle T\right\rangle$,
we have
\[\Gamma^{\prime}\cdot\left(\Gamma\otimes x\right)\cdot x^{\prime}
=A^{-\frac{1}{2}\left(\langle d(\Gamma^{\prime}),d(x)\rangle+\langle d(\Gamma^{\prime}),d(x^{\prime})\rangle+\langle d(\Gamma),d(x^{\prime})\rangle\right)}(\Gamma^{\prime}\cdot\Gamma)\otimes(x\cdot x^{\prime}).
\]
Therefore
\begin{eqnarray*}
\lefteqn{P\left(\Gamma^{\prime}\cdot\left(\Gamma\otimes x\right)\cdot x^{\prime}\right)}\\ & = &
A^{-\frac{1}{2}\left(\langle d(\Gamma^{\prime}),d(x)\rangle+\langle d(\Gamma^{\prime}),d(x^{\prime})\rangle+\langle d(\Gamma),d(x^{\prime})\rangle\right)}A^{\frac{1}{2}\langle d(\Gamma^{\prime}\Gamma),d(xx^{\prime})\rangle}Q\left(\Gamma^{\prime}\cdot\Gamma\right)(x\cdot x^{\prime})\\
& =& A^{\frac{1}{2}\langle d(\Gamma),d(x)\rangle}Q(\Gamma^{\prime})Q(\Gamma)(x\cdot x^{\prime})
=Q(\Gamma^{\prime})P\left(\Gamma\otimes x\right)\cdot x^{\prime}.
\end{eqnarray*}
\end{proof}
\endgroup

Now the right ideals $V_{T}$ and $L_{T}$ of $\mathbb{T}\langle T\rangle$
are also naturally left $R[T]$-submodules of $\mathbb{T}\langle T\rangle$
via the algebra homomorphism $Q\colon R[T]\rightarrow R$. Therefore
$V_{T}+L_{T}$ is also a $R[T]\text{-}\mathbb{T}\langle T\rangle$-sub-bimodule
of $\mathbb{T}\langle T\rangle$, and $\hat{\mathcal{G}}(T)$
is also a quotient $R[T]\text{-}\mathbb{T}\langle T\rangle$-bimodule. 

\begin{prop}
\label{prop:P descends to Tr}
The $R\left[T \right]\text{-}\mathbb{T}\left\langle T\right\rangle$-bimodule
homomorphism 
$
P\colon R\left[T \right]\otimes\mathbb{T}\left\langle T\right\rangle\rightarrow\mathbb{T}\left\langle T\right\rangle
$
descends to a well-defined $R\left[T \right]\text{-}\mathbb{T}\left\langle T\right\rangle$-bimodule
homomorphism
\[
Tr_{T}\colon\overline{\mathrm{Sk}}(T)\rightarrow\hat{\mathcal{G}}(T).
\]
\end{prop}

\begingroup
\allowdisplaybreaks
\begin{proof}
By Lemma \ref{lem:generators of Ann(emptyset) are homogeneous, so it is also an ideal in dot-product},
we only need to
show that $P$ maps the generators
\begin{eqnarray*}
v_{zz^{\prime}z^{\prime\prime}} & := & (-A^{2})^{2}\Gamma_{zz^{\prime\prime}}\Gamma_{z^{\prime\prime}z^{\prime}}\Gamma_{z^{\prime}z}-\hat{z}\hat{z}^{\prime}\hat{z}^{\prime\prime} \\
\ell_{zz^{\prime\prime}} & := &[\Gamma_{zy^{\prime}}\Gamma_{y^{\prime}z^{\prime\prime}}^{-1}]-(-A^{2})^{-1}[\Gamma_{z^{\prime\prime}y}\Gamma_{yz^{\prime}}^{-1}]\Gamma_{zz^{\prime}}^{-1}\hat{z}^{\prime\prime-1}\hat{z}^{\prime}\hat{z}
\\
& & -(-A^{2})\Gamma_{z^{\prime\prime}z^{\prime}}[\Gamma_{z^{\prime}y^{\prime\prime}}\Gamma_{y^{\prime\prime}z}^{-1}]\hat{z}^{\prime\prime-1}\hat{z}^{\prime-1}\hat{z}
\end{eqnarray*}
into $V_{T}+L_{T}$. For $v_{zz^{\prime}z^{\prime\prime}}$, we have
\begin{eqnarray*}
P(v_{zz^{\prime}z^{\prime\prime}}) &= &(-A^{2})^{2}Q(\Gamma_{zz^{\prime\prime}}\Gamma_{z^{\prime\prime}z^{\prime}}\Gamma_{z^{\prime}z})-\hat{z}\hat{z}^{\prime}\hat{z}^{\prime\prime} \\
&=&(-A^{2})^{2}Q(\Gamma_{zz^{\prime\prime}}\cdot\Gamma_{z^{\prime\prime}z^{\prime}}\cdot\Gamma_{z^{\prime}z})-\left[\hat{z}\cdot\hat{z}^{\prime}\cdot\hat{z}^{\prime\prime}\right]\\
&=&(-A^{2})^{\frac{1}{2}}-\left[\hat{z}\cdot\hat{z}^{\prime}\cdot\hat{z}^{\prime\prime}\right]\in V_{T}.
\end{eqnarray*}
In the above, we used $\hat{z}\hat{z}^{\prime}\hat{z}^{\prime\prime}=[\hat{z}\cdot\hat{z}^{\prime}\cdot\hat{z}^{\prime\prime}]$;
moreover $\Gamma_{zz^{\prime\prime}}\Gamma_{z^{\prime\prime}z^{\prime}}\Gamma_{z^{\prime}z}=\Gamma_{zz^{\prime\prime}}\cdot\Gamma_{z^{\prime\prime}z^{\prime}}\cdot\Gamma_{z^{\prime}z}$
because $\Gamma_{zz^{\prime\prime}}$, $\Gamma_{z^{\prime\prime}z^{\prime}}$
and $\Gamma_{z^{\prime}z}$ are three generators of 
boundary
face algebras $\mathbb{T}$ associated to different faces of $T$.

For $\ell_{zz^{\prime\prime}}$ we have $P(\ell_{zz^{\prime\prime}})$
equals
\begin{eqnarray*}
&&Q(\left[\Gamma_{zy^{\prime}}\Gamma_{y^{\prime}z^{\prime\prime}}^{-1}\right])-(-A^{2})^{-1}A^{\frac{1}{2}\langle d(\Gamma_{z^{\prime\prime}y}\Gamma_{yz^{\prime}}^{-1}\Gamma_{zz^{\prime}}^{-1}),d(\hat{z}^{\prime\prime-1}\hat{z}^{\prime}\hat{z})\rangle}Q(\left[\Gamma_{z^{\prime\prime}y}\Gamma_{yz^{\prime}}^{-1}\right]\Gamma_{zz^{\prime}}^{-1})\hat{z}^{\prime\prime-1}\hat{z}^{\prime}\hat{z}\\ & &
\ \ \ \ -(-A^{2})A^{\frac{1}{2}\langle d(\Gamma_{z^{\prime\prime}z^{\prime}}\Gamma_{z^{\prime}y^{\prime\prime}}\Gamma_{y^{\prime\prime}z}^{-1},d(\hat{z}^{\prime\prime-1}\hat{z}^{\prime-1}\hat{z})\rangle}Q(\Gamma_{z^{\prime\prime}z^{\prime}}\left[\Gamma_{z^{\prime}y^{\prime\prime}}\Gamma_{y^{\prime\prime}z}^{-1}\right])\hat{z}^{\prime\prime-1}\hat{z}^{\prime-1}\hat{z}\\
&=&Q(\Gamma_{zy^{\prime}}\cdot\Gamma_{y^{\prime}z^{\prime\prime}}^{-1})-(-A^{2})^{-1}Q(\Gamma_{z^{\prime\prime}y}\cdot\Gamma_{yz^{\prime}}^{-1}\cdot\Gamma_{zz^{\prime}}^{-1})\left[\hat{z}^{\prime\prime-1}\cdot\hat{z}^{\prime}\cdot\hat{z}\right]\\ & & 
\ \ \ \ -(-A^{2})Q(\Gamma_{z^{\prime\prime}z^{\prime}}\left[\Gamma_{z^{\prime}y^{\prime\prime}}\Gamma_{y^{\prime\prime}z}^{-1}\right])\hat{z}^{\prime\prime-1}\hat{z}^{\prime-1}\hat{z}\\
&=& 1-(-A^{2})^{-\frac{1}{2}}\left[\hat{z}^{\prime\prime-1}\cdot\hat{z}^{\prime}\cdot\hat{z}\right]-(-A^{2})^{\frac{1}{2}}\left[\hat{z}^{\prime\prime-1}\cdot\hat{z}^{\prime-1}\cdot\hat{z}\right].
\end{eqnarray*}
In the above, we used that 
\[
\langle d(\Gamma_{z^{\prime\prime}y}\Gamma_{yz^{\prime}}^{-1}\Gamma_{zz^{\prime}}^{-1}),d(\hat{z}^{\prime\prime-1}\hat{z}^{\prime}\hat{z})\rangle=\langle d(\Gamma_{z^{\prime\prime}z^{\prime}}\Gamma_{z^{\prime}y^{\prime\prime}}\Gamma_{y^{\prime\prime}z}^{-1}),d(\hat{z}^{\prime\prime-1}\hat{z}^{\prime-1}\hat{z})\rangle=0,
\]
which can be checked by a straightforward computation or looking at
their skein diagrams. Now the element 
\[
1-(-A^{2})^{-\frac{1}{2}}[\hat{z}^{\prime\prime-1}\cdot\hat{z}^{\prime}\cdot\hat{z}]-(-A^{2})^{\frac{1}{2}}[\hat{z}^{\prime\prime-1}\cdot\hat{z}^{\prime-1}\cdot\hat{z}]
\]
is in the ideal $V_{T}+L_{T}$. Indeed, if one manipulates this element
using the vertex relation $(-A^{2})^{\frac{1}{2}}-\left[\hat{z}\cdot\hat{z}^{\prime}\cdot\hat{z}^{\prime\prime}\right]$,
it becomes the lagrangian
\[
1-\hat{z}^{\prime\prime-2}-\hat{z}^{2}.
\]
\end{proof}
\endgroup

\begin{defn}
\label{def:Quantum trace map on T}The map 
\[
Tr_{T}\colon\overline{\mathrm{Sk}}(T)\rightarrow\hat{\mathcal{G}}(T)
\]
given by Proposition \ref{prop:P descends to Tr} is called the \emph{quantum
trace map for the ideal tetrahedron} $T$. It is a $R\left[T \right]\text{-}\mathbb{T}\left\langle T\right\rangle$-bimodule homomorphism.
\end{defn}

The following theorem gives the structural result for the corner-reduced
skein module $\overline{\mathrm{Sk}}^{c}(T)$, which was promised at
the end of \ref{subsec:corner reduction in the case of a singe tetrahedron}.

\begin{thm}
\label{thm:Tr induce isomorphism between corner reduced module and quantum module}
The quantum trace map $Tr_{T}$ is surjective, and its kernel is exactly the $R\left[T \right]\text{-}\mathbb{T}\left\langle T\right\rangle$-sub-bimodule
$I^{c}\cdot\overline{\mathrm{Sk}}(T)$; in other words,
it descends to an isomorphism of $R\left[T \right]\text{-}\mathbb{T}\left\langle T\right\rangle$-bimodules
\[
Tr_{T}^{c}\colon\overline{\mathrm{Sk}}^{c}(T)\overset{\cong}{\rightarrow}\hat{\mathcal{G}}(T).
\]
\end{thm}

In the following proof and subsequent discussions, we will be constantly
dealing with a module $M$ and its quotients or even quotients of
quotients, to ensure smoothness of discussion and avoid cumbersome
notation, we will adopt the following abuse of languages. Let $N$
be a quotient module of $M$:
\begin{itemize}
    \item The element of $N$ represented by $m\in M$ will still
be denoted as $m$. For example we will write $\Gamma\otimes x\in\overline{\mathrm{Sk}}^{c}(T)$
where $\Gamma\in\mathbb{T}^{\otimes4}$ and $x\in\mathbb{B}^{\otimes6}$,
and by that we really mean the class of $\Gamma\otimes x\in\mathbb{T}^{\otimes4}\otimes\mathbb{B}^{\otimes6}$
in $\overline{\mathrm{Sk}}^{c}(T)$.

    \item For $m_{1},m_{2}\in M$, we write ``$m_{1}=m_{2}\ \text{in}\ N$''
or ``in $N$, we have $m_{1}=m_{2}$'' to mean that
$m_{1}$ and $m_{2}$ have the same class in $N$.
\end{itemize}

\begingroup
\allowdisplaybreaks
\begin{proof}
First we find that $Tr_{T}$ maps the $R[T]$-$\mathbb{T}\langle T\rangle$-sub-bimodule
$I^{c}\cdot\overline{\mathrm{Sk}}(T)$ to $\{0\}$. This is in fact quite
obvious, recall that the ideal $I^{c}$ of $R[T]$ is generated by
elements of the form
\[
\Gamma_{ab}-(-A^{2})^{-\frac{1}{2}}.
\]
Moreover, $Tr_{T}$ is left $R[T]$-linear, thus for any $w\in\overline{\mathrm{Sk}}(T)$
we have
\begin{align*}
Tr_{T}\left(\left(\Gamma_{ab}-(-A^{2})^{-\frac{1}{2}}\right)\cdot w\right) & =Q\left(\Gamma_{ab}-(-A^{2})^{-\frac{1}{2}}\right)Tr_{T}(w)\\
 & =\left(Q(\Gamma_{ab})-(-A^{2})^{-\frac{1}{2}}\right)Tr_{T}(w)\\
 & =0.
\end{align*}
This shows that $Tr_{T}$ maps $I^{c}\cdot\overline{\mathrm{Sk}}(T)$
to $\{0\}$ and hence it induces a well-defined $R[T]$-$\mathbb{T}\langle T\rangle$-bimodule
homomorphism
\[
Tr_{T}^{c}\colon\overline{\mathrm{Sk}}^{c}(T)=\frac{\overline{\mathrm{Sk}}(T)}{I^{c}\cdot\overline{\mathrm{Sk}}(T)}\rightarrow\hat{\mathcal{G}}(T).
\]

To prove that $Tr_{T}^{c}$ is a bijection, we construct the inverse
map. Let $F$ be composition of the maps
\[
F\colon\mathbb{T}\left\langle T\right\rangle\rightarrow R\left[T \right]\otimes\mathbb{T}\left\langle T\right\rangle\twoheadrightarrow\overline{\mathrm{Sk}}^{c}(T),
\]
where $\mathbb{T}\left\langle T\right\rangle\rightarrow R\left[T \right]\otimes\mathbb{T}\left\langle T\right\rangle$
is the natural map $x\mapsto1\otimes x$. In other word $F$ maps the element $x\in\mathbb{T}\left\langle T\right\rangle$
to the element of $\overline{\mathrm{Sk}}^{c}(T)$ represented by $1\otimes x$ in $R\left[T \right]\otimes\mathbb{T}\left\langle T\right\rangle$.

We claim that $F$ maps $V_{T}+L_{T}$ to $\{0\}$.
Indeed, $F$ is a right $\mathbb{T}\left\langle T\right\rangle$-module homomorphism because both the natual map $\mathbb{T}\left\langle T\right\rangle\rightarrow R\left[T \right]\otimes\mathbb{T}\left\langle T\right\rangle$ and the quotient map $R\left[T \right]\otimes\mathbb{T}\left\langle T\right\rangle\twoheadrightarrow\overline{\mathrm{Sk}}^{c}(T)$ are right $\mathbb{T}\left\langle T\right\rangle$-module homomorphisms.
Hence to show $F$ maps $V_{T}+L_{T}$ to $\{0\}$ we only needs to show that it maps the generators of $V_{T}+L_{T}$ to $0$. In $\overline{\mathrm{Sk}}^{c}(T)$, the following holds:
\begin{align*}
F\left([\hat{z}\cdot\hat{z}^{\prime}\cdot\hat{z}^{\prime\prime}]-(-A^{2})^{\frac{1}{2}}\right)
&=[\hat{z}\cdot\hat{z}^{\prime}\cdot\hat{z}^{\prime\prime}]-(-A^{2})^{\frac{1}{2}}
\\
&=\hat{z}\hat{z}^{\prime}\hat{z}^{\prime\prime}-(-A^{2})^{2}\Gamma_{zz^{\prime\prime}}\Gamma_{z^{\prime\prime}z^{\prime}}\Gamma_{z^{\prime}z}
\\
&=0
\end{align*}
In the above, we have used the fact that $[\hat{z}\cdot\hat{z}^{\prime}\cdot\hat{z}^{\prime\prime}]=\hat{z}\hat{z}^{\prime}\hat{z}^{\prime\prime}$
and the fact that inside $\overline{\mathrm{Sk}}^{c} (T)$, we have
\[
\Gamma_{zz^{\prime\prime}}\Gamma_{z^{\prime\prime}z^{\prime}}\Gamma_{z^{\prime}z}
=\Gamma_{zz^{\prime\prime}}\cdot \Gamma_{z^{\prime\prime}z^{\prime}}\cdot \Gamma_{z^{\prime}z}
=(-A^{2})^{-\frac{3}{2}}.
\]

The other generator
\[
\bigl(1-\hat{z}^{\prime\prime-2}-\hat{z}^{2}\bigr)
\]
 is equivalent to 
\[
1-(-A^{2})^{-\frac{1}{2}}[\hat{z}^{\prime\prime-1}\cdot\hat{z}^{\prime}\cdot\hat{z}]-(-A^{2})^{\frac{1}{2}}[\hat{z}^{\prime\prime-1}\cdot\hat{z}^{\prime-1}\cdot\hat{z}],
\]
and we have, in $\overline{\mathrm{Sk}}^c (T)$:
\begin{eqnarray*}
\lefteqn{F\left(1-(-A^{2})^{-\frac{1}{2}}[\hat{z}^{\prime\prime-1}\cdot\hat{z}^{\prime}\cdot\hat{z}]-(-A^{2})^{\frac{1}{2}}[\hat{z}^{\prime\prime-1}\cdot\hat{z}^{\prime-1}\cdot\hat{z}]\right)}
\\
&=&
1-(-A^{2})^{-\frac{1}{2}}[\hat{z}^{\prime\prime-1}\cdot\hat{z}^{\prime}\cdot\hat{z}]-(-A^{2})^{\frac{1}{2}}[\hat{z}^{\prime\prime-1}\cdot\hat{z}^{\prime-1}\cdot\hat{z}]
\\
&=&\Gamma_{zy^{\prime}}\cdot\Gamma_{y^{\prime}z^{\prime\prime}}^{-1}-(-A^{2})^{-1}(\Gamma_{z^{\prime\prime}y}\cdot\Gamma_{yz^{\prime}}^{-1}\cdot\Gamma_{zz^{\prime}}^{-1})\cdot[\hat{z}^{\prime\prime-1}\cdot\hat{z}^{\prime}\cdot\hat{z}]
\\
& &-(-A^{2})(\Gamma_{z^{\prime\prime}z^{\prime}}\cdot\Gamma_{z^{\prime}y^{\prime\prime}}\cdot\Gamma_{y^{\prime\prime}z}^{-1})\cdot[\hat{z}^{\prime\prime-1}\cdot\hat{z}^{\prime-1}\cdot\hat{z}]
\\
&=&[\Gamma_{zy^{\prime}}\Gamma_{y^{\prime}z^{\prime\prime}}^{-1}]-(-A^{2})^{-1}[\Gamma_{z^{\prime\prime}y}\Gamma_{yz^{\prime}}^{-1}]\Gamma_{zz^{\prime}}^{-1}\hat{z}^{\prime\prime-1}\hat{z}^{\prime}\hat{z}\\
& &-(-A^{2})\Gamma_{z^{\prime\prime}z^{\prime}}[\Gamma_{z^{\prime}y^{\prime\prime}}\Gamma_{y^{\prime\prime}z}^{-1}]\hat{z}^{\prime\prime-1}\hat{z}^{\prime-1}\hat{z}\bigr)\\
&=&0
\end{eqnarray*}
In the above, we used the fact that 
\[
\langle d(\Gamma_{z^{\prime\prime}y}\Gamma_{yz^{\prime}}^{-1}\Gamma_{zz^{\prime}}^{-1}),d(\hat{z}^{\prime\prime-1}\hat{z}^{\prime}\hat{z})\rangle=\langle d(\Gamma_{z^{\prime\prime}z^{\prime}}\Gamma_{z^{\prime}y^{\prime\prime}}\Gamma_{y^{\prime\prime}z}^{-1}),d(\hat{z}^{\prime\prime-1}\hat{z}^{\prime-1}\hat{z})\rangle=0.
\]

We have shown that $F\colon\mathbb{T}\left\langle T\right\rangle\rightarrow\overline{\mathrm{Sk}}^{c}(T)$
maps $V_{T}+L_{T}$ to $\{0\}$, hence it descends to a right $\mathbb{T}\langle T\rangle$-module
homomorphism
\[
F^{c}\colon\hat{\mathcal{G}}(T)\rightarrow\overline{\mathrm{Sk}}^{c}(T).
\]
We now show that $F^{c}$ is the inverse map of $Tr_{T}^{c}$.
Let $x\in\hat{\mathcal{G}}(T)$ be represented by some $x\in\mathbb{T}\left\langle T\right\rangle$,
then in $\hat{\mathcal{G}}(T)$, we have
\begin{align*}
Tr_{T}^{c}F^{c}(x) & =Tr_{T}^{c}F(x)\\
 & =Tr_{T}^{c}(1\otimes x)\\
 & =P(1\otimes x)\\
 & =x.
\end{align*}
Conversely, elements of $\overline{\mathrm{Sk}}^{c}(T)$ are given by
sums of elements of the form $\Gamma\otimes x$ where $\Gamma\in R\left[\Gamma_{ab}^{\pm 1} \right]$
and $x\in \mathbb{T}\left\langle T\right\rangle$ are homogeneous, then in $\overline{\mathrm{Sk}}^c (T)$, we have
\begin{align*}
F^{c}Tr_{T}^{c}(\Gamma\otimes x) & =F^{c}\bigl(P(\Gamma\otimes x)\bigr)\\
 &=A^{\frac{1}{2}\langle d(\Gamma),d(x)\rangle}Q(\Gamma)(1\otimes x)\\
 & =A^{\frac{1}{2}\langle d(\Gamma),d(x)\rangle}\Gamma \cdot (1\otimes x)\\
 & =A^{\frac{1}{2}\langle d(\Gamma),d(x)\rangle}A^{-\frac{1}{2}\langle d(\Gamma),d(x)\rangle}\Gamma\otimes x\\
 & =\Gamma\otimes x
\end{align*}
Therefore the maps $Tr_{T}^{c}$ and $F^c$ are inverse of each other and thus $Tr_{T}^{c}$ is a bijection.
\end{proof}
\endgroup

\subsection{Quantum trace map for the 3-manifold.
\label{subsec:quantum trace map on Y}}
Let $(Y,\mathcal{T})$ be an ideally triangulated boundary marked
3-manifold, we use the splitting homomorphism to combine the quantum
trace map $Tr_{T}$ for each ideal tetrahedron $T\in\mathcal{T}$
into a quantum trace map for the 3-manifold $Y$. First, let us properly
define the codomain. Recall in Definition \ref{def:reduced tensor product}
we have the element $\hat{e}\in\bigotimes_{T\in\mathcal{T}}\left(\bigotimes_{e\in\mathbf{e}(T)}\mathbb{B}_{e}\right)=\bigotimes_{T\in\mathcal{T}}\mathbb{T}\left\langle T\right\rangle$
for an internal edge $e\in\mathcal{T}^{(1)}$, which is given by
\begin{equation}
\label{eqn:definition of e^hat}
\hat{e}=[\hat{e}_{1}\cdot\hat{e}_{2}\cdot\ \dots\ \cdot\hat{e}_{k}]=\hat{e}_{1}\hat{e}_{2}\dots\hat{e}_{k}
\end{equation}
where $e_{1},e_{2},\dots,e_{k}$ are the labeling of the bare edges identified
to $e$. 

\begin{defn}
\label{def: quantum gluing module}
(i) Let $E$ be the \emph{left} ideal of the
quantum torus $\bigotimes_{T\in\mathcal{T}}\mathbb{T}\left\langle T\right\rangle$ generated
by the \emph{edge relations}:
\[
\hat{e}-(-A^{2}).
\]
We have one such relation for each \emph{internal} edge $e\in\mathcal{T}^{(1)}$.
(It is clear that $E$ also contains the element $\hat{e}^{-1}-(-A^{2})^{-1}$
for each $e\in\mathcal{T}^{(1)}$.)

(ii) Let $\overline{E}$ be the $R$-submodule of $\bigotimes_{T\in\mathcal{T}}\hat{\mathcal{G}}(T)$
given by be the image of $E$ under the natural map $\bigotimes_{T\in\mathcal{T}}\mathbb{T}\left\langle T\right\rangle\rightarrow\bigotimes_{T\in\mathcal{T}}\hat{\mathcal{G}}(T)$.

(iii) The \emph{quantum gluing module}  $\hat{\mathcal{G}}_{\mathcal{T}}$
of the ideally triangulated boundary marked 3-manifold $(Y,\mathcal{T})$
is the $R$-module quotient
\[
\hat{\mathcal{G}}_{\mathcal{T}}=\bigotimes_{T\in\mathcal{T}}\hat{\mathcal{G}}(T)\biggm/\overline{E}.
\]
\end{defn}
Note again, if we set $A=1$, the edge relation $\hat{e}-(-A^{2})$
reduces to the classical edge equation (\ref{eq:classical edge equation}).
\begin{rem}
\label{rem: alternative description of quantum guling module}
If, by abuse of notation, we denote by $V_{T}$ and $L_{T}$ the
natural extensions into $\bigotimes_{T\in\mathcal{T}}\mathbb{T}\left\langle T\right\rangle$
of the right ideals $V_{T}$ and $L_{T}$ of $\mathbb{T}\left\langle T\right\rangle$ respectively,
we can also express $\hat{\mathcal{G}}_{\mathcal{T}}$ as the $R$-module quotient
\[
\hat{\mathcal{G}}_{\mathcal{T}}=\frac{\bigotimes_{T\in\mathcal{T}}\mathbb{T}\left\langle T\right\rangle}{E+\sum_{T\in\mathcal{T}}(V_{T}+L_{T})}.
\]

\end{rem}

\begin{prop}
\label{prop:quantum trace for T glue}
Let $(Y,\mathcal{T})$ be an ideally triangulated boundary marked
3-manifold. The map
\[
\bigotimes_{T\in\mathcal{T}}Tr_{T}^{c}\colon\bigotimes_{T\in\mathcal{T}}\overline{\mathrm{Sk}}^{c}(T)\rightarrow\bigotimes_{T\in\mathcal{T}}\hat{\mathcal{G}}(T)
\]
 maps the $R$-submodule $\mathfrak{R}_{E}^{c}$ of $\bigotimes_{T\in\mathcal{T}}\overline{\mathrm{Sk}}^{c}(T)$
(see Definition \ref{def:reduced tensor product}) into the $R$-submodule
$\overline{E}$ of $\bigotimes_{T\in\mathcal{T}}\hat{\mathcal{G}}(T)$.
Therefore it induces a well-defined $R$-module homomorphism
\[
\overline{\bigotimes_{T\in\mathcal{T}}}Tr_{T}^{c}\colon\overline{\bigotimes_{T\in\mathcal{T}}}\overline{\mathrm{Sk}}^{c}(T)\rightarrow\hat{\mathcal{G}}_{\mathcal{T}}.
\]
\end{prop}

\begin{proof}
Recall that $\mathfrak{R}_{E}^{c}$ is spanned by elements that can
be represented by 
\[
\left(\underset{T\in\mathcal{T}}{\otimes}x_{T}\right)\cup\left(\hat{e}^{\epsilon}-(-A^{2})^{\epsilon}\right)\in\mathfrak{R}_{E},\ \epsilon\in\{\pm\},
\]
where $\underset{T\in\mathcal{T}}{\otimes}x_{T}\in\bigotimes_{T\in\mathcal{T}}\overline{\mathrm{Sk}}(T)$
is balanced and $e$ an internal edge. Note that $\hat{e}^{\epsilon}-(-A^{2})^{\epsilon}$
is also balanced, therefore we have (see Remark \ref{rem:matching face, opposite skew-symmetric form})
\[
\langle d(\underset{T\in\mathcal{T}}{\otimes}x_{T}),d(\hat{e}^{\epsilon}-(-A^{2})^{\epsilon})\rangle=0.
\]
Thus 
\[
\left(\underset{T\in\mathcal{T}}{\otimes}x_{T}\right)\cup\left(\hat{e}^{\epsilon}-(-A^{2})^{\epsilon}\right)=\left(\underset{T\in\mathcal{T}}{\otimes}x_{T}\right)\cdot\left(\hat{e}^{\epsilon}-(-A^{2})^{\epsilon}\right).
\]
Now because each $Tr_T^{c}$ is a right $\mathbb{T}\left\langle T\right\rangle$-module homomorphism, the above element is mapped by $\bigotimes_{T\in\mathcal{T}}Tr_{T}$ to
\[
\left(\underset{T\in\mathcal{T}}{\otimes}Tr_{T}^{c}(x_{T})\right)\cdot(\hat{e}^{\epsilon}-(-A^{2})^{\epsilon}),
\]
which is in the left ideal $E$.
\end{proof}

We are finally in the position to give the definition of our quantum
trace map for the 3-manifold $Y$.
\begin{defn}
\label{def:quantum trace map for Y}
\emph{(Quantum trace map of ideally triangulated boundary marked 3-manifolds)} Let $(Y,\mathcal{T})$ be an ideally
triangulated boundary marked 3-manifold. The \emph{quantum trace map}
of $(Y,\mathcal{T})$ is the $R$-module homomorphism given by the
composition
\[
Tr_{\mathcal{T}}\colon\overline{\mathrm{Sk}}(Y)\stackrel{\sigma}{\rightarrow}\overline{\bigotimes_{T\in\mathcal{T}}}\overline{\mathrm{Sk}}^{c}(T)\stackrel{\overline{\bigotimes}_{T\in\mathcal{T}}Tr_{T}^{c}}{\xrightarrow{\hspace{1.8cm}}}\hat{\mathcal{G}}_{\mathcal{T}}
\]
where $\sigma$ is the splitting map given in Theorem \ref{thm:splitting homomorphism}
and $\overline{\bigotimes}_{T\in\mathcal{T}}Tr_{T}^{c}$ is the map
provided by Proposition \ref{prop:quantum trace for T glue}.
\end{defn}

\section{Equivalence with the Quantum trace map of Garoufalidis \& Yu\label{sec:equivalence with GY}}

In this subsection, we show that our construction of quantum trace
map agrees with that in \cite{GY1} in the case when $Y$ is a manifold
without boundary (with cusps), up to some conventions. 

The standing assumption throughout this section will be that the cusped 3-manifold $Y$ has empty boundary, in this case we have 
\[
\overline{\mathrm{Sk}}^{c}(Y)=\overline{\mathrm{Sk}}(Y)=\mathrm{Sk}(Y).
\]

We present, step by step but briefly, the construction in
\cite{GY1} and along the way we prove that the construction yields
the same map as ours. We refer the reader to \cite[Section 4 and 5]{GY1}
for details of Garoufalidis \& Yu's construction.

\subsection{Lantern surface and compatibility of splitting homomorphisms\label{subsec:Skein module of lantern surface and compatibility of splitting }}

One starts with a lantern surface $\mathbb{L}_{T}$ (a sphere with four discs
removed) properly embedded in the ideal tetrahedron $T$ so that every
face of $T$ contains exactly one boundary circle of $\mathbb{L}_{T}$.
In Figure \ref{fig:ideal T with an embedded lantern}, we illustrate
the embedding of a ``thickened'' lantern $\mathbb{L}_{T}\times(-1,1)$ in $T$.
\begin{figure}[h]
  \includegraphics[scale=0.2]{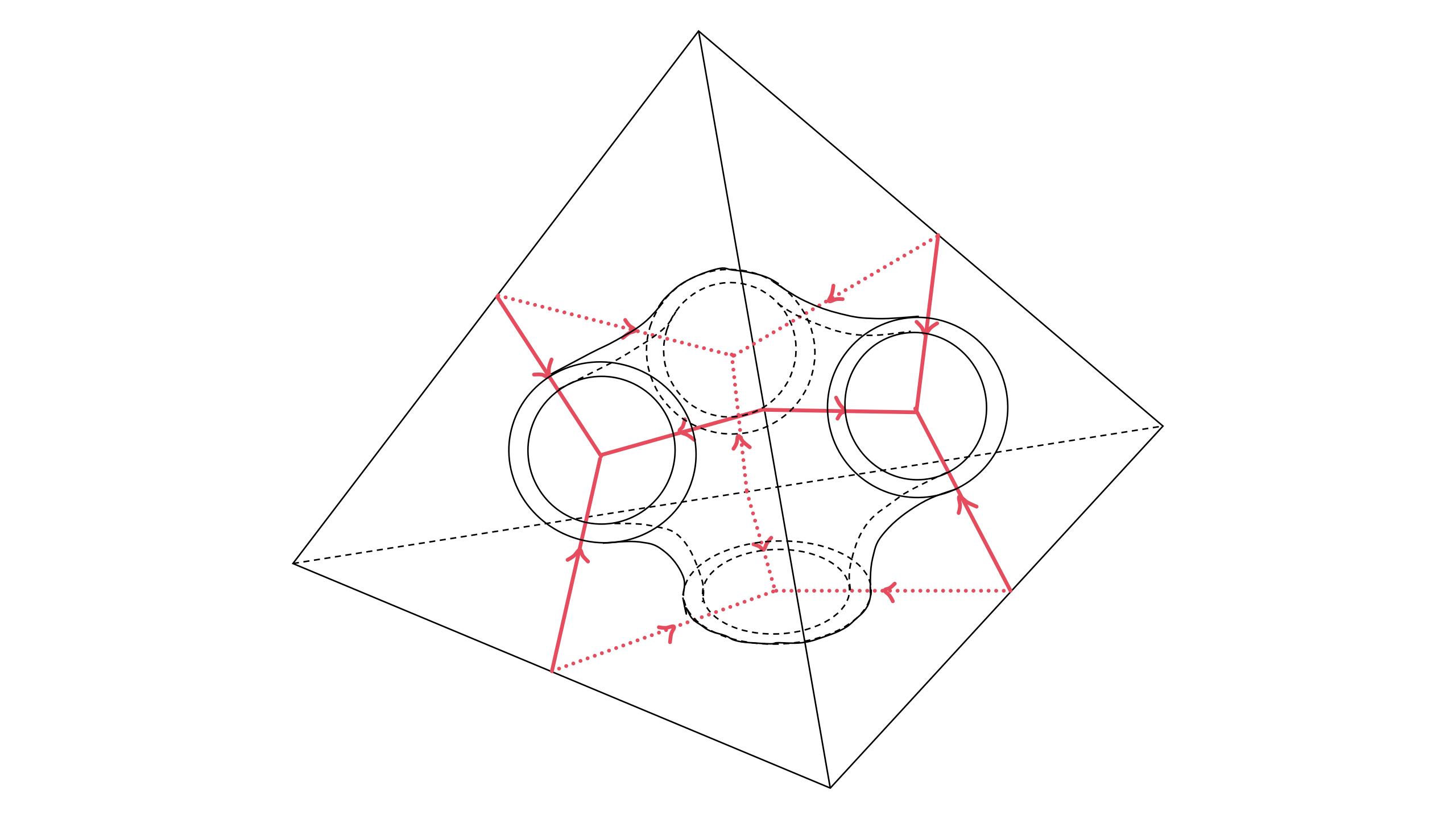} 
  \caption{ideal tetrahedron with an embedded thickened lantern surface}
  \label{fig:ideal T with an embedded lantern}
\end{figure}
We require that $\mathbb{L}_{T}$ be positioned in $T$ in a way so
that the trivalent sinks of the boundary marking of $T$ are at the
``center'' of the boundary circles of $\mathbb{L}_{T}$. When we glue the ideal tetrahedra along their faces to form the 3-manifold $Y$, the lantern surfaces in each ideal tetrahedra will glue along their boundary circles and gives us a closed surface $\Sigma_{\mathcal{T}}$ embedded in $Y$. This surface is called the \emph{dual surface} of $(Y,\mathcal{T})$.

In each ideal tetrahedron $T$,
the 3-manifold $\mathbb{L}_{T}\times(-1,1)$ is induced with a boundary marking,
given by the intersection of the boundary marking of $T$ with $\partial\mathbb{L}_{T}\times(-1,1)$.
Note that there are exactly three marking edges in each boundary component
of $\mathbb{L}_{T}\times(-1,1)$, see Figure \ref{fig:induced boundary marking on lantern}. 
\begin{figure}[h]
  \includegraphics[scale=0.15]{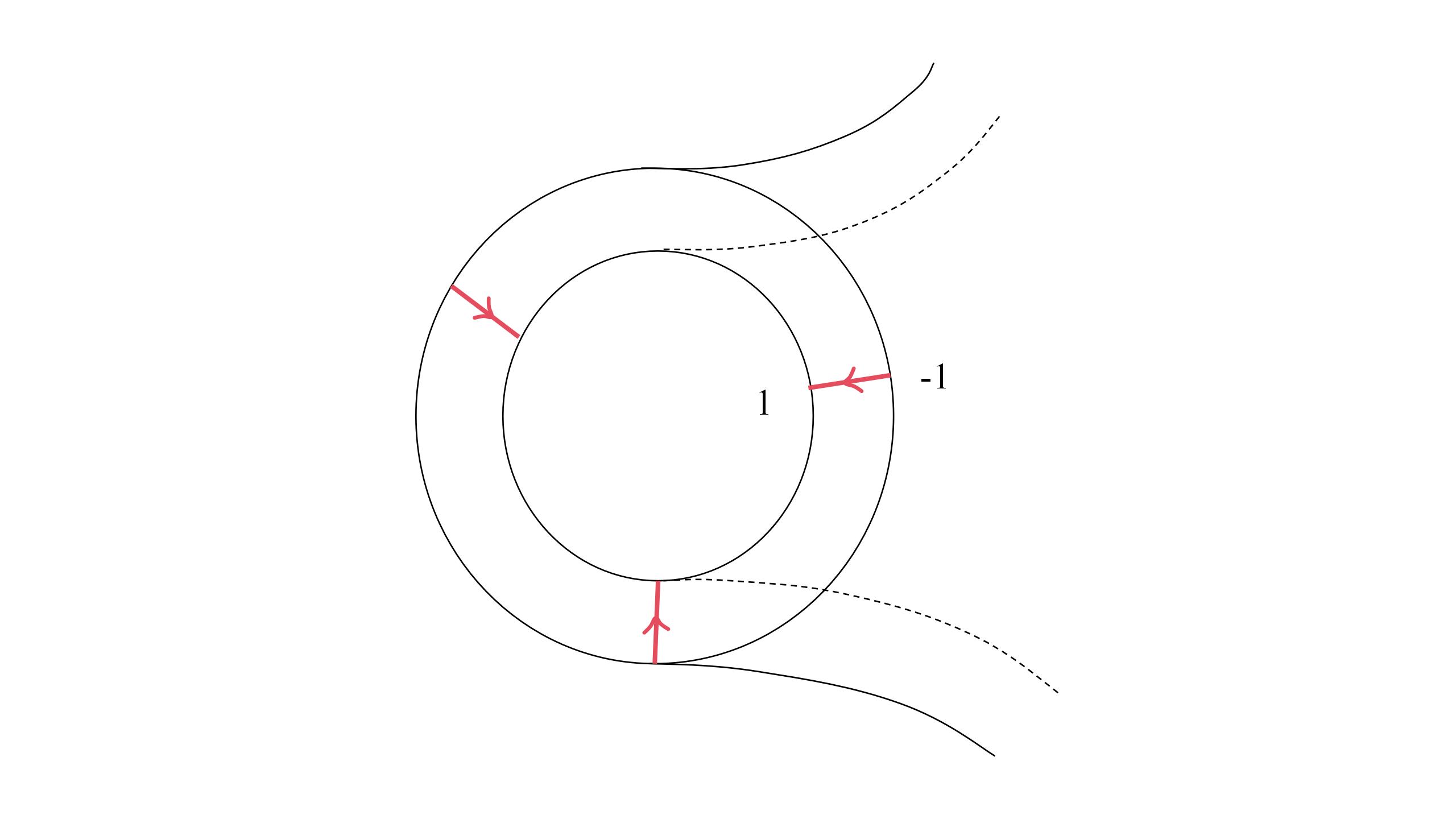} 
  \caption{induced marking on a boundary component of $\mathbb{L}_{T}\times(-1,1)$}
  \label{fig:induced boundary marking on lantern}
\end{figure}

Moreover, we require that the natural orientation of the interval $(-1,1)$ given by increasing of the coordinate agree with the direction of the marking edge. With this induced boundary marking on $\mathbb{L}_{T}\times(-1,1)$, we can consider the (reduced) skein algebra $\overline{\mathrm{SkAlg}}(\mathbb{L}_{T})$ of the lantern surface.  This algebra is opposite to the skein algebra
of lantern surface considered in \cite{GY1}, in fact, if we switch
the orientation of the boundary marking in our case, the resulting
skein algebra will be exactly the same as the one considered in \cite{GY1}. 

As mentioned earlier, our construction of corner-reduced skein module
is motivated by the corner-reduction construction in \cite{GY1},
which is a relation imposed on skein algebra of certain surfaces.
Let's discuss this construction of \cite{GY1} in the case of lantern
surface. Like what we did for our corner-reduction, the first step
is to twist the product structure on $\overline{\mathrm{SkAlg}}(\mathbb{L}_{T})$.
Recall that the skein module $\overline{\mathrm{Sk}}(T)$ has a $\prod_{f\in\mathbf{f}(T)}\mathbb{Z}^{M_{f}}$-grading,
and there is also a skew-symmetric bilinear form $\langle\ ,\ \rangle_{f}$
on each $\mathbb{Z}^{M_{f}}$. Now, the inclusion of the boundary
marked 3-manifolds $\iota_{T}\colon\mathbb{L}_{T}\times(-1,1)\hookrightarrow T$
induces a well-defined $R$-module homomorphism
\[
(\iota_{T})_{*}\colon\overline{\mathrm{SkAlg}}(\mathbb{L}_{T})=\overline{\mathrm{Sk}}(\mathbb{L}_{T}\times(-1,1))\rightarrow\overline{\mathrm{Sk}}(T),
\]
which induces $\prod_{f\in\mathbf{f}(T)}\mathbb{Z}^{M_{f}}$-grading
on $\overline{\mathrm{SkAlg}}(\mathbb{L}_{T})$ in the obvious way.
For $x,y\in\overline{\mathrm{SkAlg}}(\mathbb{L}_{T})$ which are $\prod_{f\in\mathbf{f}(T)}\mathbb{Z}^{M_{f}}$-homogeneous,
we define
\[
x\cdot y=A^{-\frac{1}{2}\langle d(x),d(y)\rangle}xy
\]
and extend by bilinearity to all elements of $\overline{\mathrm{SkAlg}}(\mathbb{L}_{T})$.

Recall in \ref{subsec:Reduced skein module of T}, we introduced the
element $\Gamma_{ab}\in\overline{\mathrm{Sk}}(T)$, which is given by an ``untwisted'' trival 
ribbon arc near a sink of the boundary marking of $T$ connecting the two marking edges 
emanating from the edges of $T$ labeled $a$ and $b$; by a small isotopy,
we can assume that this ribbon arc is contained in the embedded thickened
lantern from the outset, and this gives us elements of $\overline{\mathrm{SkAlg}}(\mathbb{L}_{T})$.
Let $\Gamma_{ab}^{\prime}\in\overline{\mathrm{SkAlg}}(\mathbb{L}_{T})$
be the element represented by this arc with both endpoints in $+$
states and $\Gamma_{ab}^{\prime\prime}\in\overline{\mathrm{SkAlg}}(\mathbb{L}_{T})$
be the nontrivial element represented by this arc with one endpoint
in $+$ state and the other in $-$ state. (Note that there is only
one choice of mixed states on the two endpoints of this arc
that gives a nontrivial element of $\overline{\mathrm{SkAlg}}(\mathbb{L}_{T})$,
the other choice gives a bad arc and hence is $0$ in $\overline{\mathrm{SkAlg}}(\mathbb{L}_{T})$).
See Figure \ref{fig:Gamma' in the lantern}.
\begin{figure}[h]
  \includegraphics[scale=0.2]{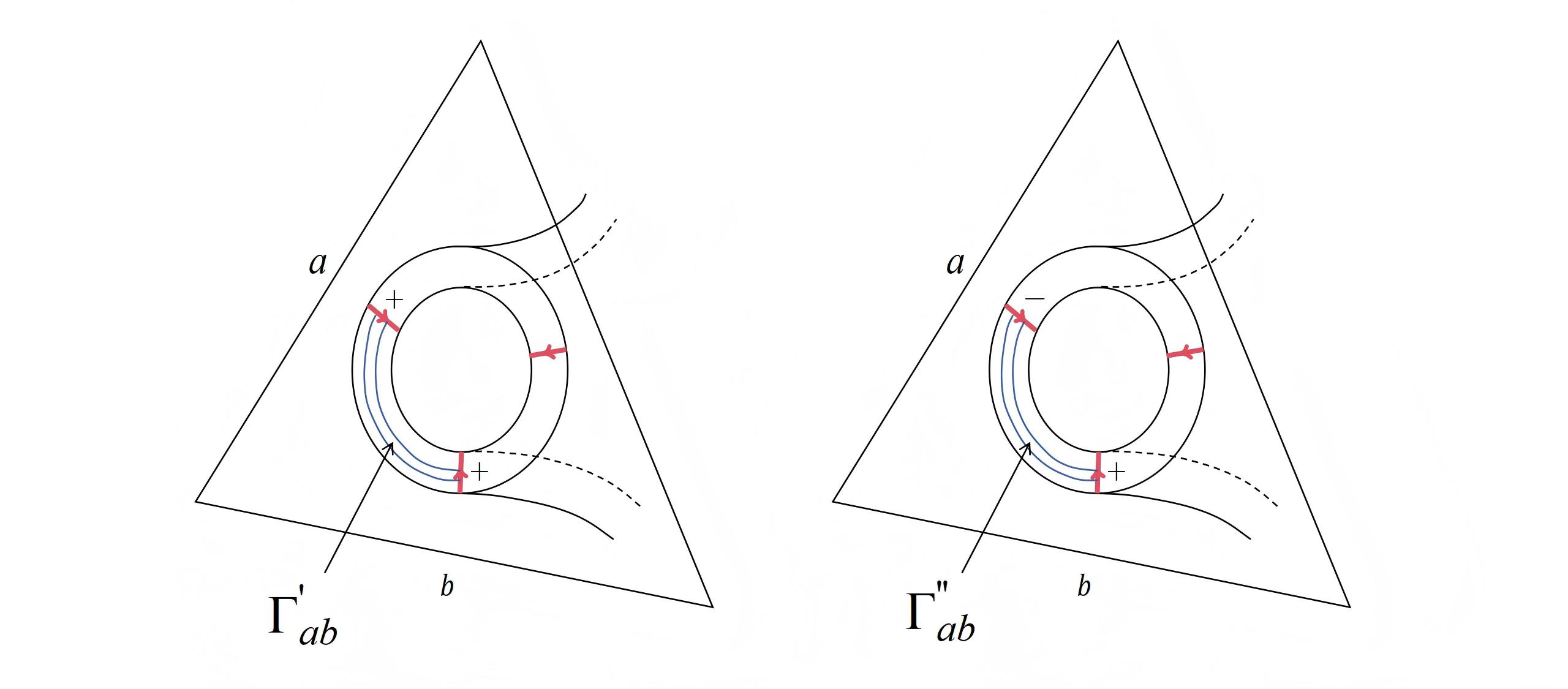} 
  \caption{}
  \label{fig:Gamma' in the lantern}
\end{figure}

\begin{defn}
\label{def:corner-reduced skein module of lantern surfaces}We consider
the right ideal $J^{c}$ of the algebra $(\overline{\mathrm{SkAlg}}(\mathbb{L}),\cdot)$
generated by the set of elements
\[
\left\{ \Gamma_{ab}^{\prime}-(-A^{2})^{-\frac{1}{2}},\ \Gamma_{ab}^{\prime\prime}-1\biggm|\begin{array}{c}
a,b\in\{z,z^{\prime},z^{\prime\prime},y,y^{\prime},y^{\prime\prime}\}\\
\text{\ensuremath{a} and \ensuremath{b} are not labels of opposite edges}
\end{array}\right\} .
\]
The \emph{corner-reduced skein module} $\overline{\mathrm{Sk}}^{c}(\mathbb{L})$
of the lantern is defined to be the quotient module
\[
\overline{\mathrm{Sk}}^{c}(\mathbb{L})=\overline{\mathrm{SkAlg}}(\mathbb{L})\biggm/J^{c}
\]
\end{defn}

\begin{rem}
The corner-reduced skein module of the lantern surface
given above is essentially the same as the one which appeared in \cite{GY1}
(although in \cite{GY1} one can consider corner-reduced module of
any so-called surface with triangular boundary, not just the lantern
surface). However since our skein algebra $\overline{\mathrm{SkAlg}}(\mathbb{L}_{T})$
has product opposite to the one in \cite{GY1}, we need to consider the right
ideal $J^{c}$, whereas in \cite{GY1} one quotients out the\emph{
left }ideal generated by the same set of elements.
\end{rem}

Before describing the next step of Garoufalidis \& Yu's construction,
let's first establish the following compatibility between the splitting
homomorphisms facilitated by inclusions.

\begin{lem}
\label{lem:lantern embeds in T induces map between corner reduced modules}
The inclusion $\mathbb{L}_{T}\hookrightarrow T$
induces a well-defined $R$-module homomorphism $\Phi_{T}\colon\overline{\mathrm{Sk}}^{c}(\mathbb{L}_{T})\rightarrow\overline{\mathrm{Sk}}^{c}(T)$
which commutes with the splitting homomorphisms. More precisely, the following
diagram commutes,
\begin{equation}
\label{eq:compatability of GY splitting with our splitting}
\begin{tikzcd}
\mathrm{Sk}(\Sigma_{\mathcal{T}}) \arrow[r, "\Theta"] \arrow[d, two heads, "\Phi"] & \bigotimes_{T\in\mathcal{T}}\overline{\mathrm{Sk}}^{c}(\mathbb{L}_{T}) \arrow[r, "\bigotimes\Phi_{T} "] & \bigotimes_{T\in\mathcal{T}}\overline{\mathrm{Sk}}^{c}(T) \arrow[d, two heads] \\
\mathrm{Sk}(Y) \arrow[rr, "\sigma"]                                     &                                                                                                       & \overline{\bigotimes}_{T\in\mathcal{T}}\overline{\mathrm{Sk}}^{c}(T)          
\end{tikzcd},
\end{equation}
where $\Phi\colon\mathrm{Sk}(\Sigma_{\mathcal{T}})\rightarrow\mathrm{Sk}(Y)$
is the map induced by the inclusion $\Sigma_{\mathcal{T}}\hookrightarrow Y$;
$\Theta$ is the splitting homomorphism in \cite{GY1} (See Theorem
4.16 of \cite{GY1}) and $\sigma\colon{\mathrm{Sk}}(Y)\rightarrow\overline{\bigotimes}_{T\in\mathcal{T}}\overline{\mathrm{Sk}}^{c}(T)$
is our splitting homomorphism given by Theorem \ref{thm:splitting homomorphism}.
\end{lem}

\begingroup
\allowdisplaybreaks
\begin{proof}
First we show that the inclusion $\mathbb{L}_{T}\hookrightarrow T$
induces a well-defined $R$-module homomorphim $\Phi_{T}\colon\overline{\mathrm{Sk}}^{c}(\mathbb{L}_{T})\rightarrow\overline{\mathrm{Sk}}^{c}(T)$.
That is, we need to show that the composition of $R$-module homomorphisms
\[
\overline{\mathrm{SkAlg}}(\mathbb{L}_{T})=\overline{\mathrm{Sk}}(\mathbb{L}_{T}\times(-1,1))\stackrel{(\iota_{T})_{*}}{\longrightarrow}\overline{\mathrm{Sk}}(T)\twoheadrightarrow\overline{\mathrm{Sk}}^{c}(T)
\]
maps the right ideal $J^{c}$ of $\left(\overline{\mathrm{SkAlg}}(\mathbb{L}_{T}),\cdot\right)$
to $\{0\}$. By construction, under the inclusion $\iota_T\colon\mathbb{L}_{T}\times(-1,1)\hookrightarrow T$,
we have 
\[
(\iota_{T})_{*}(\Gamma_{ab}^{\prime})=\Gamma_{ab}
\]
 and 
\[
(\iota_{T})_{*}(\Gamma_{ab}^{\prime\prime})=[\Gamma_{ac}^{-1}\Gamma_{cb}]=\Gamma_{cb}\cdot\Gamma_{ac}^{-1}.
\]
Here, $c$ is the label of the third edge of the face whose two other
edges are labeled $a$ and $b$; we are also assuming $\Gamma_{ab}^{\prime\prime}$
has $+$ state on the marking edge emanating from edge labeled $b$
and $-$ state on the marking edge emanating from edge labeled $a$
(as in Figure \ref{fig:Gamma' in the lantern}. Therefore, for any $\kappa\in\overline{\mathrm{SkAlg}}(\mathbb{L}_{T})$, we have
\begin{align*}
(\iota_{T})_{*}\left((\Gamma_{ab}^{\prime}-(-A^{2})^{-\frac{1}{2}})\cdot\kappa\right)
 &=(\Gamma_{ab}-(-A^{2})^{-\frac{1}{2}})\cdot\left((\iota_{T})_{*}\kappa\right)\\
 &=0\ \text{in}\ \overline{\mathrm{Sk}}(T)
\end{align*}
 and 
\begin{align*}
(\iota_{T})_{*}\left((\Gamma_{ab}^{\prime\prime}-1)\cdot\kappa\right)
 &=(\Gamma_{cb}\cdot\Gamma_{ac}^{-1}-1)\cdot\left((\iota_{T})_{*}\kappa\right)\\
 &=0\ \text{in}\ \overline{\mathrm{Sk}}(T).
\end{align*}
This shows that $(\iota_{T})_{*}$ descends to a well-defined $R$-module homomorphism $\Phi_{T}\colon\overline{\mathrm{Sk}}^{c}(\mathbb{L}_{T})\rightarrow\overline{\mathrm{Sk}}^{c}(T)$.

To see that Diagram \ref{eq:compatability of GY splitting with our splitting}
commutes: We let $\ell$ be a ribbon link in $\Sigma_{\mathcal{T}}\times(-1,1)$,
up to isotopy in $\Sigma_{\mathcal{T}}\times(-1,1)$ we assume it
is in general position in the 3-manifold $Y$ (see \ref{subsec:Splitting homomorphism}
to recall the notion of general position). Then $\sigma\Phi(\ell)$
is the element in $\overline{\bigotimes}_{T\in\mathcal{T}}\overline{\mathrm{Sk}}^{c}(T)$
represented by $\sum_{\overrightarrow{\epsilon}}\left(\underset{T\in\mathcal{T}}{\otimes}\ell_{T}^{\overrightarrow{\epsilon}}\right)$,
where $\ell_{T}$ is the part of $\ell$ in $T$ (see Theorem \ref{thm:splitting homomorphism}).
On the other hand, $\Theta(\ell)$ is also represented by $\sum_{\overrightarrow{\epsilon}}\left(\underset{T\in\mathcal{T}}{\otimes}\ell_{T}^{\overrightarrow{\epsilon}}\right)$
(see \cite[4.5]{GY1} for the description of the splitting homomorphism
$\Theta$). Therefore, the image of a ribbon link $\ell$ under both
routes is the element of $\overline{\bigotimes}_{T\in\mathcal{T}}\overline{\mathrm{Sk}}^{c}(T)$
represented by $\sum_{\overrightarrow{\epsilon}}\left(\underset{T\in\mathcal{T}}{\otimes}\ell_{T}^{\overrightarrow{\epsilon}}\right)$.
\end{proof}
\endgroup

\begin{rem}
\label{rem:surjectivity of maps induced by inclusion}The maps $\Phi\colon\mathrm{SkAlg}(\Sigma_{\mathcal{T}})\rightarrow\mathrm{Sk}(Y)$
and $\Phi_{T}\colon\overline{\mathrm{Sk}}^{c}(\mathbb{L}_{T})\rightarrow\overline{\mathrm{Sk}}^{c}(T)$
are in fact surjective, see \cite[Proposition 4.3]{GY1}. (For the surjectivity of $\Phi_{T}$:
the map $\left(\iota_{T}\right)_{*}$ is surjective, while $\Phi_{T}$ is
induced by $\left(\iota_{T}\right)_{*}$ on quotient, so it must also
be surjective.)
\end{rem}

\subsection{Construction of $Tr_{T}^{[GY]}$ and its relation with our $Tr_{T}$ \label{subsec:Quantum trace map for a single lantern and relation to our quantum trace map for a single tetrahedron}}

Next in the construction of Garoufalidis \& Yu, a quantum trace map $Tr_{T}^{[GY]}$ from
the corner reduced module $\overline{\mathrm{Sk}}^{c}(\mathbb{L}_{T})$
of the lantern in the ideal tetrahedron $T$ to the quantum module
$\hat{\mathcal{G}}(T)$ of $T$ is constructed. 

By definition, $\overline{\mathrm{Sk}}^{c}(\mathbb{L}_{T})$ is a quotient
of $\mathrm{SkAlg}(\mathbb{L}_{T})$, however, using a counit trick (see \cite[Lemma 4.22]{GY1}),
one can express $\overline{\mathrm{Sk}}^{c}(\mathbb{L}_{T})$ as the
quotient of a simpler module. We'll now recall the details of this. Let $a\in\{z,z^{\prime},z^{\prime},y,y^{\prime},y^{\prime\prime}\}$
be a labeling of an edge of $T$, we let $x_{a}$ be the ribbon tangle
in $\mathbb{L}_{T}\times(-1,1)$ given by an ``untwisted'' trivial
ribbon arc connecting the two marking edges in $\partial\mathbb{L}_{T}\times(-1,1)$
which, after embedding in $T$, are parts of the two marking edges
emanating from the edge labeled $a$, see Figure \ref{fig:embedded lantern with standard arcs}.
\begin{figure}[h]
  \includegraphics[scale=0.2]{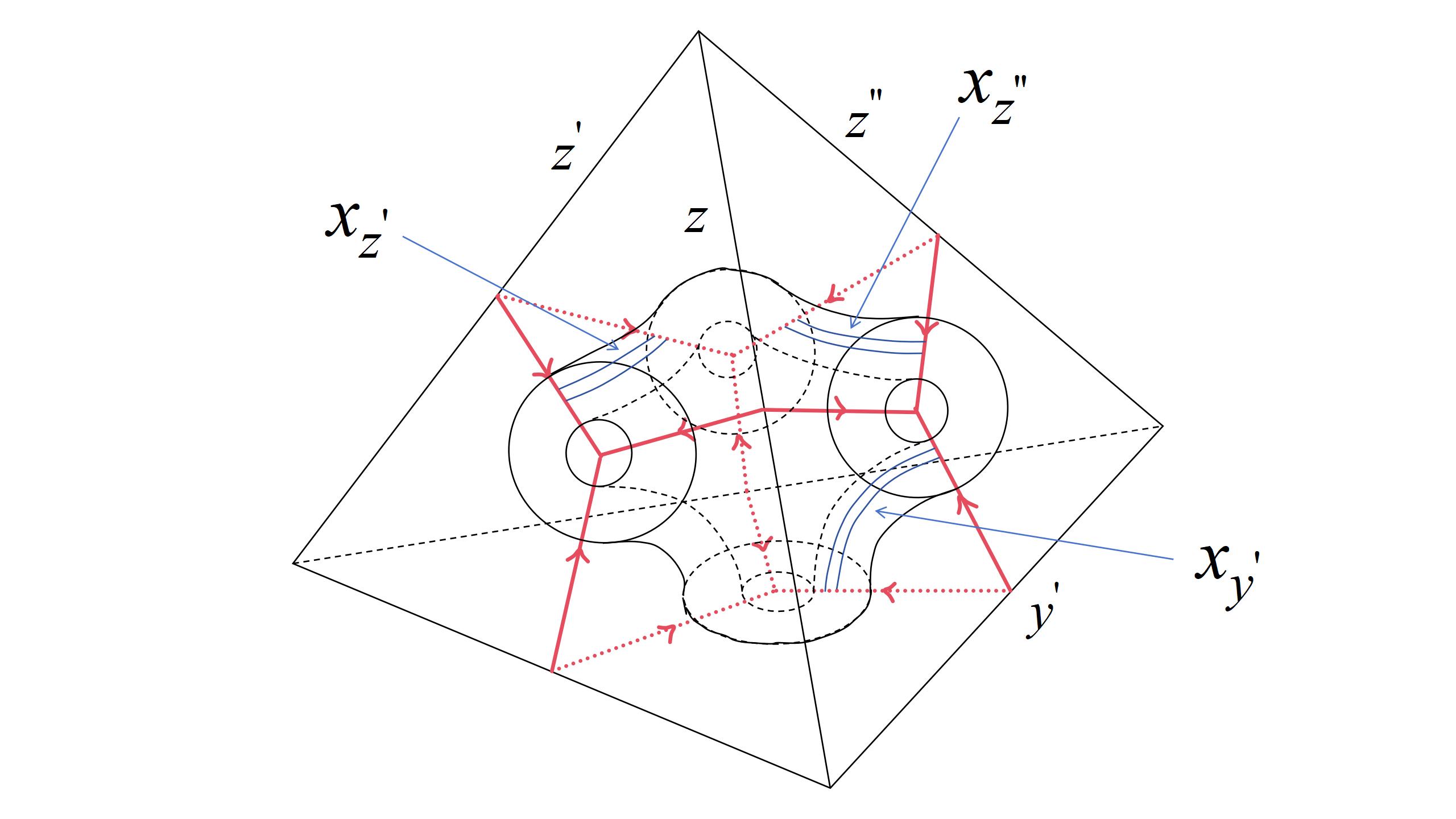} 
  \caption{}
  \label{fig:embedded lantern with standard arcs}
\end{figure}
Such ribbon arcs give us some nontrivial elements of (unreduced) skein
algebra $\mathrm{SkAlg}(\mathbb{L}_{T})$: we let $x_{a}^{++}$, $x_{a}^{--}$
be the ribbon tangle $x_{a}$ with both $+$ states and both $-$
states, respectively, assigned to its endpoints and let $x_{a}^{+-}$
and $x_{a}^{-+}$ be the one with mixed states assigned to endpoints;
for our purpose we do not need to specify which element is $x_{a}^{+-}$
or $x_{a}^{-+}$. For each $a\in\{z,z^{\prime},z^{\prime},y,y^{\prime},y^{\prime\prime}\}$,
the subalgebra of $\mathrm{SkAlg}(\mathbb{L}_{T})$ (in $\cup$-product)
generated by $x_{a}^{++},\ x_{a}^{--},\ x_{a}^{+-}\ \text{and}\ x_{a}^{-+}$
is isomorphic to $\mathcal{O}_{A^{2}}(SL_{2})$, 
such that the two mixed states elements $x_{a}^{+-}$ and $x_{a}^{-+}$
correspond to off-diagonal elements of $\mathcal{O}_{A^{2}}(SL_{2})$.
We denote by $\mathcal{S}^{0}\cong\left(\mathcal{O}_{A^{2}}(SL_{2})\right)^{\otimes6}$
the subalgebra of $\mathrm{SkAlg}(\mathbb{L}_{T})$ generated by
all 24 such elements (4 for each edge of $T$). Because $x_{a}^{+-}$
and $x_{a}^{-+}$ correspond to off-diagonal elements of $\mathcal{O}_{A^{2}}(SL_{2})$,
they commute with every element of $\mathcal{S}^{0}$ up to multiplication
of a power of $A$. The restriction of the quotient map $\mathrm{SkAlg}(\mathbb{L}_{T})\twoheadrightarrow\overline{\mathrm{Sk}}^{c}(\mathbb{L}_{T})$
to $\mathcal{S}^{0}$ is actually surjective (see \cite[Lemma 4.22]{GY1}), therefore $\overline{\mathrm{Sk}}^{c}(\mathbb{L}_{T})$
is a quotient of $\mathcal{S}^{0}$.

Observe that upon inclusion $\mathbb{L}_{T}\times(-1,1)\hookrightarrow T$,
the skeins $x_{a}^{++}$ and $x_{a}^{--}$ become skeins $\hat{a}$
and $\hat{a}^{-1}$ in $\overline{\mathrm{Sk}}(T)$ respectively; whereas
$x_{a}^{+-}$ and $x_{a}^{-+}$ become bad arcs, which are $0$
in $\overline{\mathrm{Sk}}(T)$. Thus it is natural to quotient out
these mixed-states elements, in fact we consider the\emph{ left }ideal $B$ of $\overline{\mathrm{SkAlg}}(\mathbb{L}_{T})$
generated by the elements $x_{a}^{+-}$ and $x_{a}^{-+}$, $a\in\{z,z^{\prime},z^{\prime},y,y^{\prime},y^{\prime\prime}\}$, and
note that it doesn't matter whether we consider it to be a left ideal
in $\cup$-product or $\cdot$-product because the generators
$x_{a}^{+-}$ and $x_{a}^{-+}$ are $\prod_{f\in\mathbf{f}(T)}\mathbb{Z}^{M_{f}}$-homogeneous.
Let $B^{c}$ be the $R$-submodule of $\overline{\mathrm{Sk}}^{c}(\mathbb{L}_{T})$
given by the image of $B$ under the quotient map $\overline{\mathrm{SkAlg}}(\mathbb{L}_{T})\twoheadrightarrow\overline{\mathrm{Sk}}^{c}(\mathbb{L}_{T})$
and define the $R$-module quotient
\[
\hat{\mathcal{G}}(\mathbb{L}_{T}):=\overline{\mathrm{Sk}}^{c}(\mathbb{L}_{T})\biggm/B^{c}.
\]
Because we know that $\overline{\mathrm{Sk}}^{c}(\mathbb{L}_{T})$ is
a quotient of $\mathcal{S}^{0}$, the $R$-module $\hat{\mathcal{G}}(\mathbb{L}_{T})$
is also a quotient of $\mathcal{S}^{0}$. By construction, the quotient
map $\mathcal{S}^{0}\twoheadrightarrow\hat{\mathcal{G}}(\mathbb{L}_{T})$
sends the left ideal $B^{0}$ of $\mathcal{S}^{0}$
generated by the mixed-states elements $x_{a}^{+-}$ and $x_{a}^{-+}$
(regardless of using $\cdot$-product or $\cup$-product) to $0$, thus
$\hat{\mathcal{G}}(\mathbb{L}_{T})$ is a quotient of $\mathcal{S}^{0}\bigm/B^{0}$.
Using the fact that $x_{a}^{++},\ x_{a}^{--},\ x_{a}^{+-}\ \text{and}\ x_{a}^{-+}$
are matrix elements of $\mathcal{O}_{A^{2}}(SL_{2})$, the ideal $B^{0}$
of $\mathcal{S}^{0}$ is actually \emph{2-sided} because $x_{a}^{+-}$
and $x_{a}^{-+}$ commute (in $\cdot$-product or $\cup$-product)
with elements of $\mathcal{S}^{0}$ up to scalar multiples. Moreover
we find that:
\begin{itemize}
\item In $\cup$-product, $\mathcal{S}^{0}\bigm/B^{0}$ is the Laurent
polynomial algebra 
\[
R\left[\left(x_{z}^{++}\right)^{\pm1},\left(x_{z^{\prime}}^{++}\right)^{\pm1},\left(x_{z^{\prime\prime}}^{++}\right)^{\pm1},\left(x_{y}^{++}\right)^{\pm1},\left(x_{y^{\prime}}^{++}\right)^{\pm1},\left(x_{y^{\prime\prime}}^{++}\right)^{\pm1}\right],
\]
where $\left(x_{a}^{++}\right)^{-1}=x_{a}^{-}$.
\item In $\cdot$-product, $\left(\mathcal{S}^{0},\cdot\right)\bigm/B^{0}$
becomes the quantum torus
\[
\mathbb{T}^{0}(T):=\frac{R\left\langle\left(x_{z}^{++}\right)^{\pm1},\left(x_{z^{\prime}}^{++}\right)^{\pm1},\left(x_{z^{\prime\prime}}^{++}\right)^{\pm1},\left(x_{y}^{++}\right)^{\pm1},\left(x_{y^{\prime}}^{++}\right)^{\pm1},\left(x_{y^{\prime\prime}}^{++}\right)^{\pm1}\right\rangle}{\left\langle \begin{array}{c}
x_{a}^{++}\cdot x_{b^{\prime}}^{++}=Ax_{b^{\prime}}^{++}\cdot x_{a}^{++}\\
x_{a^{\prime}}^{++}\cdot x_{b^{\prime\prime}}^{++}=Ax_{b^{\prime\prime}}^{++}\cdot x_{a^{\prime}}^{++}\\
x_{a^{\prime\prime}}^{++}\cdot x_{b}^{++}=Ax_{b}^{++}\cdot x_{a^{\prime\prime}}^{++}
\end{array}\Bigg|a,b\in\{z,y\}\right\rangle }.
\]
\end{itemize}
Of course, as $R$-modules, we have $\mathcal{S}^{0}\bigm/B^{0}=\left(\mathcal{S}^{0},\cdot\right)\bigm/B^{0}=\mathbb{T}^{0}(T)$,
and similarly we find that monomials in $\cup$-product are the Weyl-ordering
of the same monomial in $\cdot$-product. 

The next theorem is one
of the key results of \cite{GY1}. We phrase it in our notation, recall
that we also have the quantum torus $\mathbb{T}\langle T\rangle=\left(\mathbb{B}^{\otimes6},\cdot\right)$
given by the skein algebras associated with the six bivalent sources at the center
of the six edges of the ideal tetrahedron $T$ under $\cdot$-product (see Lemma \ref{lem:presntation of (T^tensor4,cdot)  and (B^tensor6, cdot)}
(ii)).

\begin{thm}
\emph{(\cite[Theorem 4.25]{GY1})} The isomorphism of quantum tori 
\[
\Psi\colon\mathbb{T}^{0}(T)\overset{\cong}{\longrightarrow}\mathbb{T}\left\langle T\right\rangle
\]
given by $x_{a}^{++}\mapsto\hat{a}$, $a\in\{z,z^{\prime},z^{\prime},y,y^{\prime},y^{\prime\prime}\}$
descends to an isomorphism of $R$-modules
\[
\tilde{\Psi}\colon\hat{\mathcal{G}}(\mathbb{L}_{T})\overset{\cong}{\longrightarrow}\hat{\mathcal{G}}(T).
\]
\end{thm}

Recall that $\hat{\mathcal{G}}(\mathbb{L}_{T})$ is also by definition
the quotient $R$-module $\overline{\mathrm{Sk}}^{c}(\mathbb{L}_{T})\bigm/B^{c}$.
Let 
\[
Tr_{T}^{[GY]}\colon\overline{\mathrm{Sk}}^{c}(\mathbb{L}_{T})\rightarrow\hat{\mathcal{G}}(T)
\]
 be the composition
\[
\overline{\mathrm{Sk}}^{c}(\mathbb{L}_{T})\twoheadrightarrow\overline{\mathrm{Sk}}^{c}(\mathbb{L}_{T})\bigm/B^{c}=\hat{\mathcal{G}}(\mathbb{L}_{T})\overset{\tilde{\Psi}_{T}}{\longrightarrow}\hat{\mathcal{G}}(T),
\]
and think of it as Garoufalidis \& Yu's version of quantum trace map
for a single tetrahedron. Let's show that this map is compatible with
our quantum trace map $Tr_{T}^{c}$ for the tetrahedron $T$ via the inclusion
$\iota_{T}\colon\mathbb{L}_{T}\hookrightarrow T$.

\begin{lem}
\label{lem:compatibility of GY Tr and our Tr for a single T}
We have the following commutative diagram:
\[
\begin{tikzcd} \overline{\mathrm{Sk}}^{c}(\mathbb{L}_{T}) \arrow[d, "\Phi_{T}"'] \arrow[rd, "{Tr_{T}^{[GY]}}"] &                      \\ \overline{\mathrm{Sk}}^{c}(T) \arrow[r, "Tr_{T}^{c}"']                                              & \hat{\mathcal{G}}(T) \end{tikzcd}
\]
where
$\Phi_{T}\colon\overline{\mathrm{Sk}}^{c}(\mathbb{L}_{T})\rightarrow\overline{\mathrm{Sk}}^{c}(T)$
is the $R$-module homomorphism induced by the inclusion $\mathbb{L}_{T}\hookrightarrow T$
provided by Lemma \ref{lem:lantern embeds in T induces map between corner reduced modules}.
\end{lem}
\begin{proof}
Since $\overline{\mathrm{Sk}}^{c}(\mathbb{L}_{T})$ is an $R$-module
quotient of $\mathcal{S}^{0}$, it is the $R$-span of monomials
\[
\left(x_{z}^{++}\right)^{m}\cdot\left(x_{z^{\prime}}^{++}\right)^{n}\cdot\left(x_{z^{\prime\prime}}^{++}\right)^{\ell}\cdot\left(x_{y}^{++}\right)^{p}\cdot\left(x_{y^{\prime}}^{++}\right)^{q}\cdot\left(x_{y^{\prime\prime}}^{++}\right)^{r}.
\]
In $\hat{\mathcal{G}}(T)$, we have on the one hand, 
\begin{eqnarray*}
\lefteqn{Tr_{T}^{[GY]}\left(\left(x_{z}^{++}\right)^{m}\cdot\left(x_{z^{\prime}}^{++}\right)^{n}\cdot\left(x_{z^{\prime\prime}}^{++}\right)^{\ell}\cdot\left(x_{y}^{++}\right)^{p}\cdot\left(x_{y^{\prime}}^{++}\right)^{q}\cdot\left(x_{y^{\prime\prime}}^{++}\right)^{r}\right)}\\
&=&\Psi\left(\left(x_{z}^{++}\right)^{m}\cdot\left(x_{z^{\prime}}^{++}\right)^{n}\cdot\left(x_{z^{\prime\prime}}^{++}\right)^{\ell}\cdot\left(x_{y}^{++}\right)^{p}\cdot\left(x_{y^{\prime}}^{++}\right)^{q}\cdot\left(x_{y^{\prime\prime}}^{++}\right)^{r}\right)\\
&=&\hat{z}^{m}\cdot\hat{z}^{\prime n}\cdot\hat{z}^{\prime\prime \ell}\cdot\hat{y}^{p}\cdot\hat{y}^{\prime q}\cdot\hat{y}^{\prime\prime r};
\end{eqnarray*}
on the other hand:
\begin{eqnarray*}
\lefteqn{Tr_{T}^{c}\Phi_{T}\left(\left(x_{z}^{++}\right)^{m}\cdot\left(x_{z^{\prime}}^{++}\right)^{n}\cdot\left(x_{z^{\prime\prime}}^{++}\right)^{\ell}\cdot\left(x_{y}^{++}\right)^{p}\cdot\left(x_{y^{\prime}}^{++}\right)^{q}\cdot\left(x_{y^{\prime\prime}}^{++}\right)^{r}\right)}\\
&=&Tr_{T}^{c}\left(\hat{z}^{m}\cdot\hat{z}^{\prime n}\cdot\hat{z}^{\prime\prime \ell}\cdot\hat{y}^{p}\cdot\hat{y}^{\prime q}\cdot\hat{y}^{\prime\prime r}\right)\\
&=&\hat{z}^{m}\cdot\hat{z}^{\prime n}\cdot\hat{z}^{\prime\prime \ell}\cdot\hat{y}^{p}\cdot\hat{y}^{\prime q}\cdot\hat{y}^{\prime\prime r}.
\end{eqnarray*}
(In the above, we recall that $\Phi_{T}$ is the map induced by inclusion
$\mathbb{L}_{T}\times(-1,1)\hookrightarrow T$, under which $x_{a}^{++}$
becomes $\hat{a}$.)
This shows that $Tr_{T}^{[GY]}$ and $Tr_{T}^{c}\Phi_{T}$ agree on a
spanning set of $\overline{\mathrm{Sk}}^{c}(\mathbb{L}_{T})$.
\end{proof}

\begin{rem}
\label{rem:kernel of =00005CPhi_T}Since we know that $Tr_{T}^{c}$ is
an isomorphism (by Theorem \ref{thm:Tr induce isomorphism between corner reduced module and quantum module}) and $Tr_{T}^{[GY]}$ is the composition 
\[
\overline{\mathrm{Sk}}^{c}(\mathbb{L}_{T})\twoheadrightarrow\overline{\mathrm{Sk}}^{c}(\mathbb{L}_{T})\bigm/B^{c}=\hat{\mathcal{G}}(\mathbb{L}_{T})\overset{\cong}{\longrightarrow}\hat{\mathcal{G}}(T),
\]
we see that the kernel of the map  $\Phi_{T}\colon\overline{\mathrm{Sk}}^{c}(\mathbb{L}_{T})\twoheadrightarrow\overline{\mathrm{Sk}}^{c}(T)$ is nothing but the $R$-submodule
$B^{c}$ of $\overline{\mathrm{Sk}}^{c}(\mathbb{L}_{T})$.
\end{rem}

\subsection{Construction of $Tr_{\mathcal{T}}^{[GY]}$ and proof of \textmd{$Tr_{\mathcal{T}}^{[GY]}=Tr_{\mathcal{T}}$}\label{subsec:Construction Tr of GY and prove it agrees with our Tr}}

Next we can finally describe $Tr_{\mathcal{T}}^{[GY]}$, the quantum
trace map in \cite{GY1} on the ideally triangulated 3-manifold $(Y,\mathcal{T})$.
Recall that there is the splitting homomorphism $\Theta\colon\mathrm{SkAlg}(\Sigma_{\mathcal{T}})\rightarrow\bigotimes_{T\in\mathcal{T}}\overline{\mathrm{Sk}}^{c}(\mathbb{L}_{T})$;
the surjective $R$-module homomorphisms $\Phi\colon\mathrm{Sk}(\Sigma_{\mathcal{T}})\twoheadrightarrow\mathrm{Sk}(Y)$
and $\Phi_{T}\colon\overline{\mathrm{Sk}}^{c}(\mathbb{L}_{T})\twoheadrightarrow\overline{\mathrm{Sk}}^{c}(T)$
induced by inclusions. Also recall the quantum gluing module $\hat{\mathcal{G}}_{\mathcal{T}}$,
which is defined to be the $R$-module quotient of $\bigotimes_{T\in\mathcal{T}}\hat{\mathcal{G}}(T)$
by the edge relations (see Definition \ref{def: quantum gluing module}). 

\begin{thm}
\emph{(\cite[Theorem 5.3]{GY1})} The composition 
\[
\mathrm{Sk}(\Sigma_{\mathcal{T}})\overset{\Theta}{\longrightarrow}\bigotimes_{T\in\mathcal{T}}\overline{\mathrm{Sk}}^{c}(\mathbb{L}_{T})\overset{\bigotimes Tr_{T}^{[GY]}}{\longrightarrow}\bigotimes_{T\in\mathcal{T}}\hat{\mathcal{G}}(T)\twoheadrightarrow\hat{\mathcal{G}}_{\mathcal{T}}
\]
descends to a well-defined $R$-module homomorphism $Tr_{\mathcal{T}}^{[GY]}\colon\mathrm{Sk}(Y)\rightarrow\hat{\mathcal{G}}_{\mathcal{T}}$.
\end{thm}

The map $Tr_{\mathcal{T}}^{[GY]}$ given above is the quantum trace
map of Garoufalidis \& Yu. Strictly speaking, this is not exactly
the same as the one appeared in \cite{GY1}. In fact, apart from the
differences induced by different choices of labeling and orientations,
the quantum gluing module appeared in \cite[5.1]{GY1} is the
quotient of our quantum gluing module $\hat{\mathcal{G}}_{\mathcal{T}}$
by identifying $\hat{z}^{\boxempty}$ and $\hat{y}^{\boxempty}$ in
every tetrahedron; that is, one identifies quantized shape parameters
associated to opposite edges in each tetrahedron. However this is
only a mild quotient (in a sense 2 to 1) because in our quantum gluing
module $\hat{\mathcal{G}}_{\mathcal{T}}$ it holds that $\left(\hat{z}^{\boxempty}\right)^{2}=\left(\hat{y}^{\boxempty}\right)^{2}$
in every tetrahedron (this can be easily deduced from the vertex relations).
On the other hand, it is also remarked in \cite[Remark 5.1]{GY1}
that their construction push through without imposing the relation
$\hat{z}^{\boxempty}=\hat{y}^{\boxempty}$.

Now we can prove our first main comparison result. The following is a restatement
of Theorem \ref{thm: Main theorem 1, our Tr agrees with Tr of GY}.

\begin{thm}
$Tr_{\mathcal{T}}^{[GY]}=Tr_{\mathcal{T}}$.
\end{thm}
\begin{proof}
We only need to show the following diagram commutes
\[
\begin{tikzcd} \mathrm{Sk}(\Sigma_{\mathcal{T}}) \arrow[r, "\Theta"] \arrow[d, "\Phi"', two heads] & \bigotimes_{T\in\mathcal{T}}\overline{\mathrm{Sk}}^{c}(\mathbb{L}_{T}) \arrow[r, "{\bigotimes Tr_{T}^{[GY]}}"] & \bigotimes_{T\in\mathcal{T}}\hat{\mathcal{G}}(T) \arrow[r, two heads] & \hat{\mathcal{G}}_{\mathcal{T}} \\ \mathrm{Sk}(Y) \arrow[rrru, "Tr_{\mathcal{T}}"']                                    &                                                                                                              &                                                                       &                                 \end{tikzcd}
\]
Then by surjectivity of the map $\Phi$, we must have $Tr_{\mathcal{T}}^{[GY]}=Tr_{\mathcal{T}}$. But the desired commutative diagram is the boundary of the following
commutative diagram
\[
\begin{tikzcd} \mathrm{Sk}(\Sigma_{\mathcal{T}}) \arrow[d, "\Phi"', two heads] \arrow[r, "\Theta"] & \bigotimes_{T\in\mathcal{T}}\overline{\mathrm{Sk}}^{c}(\mathbb{L}_{T}) \arrow[d, "\bigotimes \Phi_{T}"'] \arrow[rd, "\bigotimes Tr_T^{[\text{GY}]} "] &                                                                       \\ \mathrm{Sk}(Y) \arrow[rd, "\sigma"']                                                & \bigotimes_{T\in\mathcal{T}}\overline{\mathrm{Sk}}^{c}(T) \arrow[d, two heads] \arrow[r, "\bigotimes Tr_{T}^{c}"']                          & \bigotimes_{T\in\mathcal{T}}\hat{\mathcal{G}}(T) \arrow[d, two heads] \\                                                                                   & \overline{\bigotimes}_{T\in\mathcal{T}}\overline{\mathrm{Sk}}^{c}(T) \arrow[r, "\overline{\bigotimes}Tr_{T}^{c}"']                          & \hat{\mathcal{G}}_{\mathcal{T}}                                       \end{tikzcd}
\]
The pentagon on the left commute by Lemma \ref{lem:lantern embeds in T induces map between corner reduced modules};
the upper right triangle commutes by Lemma \ref{lem:compatibility of GY Tr and our Tr for a single T};
the lower right square commutes by Proposition \ref{prop:quantum trace for T glue}.
\end{proof}

\section{Comparison with the Quantum Trace map of Panitch \& Park\label{sec:Compare With PP}}

In this subsection, we give an exact relation between our quantum
trace map and that in \cite{PP1} when $Y$ is a manifold without
boundary (with cusps). 
Therefore we will
assume that $Y$ is a 3-manifold  with empty boundary throughout this section.

The organization of the discussion is similar to Section \ref{sec:equivalence with GY}.
We describe, step by step, the construction in \cite{PP1}
, and along the way we explain how their construction can be compared with ours, culminating in the proof of Theorem \ref{thm: Main theorem 2, exact relation between our Tr and PP's Tr}. 

\subsection{\label{subsec: face cones and face suspensions}Face cones and Face
suspensions}

We start with some basic terminology 
related to ideal tetrahedra and ideal
triangulations.

\begin{defn}
\label{def: vertex cone, edge cone, face cone, face suspension}
Let
$T$ be a single ideal tetrahedron:
\begin{itemize}
\item For a bare vertex $v\in\mathbf{v}(T)$, the \emph{vertex cone} $Cv$
is the line segment connecting $v$ and the barycenter of $T$.  (See Figure \ref{fig:vertex cone, edge cone, face cone}(i))
\item For a bare edge $e\in\mathbf{e}(T)$, the \emph{edge cone} $Ce$ if
the triangle given by the union of all line segments between points
in $e$ and the barycenter of $T$.  (See Figure \ref{fig:vertex cone, edge cone, face cone}(ii))
\item For a bare face $f\in\mathbf{f}(T)$, the\emph{ face cone} $Cf$ is
the 3-ball given by the union of all line segments between points
in $f$ and the barycenter of $T$.  (See Figure \ref{fig:vertex cone, edge cone, face cone}(iii))
\end{itemize}
\begin{figure}[!h]
  \includegraphics[scale=0.2]{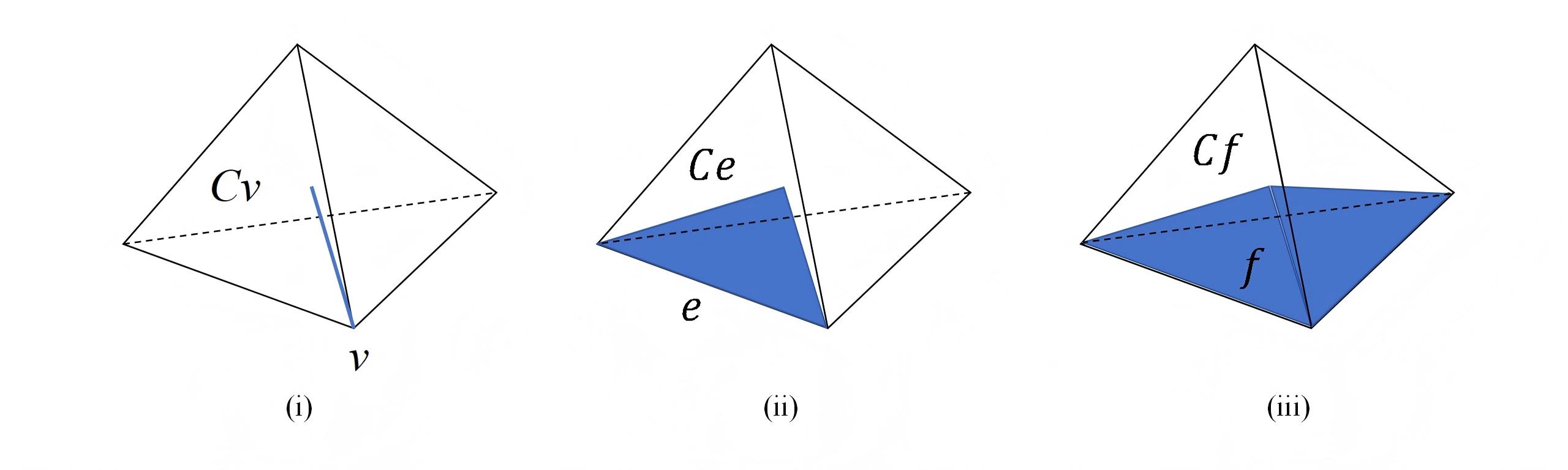} 
  \caption{}
  \label{fig:vertex cone, edge cone, face cone}
\end{figure}
For brevity of notation, if $\mathcal{T}$ is an ideal triangulation of the 3-manifold $Y$, we let $\mathbf{f}(\mathcal{T})=\underset{T\in \mathcal{T}}{\bigcup}\mathbf{f}(T)$. Namely, $\mathbf{f}(\mathcal{T})$ is the set of all bare faces of all ideal tetrahedra in the ideal triangulation $\mathcal{T}$. In this section, we are frequently exposed to situations involving both faces of the ideal triangulation and bare faces of ideal tetrahedra. To distinguish these two cases in notation, we fix an enumeration on the set of all bare faces $\mathbf{f}(\mathcal{T})$, so that $\mathbf{f}(\mathcal{T})=\{f_{1},f_{2},\dots,f_{4N}\}$ (assuming there are $N$ tetrahedra in $\mathcal{T}$). On the other hand, we will never put subscripts to elements of $\mathcal{T}^{(2)}$. Also, we do not enumerate the set $\mathcal{T}$ of ideal tetrahedra in this section. Therefore the subscript ``$i$'' in $f_{i}$ is irrelevant to the question of which tetrahedron $f_{i}$ is a bare face of.

Let $f\in\mathcal{T}^{(2)}$ be a face 
of an ideal triangulation $\mathcal{T}$, given by identifying the bare faces $f_{i}$
and $f_{j}$. The \emph{face suspension} $Sf$ is defined to be the
3-ball given by gluing the two face cones $Cf_{i}$ and $Cf_{j}$ along
the map that identifies $f_{i}$ and $f_{j}$. (See Figure \ref{fig:face suspension})
\begin{figure}[h]
  \includegraphics[scale=0.15]{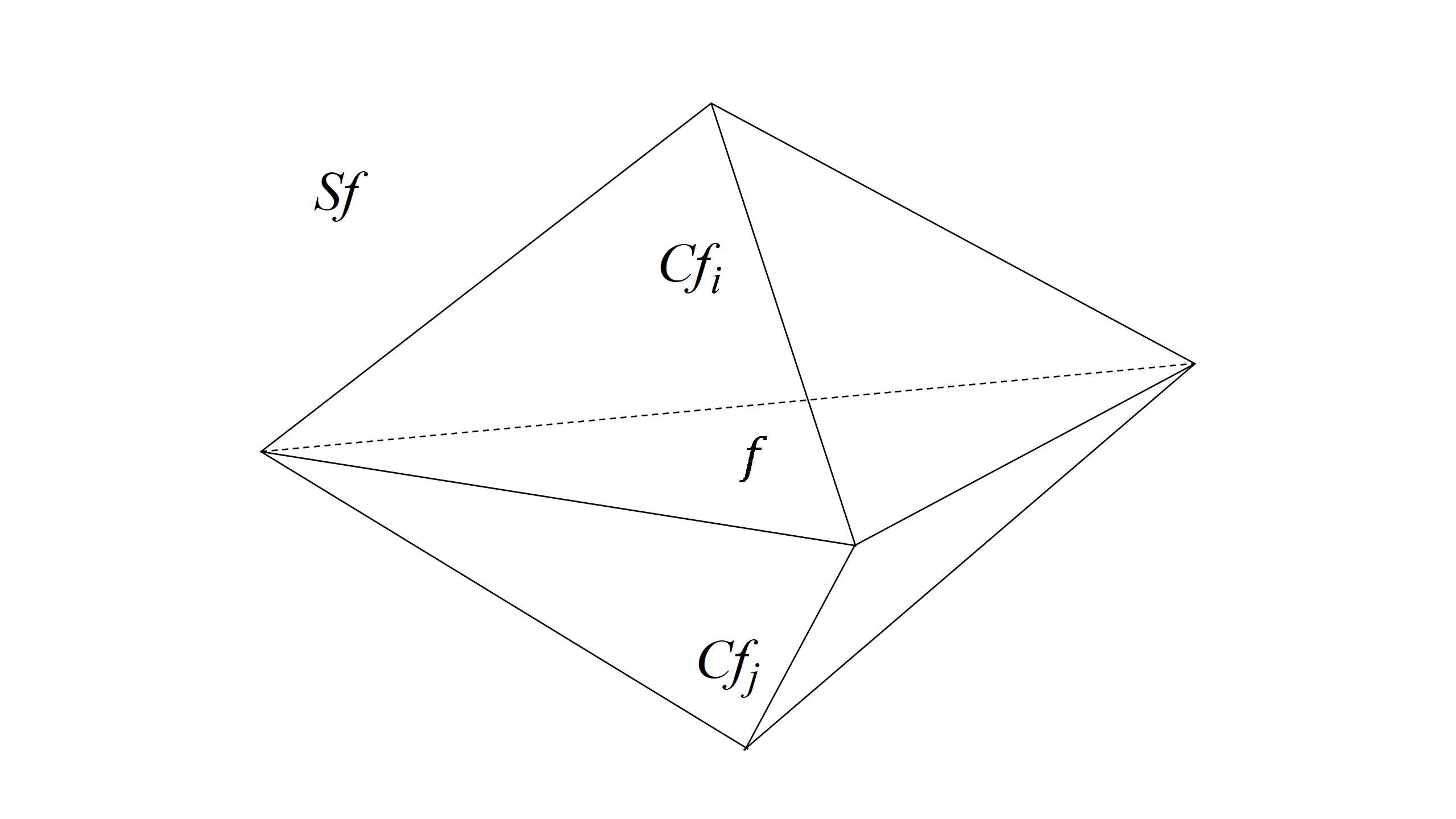} 
  \caption{}
  \label{fig:face suspension}
\end{figure}
\end{defn}

Next we describe the skein modules of the face cones and face suspensions.
The boundary markings are shown in Figure \ref{fig:face cone with bm}
and \ref{fig:face suspension with bm}.
\begin{figure}[h]
  \includegraphics[scale=0.15]{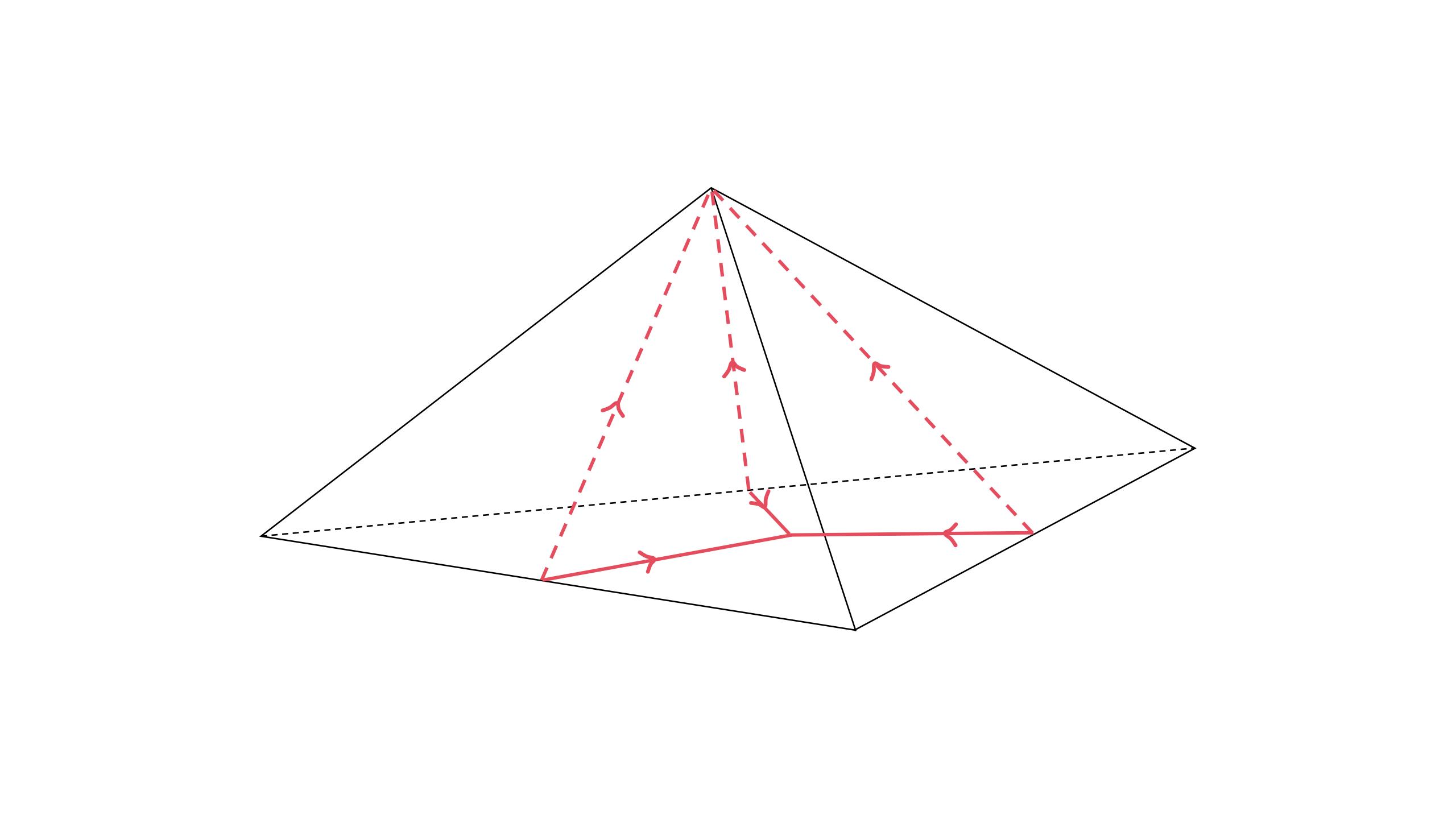} 
  \caption{Boundary marking on a face cone}
  \label{fig:face cone with bm}
\end{figure}
\begin{figure}[h]
  \includegraphics[scale=0.15]{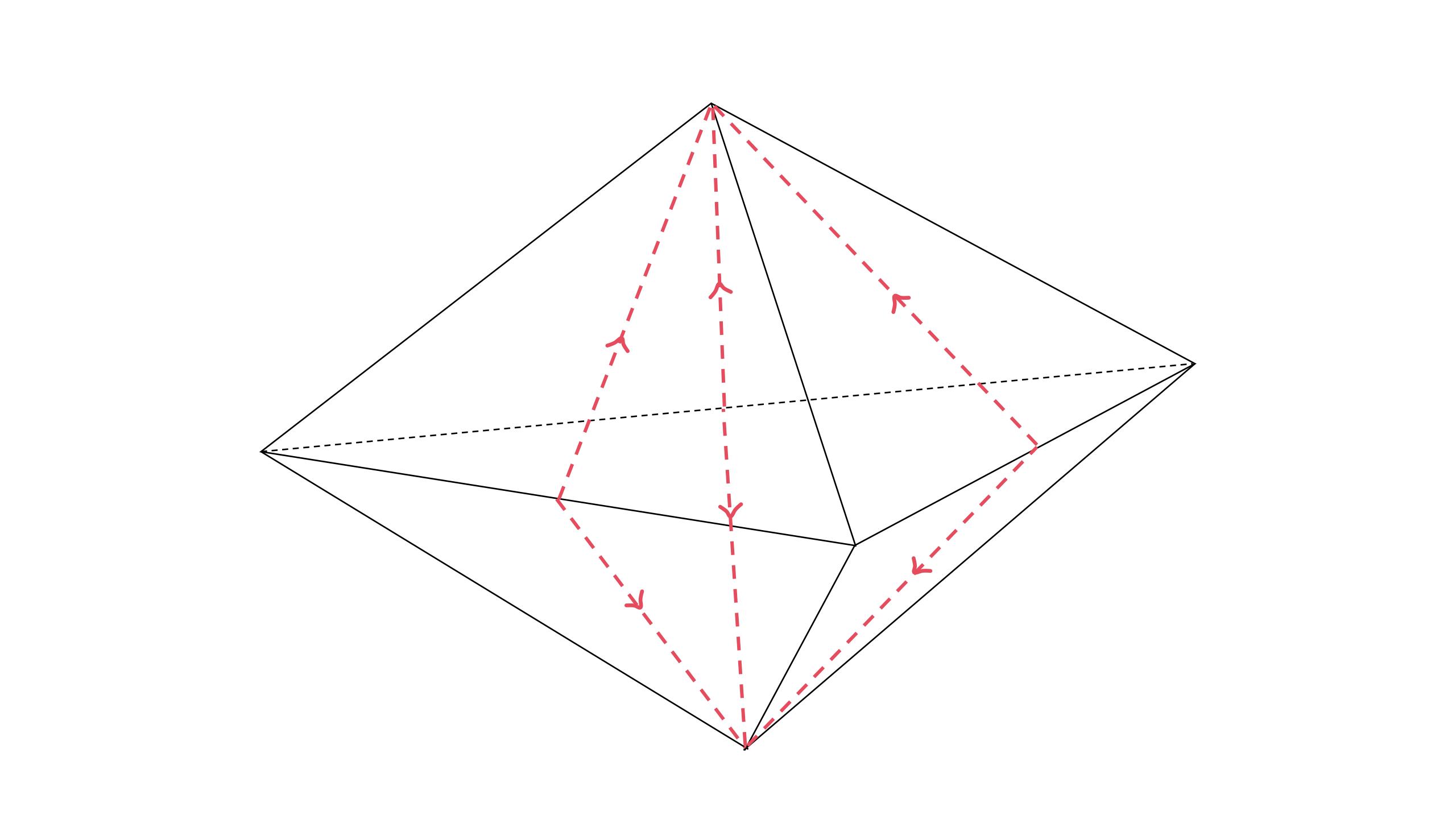} 
  \caption{Boundary marking on a face suspension}
  \label{fig:face suspension with bm}
\end{figure}
For the face cone $Cf_{i}$ of the bare face $f_{i}$ of an ideal tetrahedron,
its boundary is comprised of the bare face $f_{i}$ and three edge cones
of the three boundary edges of $f_{i}$; the boundary marking on $f_{i}$
is the one inherited from the canonical boundary marking on the ideal
tetrahedron, whereas the boundary marking on each edge cone is a single
oriented edge emanating from the barycenter of the edge to the cone
point of the face cone. For the face suspension $Sf$, its boundary
is comprised of six edge cones, two for each of the three
boundary edges of $f$, and the boundary marking on $Sf$ is so that in
each edge cone, it is a single oriented edge emanating from the barycenter
of the edge to the cone point. To distinguish in diagrams, we use solid
oriented edges for marking edges in the bare face $f_{i}$ and dashed
oriented edges for marking edges in the edge cones.

Note that as boundary marked 3-balls, a face cone is homeomorphic
to a face suspension. Moreover, in both cases the boundary markings
are admissible (see Definition \ref{def:combinatorial folation}(iii))
and thus Theorem \ref{thm:skein module of 3-ball} applies and shows
that the reduced skein modules of face cones and face suspensions
are both isomorphic to\footnote{The boundary markings in both cases consist of 2 trivalent sinks and
3 bivalent sources.} 
\[
\frac{\mathbb{T}^{\otimes2}\otimes\mathbb{B}^{\otimes3}}{\mathrm{Ann}([\emptyset])}.
\]
Like what we did in \ref{subsec:Reduced skein module of T}, we will
give explicit algebraic presentations of the reduced skein modules of face
cones and face suspensions. We begin by introducing notation
for the various generators. To start, we first label each marking edge by the following
rule, recalling that every bare edge of an ideal tetrahedron is labeled
by a shape parameter:
\begin{itemize}
\item For the boundary marking of the face cone $Cf_{i}$ of the bare face $f_{i}$:
the marking edge in the bare face $f_{i}$ emanating from the edge labeled
by $a$ will be labeled by $(a,f_{i})$; the marking edge in the edge
cone $Ce$ (being part of the boundary of $Cf_{i}$) of the boundary edge
$e$ of $f_{i}$ labeled by $a$ will be labeled by $\{a,f_{i}\}$ (see Figure \ref{fig:labeling of bm on face cone}).
\begin{figure}[h]
  \includegraphics[scale=0.2]{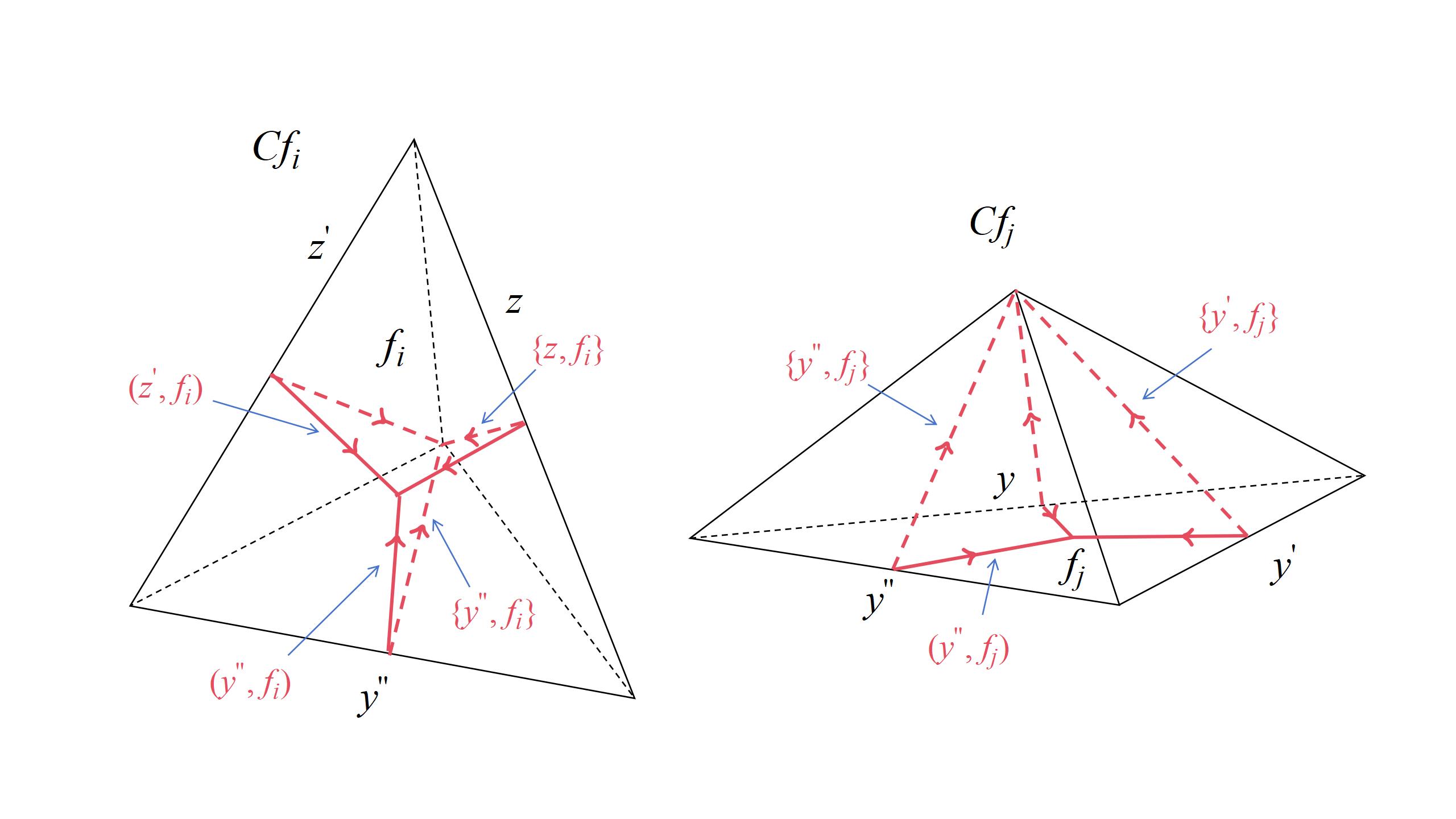} 
  \caption{Labeling of marking edges for face cone. Here $f_{i}$ and $f_{j}$ are different bare faces of the same ideal tetrahedron.}
  \label{fig:labeling of bm on face cone}
\end{figure}
\item For the boundary marking of the face suspension $Sf$ of a face
$f\in\mathcal{T}^{(2)}$: let $f_{i}$ and $f_{j}$ be the bare faces
identified to $f$, so that $Sf$ is formed by gluing the face cones
$Cf_{i}$ and $Cf_{j}$ along the map identifying $f_{i}$ and $f_{j}$.
To label the marking edges on $Sf$ we simply take the labels those edges get when we label them using the prescription just given for the face cones 
$Cf_{i}$
and $Cf_{j}$
(see Figure \ref{fig:labeling of bm on face suspension}).
\begin{figure}[h]
  \includegraphics[scale=0.2]{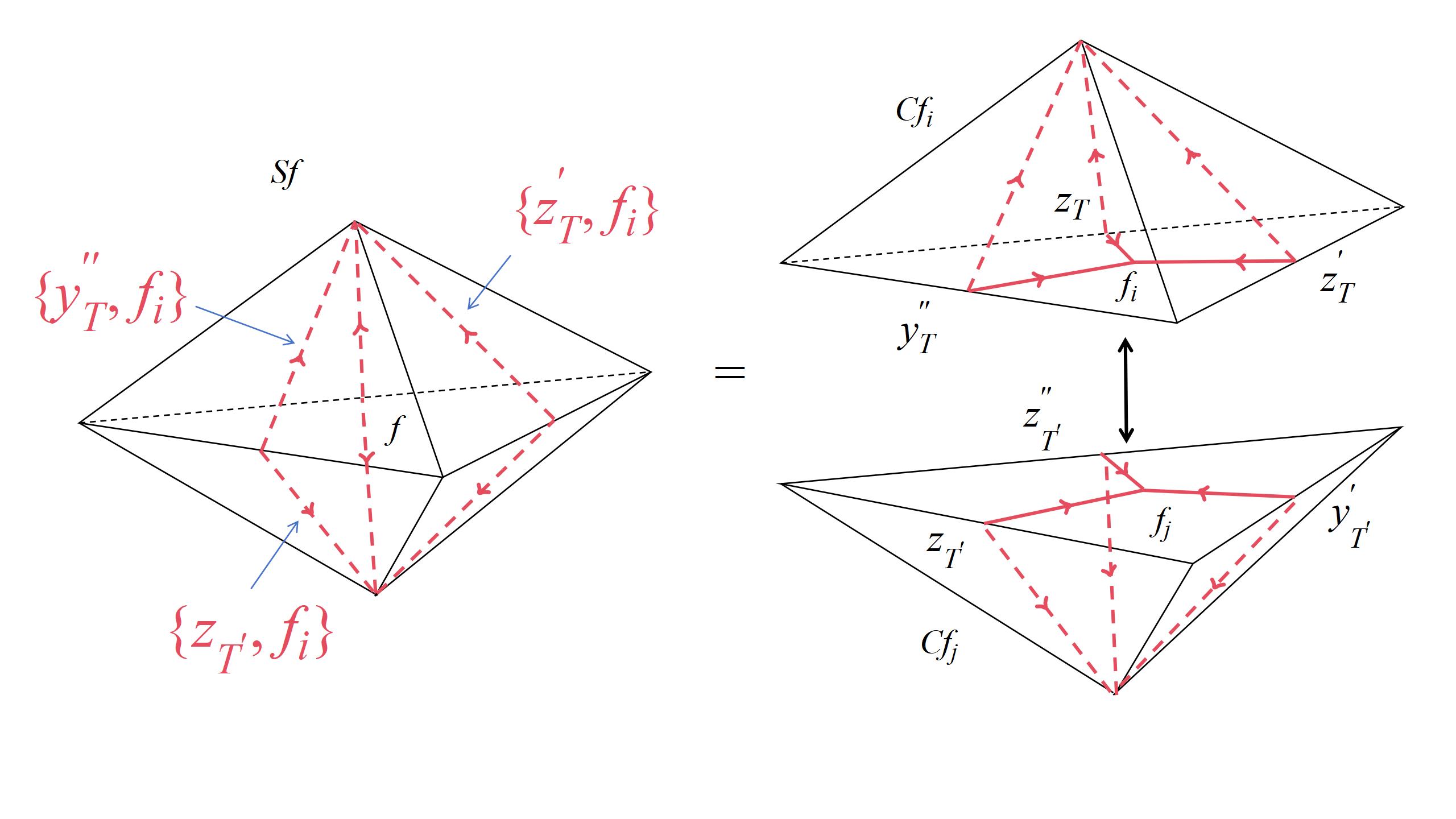} 
  \caption{An example of the labeling of marking edges for face suspension. Here $f_{i}$ and $f_{j}$ are the two bare faces identified to the face $f\in \mathcal{T}^{(2)}$. Moreover, $f_{i}$ is a bare face of the ideal tetrahedron $T$ and $f_{j}$ is a bare face of the ideal tetrahedron $T^{\prime}$}
  \label{fig:labeling of bm on face suspension}
\end{figure}
\end{itemize}

Now let's describe the generators of the various skein algebras associated
to the vertices of the boundary markings. They are demonstrated in
Figures \ref{fig:labeling of generators for face cone} and \ref{fig:labeling of generators for face suspension}:
\begin{enumerate}
\item For the face cone $Cf_{i}$ of the bare face $f_{i}$, the triangle algebra
associated with the trivalent sink at the barycenter of $f_{i}$ is nothing
but the boundary face algebra $\mathbb{T}_{f_i}$ (see Definition \ref{def:boundary face algebra, boundary edge algebra}).
We have already introduced notations for generators of $\mathbb{T}_{f_{i}}$
in \ref{subsec:Reduced skein module of T}, for example the ``untwisted''
trival ribbon arc connecting marking edges labeled $(a,f_{i})$ and $(b,f_{i})$
with both endpoints in $+$ states is denoted by\footnote{We do not need to specify the bare face $f_{i}$ in the notation because the
notation implies that it is the unique bare face abutting the two edges
labeled $a$ and $b$.} $\Gamma_{ab}$. The triangle algebra associated to the other trivalent sink (the
one at the cone point of the face cone) will be denoted as $\tilde{\mathbb{T}}_{f_{i}}$;
for its generators, we denote by $\tilde{\Gamma}_{ab}$ the ``untwisted''
trival ribbon arc connecting marking edges labeled $\{a,f_{i}\}$ and
$\{b,f_{i}\}$ with both endpoints in $+$ states. For the algebra $\mathbb{B}^{\otimes3}$
associated to the three bivalent sources at the barycenters of the
three edges of $f_{i}$, we denote the ``untwisted'' trivial ribbon arc
connecting marking edges labeled $(a,f_{i})$ and $\{a,f_{i}\}$ with both
endpoints in $+$ states is denoted by $x_{a,f_{i}}$.
\begin{figure}[h]
  \includegraphics[scale=0.2]{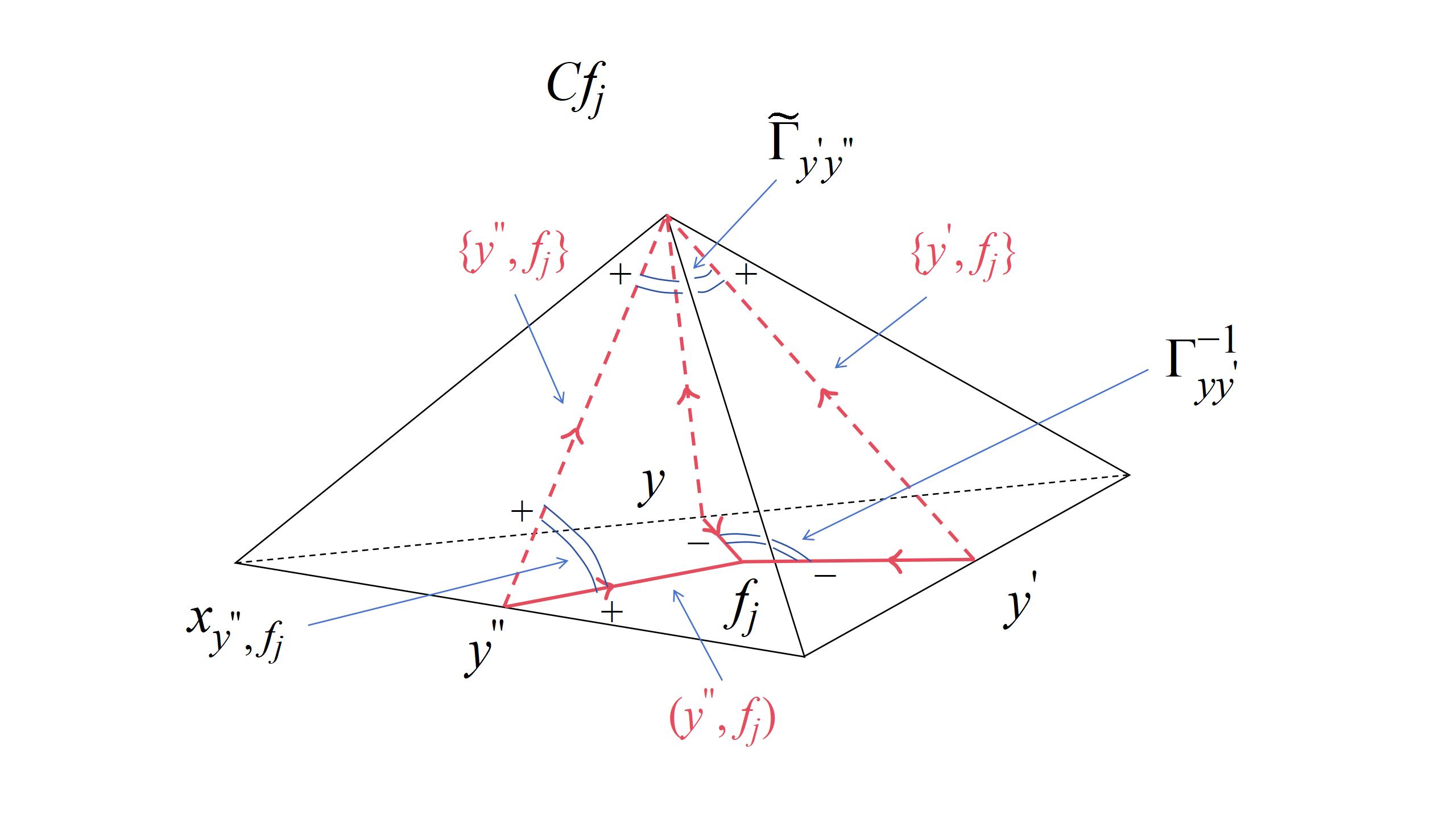} 
  \caption{Examples of generators of the skein algebras associated to the vertex of the boundary marking on the face cone $Cf_{j}$}
  \label{fig:labeling of generators for face cone}
\end{figure}
\item For the face suspension $Sf$, where $f\in\mathcal{T}^{(2)}$ is the
face given by identifying bare faces $f_{i}$ and $f_{j}$, the algebra
$\mathbb{T}^{\otimes2}$ associated to the two trivalent sinks of
the boundary marking (which are the two cone points of the face cones
$Cf_{i}$ and $Cf_{j}$) of $Sf$ is given by $\tilde{\mathbb{T}}_{f_{i}}\otimes\tilde{\mathbb{T}}_{f_{j}}$,
where $\tilde{\mathbb{T}}_{f_{i}}$ (resp. $\tilde{\mathbb{T}}_{f_{j}}$)
is the triangle algebra associated to the trivalent sink at the cone
point of the face cone $Cf_{i}$ (resp. $Cf_{j}$). For the algebra
$\mathbb{B}^{\otimes3}$ associated to the three bivalent sources
at the barycenters of the three edges of $f$ we denote the ``untwisted''
trival ribbon arc connecting marking edges emanating from the same source, one labeled $\{a,f_{i}\}$
and the other $\{b,f_{j}\}$, with both endpoints in $+$ states
by\footnote{This notation may appear to be cumberson, it is however very informative.
For example it implies that the face suspension $Sf$ is obtained
by gluing face cones $Cf_{i}$ and $Cf_{j}$, and when the bare faces
$f_{i}$ and $f_{j}$ are identified, the edge of $f_{i}$ labeled
$a$ is identified with the edge of $f_{j}$ labeled $b$,
namely they are identified to the same edge $e\in\mathcal{T}^{(1)}$.} $x_{a,f_{i};b,f_{j}}$. Note that $x_{a,f_{i};b,f_{j}}$ can be represented by a ribbon arc such that the part of the ribbon arc lying in $Cf_{i}$ (resp. $Cf_{j}$)
is the underlying ribbon arc of $x_{a,f_{i}}$ (resp. $x_{b,f_{j}}$).
\begin{figure}[h]
  \includegraphics[scale=0.2]{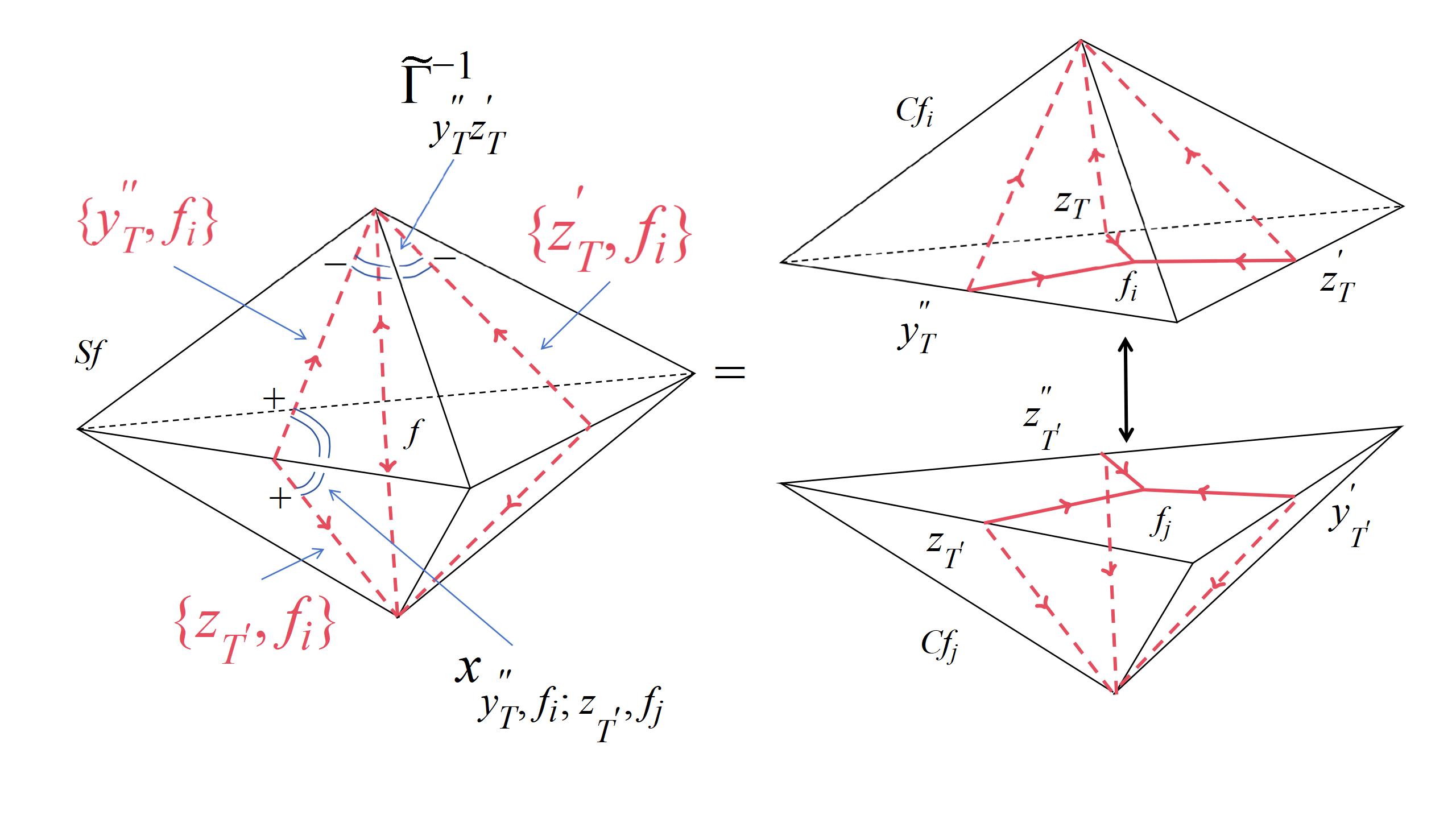} 
  \caption{Examples of the generators of the skein algebras associated to the vertices of the boundary markings on the face suspension $Sf$. Note that the notation $x_{y_{T}^{\prime\prime},f_i;\ z_{T^{\prime}},f_j}$ implies that the edge of $f_i$ labeled $y_{T}^{\prime\prime}$ is identified to the edge of $f_j$ labeled $z_{T^{\prime}}$ when $f_i$ and $f_j$ are identified.}
  \label{fig:labeling of generators for face suspension}
\end{figure}
\end{enumerate}

Now we give the explicit algebraic presentations of the reduced skein
modules $\overline{\mathrm{Sk}}(Cf_{i})$ and $\overline{\mathrm{Sk}}(Sf)$. These
are direct applications of Theorem \ref{thm:skein module of 3-ball}.
\begin{cor}
\label{cor:generators for skein modules of face cones and face suspensions}
(i) For the face cone $Cf_{i}$ of the bare face $f_{i}$, assuming the three
boundary edges of $f_{i}$ are labeled by shape parameters $a,b$ and
$c$ (where $a,b,c$ are in clockwise order with respect to an outward normal vector of the bare face $f_{i}$), then the reduced skein module $\overline{\mathrm{Sk}}(Cf_{i})$
has presentation
\[
\overline{\mathrm{Sk}}(Cf_{i}) =\frac{\mathbb{T}_{f_{i}}\otimes\tilde{\mathbb{T}}_{f_{i}}\otimes\left(\mathbb{B}^{\otimes3}\right)_{Cf_{i}}}{\mathrm{Ann}([\emptyset])},
\]
where $\left(\mathbb{B}^{\otimes3}\right)_{Cf_{i}}$ is the skein algebra associated to the three bivalent sources of the boundary marking on $Cf_{i}$ and is given by the Laurent polynomial algebra
\[
\left(\mathbb{B}^{\otimes3}\right)_{Cf_{i}}=R[x_{a,f_{i}}^{\pm1},\ x_{b,f_{i}}^{\pm1},\ x_{c,f_{i}}^{\pm1}];
\]
and $\mathbb{T}_{f_{i}}$ and $\tilde{\mathbb{T}}_{f_{i}}$ are the skein algebras associated to the trivalent sink at the center of $f_{i}$ and the cone point of $Cf_{i}$, respectively, and are isomorphic to the quantum tori given by 
\[
\mathbb{T}_{f_{i}}=\frac{R\langle\Gamma_{ab}^{\pm1},\Gamma_{bc}^{\pm1},\Gamma_{ca}^{\pm1}\rangle}{\left\langle \begin{array}{c}
\Gamma_{bc}\Gamma_{ab}=A^{-1}\Gamma_{ab}\Gamma_{bc},\\
\Gamma_{ca}\Gamma_{bc}=A^{-1}\Gamma_{bc}\Gamma_{ca},\\
\Gamma_{ab}\Gamma_{ca}=A^{-1}\Gamma_{ca}\Gamma_{ab}
\end{array}\right\rangle }\ ,\ \ 
\tilde{\mathbb{T}}_{f_{i}}=\frac{R\langle\tilde{\Gamma}_{ab}^{\pm1},\tilde{\Gamma}_{bc}^{\pm1},\tilde{\Gamma}_{ca}^{\pm1}\rangle}{\left\langle \begin{array}{c}
\tilde{\Gamma}_{bc}\tilde{\Gamma}_{ab}=A\tilde{\Gamma}_{ab}\tilde{\Gamma}_{bc},\\
\tilde{\Gamma}_{ca}\tilde{\Gamma}_{bc}=A\tilde{\Gamma}_{bc}\tilde{\Gamma}_{ca},\\
\tilde{\Gamma}_{ab}\tilde{\Gamma}_{ca}=A\tilde{\Gamma}_{ca}\tilde{\Gamma}_{ab}
\end{array}\right\rangle }.
\]
The annihilator $Ann([\emptyset])$, as a $\mathbb{T}_{f_{i}}\otimes\tilde{\mathbb{T}}_{f_{i}}\text{-}\left(\mathbb{B}^{\otimes3}\right)_{Cf_{i}}$-sub-bimodule of $\mathbb{T}_{f_{i}}\otimes\tilde{\mathbb{T}}_{f_{i}}\otimes\left(\mathbb{B}^{\otimes3}\right)_{Cf_{i}},$
 is generated by the following three elements
\begin{equation}\label{eq: generators of Ann for face cone}
(-A^{2})\Gamma_{ab}\tilde{\Gamma}_{ab}-x_{a,f_{i}}x_{b,f_{i}}\ ,\ \ 
(-A^{2})\Gamma_{bc}\tilde{\Gamma}_{bc}-x_{b,f_{i}}x_{c,f_{i}}\ ,\ \
(-A^{2})\Gamma_{ca}\tilde{\Gamma}_{ca}-x_{c,f_{i}}x_{a,f_{i}}.
\end{equation}

(ii) For the face suspension $Sf$ of the face $f\in\mathcal{T}^{(2)}$given
by identifying bare faces $f_{i}$ and $f_{j}$. Suppose that the
edges of $f_{i}$ (resp. $f_{j}$) are labeled $a_{i}$, $b_{i}$
and $c_{i}$ (resp. $a_{j}$, $b_{j}$ and $c_{j}$) such that when
$f_{i}$ and $f_{j}$ are identified, the edge labeled $a_{i}$ (resp.
$b_{i}$ or $c_{i}$) is identified with $a_{j}$ (resp. $b_{j}$
or $c_{j}$). Moreover we assume that the labeling $a_{i}$, $b_{i}$ and $c_{i}$ are
clockwise with respect to an outward normal vector of $f_{i}$. (Consequently, the labeling $a_{j}$, $b_{j}$ and $c_{j}$ will be counterclockwise with respect to an outward normal vector of $f_{j}$.)
Then the reduced skein module $\overline{\mathrm{Sk}}(Sf)$ has presentation
\[
\overline{\mathrm{Sk}}(Sf) 
  =\frac{\mathbb{\tilde{T}}_{f_{i}}\otimes\tilde{\mathbb{T}}_{f_{j}}\otimes\left(\mathbb{B}^{\otimes3}\right)_{Sf}}{\mathrm{Ann}([\emptyset])},
\]
where $\left(\mathbb{B}^{\otimes3}\right)_{Sf}$ is the skein algebra associated with the three bivalent sources of the boundary marking on $Sf$ and is given by the Laurent polynomial algebra
\[
\left(\mathbb{B}^{\otimes3}\right)_{Sf}=R[x_{a_{i},f_{i};a_{j},f_{j}}^{\pm1},x_{b_{i},f_{i};b_{j},f_{j}}^{\pm1},x_{c_{i},f_{i};c_{j},f_{j}}^{\pm1}];
\]
and $\tilde{\mathbb{T}}_{f_{i}}$ and $\tilde{\mathbb{T}}_{f_{j}}$
are skein algebras associated to the two cone points of $Sf$ and are isomorphic to the quantum tori given by 
\[
\mathbb{\tilde{T}}_{f_{i}}=\frac{R\langle\tilde{\Gamma}_{a_{i}b_{i}}^{\pm1},\tilde{\Gamma}_{b_{i}c_{i}}^{\pm1},\tilde{\Gamma}_{c_{i}a_{i}}^{\pm1}\rangle}{\left\langle \begin{array}{c}
\tilde{\Gamma}_{b_{i}c_{i}}\tilde{\Gamma}_{a_{i}b_{i}}=A\tilde{\Gamma}_{a_{i}b_{i}}\tilde{\Gamma}_{b_{i}c_{i}},\\
\tilde{\Gamma}_{c_{i}a_{i}}\tilde{\Gamma}_{b_{i}c_{i}}=A\tilde{\Gamma}_{b_{i}c_{i}}\tilde{\Gamma}_{c_{i}a_{i}},\\
\tilde{\Gamma}_{a_{i}b_{i}}\tilde{\Gamma}_{c_{i}a_{i}}=A\tilde{\Gamma}_{c_{i}a_{i}}\tilde{\Gamma}_{a_{i}b_{i}}
\end{array}\right\rangle }\ ,\ \ 
\tilde{\mathbb{T}}_{f_{j}}=\frac{R\langle\tilde{\Gamma}_{a_{j}b_{j}}^{\pm1},\tilde{\Gamma}_{b_{j}c_{j}}^{\pm1},\tilde{\Gamma}_{c_{j}a_{j}}^{\pm1}\rangle}{\left\langle \begin{array}{c}
\tilde{\Gamma}_{b_{j}c_{j}}\tilde{\Gamma}_{a_{j}b_{j}}=A^{-1}\tilde{\Gamma}_{a_{j}b_{j}}\tilde{\Gamma}_{b_{j}c_{j}},\\
\tilde{\Gamma}_{c_{j}a_{j}}\tilde{\Gamma}_{b_{j}c_{j}}=A^{-1}\tilde{\Gamma}_{b_{j}c_{j}}\tilde{\Gamma}_{c_{j}a_{j}},\\
\tilde{\Gamma}_{a_{j}b_{j}}\tilde{\Gamma}_{c_{j}a_{j}}=A^{-1}\tilde{\Gamma}_{c_{j}a_{j}}\tilde{\Gamma}_{a_{j}b_{j}}
\end{array}\right\rangle }.
\]
The annihilator $Ann([\emptyset])$, as a $\mathbb{\tilde{T}}_{f_{i}}\otimes\tilde{\mathbb{T}}_{f_{j}}\text{-}\left(\mathbb{B}^{\otimes3}\right)_{Sf}$-sub-bimodule of $\mathbb{\tilde{T}}_{f_{i}}\otimes\tilde{\mathbb{T}}_{f_{j}}\otimes\left(\mathbb{B}^{\otimes3}\right)_{Sf}$, is generated by the following three elements
\begin{equation}
\begin{array}{c}
(-A^{2})\tilde{\Gamma}_{a_{i}b_{i}}\tilde{\Gamma}_{a_{j}b_{j}}-x_{a_{i},f_{i};a_{j},f_{j}}x_{b_{i},f_{i};b_{j},f_{j}},\\
(-A^{2})\tilde{\Gamma}_{b_{i}c_{i}}\tilde{\Gamma}_{b_{j}c_{j}}-x_{b_{i},f_{i};b_{j},f_{j}}x_{c_{i},f_{i};c_{j},f_{j}},\\
(-A^{2})\tilde{\Gamma}_{c_{i}a_{i}}\tilde{\Gamma}_{c_{j}a_{j}}-x_{c_{i},f_{i};c_{j},f_{j}}x_{a_{i},f_{i};a_{j},f_{j}}.
\end{array}\label{eq:generators of Ann for face suspension}
\end{equation}
\end{cor}

\subsection{Partial corner reduction for face cones\label{subsec: partial corner reduction}}

As mentioned in the introduction, we require a further quotient of the reduced skein module $\overline{\mathrm{Sk}}(Cf_{i})$ of a face cone by taking corner-reduction on the bare face $f_{i}$ only. Let's discuss this construction in detail. 

Like what we did when we introduced corner-reduction in \ref{subsec:Corner reduction}, first we need
to introduce a $\cdot$-product on the various skein algebras
associated with the vertices of the boundary marking.
Different to \ref{subsec:Corner reduction}, the face cone $Cf_{i}$ of a bare face
$f_{i}$ is not an ideally triangulated boundary marked 3-manifold
(as given by Definition \ref{def:ideally triangulated boundary marked 3-manifolds}).
However we can still define a suitable version of the $\cdot$-product. Just as
in \ref{subsec:Corner reduction}, we let $M_{f_{i}}$ be the set
of marking edges in the bare face $f_{i}$. Recall that the skein
module $\overline{\mathrm{Sk}}(Cf_{i})$ of the face cone of $f_{i}$,
the skein algebra $\mathbb{T}_{f_{i}}$ associated to the trivalent
sink in $f_{i}$ and the skein algebra $\left(\mathbb{B}^{\otimes3}\right)_{Cf_{i}}$
associated with the three bivalent sources all have $\mathbb{Z}^{M_{f_{i}}}$-gradings
(given by the sum of states of endpoints on the marking edges in $f_{i}$).
Also recall that there is a skew-symmetric bilinear form $\langle\ ,\ \rangle_{f_{i}}$
on $\mathbb{Z}^{M_{f_{i}}}$ given by the skew-symmetric bilinear
form on $\mathbb{Z}^{3}$ with matrix $\left(\begin{array}{ccc}
0 & 1 & -1\\
-1 & 0 & 1\\
1 & -1 & 0
\end{array}\right)$ along with an identification $\mathbb{Z}^{M_{f_{i}}}\cong\mathbb{Z}^{3}$
induced by a clockwise labeling of the marking edges on $f_{i}$ with
respect to an outward normal of $f_{i}$ .

\begin{defn}
\label{def: defn of cdot-product in the case of face cone}
(i) Let $\Gamma$
and $\Gamma^{\prime}$ be $\mathbb{Z}^{M_{f_{i}}}$-homogeneous
elements in $\mathbb{T}_{f_{i}}$, the skein algebra associated to
the trivalent sink at the center of the bare face $f_{i}$. We define
\[
\Gamma\cdot\Gamma^{\prime}=A^{-\frac{1}{2}\langle d_{f_{i}}(\Gamma),d_{f_{i}}(\Gamma^{\prime})\rangle_{f_{i}}}\Gamma\Gamma^{\prime}
\]
and extend by bilinearity to all elements of $\mathbb{T}_{f_{i}}$.
This is in fact the same $\cdot$-product on $\mathbb{T}_{f_{i}}$
we have introduced in \ref{subsec:Corner reduction} (see the discussion
right after Definition \ref{def: defn of cdot-product}).

(ii) Let $x$ and $x^{\prime}$ be $\mathbb{Z}^{M_{f_{i}}}$-homogeneous
elements in $\left(\mathbb{B}^{\otimes3}\right)_{Cf_{i}}$, the skein
algebra associated to the three bivalent sources. We define
\[
x\cdot x^{\prime}=A^{-\frac{1}{2}\langle d_{f_{i}}(x),d_{f_{i}}(x^{\prime})\rangle_{f_{i}}}xx^{\prime}
\]
and extend by bilinearity to all elements of $\left(\mathbb{B}^{\otimes3}\right)_{Cf_{i}}$.
Note that although the elements of $\left(\mathbb{B}^{\otimes3}\right)_{Cf_{i}}$
can have endpoints on both the marking edges in the bare face $f_{i}$
and the marking edges in the edge cones, the $A$-factor in the
definition of $\cdot$-product only ``sees'' the endpoints on
the marking edges in $f_{i}$.

(iii) We can also twist the the module structure on $\overline{\mathrm{Sk}}(Cf_{i})$
over these algebras. Let $\Gamma\in\mathbb{T}_{f_{i}}$, $x\in\left(\mathbb{B}^{\otimes3}\right)_{Cf_{i}}$
and $w\in\overline{\mathrm{Sk}}(Cf_{i})$ be $\mathbb{Z}^{M_{f_{i}}}$-homogeneous,
we define 
\[
\Gamma\cdot w=A^{-\frac{1}{2}\langle d_{f_{i}}(\Gamma),d_{f_{i}}(w)\rangle_{f_{i}}}\Gamma w
\]
and 
\[
w\cdot x=A^{-\frac{1}{2}\langle d_{f_{i}}(w),d_{f_{i}}(x)\rangle_{f_{i}}}wx
\]
and extend bilinearly. Once again, the $A$-factors only ``see''
the endpoints on the marking edges in the bare face $f_{i}$.

\end{defn}

The following is immediate from the last definition.

\begin{lem}
\label{lem:presentation of (T,cdot) and (B^tensor 3, cdot)}
(i) The
algebra $\left(\mathbb{T}_{f_{i}},\cdot\right)$ is  the Laurent polynomial
algebra (assuming the boundary edges of $f$ are labeled $a$, $b$
and $c$)
\[
R\left[Cf_{i}\right]:=R\left[\Gamma_{ab}^{\pm1},\ \Gamma_{bc}^{\pm1},\ \Gamma_{ca}^{\pm1}\right].
\]
In this case, a monomial in $\left(\mathbb{T}_{f_{i}},\cdot\right)=R\left[Cf_{i}\right]$
equals the Weyl-ordering of the same monomial under the $\cup$-product
in the common underlying set $\mathbb{T}_{f_{i}}$.

(ii) The algebra $\left(\left(\mathbb{B}^{\otimes3}\right)_{Cf_{i}},\cdot\right)$
is the quantum torus (assuming the labeling of edges by $a$, $b$
and $c$ are in clockwise order with respect to an outward normal
of $f_{i}$)
\[
\mathbb{T}\langle Cf_{i}\rangle:=\frac{R\langle x_{a,f_{i}}^{\pm1},\ x_{b,f_{i}}^{\pm1},\ x_{c,f_{i}}^{\pm1}\rangle}{\left\langle \begin{array}{c}
x_{b,f_{i}}\cdot x_{a,f_{i}}=Ax_{a,f_{i}}\cdot x_{b,f_{i}},\\
x_{c,f_{i}}\cdot x_{b,f_{i}}=Ax_{b,f_{i}}\cdot x_{c,f_{i}},\\
x_{a,f_{i}}\cdot x_{c,f_{i}}=Ax_{c,f_{i}}\cdot x_{a,f_{i}}
\end{array}\right\rangle }.
\]
In this case, a monomial in $\left(\mathbb{B}^{\otimes3}\right)_{Cf_{i}}$
(which is with respect to $\cup$-product) equals the Weyl-ordering
of the same monomial under the $\cdot$-product. 
\end{lem}

Once again for consistency, the multiplications in the algebras $R[Cf_{i}]$ and $\mathbb{T}\langle Cf_{i}\rangle$ will always be written as $\cdot$. Also note that $\underset{f_{i}\in\mathbf{f}(T)}{\bigotimes}R[Cf_{i}]$
is nothing but the algebra $R[T]$.

The upshot is that the skein module $\overline{\mathrm{Sk}}(Cf_{i})$ is
now a left $\left(\mathbb{T}_{f_{i}},\cdot\right)=R\left[Cf_{i}\right]$-module
under the $\cdot$-action, a right $\left(\left(\mathbb{B}^{\otimes3}\right)_{Cf_{i}},\cdot\right)=\mathbb{T}\left\langle Cf_{i}\right\rangle $-module
under the $\cdot$-action, and still a left $\tilde{\mathbb{T}}_{f_{i}}$-module
under the $\cup$-action at the cone point. These module structures are compatible
with each other, meaning $\overline{\mathrm{Sk}}(Cf_{i})$ is a 
\[
\left(\mathbb{T}_{f_{i}},\cdot\right)\otimes\tilde{\mathbb{T}}_{f_{i}}\text{-}\left(\left(\mathbb{B}^{\otimes3}\right)_{Cf_{i}},\cdot\right)=R\left[Cf_{i}\right]\otimes\tilde{\mathbb{T}}_{f_{i}}\text{-}\mathbb{T}\left\langle Cf_{i}\right\rangle 
\]
-bimodule\footnote{Here, the ring structure on $R\left[Cf\right]\otimes\tilde{\mathbb{T}}_{f}$
is the tensor product ring structure.}.

Similar to what we did in \ref{subsec:corner reduction in the case of a singe tetrahedron},
we can also describe this $R\left[Cf_{i}\right]\otimes\tilde{\mathbb{T}}_{f_{i}}\text{-}\mathbb{T}\left\langle Cf_{i}\right\rangle$-bimodule structure on $\overline{\mathrm{Sk}}(Cf_{i})$ in an alternative manner. We first introduce a 
\[
\left(\mathbb{T}_{f_{i}},\cdot\right)\otimes\tilde{\mathbb{T}}_{f_{i}}\text{-}\left(\left(\mathbb{B}^{\otimes3}\right)_{Cf_{i}},\cdot\right)=R\left[Cf_{i}\right]\otimes\tilde{\mathbb{T}}_{f_{i}}\text{-}\mathbb{T}\left\langle Cf_{i}\right\rangle
\]
-bimodule structure on 
\[
\left(\mathbb{T}_{f_{i}},\cdot\right)\otimes\tilde{\mathbb{T}}_{f_{i}}\otimes\left(\left(\mathbb{B}^{\otimes3}\right)_{Cf_{i}},\cdot\right)=R\left[Cf_{i}\right]\otimes\tilde{\mathbb{T}}_{f_{i}}\otimes\mathbb{T}\left\langle Cf_{i}\right\rangle
\]
and then realize $\overline{\mathrm{Sk}}(Cf_{i})$ as a quotient $R\left[Cf_{i}\right]\otimes\tilde{\mathbb{T}}_{f_{i}}\text{-}\mathbb{T}\left\langle Cf_{i}\right\rangle$-bimodule.
The $R\left[Cf_{i}\right]\otimes\tilde{\mathbb{T}}_{f_{i}}\text{-}\mathbb{T}\left\langle Cf_{i}\right\rangle$-bimodule
structure on $R\left[Cf_{i}\right]\otimes\tilde{\mathbb{T}}_{f_{i}}\otimes\mathbb{T}\left\langle Cf_{i}\right\rangle$
is defined as follows:
\begin{enumerate}
\item Let $\Gamma,\Gamma^{\prime}\in R\left[Cf_{i}\right]$, $\tilde{\Gamma}\in\tilde{\mathbb{T}}_{f_{i}}$
and $x\in\mathbb{T}\left\langle Cf_{i}\right\rangle$ be $\mathbb{Z}^{M_{f_{i}}}$-homogeneous,
define
\begin{align*}
\Gamma^{\prime}\cdot\left(\Gamma\otimes\tilde{\Gamma}\otimes x\right) & =A^{-\frac{1}{2}\langle d_{f_{i}}(\Gamma^{\prime}),d_{f_{i}}(x)\rangle_{f_{i}}}\left(\Gamma^{\prime}\cdot\Gamma\right)\otimes\tilde{\Gamma}\otimes x\\
 & =A^{-\frac{1}{2}\langle d_{f_{i}}(\Gamma^{\prime}),d_{f_{i}}(\Gamma)\rangle_{f_{i}}-\frac{1}{2}\langle d_{f_{i}}(\Gamma^{\prime}),d_{f_{i}}(x)\rangle_{f_{i}}}\left(\Gamma^{\prime}\Gamma\right)\otimes\tilde{\Gamma}\otimes x
\end{align*}
and extend by bilinearity. This defines a left $R\left[Cf_{i}\right]$-module
structure on $R\left[Cf_{i}\right]\otimes\tilde{\mathbb{T}}_{f_{i}}\otimes\mathbb{T}\left\langle Cf_{i}\right\rangle$.
\item Let $\Gamma\in R\left[Cf_{i}\right]$, $\tilde{\Gamma}\in\tilde{\mathbb{T}}_{f_{i}}$,
$x,x^{\prime}\in\mathbb{T}\left\langle Cf_{i}\right\rangle$ be $\mathbb{Z}^{M_{f_{i}}}$-homogeneous,
define
\begin{align*}
\left(\Gamma\otimes\tilde{\Gamma}\otimes x\right)\cdot x^{\prime} & =A^{-\frac{1}{2}\langle d_{f_{i}}(\Gamma),d_{f_{i}}(x^{\prime})\rangle_{f_{i}}}\Gamma\otimes\tilde{\Gamma}\otimes\left(x\cdot x^{\prime}\right)\\
 & =A^{-\frac{1}{2}\langle d_{f_{i}}(\Gamma),d_{f_{i}}(x^{\prime})\rangle_{f_{i}}-\frac{1}{2}\langle d_{f_{i}}(x),d_{f_{i}}(x^{\prime})\rangle_{f_{i}}}\Gamma\otimes\tilde{\Gamma}\otimes\left(xx^{\prime}\right)
\end{align*}
and extend by bilinearity. This defines a right $\mathbb{T}\left\langle Cf_{i}\right\rangle$-module
structure on $R\left[Cf_{i}\right]\otimes\tilde{\mathbb{T}}_{f_{i}}\otimes\mathbb{T}\left\langle Cf_{i}\right\rangle$.
\item Therefore if $\Gamma,\Gamma^{\prime}\in R\left[Cf_{i}\right]$, $\tilde{\Gamma}\in\tilde{\mathbb{T}}_{f_{i}}$
and $x,x^{\prime}\in\mathbb{T}\left\langle Cf_{i}\right\rangle$ are $\mathbb{Z}^{M_{f_{i}}}$-homogeneous,
we have
\begin{eqnarray*}
\lefteqn{\left(\Gamma^{\prime}\cdot\left(\Gamma\otimes\tilde{\Gamma}\otimes x\right)\right)\cdot x^{\prime}=\Gamma^{\prime}\cdot\left(\left(\Gamma\otimes\tilde{\Gamma}\otimes x\right)\cdot x^{\prime}\right)}\\
&=&A^{-\frac{1}{2}\langle d_{f_{i}}(\Gamma^{\prime}),d_{f_{i}}(x)\rangle_{f_{i}}-\frac{1}{2}\langle d_{f_{i}}(\Gamma^{\prime}),d_{f_{i}}(x^{\prime})\rangle_{f_{i}}-\frac{1}{2}\langle d_{f_{i}}(\Gamma),d_{f_{i}}(x^{\prime})\rangle_{f_{i}}}\left(\Gamma^{\prime}\cdot\Gamma\right)\otimes\tilde{\Gamma}\otimes\left(x\cdot x^{\prime}\right)\\
&=&A^{-\frac{1}{2}\langle d_{f_{i}}(\Gamma^{\prime}),d_{f_{i}}(\Gamma)\rangle_{f_{i}}-\frac{1}{2}\langle d_{f_{i}}(\Gamma^{\prime}),d_{f_{i}}(x)\rangle_{f_{i}}-\frac{1}{2}\langle d_{f_{i}}(\Gamma^{\prime}),d_{f_{i}}(x^{\prime})\rangle_{f_{i}}-\frac{1}{2}\langle d_{f_{i}}(\Gamma),d_{f_{i}}(x^{\prime})\rangle_{f_{i}}-\frac{1}{2}\langle d_{f_{i}}(x),d_{f_{i}}(x^{\prime})\rangle_{f_{i}}}\\
& &\times\left(\Gamma^{\prime}\Gamma\right)\otimes\tilde{\Gamma}\otimes\left(xx^{\prime}\right).
\end{eqnarray*}
Hence $R\left[Cf_{i}\right]\otimes\tilde{\mathbb{T}}_{f_{i}}\otimes\mathbb{T}\left\langle Cf_{i}\right\rangle$
becomes a $R\left[Cf_{i}\right]\text{-}\mathbb{T}\left\langle Cf_{i}\right\rangle$-bimodule.
\item $R\left[Cf_{i}\right]\otimes\tilde{\mathbb{T}}_{f_{i}}\otimes\mathbb{T}\left\langle Cf_{i}\right\rangle$
is naturally a left $\tilde{\mathbb{T}}_{f_{i}}$-module ($\tilde{\mathbb{T}}_{f_{i}}$
is in the usual $\cup$-product structure) coming from the tensor
product. This left $\tilde{\mathbb{T}}_{f_{i}}$-module structure is
compatible with the $R\left[Cf_{i}\right]\text{-}\mathbb{T}\left\langle Cf_{i}\right\rangle$-bimodule
structure given above, as can be quickly checked.
Hence $R\left[Cf_{i}\right]\otimes\tilde{\mathbb{T}}_{f_{i}}\otimes\mathbb{T}\left\langle Cf_{i}\right\rangle$
is a $R\left[Cf_{i}\right]\otimes\tilde{\mathbb{T}}_{f_{i}}\text{-}\mathbb{T}\left\langle Cf_{i}\right\rangle$-bimodule.
\end{enumerate}

We can immediately give the desired algebraic presentation of 
$\overline{\mathrm{Sk}}(Cf_{i})$.

\begin{lem}
\label{lem: generators of Ann(emptyset) for Cf are homogeneous}
The generators of $\mathrm{Ann}(\emptyset)$ for the face cone $Cf_{i}$ given
by (\ref{eq: generators of Ann for face cone}) in Corollary \ref{cor:generators for skein modules of face cones and face suspensions}(i) are $\mathbb{Z}^{M_{f_{i}}}$-homogeneous, therefore $\mathrm{Ann}(\emptyset)$
for $Cf_{i}$ can be viewed as the $R\left[Cf_{i}\right]\otimes\tilde{\mathbb{T}}_{f_{i}}\text{-}\mathbb{T}\left\langle Cf_{i}\right\rangle $-sub-bimodule
of $R\left[Cf_{i}\right]\otimes\tilde{\mathbb{T}}_{f_{i}}\otimes\mathbb{T}\left\langle Cf_{i}\right\rangle $
generated by the elements given by (\ref{eq: generators of Ann for face cone}).
Therefore $\overline{\mathrm{Sk}}(Cf_{i})$ can be alternatively given as
the $R\left[Cf_{i}\right]\otimes\tilde{\mathbb{T}}_{f_{i}}\text{-}\mathbb{T}\left\langle Cf_{i}\right\rangle $-bimodule quotient
of $R\left[Cf_{i}\right]\otimes\tilde{\mathbb{T}}_{f_{i}}\otimes\mathbb{T}\left\langle Cf_{i}\right\rangle$ by $\mathrm{Ann}[\emptyset]$.
\end{lem}
Now we can define the partial corner reduction for the face cone
$Cf_{i}$ of a bare face $f_{i}$. For the commutative algebra $(\mathbb{T}_{f_{i}},\cdot)=R\left[Cf_{i}\right]$,
we have the ideal $I_{f_{i}}^{c}$ introduced in \ref{subsec:Corner reduction}
(see Definition \ref{def:corner reduced module} (i)), which
in the notation of the current section is (assuming the boundary
edges of $f_{i}$ are labeled $a$, $b$ and $c$)
\[
I_{f_{i}}^{c}=\left\langle \Gamma_{ab}-(-A^{2})^{-\frac{1}{2}},\ \Gamma_{bc}-(-A^{2})^{-\frac{1}{2}},\ \Gamma_{ca}-(-A^{2})^{-\frac{1}{2}}\right\rangle .
\]
\begin{defn}
\label{def: partial corner reduced module of face cone} 
The \emph{partially
corner-reduced skein module} $\overline{\mathrm{Sk}}^{pc}(Cf_{i})$ of the
face cone is defined to be the $R$-module quotient
\[
\overline{\mathrm{Sk}}^{pc}(Cf_{i})=\frac{\overline{\mathrm{Sk}}(Cf_{i})}{I_{f_{i}}^{c}\cdot\overline{\mathrm{Sk}}(Cf_{i})}.
\]
\end{defn}
The $R$-submodule $I_{f_{i}}^{c}\cdot\overline{\mathrm{Sk}}(Cf_{i})$ of
$\overline{\mathrm{Sk}}(Cf_{i})$ is clearly also a 
$
R\left[Cf_{i}\right]\otimes\tilde{\mathbb{T}}_{f_{i}}\text{-}\mathbb{T}\left\langle Cf_{i}\right\rangle
$
sub-bimodule, therefore $\overline{\mathrm{Sk}}^{pc}(Cf_{i})$ is also a quotient 
$
R\left[Cf_{i}\right]\otimes\tilde{\mathbb{T}}_{f_{i}}\text{-}\mathbb{T}\left\langle Cf_{i}\right\rangle
$
bimodule of $\overline{\mathrm{Sk}}(Cf_{i})$ and thus of $R\left[Cf_{i}\right]\otimes\tilde{\mathbb{T}}_{f_{i}}\otimes\mathbb{T}\left\langle Cf_{i}\right\rangle $,
where $R\left[Cf_{i}\right]$ and $\mathbb{T}\left\langle Cf_{i}\right\rangle $ act by
$\cdot$ and $\tilde{\mathbb{T}}_{f_{i}}$ acts by $\cup$. 

\subsection{Splitting of the 3-manifolds into face suspensions and face cones\label{subsec:splitting homomorphisms for face cones and face suspensions}}

The first step of Panitch \& Park's construction of their quantum trace
map is to split the ideally triangulated 3-manifold into face suspensions.
On the other hand, the first step of our construction is to split the 3-manifold
into ideal tetrahedra. If, in the approach of Panitch \& Park, we
further split each face suspension $Sf$ into two face cones $Cf_{i}$
and $Cf_{j}$ where $f_{i}$ and $f_{j}$ are the bare faces identified
to $f$; and if, in our approach, we further split each tetrahedron $T$
into four face cones, one for each of the four bare faces of $T$;
then the result in both cases is the splitting of the 3-manifold into face cones. This
serves as the common refinement of both approaches. 

In summary, if we let $(Y,\mathcal{T})$
be our ideally triangulated 3-manifold, then we have the following commutative
diagram of splittings
\[
\begin{tikzcd}
                        & \underset{f\in\mathcal{T}^{(2)}}{\coprod}Sf  \arrow[rd] &                                                              \\
Y \arrow[ru] \arrow[rd] &                                                         & {{\underset{f_{i}\in\mathbf{f}(\mathcal{T})}{\coprod}Cf_{i}} } \\
                        &  \underset{T\in\mathcal{T}}{\coprod}T \arrow[ru]        &                                                             
\end{tikzcd}
\]

Our task in this subsection is to study the homomorphisms of skein
modules induced by these splittings and produce the following commutative
diagram:
\[
\begin{tikzcd}
               & \underset{f\in\mathcal{T}^{(2)}}{\overline{\bigotimes}}\overline{\mathrm{Sk}}(Sf) \arrow[rd, "\large{\textcircled{\small{2}}}"] &          \\
\mathrm{Sk}(Y)  \arrow[ru, "\large{\textcircled{\small{1}}}"] \arrow[rd, "\large{\textcircled{\small{3}}}"'] &                                       & {{\underset{f_{i}\in\mathbf{f}(\mathcal{T})}{\overline{\bigotimes}}\overline{\mathrm{Sk}}^{pc}(Cf_{i})}} \\
               & \underset{T\in\mathcal{T}}{\overline{\bigotimes}}\overline{\mathrm{Sk}}^{c}(T) \arrow[ru, "\large{\textcircled{\small{4}}}"']   &                                                                                                     
\end{tikzcd}
\]

We first describe these homomorphisms:

{\large{\textcircled{\small{1}}}}:
This map is the splitting homomorphism
\[
\tilde{\sigma}\colon\mathrm{Sk}(Y)\rightarrow\underset{f\in\mathcal{T}^{(2)}}{\overline{\bigotimes}}\overline{\mathrm{Sk}}(Sf)
\]
induced by splitting the 3-manifold $Y$ along the edge cones into face suspensions. Let's first describe the codomain $\underset{f\in\mathcal{T}^{(2)}}{\overline{\bigotimes}}\overline{\mathrm{Sk}}(Sf)$.

\begin{defn}
\label{def:reduced tensor product of skein modules of face suspensions}
(i) Let $T$ be an ideal tetrahedron and suppose that the labeling of its edges by shape parameters is given as in Figure \ref{fig:ideal T with bm and labeling of generators}. Also recall our naming convention of generators of various skein algebras associated with the vertices of the boundary marking. Each ideal vertex of $T$ gives us a pair of elements of the algebra
$\underset{f_{i}\in\mathbf{f}(T)}{\bigotimes}\tilde{\mathbb{T}}_{f_{i}}$.
For simplicity we only write down the pair corresponding to
the vertex incident to edges labeled $z$, $z^{\prime}$ and $z^{\prime\prime}$:
\[
\tilde{v}_{zz^{\prime}z^{\prime\prime}}:=\tilde{\Gamma}_{zz^{\prime}}\tilde{\Gamma}_{z^{\prime}z^{\prime\prime}}\tilde{\Gamma}_{z^{\prime\prime}z}-(-A^{2})^{-1},
\]
\[
\tilde{\ell}_{zz^{\prime\prime}}:=\left[\tilde{\Gamma}_{zy^{\prime}}^{-1}\tilde{\Gamma}_{y^{\prime}z^{\prime\prime}}\right]-\left[\tilde{\Gamma}_{z^{\prime\prime}y}^{-1}\tilde{\Gamma}_{yz^{\prime}}\right]\tilde{\Gamma}_{zz^{\prime}}-\tilde{\Gamma}_{z^{\prime}z^{\prime\prime}}^{-1}\left[\tilde{\Gamma}_{z^{\prime}y^{\prime\prime}}^{-1}\tilde{\Gamma}_{y^{\prime\prime}z}\right];
\]
we let $\tilde{I}_{T}$ be the\emph{ right} ideal of $\underset{f_{i}\in\mathbf{f}(T)}{\bigotimes}\tilde{\mathbb{T}}_{f_{i}}$
generated by all four pairs of elements. Note that the ideal $\tilde{I}_{T}$
is independent of the labelings. By abuse of notation, if $T$ is an ideal tetrahedron in an ideal triangulation $\mathcal{T}$, we will also denote by $\tilde{I}_{T}$ the left ideal of $\underset{f_{i}\in\mathbf{f}(\mathcal{T})}{\bigotimes}\tilde{\mathbb{T}}_{f_{i}}$ generated by the same set of elements.

(ii) Let $\widetilde{\mathcal{\mathfrak{R}}}_{T}$ be the $R$-submodule
of $\underset{f\in\mathcal{T}^{(2)}}{\bigotimes}\overline{\mathrm{Sk}}(Sf)$
given by\footnote{Recall by construction $\underset{f\in\mathcal{T}^{(2)}}{\bigotimes}\overline{\mathrm{Sk}}(Sf)$ is a left $\underset{f_{i}\in\mathbf{f}(\mathcal{T})}{\bigotimes}\tilde{\mathbb{T}}_{f_{i}}$-module under $\cup$-action.} $\tilde{I}_{T}\left(\underset{f\in\mathcal{T}^{(2)}}{\bigotimes}\overline{\mathrm{Sk}}(Sf)\right).$

(iii) Let $e\in\mathcal{T}^{(1)}$ be an edge of the ideal triangulation $\mathcal{T}$ and let $Sf^{(1)},Sf^{(2)},\dots,Sf^{(k)}$
be the sequence of face suspensions around $e$ (see Figure \ref{fig:face suspensions around an edge}). Suppose
each face $f^{(i)}$ is the identification of bare faces $f_{1}^{(i)}$
and $f_{2}^{(i)}$ , so that $Sf^{(i)}$ is the union of face cones
$Cf_{1}^{(i)}$ and $Cf_{2}^{(i)}$ (so $f_{2}^{(i)}$
and $f_{1}^{(i+1)}$ are bare faces of the same ideal tetrahedron,
see Figure \ref{fig:face cones around an edge}). We also let $e_{i}$ be the shape parameter that labels
the common bare edge of both $f_{2}^{(i)}$ and $f_{1}^{(i+1)}$,
and let $e_{k}$ be the shape parameter that labels the common bare
edge of both $f_{2}^{(k)}$ and $f_{1}^{(1)}$. Thus $e_{1},e_{2}\dots,e_{k}$
is the sequence of bare edges identified to the edge $e\in\mathcal{T}^{(1)}$.
We define the element
\[
\tilde{e}=x_{e_{k},f_{1}^{(1)};e_{1},f_{2}^{(1)}}x_{e_{1},f_{1}^{(2)};e_{2};f_{2}^{(2)}}\dots x_{e_{k-1},f_{1}^{(k)};e_{k},f_{2}^{(k)}}\in\bigotimes_{f\in\mathcal{T}^{(2)}}\left({\mathbb{B}^{\otimes3}}\right)_{Sf}.
\]
 (As in Corollary \ref{cor:generators for skein modules of face cones and face suspensions} (ii),
$\left({\mathbb{B}^{\otimes3}}\right)_{Sf}$ denotes the skein algebra associated
to the three bivalent sources of the boundary marking on the face
suspension $Sf$). Let $\tilde{I}_{E}$ be the \emph{left} ideal of $\bigotimes_{f\in\mathcal{T}^{(2)}}\left({\mathbb{B}^{\otimes3}}\right)_{Sf}$
generated by the elements 
\[
\tilde{e}^{\epsilon}-(-A^{2})^{\epsilon},\ \epsilon\in\{\pm\}
\]
for every edge $e\in\mathcal{T}^{(1)}$.
\begin{figure}[h]
  \includegraphics[scale=0.2]{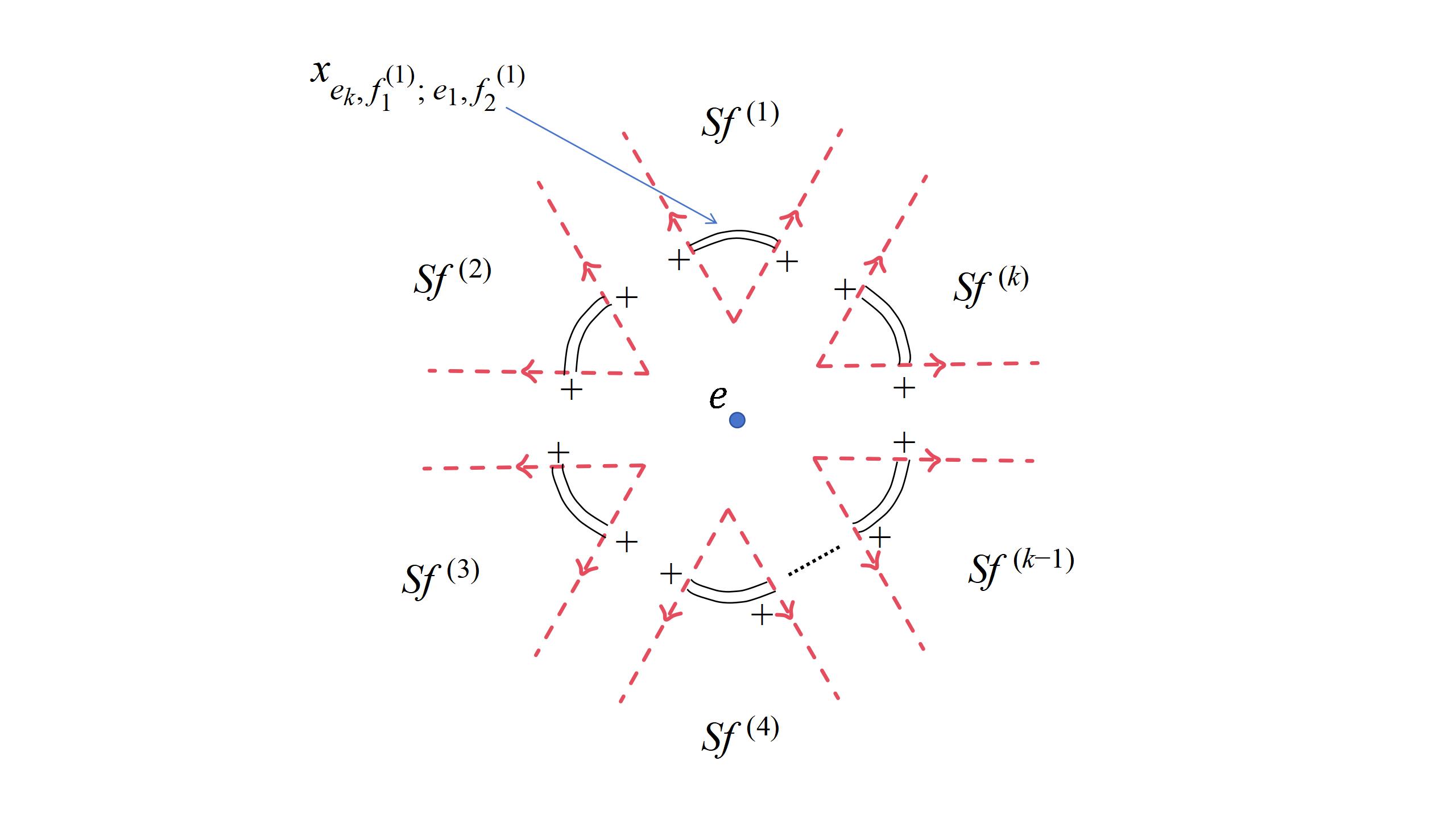} 
  \caption{Face suspensions around an edge $e$ and the skein diagram of the element $\tilde{e}$.}
  \label{fig:face suspensions around an edge}
\end{figure}
\begin{figure}[h]
  \includegraphics[scale=0.2]{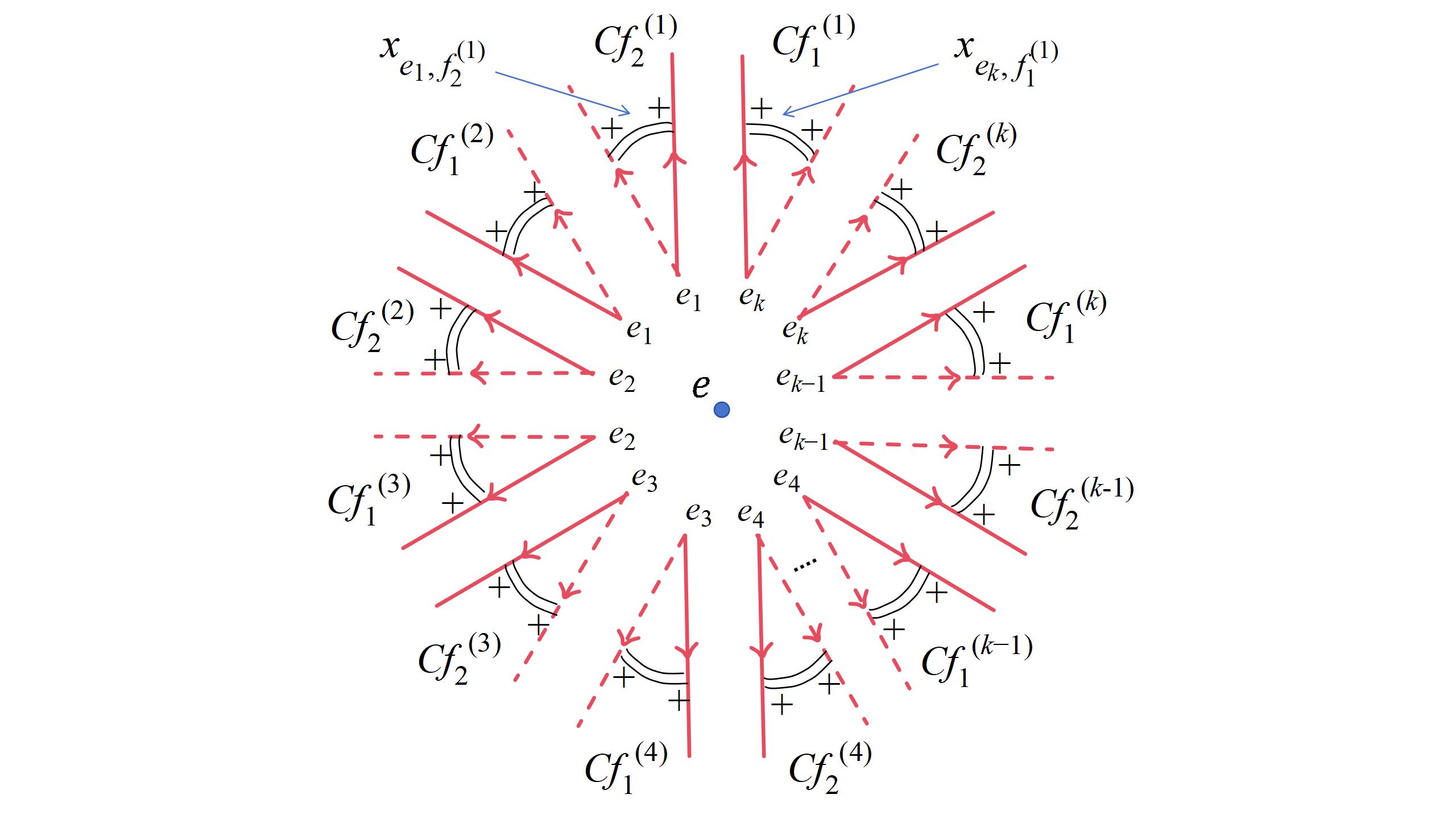} 
  \caption{Face cones around an edge $e$ and the skein diagram of the element $\bar{e}$. This diagram also implies the sequence of face suspensions and the sequence of ideal tetrahedra around the edge $e$.}
  \label{fig:face cones around an edge}
\end{figure}

(iv) Let $\widetilde{\mathfrak{R}}_{E}$ be the $R$-submodule of $\underset{f\in\mathcal{T}^{(2)}}{\bigotimes}\overline{\mathrm{Sk}}(Sf)$
given by $\left(\underset{f\in\mathcal{T}^{(2)}}{\bigotimes}\overline{\mathrm{Sk}}(Sf)\right)\tilde{I}_{E}$.

(v) Lastly, we define the \emph{reduced tensor product} $\underset{f\in\mathcal{T}^{(2)}}{\overline{\bigotimes}}\overline{\mathrm{Sk}}(Sf)$
to be the $R$-module quotient\footnote{For comparison, the submodule $\widetilde{\mathfrak{R}}_{E}$ is the submodule $\left(\underset{f\in\mathcal{T}^{(2)}}{\bigotimes}\overline{\mathrm{Sk}}(Sf)\right)U_{-}$ in \cite{PP1}; the submodule $\sum_{T\in\mathcal{T}}\widetilde{\mathcal{\mathfrak{R}}}_{T}$ is the submodule $U_{+}\left(\underset{f\in\mathcal{T}^{(2)}}{\bigotimes}\overline{\mathrm{Sk}}(Sf)\right)$ in \cite{PP1}.}
\[
\underset{f\in\mathcal{T}^{(2)}}{\overline{\bigotimes}}\overline{\mathrm{Sk}}(Sf):=\frac{\underset{f\in\mathcal{T}^{(2)}}{\bigotimes}\overline{\mathrm{Sk}}(Sf)}{\widetilde{\mathfrak{R}}_{E}+\sum_{T\in\mathcal{T}}\widetilde{\mathcal{\mathfrak{R}}}_{T}}.
\]

\end{defn}

\begin{rem}
Similar to what we have in Corollary \ref{cor:reduced skein module of T}, the element
$\tilde{\ell}_{zz^{\prime\prime}}$ given above is in fact a choice
out of three possibilities, there are two other elements $\tilde{\ell}_{z^{\prime\prime}z^{\prime}}$
and $\tilde{\ell}_{z^{\prime}z}$ in the form of $\tilde{\ell}_{zz^{\prime\prime}}$.
However these three choices are equivalent in the sense that if one
replace the element $\tilde{\ell}_{zz^{\prime\prime}}$ in the above definition by $\tilde{\ell}_{z^{\prime\prime}z^{\prime}}$
or $\tilde{\ell}_{z^{\prime}z}$, the resulting right ideal $\tilde{I}_{T}$
is the same.
\end{rem}

Now let us describe the splitting homomorphism 
\[
\tilde{\sigma}\colon\overline{\mathrm{Sk}}(Y)\rightarrow\underset{f\in\mathcal{T}^{(2)}}{\overline{\bigotimes}}\overline{\mathrm{Sk}}(Sf)
\]
induced by splitting the 3-manifold $Y$ into the face suspensions
along the edge cones. Starting with a ribbon tangle $\ell$
in $Y$ in general position with respect to the foliations on every
edge cone (every edge cone is a single elementary quadrilateral, so
we have a foliation on it, see Figure \ref{fig:foliations on elementary Q and face}). Again we can isotope it
by moving its part near an edge cone along the leaves of the foliation
so that its intersection with every edge cone falls on marking edges.
Then the part $\ell_{Sf}$ of $\ell$ in the face suspension $Sf,\ f\in\mathcal{T}^{(2)}$
is a ribbon tangle in $Sf$ whose endpoints fall on the boundary markings
of $Sf$. Therefore we can consider the element 
\[
\tilde{\sigma}(\ell)\colon=\sum_{\overrightarrow{\epsilon}}\left(\underset{f\in\mathcal{T}^{(2)}}{\otimes}\ell_{Sf}^{\overrightarrow{\epsilon}}\right)\in\underset{f\in\mathcal{T}^{(2)}}{\bigotimes}\overline{\mathrm{Sk}}(Sf),
\]
where the summation is over all compatible states. By abuse of notation,
we will denote the class of $\tilde{\sigma}(\ell)$ in the reduced
tensor product $\underset{f\in\mathcal{T}^{(2)}}{\overline{\bigotimes}}\overline{\mathrm{Sk}}(Sf)$
still by $\tilde{\sigma}(\ell)$.

\begin{prop}
\label{prop:Splitting homomorphism induced by splitting into face suspensions}
\emph{(\cite[Corollary 3.37]{PP1})} The assignment to any ribbon
link $\ell$ in general position of the element
\[
\tilde{\sigma}(\ell)=\sum_{\overrightarrow{\epsilon}}\left(\underset{f\in\mathcal{T}^{(2)}}{\otimes}\ell_{Sf}^{\overrightarrow{\epsilon}}\right)\in\underset{f\in\mathcal{T}^{(2)}}{\overline{\bigotimes}}\overline{\mathrm{Sk}}(Sf),
\]
where the sum is over all compatible states, gives a well-defined $R$-module
homomorphism
\[
\tilde{\sigma}\colon\overline{\mathrm{Sk}}(Y)\rightarrow\underset{f\in\mathcal{T}^{(2)}}{\overline{\bigotimes}}\overline{\mathrm{Sk}}(Sf).
\]
\end{prop}

{\large{\textcircled{\small{2}}}}:
This map is induced by splitting each face suspension into two face
cones. Again, let's start by defining the codomain $\underset{f_{i}\in\mathbf{f}(\mathcal{T})}{\overline{\bigotimes}}\overline{\mathrm{Sk}}^{pc}(Cf_{i})$. 

\begin{defn}
\label{def:partially balanced elements} 
An element $\underset{f_{i}\in\mathbf{f}(\mathcal{T})}{\otimes}x_{Cf_{i}}\in\underset{f_{i}\in\mathbf{f}(\mathcal{T})}{\bigotimes}\overline{\mathrm{Sk}}(Cf_{i})$
is said to be\emph{ partially balanced} if whenever $f_{i}$ and $f_{j}$
is a pair of bare faces that are identified in $\mathcal{T}$, 
we have $d_{f_{i}}(x_{Cf_{i}})=d_{f_{j}}(x_{Cf_{j}})$. 
\end{defn}

Here, ``$d_{f_{i}}(x_{Cf_{i}})=d_{f_{j}}(x_{Cf_{j}})"$ is in the
same sense as in Definition \ref{def:balanced elements}. Namely,
if we labeled the marking edges in $f_{i}$ and $f_{j}$ by $e_{1}$,
$e_{2}$ and $e_{3}$ where the same labeling $e_{k}$ is given to
the marking edge in $f_{i}$ and the marking edge in $f_{j}$ that
are identified after we glue $f_{i}$ and $f_{j}$, then $d_{f_{i}}(x_{Cf_{i}})=d_{f_{j}}(x_{Cf_{j}})\in\mathbb{Z}^{\{e_{1},e_{2},e_{3}\}}$.
As in Remark \ref{rem:matching face, opposite skew-symmetric form},
if $x=\underset{f_{i}\in\mathbf{f}(\mathcal{T})}{\otimes}x_{Cf_{i}}$
and $y=\underset{f_{i}\in\mathbf{f}(\mathcal{T})}{\otimes}y_{Cf_{i}}$
are both partially balanced, then we also have
\[
\langle d_{f_{i}}(x),d_{f_{i}}(y)\rangle_{f_{i}}=-\langle d_{f_{j}}(x),d_{f_{j}}(y)\rangle_{f_{j}}.
\]

From the discussion in \ref{subsec: partial corner reduction},
the partially corner-reduced skein module $\overline{\mathrm{Sk}}^{pc}(Cf_{i})$
is still a left $\tilde{\mathbb{T}}_{f_{i}}$-module (under $\cup$-action)
and thus $\underset{f_{i}\in\mathbf{f}(\mathcal{T})}{\bigotimes}\overline{\mathrm{Sk}}^{pc}(Cf_{i})$
is a left $\underset{f_{i}\in\mathbf{f}(\mathcal{T})}{\bigotimes}\tilde{\mathbb{T}}_{f_{i}}$-module.

\begin{defn}
\label{def:reduced tensor peoduct of partially corner-reduced modules of face cones}(i)
Recall that in Definition \ref{def:reduced tensor product of skein modules of face suspensions}
(i), we gave a right ideal $\tilde{I}_{T}$ of $\underset{f_{i}\in\mathbf{f}(T)}{\bigotimes}\tilde{\mathbb{T}}_{f_{i}}$ and by abuse of notation we denote the natural extension of $\tilde{I}_{T}$
into $\underset{f_{i}\in\mathbf{f}(\mathcal{T})}{\bigotimes}\tilde{\mathbb{T}}_{f_{i}}$
still by $\tilde{I}_{T}$. Let $\overline{\mathfrak{R}}_{T}$ be the
$R$-submodule of $\underset{f_{i}\in\mathbf{f}(\mathcal{T})}{\bigotimes}\overline{\mathrm{Sk}}^{pc}(Cf_{i})$
given by\footnote{Recall by construction $\underset{f_{i}\in\mathbf{f}(\mathcal{T})}{\bigotimes}\overline{\mathrm{Sk}}^{pc}(Cf_{i})$ is a left $\underset{f_{i}\in\mathbf{f}(\mathcal{T})}{\bigotimes}\tilde{\mathbb{T}}_{f_{i}}$-module under $\cup$-action.} $\tilde{I}_{T}\left(\underset{f_{i}\in\mathbf{f}(\mathcal{T})}{\bigotimes}\overline{\mathrm{Sk}}^{pc}(Cf_{i})\right)$.

(ii) Let $e\in\mathcal{T}^{(1)}$, and suppose we have the sequence
of face cones around $e$ as given by Figure \ref{fig:face cones around an edge}.
We define the element
\[
\overline{e}=x_{e_{k},f_{1}^{(1)}}x_{e_{1},f_{2}^{(1)}}x_{e_{1},f_{1}^{(2)}}x_{e_{2};f_{2}^{(2)}}\dots x_{e_{k-1},f_{1}^{(k)}}x_{e_{k};f_{2}^{(k)}}\in\underset{f_{i}\in\mathbf{f}(\mathcal{T})}{\bigotimes}\left(\mathbb{B}^{\otimes3}\right)_{Cf_{i}}.
\]
(As in Corollary \ref{cor:generators for skein modules of face cones and face suspensions}
(i), $\left(\mathbb{B}^{\otimes3}\right)_{Cf_{i}}$ denotes the the
skein algebra associated to the three bivalent sources of the boundary
marking on the face cones $Cf_{i}$). Let $\overline{\mathfrak{R}}_{E}$
be the $R$-submodule of $\underset{f_{i}\in\mathbf{f}(\mathcal{T})}{\bigotimes}\overline{\mathrm{Sk}}(Cf_{i})$
spanned by elements of the form
\[
\left(\underset{f_{i}\in\mathbf{f}(\mathcal{T})}{\otimes}x_{Cf_{i}}\right)\cup\left(\overline{e}^{\epsilon}-(-A^{2})^{\epsilon}\right),\ \epsilon\in\{\pm\}
\]
where $\underset{f_{i}\in\mathbf{f}(\mathcal{T})}{\bigotimes}x_{Cf_{i}}$
is partially balanced, $e\in\mathcal{T}^{(1)}$ and $\overline{e}$
is the element defined above.

(iii) Let $\overline{\mathfrak{R}}_{E}^{pc}$ be the $R$-submodule
of $\underset{f_{i}\in\mathbf{f}(\mathcal{T})}{\bigotimes}\overline{\mathrm{Sk}}^{pc}(Cf_{i})$
given by the image of $\overline{\mathfrak{R}}_{E}$ defined above
under the natural map $\underset{f_{i}\in\mathbf{f}(\mathcal{T})}{\bigotimes}\overline{\mathrm{Sk}}(Cf_{i})\twoheadrightarrow\underset{f_{i}\in\mathbf{f}(\mathcal{T})}{\bigotimes}\overline{\mathrm{Sk}}^{pc}(Cf_{i})$.

(iv) The reduced tensor product $\underset{f_{i}\in\mathbf{f}(\mathcal{T})}{\overline{\bigotimes}}\overline{\mathrm{Sk}}^{pc}(Cf_{i})$
is defined to be the $R$-module quotient
\[
\underset{f_{i}\in\mathbf{f}(\mathcal{T})}{\overline{\bigotimes}}\overline{\mathrm{Sk}}^{pc}(Cf_{i}):=\raisebox{1.5ex}{$\underset{f_{i}\in\mathbf{f}(\mathcal{T})}{\bigotimes}\overline{\mathrm{Sk}}^{pc}(Cf_{i})$}\biggm/\raisebox{-1.5ex}{${\overline{\mathfrak{R}}_{E}^{pc}+\sum_{T\in\mathcal{T}}\overline{\mathfrak{R}}_{T}}$}.
\]

\end{defn}

Now that we have the codomain of the desired map, our next step is to understand
the case of a single face suspension. Starting with a stated ribbon
tangle $\ell$ in $Sf$ in general position with respect to the foliation
structure on the face $f$ ($f$ is the union of three elementary
quadrilaterals, it has a canonical foliation structure, see Figure \ref{fig:foliations on elementary Q and face}). Again
we can isotope $\ell$ by moving its part near the face $f$ along
the leaves of the foliation so that its intersection with the face
$f$ falls on marking edges. Then the part $\ell_{Cf_{i}}$ and $\ell_{Cf_{j}}$
of $\ell$ in the face cones $Cf_{i}$ and $Cf_{j}$, respectively,
are ribbon tangles whose endpoints fall on the boundary markings of
the face cones $Cf_{i}$ and $Cf_{j}$. Therefore we can consider
the element 
\[
\sigma_{Sf}(\ell):=\sum_{\overrightarrow{\epsilon}}\left(\ell_{Cf_{i}}^{\overrightarrow{\epsilon}}\otimes\ell_{Cf_{j}}^{\overrightarrow{\epsilon}}\right)\in\overline{\mathrm{Sk}}^{pc}(Cf_{i})\otimes\overline{\mathrm{Sk}}^{pc}(Cf_{j}),
\]
where the sum is over all compatible states.
\begin{lem}
\label{lem: Splitting homomorphism induced by splitting a single face suspension into two face cones}
The assignment to any stated ribbon tangle $\ell$ in $Sf$ in general
position of the element
\[
\sigma_{Sf}(\ell)=\sum_{\overrightarrow{\epsilon}}\left(\ell_{Cf_{i}}^{\overrightarrow{\epsilon}}\otimes\ell_{Cf_{j}}^{\overrightarrow{\epsilon}}\right)\in\overline{\mathrm{Sk}}^{pc}(Cf_{i})\otimes\overline{\mathrm{Sk}}^{pc}(Cf_{j}),
\]
where the sum is over all compatible states gives a well-defined $R$-module
homomorphism
\[
\sigma_{Sf}\colon\overline{\mathrm{Sk}}(Sf)\rightarrow\overline{\mathrm{Sk}}^{pc}(Cf_{i})\otimes\overline{\mathrm{Sk}}^{pc}(Cf_{j}).
\]
\end{lem}
\begin{proof}
The argument is similar to the proof of Theorem \ref{thm:splitting homomorphism}.
We only need to show that the element $\sigma_{Sf}(\ell)\in\overline{\text{Sk}}^{pc}(Cf_{i})\otimes\overline{\text{Sk}}^{pc}(Cf_{j})$
is invariant under an isotopy of the stated tangle $\ell$. Again,
an isotopy of $\ell$ can be decomposed into a finite sequence of
elementary moves of type (\Romannum{1})-(\Romannum{5}) (See the
proof of Theorem \ref{thm:splitting homomorphism}). Once again, using only skein relations, we find that the element $\sigma_{Sf}(\ell)$
is invariant under moves of type (\Romannum{1})-(\Romannum{4}).
For invariance under move of type (\Romannum{5}), we only need to consider
the case of isotopy across a singular leaf which is an edge connecting
the barycenter of the face $f$ to a vertex of $f$; it then follows
from the same calculation as that in the proof of Theorem \ref{thm:splitting homomorphism}, 
making use of the relations coming from the corner-reductions on the
bare faces $f_{i}$ and $f_{j}$, that $\sigma_{Sf}(\ell)$ is invariant under move of type (\Romannum{5}).
\end{proof}
Note, if $\tilde{\Gamma}\in\tilde{\mathbb{T}}_{f_{i}}$ or $\tilde{\mathbb{T}}_{f_{j}}$
and $x\in\overline{\mathrm{Sk}}(Sf)$, then by definition we have 
\[
\sigma_{Sf}(\tilde{\Gamma}\cup x)=\tilde{\Gamma}\cup\sigma_{Sf}(x)\in\overline{\mathrm{Sk}}^{pc}(Cf_{i})\otimes\overline{\mathrm{Sk}}^{pc}(Cf_{j}).
\]
On the other hand, if an edges of $f_{i}$ labeled by $a$ is
identified with the edge of $f_{j}$ labeled $b$ when $f_{i}$
and $f_{j}$ are glued, then for any $z\in\overline{\mathrm{Sk}}(Sf)$
we have
\[
\sigma_{Sf}(z\cup x_{a,f_{i};b,f_{j}})=\sigma_{Sf}(z)\cup(x_{a,f_{i}}x_{b,f_{j}}).
\]

\begin{prop}
\label{prop:gluing splitting homomorphism coming from splitting every single face suspension}
The $R$--module homomorphism
\[
\underset{f\in\mathcal{T}^{(2)}}{\bigotimes}\sigma_{Sf}\colon\underset{f\in\mathcal{T}^{(2)}}{\bigotimes}\overline{\mathrm{Sk}}(Sf)\rightarrow\underset{f_{i}\in\mathbf{f}(\mathcal{T})}{\bigotimes}\overline{\mathrm{Sk}}^{pc}(Cf_{i})
\]
given by the (usual) tensor product of the individual splitting homomorphism
$\sigma_{Sf}$ given in the last lemma over all $f\in\mathcal{T}^{(2)}$
descends to a well-defined $R$--module homomorphism
\[
\underset{f\in\mathcal{T}^{(2)}}{\overline{\bigotimes}}\sigma_{Sf}\colon\underset{f\in\mathcal{T}^{(2)}}{\overline{\bigotimes}}\overline{\mathrm{Sk}}(Sf)\rightarrow\underset{f_{i}\in\mathbf{f}(\mathcal{T})}{\overline{\bigotimes}}\overline{\mathrm{Sk}}^{pc}(Cf_{i}).
\]
\end{prop}

\begingroup
\allowdisplaybreaks
\begin{proof}
We prove it by showing that $\underset{f\in\mathcal{T}^{(2)}}{\bigotimes}\sigma_{Sf}$
maps the $R$--submodule $\tilde{\mathfrak{R}}_{T}$ of $\underset{f\in\mathcal{T}^{(2)}}{\bigotimes}\overline{\mathrm{Sk}}(Sf)$
given in Definition \ref{def:reduced tensor product of skein modules of face suspensions}
(ii) into the $R$--submodule $\overline{\mathfrak{R}}_{T}$ of $\underset{f_{i}\in\mathbf{f}(\mathcal{T})}{\bigotimes}\overline{\mathrm{Sk}}^{pc}(Cf_{i})$
given in Definition \ref{def:reduced tensor peoduct of partially corner-reduced modules of face cones}
(i); and maps the $R$--submodule $\tilde{\mathfrak{R}}_{E}$ of
$\underset{f\in\mathcal{T}^{(2)}}{\bigotimes}\overline{\mathrm{Sk}}(Sf)$
given in Definition \ref{def:reduced tensor product of skein modules of face suspensions}
(iv) into the $R$--submodule $\overline{\mathfrak{R}}_{E}^{pc}$
of $\underset{f_{i}\in\mathbf{f}(\mathcal{T})}{\bigotimes}\overline{\mathrm{Sk}}^{pc}(Cf_{i})$
given in Definition \ref{def:reduced tensor peoduct of partially corner-reduced modules of face cones}
(iii).

To see that $\underset{f\in\mathcal{T}^{(2)}}{\bigotimes}\sigma_{Sf}$
maps $\tilde{\mathfrak{R}}_{T}$ into $\overline{\mathfrak{R}}_{T}$.
By definition, $\tilde{\mathfrak{R}}_{T}$ is the $R$-span of elements
of the form $\tilde{v}_{zz^{\prime}z^{\prime\prime}}\cup x$ and $\tilde{\ell}_{zz^{\prime\prime}}\cup x$,
where $x\in\underset{f\in\mathcal{T}^{(2)}}{\bigotimes}\overline{\mathrm{Sk}}(Sf)$,
$\tilde{v}_{zz^{\prime}z^{\prime\prime}}$ and $\tilde{\ell}_{zz^{\prime\prime}}$
are elements of $\underset{f_{i}\in\mathbf{f}(T)}{\bigotimes}\tilde{\mathbb{T}}_{f_{i}}$
given in Definition \ref{def:reduced tensor product of skein modules of face suspensions}
(i). By definition of $\sigma_{Sf}$, we have
\[
\left(\underset{f\in\mathcal{T}^{(2)}}{\bigotimes}\sigma_{Sf}\right)\left(\tilde{v}_{zz^{\prime}z^{\prime\prime}}\cup x\right)=\tilde{v}_{zz^{\prime}z^{\prime\prime}}\cup\left(\underset{f\in\mathcal{T}^{(2)}}{\bigotimes}\sigma_{Sf}\right)(x)\in\overline{\mathfrak{R}}_{T}
\]
and
\[
\left(\underset{f\in\mathcal{T}^{(2)}}{\bigotimes}\sigma_{Sf}\right)\left(\tilde{\ell}_{zz^{\prime\prime}}\cup x\right)=\tilde{\ell}_{zz^{\prime\prime}}\cup\left(\underset{f\in\mathcal{T}^{(2)}}{\bigotimes}\sigma_{Sf}\right)(x)\in\overline{\mathfrak{R}}_{T}.
\]
This shows that $\underset{f\in\mathcal{T}^{(2)}}{\bigotimes}\sigma_{Sf}$
maps $\tilde{\mathfrak{R}}_{T}$ into $\overline{\mathfrak{R}}_{T}$. 

To see that $\underset{f\in\mathcal{T}^{(2)}}{\bigotimes}\sigma_{Sf}$
maps $\widetilde{\mathfrak{R}}_{E}$ into $\overline{\mathfrak{R}}_{E}^{pc}$.
By definition, $\widetilde{\mathfrak{R}}_{E}$ is the $R$-span of elements
of the form $\left(\underset{f\in\mathcal{T}^{(2)}}{\otimes}x_{Sf}\right)\cup\left(\tilde{e}^{\epsilon}-(-A^{2})^{\epsilon}\right)$,
where $\underset{f\in\mathcal{T}^{(2)}}{\otimes}x_{Sf}\in\underset{f\in\mathcal{T}^{(2)}}{\bigotimes}\overline{\mathrm{Sk}}(Sf)$
and $\tilde{e}$ is the element of $\underset{f\in\mathcal{T}^{(2)}}{\bigotimes}\left(\mathbb{B}^{\otimes3}\right)_{Sf}$
given by Definition \ref{def:reduced tensor product of skein modules of face suspensions}
(iii). Now, by definition of $\sigma_{Sf}$ again, we have
\begin{eqnarray*}
\left(\underset{f\in\mathcal{T}^{(2)}}{\bigotimes}\sigma_{Sf}\right)\left(\left(\underset{f\in\mathcal{T}^{(2)}}{\otimes}x_{Sf}\right)\cup\left(\tilde{e}^{\epsilon}-(-A^{2})^{\epsilon}\right)\right)\\
=\left(\underset{f\in\mathcal{T}^{(2)}}{\bigotimes}\sigma_{Sf}(x_{Sf})\right)\cup\left(\overline{e}^{\epsilon}-(-A^{2})^{\epsilon}\right).
\end{eqnarray*}
Here $\overline{e}$ is the element of $\underset{f_{i}\in\mathbf{f}(\mathcal{T})}{\bigotimes}\left(\mathbb{B}^{\otimes3}\right)_{Cf_{i}}$
given in Definition \ref{def:reduced tensor peoduct of partially corner-reduced modules of face cones}
(ii). Moreover, the element $\underset{f\in\mathcal{T}^{(2)}}{\bigotimes}\sigma_{Sf}(x_{Sf})$
is given by a sum of partially balanced elements (see Definition \ref{def:partially balanced elements}).
Therefore
\[
\left(\underset{f\in\mathcal{T}^{(2)}}{\bigotimes}\sigma_{Sf}(x_{Sf})\right)\cup\left(\overline{e}^{\epsilon}-(-A^{2})^{\epsilon}\right)\in\overline{\mathfrak{R}}_{E}^{pc}
\]
and it shows that $\underset{f\in\mathcal{T}^{(2)}}{\bigotimes}\sigma_{Sf}$
maps $\widetilde{\mathfrak{R}}_{E}$ into $\overline{\mathfrak{R}}_{E}^{pc}$.

\end{proof}
\endgroup

{\large{\textcircled{\small{3}}}}:
This map is induced by decomposing the triangulated manifold $Y$ into ideal tetrahedra,
and is just the splitting homomorphism
\[
\sigma\colon\overline{\mathrm{Sk}}(Y)\rightarrow\underset{T\in\mathcal{T}}{\overline{\bigotimes}}\overline{\mathrm{Sk}}^{c}(T)
\]
given by Theorem \ref{thm:splitting homomorphism}.

{\large{\textcircled{\small{4}}}}:
This map is induced by decomposing each ideal tetrahedron into the four
face cones of its bare faces. We have already defined its codomain
earlier. To describe the map, let's first understand the case of a single tetrahedron $T$.
Starting with a stated ribbon tangle $\ell$ in $T$ in general position
with respect to the foliations on every edge cone. We can isotope
it by moving its part near an edge cone along the leaves of the foliation
so that its intersection with every edge cone falls on marking edges.
Then for each bare face $f_{i}$ of $T$, the part $\ell_{Cf_{i}}$ of $\ell$ in the face cone $Cf_{i}$ is a ribbon tangle whose endpoints
fall on the boundary markings of $Cf_{i}$ . Therefore we can consider
the element
\[
\sigma_{T}(\ell):=\sum_{\overrightarrow{\epsilon}}\left(\underset{f_{i}\in\mathbf{f}(T)}{\otimes}\ell_{Cf_{i}}^{\overrightarrow{\epsilon}}\right)\in\underset{f_{i}\in\mathbf{f}(T)}{\bigotimes}\overline{\mathrm{Sk}}(Cf_{i}),
\]
where the summation is over all compatible states. Note, the compatible
states $\overrightarrow{\epsilon}$ are assigned to endpoints created
by splitting along the edge cones, the endpoints on the marking edges
in the bare faces of $T$ are not affected. In other word, we have $d_{f_{i}}(\ell)=d_{f_{i}}(\ell_{Cf_{i}}^{\overrightarrow{\epsilon}})$
for any $\overrightarrow{\epsilon}$.

\begin{lem}
\label{lem: Splitting homomorphism from splitting a single tetrahedra into face cones}
The assignment to any stated ribbon tangle $\ell$ in $T$ of the element
\[
\sigma_{T}(\ell)=\sum_{\overrightarrow{\epsilon}}\left(\underset{f_{i}\in\mathbf{f}(T)}{\otimes}\ell_{Cf_{i}}^{\overrightarrow{\epsilon}}\right)\in\underset{f_{i}\in\mathbf{f}(T)}{\bigotimes}\overline{\mathrm{Sk}}(Cf_{i})
\]
descends to a well-defined $R$-module homomorphism\footnote{Recall that $\underset{f_{i}\in\mathbf{f}(T)}{\bigotimes}\overline{\mathrm{Sk}}^{pc}(Cf_{i})$
is a left $\underset{f_{i}\in\mathbf{f}(T)}{\bigotimes}\tilde{\mathbb{T}}_{f_{i}}$-module
and $\tilde{I}_{T}$ is a right ideal of $\underset{f_{i}\in\mathbf{f}(T)}{\bigotimes}\tilde{\mathbb{T}}_{f_{i}}$,
so the codomain makes sense.}
\[
\sigma_{T}\colon\overline{\mathrm{Sk}}^{c}(T)\rightarrow\frac{\underset{f_{i}\in\mathbf{f}(T)}{\bigotimes}\overline{\mathrm{Sk}}^{pc}(Cf_{i})}{\tilde{I}_{T}\left(\underset{f_{i}\in\mathbf{f}(T)}{\bigotimes}\overline{\mathrm{Sk}}^{pc}(Cf_{i})\right)}
\]
\end{lem}

\begingroup
\allowdisplaybreaks
\begin{proof}
The general machinary of splitting homomorphism developed in \cite[3.2]{PP1}
guarantees that the assigment to a stated ribbon link $\ell$ in $T$ in general position of
the element $\sigma_{T}(\ell)$
gives us a well-defined $R$-module homomorphism
\[
\sigma_{T}\colon\overline{\mathrm{Sk}}(T)\rightarrow\frac{\underset{f_{i}\in\mathbf{f}(T)}{\bigotimes}\overline{\mathrm{Sk}}(Cf_{i})}{\tilde{I}_{T}\left(\underset{f_{i}\in\mathbf{f}(T)}{\bigotimes}\overline{\mathrm{Sk}}(Cf_{i})\right)}.
\]
We need to show that this map descends to the (partially) corner-reduced
setting. Recall from our discussion in \ref{subsec: partial corner reduction}  that the reduced skein module $\overline{\mathrm{Sk}}(Cf_{i})$
is a $R[Cf_{i}]\otimes\tilde{\mathbb{T}}_{f_{i}}\text{-}\mathbb{T}\langle Cf_{i}\rangle$-bimodule
(where $R[Cf_{i}]$ and $\mathbb{T}\langle Cf_{i}\rangle$ acts by
$\cdot$ and $\tilde{\mathbb{T}}_{f_{i}}$ acts by $\cup$) and the
partially corner-reduced module $\overline{\mathrm{Sk}}^{pc}(Cf_{i})$
is defined as the $R[Cf_{i}]\otimes\tilde{\mathbb{T}}_{f_{i}}\text{-}\mathbb{T}\langle Cf_{i}\rangle$-bimodule
quotient
\[
\frac{\overline{\mathrm{Sk}}(Cf_{i})}{I_{f_{i}}^{c}\cdot\overline{\mathrm{Sk}}(Cf_{i})},
\]
where $I_{f_{i}}^{c}$ is the ideal of the commutative algebra $R[Cf_{i}]$
given right before Definition \ref{def: partial corner reduced module of face cone}.
In particular the natural quotient map $\overline{\mathrm{Sk}}(Cf_{i})\twoheadrightarrow\overline{\mathrm{Sk}}^{pc}(Cf_{i})$
is a left $\tilde{\mathbb{T}}_{f_{i}}$-module homomorphism and
thus the map
\[
\underset{f_{i}\in\mathbf{f}(T)}{\bigotimes}\overline{\mathrm{Sk}}(Cf_{i})\twoheadrightarrow\underset{f_{i}\in\mathbf{f}(T)}{\bigotimes}\overline{\mathrm{Sk}}^{pc}(Cf_{i})
\]
is a left $\underset{f_{i}\in\mathbf{f}(T)}{\bigotimes}\tilde{\mathbb{T}}_{f_{i}}$-module
homomorphism. Therefore it maps $\tilde{I}_{T}\left(\underset{f_{i}\in\mathbf{f}(T)}{\bigotimes}\overline{\mathrm{Sk}}(Cf_{i})\right)$
into $\tilde{I}_{T}\left(\underset{f_{i}\in\mathbf{f}(T)}{\bigotimes}\overline{\mathrm{Sk}}^{pc}(Cf_{i})\right)$.
Consequently, we have the induced map
\[
\Pi\colon\frac{\underset{f_{i}\in\mathbf{f}(T)}{\bigotimes}\overline{\mathrm{Sk}}(Cf_{i})}{\tilde{I}_{T}\left(\underset{f_{i}\in\mathbf{f}(T)}{\bigotimes}\overline{\mathrm{Sk}}(Cf_{i})\right)}\twoheadrightarrow\frac{\underset{f_{i}\in\mathbf{f}(T)}{\bigotimes}\overline{\mathrm{Sk}}^{pc}(Cf_{i})}{\tilde{I}_{T}\left(\underset{f_{i}\in\mathbf{f}(T)}{\bigotimes}\overline{\mathrm{Sk}}^{pc}(Cf_{i})\right)}.
\]
We're left to show that the composition 
\[
\Pi\circ\sigma_{T}\colon\overline{\mathrm{Sk}}(T)\rightarrow\frac{\underset{f_{i}\in\mathbf{f}(T)}{\bigotimes}\overline{\mathrm{Sk}}^{pc}(Cf_{i})}{\tilde{I}_{T}\left(\underset{f_{i}\in\mathbf{f}(T)}{\bigotimes}\overline{\mathrm{Sk}}^{pc}(Cf_{i})\right)}
\]
maps the $R$-submodule\footnote{Recall that the corner-reduced module $\overline{\mathrm{Sk}}^{c}(T)$
is defined to be the quotient $
\frac{\overline{\mathrm{Sk}}(T)}{I^{c}\cdot\overline{\mathrm{Sk}}(T)},$
where $I^{c}$ is the ideal of $R[T]=\left(\underset{f_{i}\in\mathbf{f}(T)}{\bigotimes}\mathbb{T}_{f_{i}},\cdot\right)$
generated by elements of the form $\Gamma_{ab}-(-A^{2})^{-\frac{1}{2}}$.} $I^{c}\cdot\overline{\mathrm{Sk}}(T)$ to $\{0\}$. $I^{c}\cdot\overline{\mathrm{Sk}}(T)$
is the $R$-span of elements of the form $\left(\Gamma_{ab}-(-A^{2})^{-\frac{1}{2}}\right)\cdot\ell$,
where $\ell$ is a stated ribbon tangle in $T$, we have (in the following
we assume without lost of generality that $\Gamma_{ab}$ is in $\mathbb{T}_{f_{j}}$
)
\[
\left(\Gamma_{ab}-(-A^{2})^{-\frac{1}{2}}\right)\cdot\ell=\left(A^{-\frac{1}{2}\langle d_{f_{j}}(\Gamma_{ab}),d_{f_{j}}(\ell)\rangle_{f_{j}}}\Gamma_{ab}-(-A^{2})^{-\frac{1}{2}}\right)\cup\ell.
\]
If we apply $\Pi\circ\sigma_{T}$ to it and use the fact that $d_{f_{i}}(\ell)=d_{f_{i}}(\ell_{Cf_{i}}^{\overrightarrow{\epsilon}})$,
the result becomes
\begin{eqnarray*}
\sum_{\overrightarrow{\epsilon}}\left(\left(A^{-\frac{1}{2}\langle d_{f_{j}}(\Gamma_{ab}),d_{j}(\ell_{C_{j}}^{\overrightarrow{\epsilon}})\rangle_{f_{j}}}\Gamma_{ab}-(-A^{2})^{-\frac{1}{2}}\right)\cup\left(\underset{f_{i}\in\mathbf{f}(T)}{\otimes}\ell_{Cf_{i}}^{\overrightarrow{\epsilon}}\right)\right)\\
=\sum_{\overrightarrow{\epsilon}}\left(\left(\Gamma_{ab}-(-A^{2})^{-\frac{1}{2}}\right)\cdot\ell_{Cf_{j}}^{\overrightarrow{\epsilon}}\right)\otimes\left(\underset{i\neq j}{\underset{f_{i}\in\mathbf{f}(T)}{\otimes}}\ell_{Cf_{i}}^{\overrightarrow{\epsilon}}\right).
\end{eqnarray*}
That is, $\Pi\circ\sigma_{T}\left(\left(\Gamma_{ab}-(-A^{2})^{-\frac{1}{2}}\right)\cdot\ell\right)$
is represented by an element in $\left(I_{f_{j}}^{c}\cdot\overline{\mathrm{Sk}}(Cf_{j})\right)\otimes\left(\underset{i\neq j}{\underset{f_{i}\in\mathbf{f}(T)}{\bigotimes}}\overline{\mathrm{Sk}}(Cf_{i})\right)$,
which is already $0$ in $\underset{f_{i}\in\mathbf{f}(T)}{\bigotimes}\overline{\mathrm{Sk}}^{pc}(Cf_{i})$.
\end{proof}
\endgroup

\begin{prop}
\label{prop:gluing splitting homomorphism coming from splitting every ideal tetrahedron}The
usual tensor product 
\[
\underset{T\in\mathcal{T}}{\bigotimes}\sigma_{T}\colon\underset{T\in\mathcal{T}}{\bigotimes}\overline{\mathrm{Sk}}^{c}(T)  \rightarrow\underset{T\in\mathcal{T}}{\bigotimes}\left(\frac{\underset{f_{i}\in\mathbf{f}(T)}{\bigotimes}\overline{\mathrm{Sk}}^{pc}(Cf_{i})}{\tilde{I}_{T}\left(\underset{f_{i}\in\mathbf{f}(T)}{\bigotimes}\overline{\mathrm{Sk}}^{pc}(Cf_{i})\right)}\right)
  =\left(\underset{f_{i}\in\mathbf{f}(\mathcal{T})}{\bigotimes}\overline{\mathrm{Sk}}^{pc}(Cf_{i})\right)\biggm/\sum_{T\in\mathcal{T}}\overline{\mathfrak{R}}_{T}.
\]
of the splitting homomorphisms induced by splitting every ideal tetrahedron
into face cones given above descends to a well-defined $R$-module
homomorphism
\[
\underset{T\in\mathcal{T}}{\overline{\bigotimes}}\sigma_{T}\colon\underset{T\in\mathcal{T}}{\overline{\bigotimes}}\overline{\mathrm{Sk}}^{c}(T)\rightarrow\underset{f_{i}\in\mathbf{f}(\mathcal{T})}{\overline{\bigotimes}}\overline{\mathrm{Sk}}^{pc}(Cf_{i}).
\]
\end{prop}
\begingroup
\allowdisplaybreaks
\begin{proof}
Recall that the reduced tensor product $\underset{T\in\mathcal{T}}{\overline{\bigotimes}}\overline{\mathrm{Sk}}^{c}(T)$
is given by the $R$-module quotient 
$
\frac{\underset{T\in\mathcal{T}}{\bigotimes}\overline{\mathrm{Sk}}^{c}(T)}{\mathfrak{R}_{E}^{c}},
$
(see Definition \ref{def:reduced tensor product}), whereas the reduced tensor product $\underset{f_{i}\in\mathbf{f}(\mathcal{T})}{\overline{\bigotimes}}\overline{\mathrm{Sk}}^{pc}(Cf_{i})$
is given by the $R$-module quotient 
$
\raisebox{1.5ex}{$\underset{f_{i}\in\mathbf{f}(\mathcal{T})}{\bigotimes}\overline{\mathrm{Sk}}^{pc}(Cf_{i})$}\biggm/\raisebox{-1.5ex}{${\overline{\mathfrak{R}}_{E}^{pc}+\sum_{T\in\mathcal{T}}\overline{\mathfrak{R}}_{T}}$},
$
(see Definition \ref{def:reduced tensor peoduct of partially corner-reduced modules of face cones}).

Let 
\[
\Pi^{\prime}\colon\frac{\underset{f_{i}\in\mathbf{f}(\mathcal{T})}{\bigotimes}\overline{\mathrm{Sk}}^{pc}(Cf_{i})}{\sum_{T\in\mathcal{T}}\overline{\mathfrak{R}}_{T}} \twoheadrightarrow\frac{\underset{f_{i}\in\mathbf{f}(\mathcal{T})}{\bigotimes}\overline{\mathrm{Sk}}^{pc}(Cf_{i})}{\overline{\mathfrak{R}}_{E}^{pc}+\sum_{T\in\mathcal{T}}\overline{\mathfrak{R}}_{T}}=\underset{f_{i}\in\mathbf{f}(\mathcal{T})}{\overline{\bigotimes}}\overline{\mathrm{Sk}}^{pc}(Cf_{i}),
\]
be the natural quotient map. We must show that the composition $\Pi^{\prime}\circ\underset{T\in\mathcal{T}}{\bigotimes}\sigma_{T}$
maps the $R$-submodule $\mathfrak{R}_{E}^{c}$ of $\underset{T\in\mathcal{T}}{\bigotimes}\overline{\mathrm{Sk}}^{c}(T)$
to $\{0\}$. Recall that $\mathfrak{R}_{E}^{c}$ is the image of $\mathfrak{R}_{E}$
under the natural map $\underset{T\in\mathcal{T}}{\bigotimes}\overline{\mathrm{Sk}}(T)\twoheadrightarrow\underset{T\in\mathcal{T}}{\bigotimes}\overline{\mathrm{Sk}}^{c}(T)$
(see Definition \ref{def:reduced tensor product}). Therefore $\mathfrak{R}_{E}^{c}$
is the $R$-span of elements which have representatives in $\underset{T\in\mathcal{T}}{\bigotimes}\overline{\mathrm{Sk}}(T)$
of the form
\[
\left(\underset{T\in\mathcal{T}}{\otimes}x_{T}\right)\cup\left(\hat{e}^{\epsilon}-(-A^{2})^{\epsilon}\right),\ \epsilon\in\{\pm1\},
\]
where $\underset{T\in\mathcal{T}}{\otimes}x_{T}$ is balanced and $\hat{e}$ is the element given by (\ref{eqn:definition of e^hat}) associated to an edge $e\in \mathcal{T}^{(1)}$. Apply
$\Pi^{\prime}\circ\underset{T\in\mathcal{T}}{\bigotimes}\sigma_{T}$
to $\left(\underset{T\in\mathcal{T}}{\otimes}x_{T}\right)\cup\left(\hat{e}^{\epsilon}-(-A^{2})^{\epsilon}\right)$,
the result is the element represented by
\[
\left(\underset{T\in\mathcal{T}}{\otimes}\sigma_{T}(x_{T})\right)\cup\left(\overline{e}^{\epsilon}-(-A^{2})^{\epsilon}\right).
\]
Note that $\underset{T\in\mathcal{T}}{\otimes}\sigma_{T}(x_{T})$
is partially balanced, therefore the above element represents an element
of the $R$-submodule $\overline{\mathfrak{R}}_{E}^{pc}$ of $\underset{f_{i}\in\mathbf{f}(\mathcal{T})}{\bigotimes}\overline{\mathrm{Sk}}^{pc}(Cf_{i})$
and thus is $0$ in $\underset{f_{i}\in\mathbf{f}(\mathcal{T})}{\overline{\bigotimes}}\overline{\mathrm{Sk}}^{pc}(Cf_{i})$.
\end{proof}
\endgroup

\vspace{1cm}

We now have discussed all four relevant maps, we can assemble them into the
desired commutative diagram.

\begin{prop}
\label{prop:commutative diagram of various splitting homomorphisms}The
following diagram commutes,
\begin{equation}
\label{eq:commutatve diagram of splitting homomorphisms}
\begin{tikzcd}                                                                 & \underset{f\in\mathcal{T}^{(2)}}{\overline{\bigotimes}}\overline{\mathrm{Sk}}(Sf)  \arrow[rd, "\underset{f\in\mathcal{T}^{(2)}}{\overline{\bigotimes}}\sigma_{Sf}"] &                                                                                                                \\ \mathrm{Sk}(Y) \arrow[ru, "\tilde{\sigma}"] \arrow[rd, "\sigma"'] &                                                                                                                                                                   & {{{\underset{f_{i}\in\mathbf{f}(\mathcal{T})}{\overline{\bigotimes}}\overline{\mathrm{Sk}}^{pc}(Cf_{i})}}} \\                                                                 & \underset{T\in\mathcal{T}}{\overline{\bigotimes}}\overline{\mathrm{Sk}}^{c}(T) \arrow[ru, "\underset{T\in\mathcal{T}}{\overline{\bigotimes}}\sigma_{T}"']           &                                                                                                                \end{tikzcd}.
\end{equation}
\end{prop}

\begin{proof}
Starting with a ribbon link $\ell\in\mathrm{Sk}(Y)$, following both
routes the result are both given by the element of $\underset{f_{i}\in\mathbf{f}(\mathcal{T})}{\overline{\bigotimes}}\overline{\mathrm{Sk}}^{pc}(Cf_{i})$
represented by
\[
\sum_{\overrightarrow{\epsilon}}\left(\underset{f_{i}\in\mathbf{f}(\mathcal{T})}{\otimes}\ell_{Cf_{i}}^{\overrightarrow{\epsilon}}\right).
\]
\end{proof}

\subsection{The construction of $Tr_{\mathcal{T}}^{[\text{PP}]}$. \label{subsec: construction of PP's quantum trace map}}

Consider an ideally triangulated $3$-manifold $(Y,\mathcal{T})$.
The first step of Panitch \& Park's construction of their quantum
trace map is to decompose $Y$ into face suspensions $Y=\underset{f\in\mathcal{T}^{(2)}}{\bigcup}Sf$.
Next, a quantum trace map $Tr_{Sf}\colon\overline{\mathrm{Sk}}(Sf)\rightarrow\mathbf{S}f$
is contructed for each face suspension. Finally, the tensor product of these quantum trace maps for the separate face suspensions is composed with the splitting homomorphism
$\tilde{\sigma}$ (see Proposition \ref{prop:Splitting homomorphism induced by splitting into face suspensions})
to obtain the quantum trace map $Tr_{\mathcal{T}}^{[\text{PP}]}\colon\overline{\mathrm{Sk}}(Y)\rightarrow\hat{\mathcal{G}}_{\mathcal{T}}^{[\text{PP}]}$.

We first focus on the quantum trace map $Tr_{Sf}$ for a single face
suspension. To begin, let us properly define the codomain. In what follows, we
will assume that the face suspension $Sf$ is given by gluing face
cones $Cf_{i}$ and $Cf_{j}$ of bare faces $f_{i}$ and $f_{j}$
which are identified to $f$. Moreover, we label the edges of $f_{i}$
by $a_{i}$, $b_{i}$ and $c_{i}$ in clockwise order with respect
to an outward normal vector of $f_{i}$; and label the edges of $f_{j}$
by $a_{j}$, $b_{j}$ and $c_{j}$ in counterclockwise order with
respect to an outward normal vector of $f_{j}$ such that when $f_{i}$
and $f_{j}$ are identified, $a_{i}$ (resp. $b_{i}$ and $c_{i}$)
is identified with $a_{j}$ (resp. $b_{j}$ and $c_{j}$) (as in Corollary
\ref{cor:generators for skein modules of face cones and face suspensions}).  Recall from \ref{subsec: face cones and face suspensions}
that we labeled the boundary marking edges on the face cones and face
suspensions. We have also given explicit algebraic description of
the various skein algebras associated with the vertices of the boundary
markings.
\begin{defn}
\label{def:face cone module and face suspension module}
(i) The \emph{face
cone module} $\mathbf{C}f_{i}$ associated with the face cone $Cf_{i}$
is the quantum torus generated by the formal variables $\{a,f_{i}\}$,
$\{b,f_{i}\}$ and $\{c,f_{i}\}$ corresponding to the marking edges
in the edge cones (so we name the variables by the labels of the marking
edges they correspond to), given by
\[
\mathbf{C}f_{i}:=\frac{R\langle\{a_{i},f_{i}\}^{\pm1},\ \{b_{i},f_{i}\}^{\pm1},\ \{c_{i},f_{i}\}^{\pm1}\rangle}{\left\langle \begin{array}{c}
\{b_{i},f_{i}\}\{a_{i},f_{i}\}=A\{a_{i},f_{i}\}\{b_{i},f_{i}\},\\
\{c_{i},f_{i}\}\{b_{i},f_{i}\}=A\{b_{i},f_{i}\}\{c_{i},f_{i}\},\\
\{a_{i},f_{i}\}\{c_{i},f_{i}\}=A\{c_{i},f_{i}\}\{a_{i},f_{i}\}
\end{array}\right\rangle }.
\]
(ii) The\emph{ face suspension module} $\mathbf{S}f$ associated with the
face suspension $Sf$ is defined to be the quantum torus $\mathbf{C}f_{i}\otimes\ensuremath{\mathbf{C}}f_{j}$.
Thus $\mathbf{S}f$ is given by
\[
 \frac{R\langle\{a_{i},f_{i}\}^{\pm1},\ \{b_{i},f_{i}\}^{\pm1},\ \{c_{i},f_{i}\}^{\pm1}\rangle}{\left\langle \begin{array}{c}
\{b_{i},f_{i}\}\{a_{i},f_{i}\}=A\{a_{i},f_{i}\}\{b_{i},f_{i}\},\\
\{c_{i},f_{i}\}\{b_{i},f_{i}\}=A\{b_{i},f_{i}\}\{c_{i},f_{i}\},\\
\{a_{i},f_{i}\}\{c_{i},f_{i}\}=A\{c_{i},f_{i}\}\{a_{i},f_{i}\}
\end{array}\right\rangle }\otimes\frac{R\langle\{a_{j},f_{j}\}^{\pm1},\ \{b_{j},f_{j}\}^{\pm1},\ \{c_{j},f_{j}\}^{\pm1}\rangle}{\left\langle \begin{array}{c}
\{b_{j},f_{j}\}\{a_{j},f_{j}\}=A^{-1}\{a_{j},f_{j}\}\{b_{j},f_{j}\},\\
\{c_{j},f_{j}\}\{b_{j},f_{j}\}=A^{-1}\{b_{j},f_{j}\}\{c_{j},f_{j}\},\\
\{a_{j},f_{j}\}\{c_{j},f_{j}\}=A^{-1}\{c_{j},f_{j}\}\{a_{j},f_{j}\}
\end{array}\right\rangle }.
\]
\end{defn}

Recall that the reduced skein module $\overline{\mathrm{Sk}}(Sf)$ of
the face suspension is given by an $\tilde{\mathbb{T}}_{f_{i}}\otimes\tilde{\mathbb{T}}_{f_{j}}$-$\left(\mathbb{B}^{\otimes3}\right)_{Sf}$-bimodule
quotient of the algebra $\tilde{\mathbb{T}}_{f_{i}}\otimes\tilde{\mathbb{T}}_{f_{j}}\otimes\left(\mathbb{B}^{\otimes3}\right)_{Sf}$
(see Corollary \ref{cor:generators for skein modules of face cones and face suspensions}
(ii)). Let's first consider a map on the level of this algebra.

\begin{lem}
\label{lem: map of algebra that will descend to quantum trace on Sf}
\emph{(\cite[Lemma 5.4]{PP1})} The following assignments give well-defined $R$-algebra homomorphisms. (Note that these conventions follow the ones in \cite{PP2}.) 

Define $\rho_{f_{i}}\colon\tilde{\mathbb{T}}_{f_{i}} \rightarrow \mathbf{C}f_{i}$ by:
\begin{eqnarray*}
\tilde{\Gamma}_{a_{i}b_{i}} & \mapsto & (-A^{2})^{-\frac{1}{2}}\left[\{a_{i},f_{i}\}\{b_{i},f_{i}\}\right]\\
\tilde{\Gamma}_{b_{i}c_{i}} & \mapsto & (-A^{2})^{-\frac{1}{2}}\left[\{b_{i},f_{i}\}\{c_{i},f_{i}\}\right]\\
\tilde{\Gamma}_{c_{i}a_{i}} & \mapsto & (-A^{2})^{-\frac{1}{2}}\left[\{c_{i},f_{i}\}\{a_{i},f_{i}\}\right].
\end{eqnarray*}
And define $\eta_{f}\colon\left(\mathbb{B}^{\otimes3}\right)_{Sf} \rightarrow  \mathbf{S}f$ by:
\begin{eqnarray*}
x_{a_{i},f_{i};a_{j},f_{j}} & \mapsto & \{a_{i},f_{i}\}\otimes\{a_{j},f_{j}\}\\
x_{b_{i},f_{i};b_{j},f_{j}} & \mapsto & \{b_{i},f_{i}\}\otimes\{b_{j},f_{j}\}\\
x_{c_{i},f_{i};c_{j},f_{j}} & \mapsto & \{c_{i},f_{i}\}\otimes\{c_{j},f_{j}\}.
\end{eqnarray*}
\end{lem}
The proof is a straightforward check. A detail worth noticing is
that because $\left(\mathbb{B}^{\otimes3}\right)_{Sf}$ is a commutative
algebra, its image under $\eta_{f}$ in the tensor product of quantum tori $\mathbf{S}f$ should be a commutative
subalgebra. This means, for example, that $\{a_{i},f_{i}\}\{a_{j},f_{j}\}$
and $\{b_{i},f_{i}\}\{b_{j},f_{j}\}$ should commute
in $\mathbf{S}f$ (which is not completely trivial but is easy to check).

The tensor product $\rho_{f_{i}}\otimes\rho_{f_{j}}$ then gives us
an algebra homomorphism $\tilde{\mathbb{T}}_{f_{i}}\otimes\tilde{\mathbb{T}}_{f_{j}}\rightarrow\mathbf{S}f$.
Together with $\eta_{f}$, this equips the face suspension module
$\mathbf{S}f$ with a $\tilde{\mathbb{T}}_{f_{i}}\otimes\tilde{\mathbb{T}}_{f_{j}}$-$\left(\mathbb{B}^{\otimes3}\right)_{Sf}$-bimodule
structure. 

\begin{prop}
\label{prop: quantum trace map on Sf} \emph{(\cite[Theorem 5.5]{PP1})} The
$\tilde{\mathbb{T}}_{f_{i}}\otimes\tilde{\mathbb{T}}_{f_{j}}$-$\left(\mathbb{B}^{\otimes3}\right)_{Sf}$-bimodule
homomorphism 
\[
\rho_{f_{i}}\otimes\rho_{f_{j}}\otimes\eta_{f}\colon\tilde{\mathbb{T}}_{f_{i}}\otimes\tilde{\mathbb{T}}_{f_{j}}\otimes\left(\mathbb{B}^{\otimes3}\right)_{Sf}\rightarrow\mathbf{S}f
\]
maps the $\tilde{\mathbb{T}}_{f_{i}}\otimes\tilde{\mathbb{T}}_{f_{j}}$-$\left(\mathbb{B}^{\otimes3}\right)_{Sf}$-sub-bimodule
$\mathrm{Ann}([\emptyset])$ to $\{0\}$, therefore it descends to a
well-defined $\tilde{\mathbb{T}}_{f_{i}}\otimes\tilde{\mathbb{T}}_{f_{j}}$-$\left(\mathbb{B}^{\otimes3}\right)_{Sf}$-bimodule
homomorphism 
\[
Tr_{Sf}\colon\overline{\mathrm{Sk}}(Sf)\rightarrow\mathbf{S}f.
\]
(See Corollary \ref{cor:generators for skein modules of face cones and face suspensions}
(ii) for the presentation of $\overline{\mathrm{Sk}}(Sf)$.)
\end{prop}
The proof is once again straightforward checking.

\begin{defn}
\label{def:quantum trace map for Sf}The map 
\[
Tr_{Sf}\colon\overline{\mathrm{Sk}}(Sf)\rightarrow\mathbf{S}f
\]
given by Proposition \ref{prop: quantum trace map on Sf} is the \emph{quantum
trace map for the face suspension} $Sf$. It is a $\tilde{\mathbb{T}}_{f_{i}}\otimes\tilde{\mathbb{T}}_{f_{j}}$-$\left(\mathbb{B}^{\otimes3}\right)_{Sf}$-bimodule
homomorphism.
\end{defn}

Now we have the quantum trace map for each face suspension, we next
need to glue them into a quantum trace map for the 3-manifold $Y$.
Recall from Definition \ref{def:reduced tensor product of skein modules of face suspensions}
that the reduced tensor product $\underset{f\in\mathcal{T}^{(2)}}{\overline{\bigotimes}}\overline{\mathrm{Sk}}(Sf)$
is defined to be the $R$-module quotient
\[
\underset{f\in\mathcal{T}^{(2)}}{\overline{\bigotimes}}\overline{\mathrm{Sk}}(Sf)=\frac{\underset{f\in\mathcal{T}^{(2)}}{\bigotimes}\overline{\mathrm{Sk}}(Sf)}{\widetilde{\mathfrak{R}}_{E}+\sum_{T\in\mathcal{T}}\widetilde{\mathcal{\mathfrak{R}}}_{T}}
\]
where $\widetilde{\mathfrak{R}}_{E}=\left(\underset{f\in\mathcal{T}^{(2)}}{\bigotimes}\overline{\mathrm{Sk}}(Sf)\right)\tilde{I}_{E}$
and $\tilde{\mathfrak{R}}_{T}=\tilde{I}_{T}\left(\underset{f\in\mathcal{T}^{(2)}}{\bigotimes}\overline{\mathrm{Sk}}(Sf)\right)$ (see Definition \ref{def:reduced tensor product of skein modules of face suspensions}).
Now the tensor product 
\[
\underset{f\in\mathcal{T}^{(2)}}{\bigotimes}Tr_{Sf}\colon\underset{f\in\mathcal{T}^{(2)}}{\bigotimes}\overline{\mathrm{Sk}}(Sf)\rightarrow\underset{f\in\mathcal{T}^{(2)}}{\bigotimes}\mathbf{S}f
\]
is a $\underset{f_{i}\in\mathbf{f}(\mathcal{T})}{\bigotimes}\tilde{\mathbb{T}}_{f_{i}}$-$\underset{f\in\mathcal{T}^{(2)}}{\bigotimes}\left(\mathbb{B}^{\otimes3}\right)_{Sf}$-bimodule
homomorphism, with the $\underset{f_{i}\in\mathbf{f}(\mathcal{T})}{\bigotimes}\tilde{\mathbb{T}}_{f_{i}}$-$\underset{f\in\mathcal{T}^{(2)}}{\bigotimes}\left(\mathbb{B}^{\otimes3}\right)_{Sf}$-bimodule
structure on $\underset{f\in\mathcal{T}^{(2)}}{\bigotimes}\mathbf{S}f$
given by the algebra homomorphisms
\begin{equation}
\rho=\underset{f_{i}\in\mathbf{f}(\mathcal{T})}{\bigotimes}\rho_{i}\colon\underset{f_{i}\in\mathbf{f}(\mathcal{T})}{\bigotimes}\tilde{\mathbb{T}}_{f_{i}}\rightarrow\underset{f_{i}\in\mathbf{f}(\mathcal{T})}{\bigotimes}\mathbf{C}f_{i}=\underset{f\in\mathcal{T}^{(2)}}{\bigotimes}\mathbf{S}f\label{eq:the map tilde=00007BQ=00007D}
\end{equation}
and 
\begin{equation}
\eta=\underset{f\in\mathcal{T}^{(2)}}{\bigotimes}\eta_{f}\colon\underset{f\in\mathcal{T}^{(2)}}{\bigotimes}\left(\mathbb{B}^{\otimes3}\right)_{Sf}\rightarrow\underset{f\in\mathcal{T}^{(2)}}{\bigotimes}\mathbf{S}f.\label{eq:the map tilde=00007BB=00007D}
\end{equation}
Therefore, by definiton $\underset{f\in\mathcal{T}^{(2)}}{\bigotimes}Tr_{Sf}$
maps the $R$-submodule 
\[
\widetilde{\mathfrak{R}}_{E}+\sum_{T\in\mathcal{T}}\widetilde{\mathcal{\mathfrak{R}}}_{T}=\left(\underset{f\in\mathcal{T}^{(2)}}{\bigotimes}\overline{\mathrm{Sk}}(Sf)\right)\tilde{I}_{E}+\sum_{T\in\mathcal{T}}\tilde{I}_{T}\left(\underset{f\in\mathcal{T}^{(2)}}{\bigotimes}\overline{\mathrm{Sk}}(Sf)\right)
\]
of $\underset{f\in\mathcal{T}^{(2)}}{\bigotimes}\overline{\mathrm{Sk}}(Sf)$
into the $R$-submodule 
\[
\left(\underset{f\in\mathcal{T}^{(2)}}{\bigotimes}\mathbf{S}f\right)\left(\eta(\tilde{I}_{E})\right)+\sum_{T\in\mathcal{T}}\left(\rho(\tilde{I}_{T})\right)\left(\underset{f\in\mathcal{T}^{(2)}}{\bigotimes}\mathbf{S}f\right)
\]
of $\underset{f\in\mathcal{T}^{(2)}}{\bigotimes}\mathbf{S}f$. Thus
if we define the reduced tensor product of face suspension modules as the quotient $R$-module
\begin{equation}
\label{eq:reduced tensor product of face suspension modules}
\underset{f\in\mathcal{T}^{(2)}}{\overline{\bigotimes}}\mathbf{S}f:=\frac{\underset{f\in\mathcal{T}^{(2)}}{\bigotimes}\mathbf{S}f}{\left(\underset{f\in\mathcal{T}^{(2)}}{\bigotimes}\mathbf{S}f\right)\left(\eta(\tilde{I}_{E})\right)+\sum_{T\in\mathcal{T}}\left(\rho(\tilde{I}_{T})\right)\left(\underset{f\in\mathcal{T}^{(2)}}{\bigotimes}\mathbf{S}f\right)},
\end{equation}
 then, $\underset{f\in\mathcal{T}^{(2)}}{\bigotimes}Tr_{Sf}$ descends
to an $R$-homomorphism
\[
\underset{f\in\mathcal{T}^{(2)}}{\overline{\bigotimes}}Tr_{Sf}\colon\underset{f\in\mathcal{T}^{(2)}}{\overline{\bigotimes}}\overline{\mathrm{Sk}}(Sf)\rightarrow\underset{f\in\mathcal{T}^{(2)}}{\overline{\bigotimes}}\mathbf{S}f.
\]

\begin{defn}
\label{def: quantum trace map of PP} The quantum trace map $Tr_{\mathcal{T}}^{[\text{PP}]}$
of Panitch \& Park is given by the composition of $R$-module homomorphisms
\[
Tr_{\mathcal{T}}^{[\text{PP}]}:=\left(\underset{f\in\mathcal{T}^{(2)}}{\overline{\bigotimes}}Tr_{Sf}\right)\circ\tilde{\sigma}\colon\mathrm{Sk}(Y)\rightarrow\underset{f\in\mathcal{T}^{(2)}}{\overline{\bigotimes}}\overline{\mathrm{Sk}}(Sf)\rightarrow\underset{f\in\mathcal{T}^{(2)}}{\overline{\bigotimes}}\mathbf{S}f,
\]
where $\tilde{\sigma}$ is the splitting homomorphism induced by decomposing
$Y$ into face suspensions described in \ref{subsec:splitting homomorphisms for face cones and face suspensions}.
\end{defn}

Having described the map, next we need to refine the codomain and
bring up the notions of quantized shape parameters and quantum gluing module, as defined by Panitch \& Park. 

\begin{defn}
\label{def:quantized shape parameters and quantum gluing module as in PP}
(i) Let $T$ be an ideal tetrahedron, and, without loss of generality, assume its bare faces are $f_{1}$, $f_{2}$, $f_{3}$ and $f_{4}$. The bare edges of $T$ are labeled by shape
parameters as in Figure \ref{fig:ideal T with edges and faces all labeled}.

\begin{figure}[h]
  \includegraphics[scale=0.15]{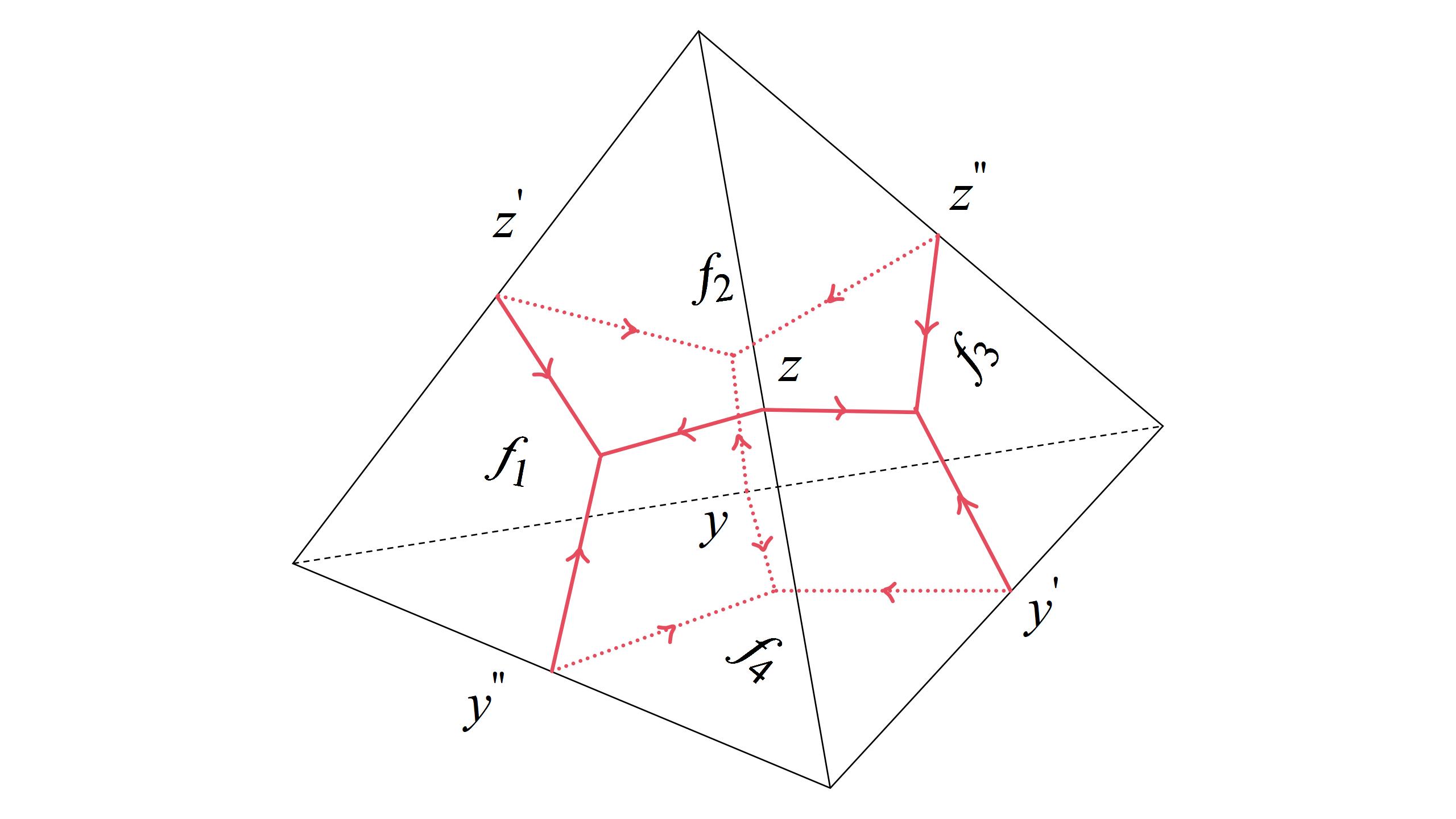} 
  \caption{}
  \label{fig:ideal T with edges and faces all labeled}
\end{figure}

Define the following elements in $\underset{f_{i}\in\mathbf{f}(T)}{\bigotimes}\mathbf{C}f_{i}\subset\underset{f_{i}\in\mathbf{f}(\mathcal{T})}{\bigotimes}\mathbf{C}f_{i}=\underset{f\in\mathcal{T}^{(2)}}{\bigotimes}\mathbf{S}f$:
\[
\begin{array}{lll}
\tilde{z}:=\{z,f_{1}\}\otimes\{z,f_{3}\}, & 
\tilde{z}^{\prime}:=\{z^{\prime},f_{1}\}\otimes\{z^{\prime},f_{2}\}, & 
\tilde{z}^{\prime\prime}:=\{z^{\prime\prime},f_{2}\}\otimes\{z^{\prime\prime},f_{3}\}, \\
\tilde{y}:=\{y,f_{2}\}\otimes\{y,f_{4}\},
 &
\tilde{y}^{\prime}:=\{y^{\prime},f_{3}\}\otimes\{y^{\prime},f_{4}\},
&  
\tilde{y}^{\prime\prime}:=\{y^{\prime\prime},f_{1}\}\otimes\{y^{\prime\prime},f_{4}\}.
\end{array}
\]

These are the \emph{quantized shape parameters}\footnote{The definition is basically saying that if an edge of $T$ is labeled
with shape parameter $a$, then the corresponding quantized shape
parameter $\tilde{a}$ is given by the product of the two generators
of $\underset{f_{i}\in\mathbf{f}(T)}{\bigotimes}\mathbf{C}f_{i}$
corresponding to the two (dashed) marking edges in the edge cones
emanating from the edge labeled with $a$.} introduced
in {\cite{PP1}}. 

(ii) Let $\tilde{\mathbb{T}}\langle T\rangle$ be the $R$-subalgebra
of $\underset{f\in\mathcal{T}^{(2)}}{\bigotimes}\mathbf{S}f$ generated
by $\tilde{z},\tilde{z}^{\prime},\tilde{z}^{\prime\prime},\tilde{y},\tilde{y}^{\prime},\tilde{y}^{\prime\prime}$.

(iii) Let $(Y,\mathcal{T})$ be an ideally triangulated 3-manifold.
The \emph{quantum gluing module} $\hat{\mathcal{G}}_{\mathcal{T}}^{[\text{PP}]}$ introduced
in \cite{PP1} is given by the image of $\underset{T\in\mathcal{T}}{\bigotimes}\tilde{\mathbb{T}}\langle T\rangle\subset\underset{f\in\mathcal{T}^{(2)}}{\bigotimes}\mathbf{S}f$
in $\underset{f\in\mathcal{T}^{(2)}}{\overline{\bigotimes}}\mathbf{S}f$
under the quotient map $\underset{f\in\mathcal{T}^{(2)}}{\bigotimes}\mathbf{S}f\twoheadrightarrow\underset{f\in\mathcal{T}^{(2)}}{\overline{\bigotimes}}\mathbf{S}f$.
\end{defn}

\begin{rem}
The algebra $\tilde{\mathbb{T}}\langle T\rangle$ is a quantum torus,
indeed it is straightforward to check by definition that is has a presentation\footnote{For example, $z^{\prime},z,y^{\prime\prime}$ is a clockwise labeling of the three edges
of $f_{1}$, clockwise with respect to an outward normal of $f_{1}$.
Therefore by Definition \ref{def:face cone module and face suspension module}
of the face cone module $\mathbf{C}f_{1}$, we have $\{z,f_{1}\}\{z^{\prime},f_{1}\}=A\{z^{\prime},f_{1}\}\{z,f_{1}\}$,
which implies that $\tilde{z}\tilde{z}^{\prime}=A\tilde{z}^{\prime}\tilde{z}$.} 
\begin{equation}
\tilde{\mathbb{T}}\langle T\rangle=\frac{R\langle\tilde{z}^{\pm1},\tilde{z}^{\prime\pm1},\tilde{z}^{\prime\prime\pm1},\tilde{y}^{\pm1},\tilde{y}^{\prime\pm1},\tilde{y}^{\prime\prime\pm1}\rangle}{\left\langle \begin{array}{c}
\tilde{a}\tilde{b}^{\prime}=A\tilde{b}^{\prime}\tilde{a}\\
\tilde{a}^{\prime}\tilde{b}^{\prime\prime}=A\tilde{b}^{\prime\prime}\tilde{a}^{\prime}\\
\tilde{a}^{\prime\prime}\tilde{b}=A\tilde{b}\tilde{a}^{\prime\prime}
\end{array}\bigg|a,b\in\{z,y\}\right\rangle }.\label{eq:presentation of tilde mathbbT(T)}
\end{equation}
This shows that $\hat{\mathcal{G}}_{\mathcal{T}}^{[\text{PP}]}$ is
given by a quotient of the tensor product of the quantum tori $\underset{T\in\mathcal{T}}{\bigotimes}\tilde{\mathbb{T}}\langle T\rangle$.
This is similar to our quantum gluing module $\hat{\mathcal{G}}_{\mathcal{T}}$, which we had given as an explicit quotient of the quantum
tori $\underset{T\in\mathcal{T}}{\bigotimes}\mathbb{T}\langle T\rangle$
(see Lemma \ref{lem:presntation of (T^tensor4,cdot)  and (B^tensor6, cdot)}
(ii), Definition \ref{def: quantum module of T} and Definition \ref{def: quantum gluing module}). 
\end{rem}

\begin{prop}
\label{prop:PP quantum trace map takes value in their quantum gluing module}
\emph{(\cite[Lemma 5.14]{PP1})}The image of the quantum trace map
$Tr_{\mathcal{T}}^{[\text{PP}]}$ is contained in $\hat{\mathcal{G}}_{\mathcal{T}}^{[\text{PP}]}$.
That is $Tr_{\mathcal{T}}^{[\text{PP}]}$ is an $R$-module homomorphism
\[
Tr_{\mathcal{T}}^{[\text{PP}]}\colon\mathrm{Sk}(Y)\rightarrow\hat{\mathcal{G}}_{\mathcal{T}}^{[\text{PP}]}.
\]
\end{prop}

\subsection{The map between the quantum gluing modules.\label{subsec:map between the quantum gluing modules}}

In this subsection, we study the map $\mu$ from our quantum gluing
module $\hat{\mathcal{G}}_{\mathcal{T}}$ to the quantum gluing module
$\hat{\mathcal{G}}_{\mathcal{T}}^{[\text{PP}]}$ of Panitch \& Park.
It is clear from their presentations that $\mathbb{T}\langle T\rangle$
and $\tilde{\mathbb{T}}\langle T\rangle$ are isomorphic as quantum
tori, and the isomorphism is simply given by ``matching'' the quantized
shape parameters. This isomorphism will induce a map between the two
quantum gluing modules. More specifically, we have the following
\begin{prop}
\label{prop:maps from our quantum gluing module to PP's quantum gluing module}
For each $T\in\mathcal{T}$, the map 
\[
\mu_{T}\colon\mathbb{T}\langle T\rangle\rightarrow\tilde{\mathbb{T}}\langle T\rangle
\]
 of quantum tori given by
 \[
\begin{array}{ccc}
\hat{z}\mapsto\tilde{z} & , & \hat{y}\mapsto\tilde{y}\\
\hat{z}^{\prime}\mapsto\tilde{z}^{\prime} & , & \hat{y}^{\prime}\mapsto\tilde{y}^{\prime}\\
\hat{z}^{\prime\prime}\mapsto\tilde{z}^{\prime\prime} & , & \hat{y}^{\prime\prime}\mapsto\tilde{y}^{\prime\prime}
\end{array}
\]
is an isomorphism. The tensor product of these maps 
\[
\underset{T\in\mathcal{T}}{\bigotimes}\mu_{T}\colon\underset{T\in\mathcal{T}}{\bigotimes}\mathbb{T}\langle T\rangle\rightarrow\underset{T\in\mathcal{T}}{\bigotimes}\tilde{\mathbb{T}}\langle T\rangle
\]
 descends to a surjective $R$-module homomorphism
\[
\mu\colon\hat{\mathcal{G}}_{\mathcal{T}}\twoheadrightarrow\hat{\mathcal{G}}_{\mathcal{T}}^{[\text{PP}]}.
\]
\end{prop}

\begingroup
\allowdisplaybreaks
\begin{proof}
It is clear that $\mu_{T}$ is an isomorphism, see (\ref{eq:quantum torus T(T)})
for the presentation of $\mathbb{T}\langle T\rangle$ and (\ref{eq:presentation of tilde mathbbT(T)})
of $\tilde{\mathbb{T}}\langle T\rangle$. The fact that $\underset{T\in\mathcal{T}}{\bigotimes}\mu_{T}$
descends to the quantum gluing modules is actually implied by \cite[Proposition 5.8]{PP1}.
For the sake of self-containedness, we still give a proof of it, which is essentially a careful rephrase and expansion with details of the calculation done in \cite[Proposition 5.8]{PP1} in our notation.
We show that $\underset{T\in\mathcal{T}}{\bigotimes}\mu_{T}$ descends
to a well-defined map $\hat{\mathcal{G}}_{\mathcal{T}}\rightarrow\underset{f\in\mathcal{T}^{(2)}}{\overline{\bigotimes}}\mathbf{S}f$
whose image is exactly the $R$-submodule $\hat{\mathcal{G}}_{\mathcal{T}}^{[\text{PP}]}$
of $\underset{f\in\mathcal{T}^{(2)}}{\overline{\bigotimes}}\mathbf{S}f$.
Recall that $\hat{\mathcal{G}}_{\mathcal{T}}$ is given by the $R$-module
quotient (see Remark \ref{rem: alternative description of quantum guling module})
\[
\frac{\underset{T\in\mathcal{T}}{\bigotimes}\mathbb{T}\langle T\rangle}{E+\sum_{T\in\mathcal{T}}(V_{T}+L_{T})},
\]
whereas the reduced tensor product $\underset{f\in\mathcal{T}^{(2)}}{\overline{\bigotimes}}\mathbf{S}f$
is defined to be the $R$-module quotient (see (\ref{eq:reduced tensor product of face suspension modules}))
\[
\frac{\underset{f\in\mathcal{T}^{(2)}}{\bigotimes}\mathbf{S}f}{\left(\underset{f\in\mathcal{T}^{(2)}}{\bigotimes}\mathbf{S}f\right)\left(\eta(\tilde{I}_{E})\right)+\sum_{T\in\mathcal{T}}\left(\rho(\tilde{I}_{T})\right)\left(\underset{f\in\mathcal{T}^{(2)}}{\bigotimes}\mathbf{S}f\right)}.
\]
Thus we need to show that $\underset{T\in\mathcal{T}}{\bigotimes}\mu_{T}$
maps the $R$-submodule $
E+\sum_{T\in\mathcal{T}}(V_{T}+L_{T})
$
 into the $R$-submodule 
$
\left(\underset{f\in\mathcal{T}^{(2)}}{\bigotimes}\mathbf{S}f\right)\left(\eta(\tilde{I}_{E})\right)+\sum_{T\in\mathcal{T}}\left(\rho(\tilde{I}_{T})\right)\left(\underset{f\in\mathcal{T}^{(2)}}{\bigotimes}\mathbf{S}f\right).
$

Recall that $E$ is the left ideal of $\underset{T\in\mathcal{T}}{\bigotimes}\mathbb{T}\langle T\rangle$
generated by edge relations, namely, by elements of the form 
$\hat{e}-(-A^{2})$
where $e\in\mathcal{T}^{(1)}$ is an edge of the triangulation. Let's
also recall that the element $\hat{e}$ is given in the following
way: if $e_{1},e_{2},\dots,e_{k}$ are the labeling of bare edges
that are identified to $e$, then 
\[
\hat{e}=\left[\hat{e}_{1}\cdot\hat{e}_{2}\cdot\ldots\cdot\hat{e}_{k}\right].
\]
Thus, we have 
\[
\left(\underset{T\in\mathcal{T}}{\bigotimes}\mu_{T}\right)\left(\hat{e}-(-A^{2})\right)=\left[\tilde{e}_{1}\tilde{e}_{2}\dots\tilde{e}_{k}\right]-(-A^{2})
\]
On the other hand, if the sequence of face suspensions around the edge $e$
is as in Figure \ref{fig:face suspensions around an edge}, then we have
a generator of the left ideal $\tilde{I}_{E}$ given by (see Definition
\ref{def:reduced tensor product of skein modules of face suspensions})
\[
\tilde{e}-(-A^{2})
\]
where 
\[
\tilde{e}=x_{e_{k},f_{1}^{(1)};e_{1},f_{2}^{(1)}}x_{e_{1},f_{1}^{(2)};e_{2};f_{2}^{(2)}}\dots x_{e_{k-1},f_{1}^{(k)};e_{k},f_{2}^{(k)}}\in\underset{f\in\mathcal{T}^{(2)}}{\bigotimes}\left(\mathbb{B}^{\otimes3}\right)_{Sf}.
\]
Then by definition (see (\ref{eq:the map tilde=00007BB=00007D}),
Lemma \ref{lem: map of algebra that will descend to quantum trace on Sf}
and Definition \ref{def:quantized shape parameters and quantum gluing module as in PP}
(i))
\begin{align*}
\eta(\tilde{e}) & =\{e_{k},f_{1}^{(1)}\}\{e_{1},f_{2}^{(1)}\}\{e_{1},f_{1}^{(2)}\}\{e_{2};f_{2}^{(2)}\}\dots\{e_{k-1},f_{1}^{(k)}\}\{e_{k},f_{2}^{(k)}\}\\
 & =\left[\{e_{k},f_{1}^{(1)}\}\{e_{1},f_{2}^{(1)}\}\{e_{1},f_{1}^{(2)}\}\{e_{2};f_{2}^{(2)}\}\dots\{e_{k-1},f_{1}^{(k)}\}\{e_{k},f_{2}^{(k)}\}\right]\\
 & =\left[\{e_{1},f_{2}^{(1)}\}\{e_{1},f_{1}^{(2)}\}\{e_{2};f_{2}^{(2)}\}\dots\{e_{k-1},f_{1}^{(k)}\}\{e_{k},f_{2}^{(k)}\}\{e_{k},f_{1}^{(1)}\}\right]\\
 & =\left[\tilde{e}_{1}\tilde{e}_{2}\dots\tilde{e}_{k}\right].
\end{align*}
(Here we are using the fact that all the terms of the form $\{e_{p-1},f_{1}^{(p)}\}\{e_{p},f_{2}^{(p)}\}$
as well as $\{e_{k},f_{1}^{(1)}\}\{e_{1},f_{2}^{(1)}\}$ commute with
each other, therefore the $A$ factors one gains when expanding the Weyl-ordering
of the product of those terms cancel in pairs.) This shows that
\[
\left(\underset{T\in\mathcal{T}}{\bigotimes}\mu_{T}\right)\left(\hat{e}-(-A^{2})\right)=\eta\left(\tilde{e}-(-A^{2})\right)
\]
and thus $\underset{T\in\mathcal{T}}{\bigotimes}\mu_{T}$ maps the
left ideal $E$ of $\underset{T\in\mathcal{T}}{\bigotimes}\mathbb{T}\langle T\rangle$
into the left ideal $\left(\underset{f\in\mathcal{T}^{(2)}}{\bigotimes}\mathbf{S}f\right)\left(\eta(\tilde{I}_{E})\right)$
of $\underset{f\in\mathcal{T}^{(2)}}{\bigotimes}\mathbf{S}f$.

Next we focus on the right ideals $V_{T}$ and $L_T$ of $\underset{T\in\mathcal{T}}{\bigotimes}\mathbb{T}\langle T\rangle$, assuming the bare faces of the tetrahedron $T$ at hand are $f_1$, $f_{2}$, $f_{3}$ and $f_{4}$, and the bare egdes are labeled as in Figure \ref{fig:ideal T with edges and faces all labeled}. The right ideal $V_{T}$ is generated by vertex relations, for instance 
\[
\left[\hat{z}\cdot\hat{z}^{\prime}\cdot\hat{z}^{\prime\prime}\right]-(-A^{2})^{\frac{1}{2}},
\]
and we have
\[
\left(\underset{T\in\mathcal{T}}{\bigotimes}\mu_{T}\right)\left(\left[\hat{z}\cdot\hat{z}^{\prime}\cdot\hat{z}^{\prime\prime}\right]-(-A^{2})^{\frac{1}{2}}\right)=\left[\tilde{z}\tilde{z}^{\prime}\tilde{z}^{\prime\prime}\right]-(-A^{2})^{\frac{1}{2}}.
\]
On the other hand, one of generators of the right ideal $\tilde{I}_{T}$
is (Definition \ref{def:reduced tensor product of skein modules of face suspensions}
(i))
\[
\tilde{\Gamma}_{zz^{\prime}}\tilde{\Gamma}_{z^{\prime}z^{\prime\prime}}\tilde{\Gamma}_{z^{\prime\prime}z}-(-A^{2})^{-1},
\]
by definition we have (see (\ref{eq:the map tilde=00007BQ=00007D}),
Lemma \ref{lem: map of algebra that will descend to quantum trace on Sf}
and Definition \ref{def:quantized shape parameters and quantum gluing module as in PP}
and also assume that the faces of the tetrahedron are labeled as in
Figure \ref{fig:ideal T with edges and faces all labeled})
\begin{eqnarray*}
\lefteqn{\rho\left(\tilde{\Gamma}_{zz^{\prime}}\tilde{\Gamma}_{z^{\prime}z^{\prime\prime}}\tilde{\Gamma}_{z^{\prime\prime}z}-(-A^{2})^{-1}\right)}\\
&=&(-A^{2})^{-\frac{3}{2}}\left[\{z,f_{1}\}\{z^{\prime},f_{1}\}\{z^{\prime},f_{2}\}\{z^{\prime\prime},f_{2}\}\{z^{\prime\prime},f_{3}\}\{z,f_{3}\}\right]-(-A^{2})^{-1}\\
&=&(-A^{2})^{-\frac{3}{2}}\left[\{z,f_{3}\}\{z,f_{1}\}\{z^{\prime},f_{1}\}\{z^{\prime},f_{2}\}\{z^{\prime\prime},f_{2}\}\{z^{\prime\prime},f_{3}\}\right]-(-A^{2})^{-1}\\
&=&(-A^{2})^{-\frac{3}{2}}\left(\left[\tilde{z}\tilde{z}^{\prime}\tilde{z}^{\prime\prime}\right]-(-A^{2})^{\frac{1}{2}}\right).
\end{eqnarray*}
Similar calculations also apply to the other vertex relations. This
shows that $\underset{T\in\mathcal{T}}{\bigotimes}\mu_{T}$ maps the
$V_{T}$ into the right ideal $\left(\rho(\tilde{I}_{T})\right)\left(\underset{f\in\mathcal{T}^{(2)}}{\bigotimes}\mathbf{S}f\right)$
of $\underset{f\in\mathcal{T}^{(2)}}{\bigotimes}\mathbf{S}f$.

The right ideal $L_{T}$ of $\underset{T\in\mathcal{T}}{\bigotimes}\mathbb{T}\langle T\rangle$ is generated by the lagrangian relation
\[
\hat{z}^{2}+\hat{z}^{\prime\prime-2}-1.
\]
We have
\[
\left(\underset{T\in\mathcal{T}}{\bigotimes}\mu_{T}\right)\left(\hat{z}^{2}+\hat{z}^{\prime\prime-2}-1\right)=\tilde{z}^{2}+\tilde{z}^{\prime\prime-2}-1.
\]
On the other hand, one of the generators of the right ideal $\tilde{I}_{T}$
is (Definition \ref{def:reduced tensor product of skein modules of face suspensions}
(i))
\[
\left[\tilde{\Gamma}_{zy^{\prime}}^{-1}\tilde{\Gamma}_{y^{\prime}z^{\prime\prime}}\right]-\left[\tilde{\Gamma}_{z^{\prime\prime}y}^{-1}\tilde{\Gamma}_{yz^{\prime}}\right]\tilde{\Gamma}_{zz^{\prime}}-\tilde{\Gamma}_{z^{\prime}z^{\prime\prime}}^{-1}\left[\tilde{\Gamma}_{z^{\prime}y^{\prime\prime}}^{-1}\tilde{\Gamma}_{y^{\prime\prime}z}\right],
\]
we have (again assume that the faces of the tetrahedron are labeled
as in Figure \ref{fig:ideal T with edges and faces all labeled})
\begin{eqnarray*}
\lefteqn{\rho\left(\left[\tilde{\Gamma}_{zy^{\prime}}^{-1}\tilde{\Gamma}_{y^{\prime}z^{\prime\prime}}\right]-\left[\tilde{\Gamma}_{z^{\prime\prime}y}^{-1}\tilde{\Gamma}_{yz^{\prime}}\right]\tilde{\Gamma}_{zz^{\prime}}-\tilde{\Gamma}_{z^{\prime}z^{\prime\prime}}^{-1}\left[\tilde{\Gamma}_{z^{\prime}y^{\prime\prime}}^{-1}\tilde{\Gamma}_{y^{\prime\prime}z}\right]\right)}\\
&=&A^{-\frac{1}{2}}\{z,f_{3}\}^{-1}\{z^{\prime\prime},f_{3}\}-(-A^{2})^{-\frac{1}{2}}A^{-1}\{z^{\prime\prime},f_{2}\}^{-1}\{z^{\prime},f_{2}\}\{z,f_{1}\}\{z^{\prime},f_{1}\}\\
& &-(-A^{2})^{\frac{1}{2}}A^{-1}\{z^{\prime},f_{2}\}^{-1}\{z^{\prime\prime},f_{2}\}^{-1}\{z^{\prime},f_{1}\}^{-1}\{z,f_{1}\}\\
&=&A^{-\frac{1}{2}}\{z,f_{3}\}^{-1}\{z^{\prime\prime},f_{3}\}-(-A^{2})^{-\frac{1}{2}}\{z^{\prime\prime},f_{2}\}^{-1}\tilde{z}^{\prime}\{z,f_{1}\}-(-A^{2})^{\frac{1}{2}}\{z^{\prime\prime},f_{2}\}^{-1}\tilde{z}^{\prime-1}\{z,f_{1}\}.
\end{eqnarray*}
Multiply the last expression by $\{z^{\prime\prime},f_{3}\}^{-1}\{z,f_{3}\}$
on the right, it becomes
\[
A^{-\frac{1}{2}}-(-A^{2})^{-\frac{1}{2}}\tilde{z}^{\prime\prime-1}\tilde{z}^{\prime}\tilde{z}-(-A^{2})^{\frac{1}{2}}\tilde{z}^{\prime\prime-1}\tilde{z}^{\prime-1}\tilde{z}.
\]
Now if we modify the last expression with $\left[\tilde{z}\tilde{z}^{\prime}\tilde{z}^{\prime\prime}\right]-(-A^{2})^{\frac{1}{2}}$,
which was shown to lie in the right ideal $\left(\rho(\tilde{I}_{T})\right)\left(\underset{f\in\mathcal{T}^{(2)}}{\bigotimes}\mathbf{S}f\right)$
previously, it becomes
\[
A^{-\frac{1}{2}}-A^{-\frac{1}{2}}\tilde{z}^{\prime\prime-2}-A^{-\frac{1}{2}}\tilde{z}^{2},
\]
which is nothing but $-A^{-\frac{1}{2}}\left(\underset{T\in\mathcal{T}}{\bigotimes}\mu_{T}\right)\left(\hat{z}^{2}+\hat{z}^{\prime\prime-2}-1\right)$.
This shows that $\underset{T\in\mathcal{T}}{\bigotimes}\mu_{T}$
maps the right ideal $L_{T}$ of $\underset{T\in\mathcal{T}}{\bigotimes}\mathbb{T}\langle T\rangle$
into the right ideal $\left(\rho(\tilde{I}_{T})\right)\left(\underset{f\in\mathcal{T}^{(2)}}{\bigotimes}\mathbf{S}f\right)$
of $\underset{f\in\mathcal{T}^{(2)}}{\bigotimes}\mathbf{S}f$.

We have shown that $\underset{T\in\mathcal{T}}{\bigotimes}\mu_{T}$
descends to a well-defined $R$-module homomorphism
\[
\mu\colon\hat{\mathcal{G}}_{\mathcal{T}}\rightarrow\underset{f\in\mathcal{T}^{(2)}}{\overline{\bigotimes}}\mathbf{S}f.
\]
Now because $\underset{T\in\mathcal{T}}{\bigotimes}\mu_{T}$ is an
isomorphism of quantum tori $\underset{T\in\mathcal{T}}{\bigotimes}\mathbb{T}\langle T\rangle\overset{\cong}{\rightarrow}\underset{T\in\mathcal{T}}{\bigotimes}\tilde{\mathbb{T}}\langle T\rangle$,
and by definition $\hat{\mathcal{G}}_{\mathcal{T}}^{[\text{PP}]}$
is the image of $\underset{T\in\mathcal{T}}{\bigotimes}\tilde{\mathbb{T}}\langle T\rangle$
in $\underset{f\in\mathcal{T}^{(2)}}{\overline{\bigotimes}}\mathbf{S}f$
under the natural quotient map, we conclude that the image of $\mu$ is
just $\hat{\mathcal{G}}_{\mathcal{T}}^{[\text{PP}]}$.

This completes the proof.
\end{proof}
\endgroup

\begin{rem}
\label{rem: vertex, lagrangian and edge relations hold in PP's quantum gluing module}A
consequence of the last proposition is that the relations satisfied
by the quantized shape parameters in $\hat{\mathcal{G}}_{\mathcal{T}}$,
namely the vertex relations, lagrangian relations and edge relations
are also satisfied by the quantized shape parameters in $\hat{\mathcal{G}}_{\mathcal{T}}^{[\text{PP}]}$.
That is, the followings hold in $\hat{\mathcal{G}}_{\mathcal{T}}^{[\text{PP}]}$.
\begin{itemize}
\item If $T\in\mathcal{T}$ is an ideal tetrahedron with edges labeled as
in Figure \ref{fig:ideal T with edges and faces all labeled}, then
we have
\[
\left[\tilde{z}\tilde{z}^{\prime}\tilde{z}^{\prime\prime}\right]=\left[\tilde{z}\tilde{y}^{\prime}\tilde{y}^{\prime\prime}\right]=\left[\tilde{y}\tilde{z}^{\prime}\tilde{y}^{\prime\prime}\right]=\left[\tilde{y}\tilde{y}^{\prime}\tilde{z}^{\prime\prime}\right]=(-A^{2})^{\frac{1}{2}}
\]
\item Similarly, we also have
\[
\tilde{z}^{\prime\prime-2}+\tilde{z}^{2}-1=0
\]
as well as other lagrangian relations identifying left action.
\item If $e\in\mathcal{T}^{(1)}$ is an edge of the triangulation and $e_{1},e_{2},\dots,e_{k}$
are the labels of bare edges (by shape parameters) identified to $e$
. We have
\[
\left[\tilde{e}_{1}\tilde{e}_{2}\dots\tilde{e}_{k}\right]=(-A^{2})
\]
identifying right action.
\end{itemize}
\end{rem}

These are the defining relations for our quantum gluing module $\hat{\mathcal{G}}_{\mathcal{T}}$,
however it is not clear to us if it is also the case for $\hat{\mathcal{G}}_{\mathcal{T}}^{[\text{PP}]}$. After all, the quantum gluing module $\hat{\mathcal{G}}_{\mathcal{T}}^{[\text{PP}]}$ of Panitch \& Park is by definition a submodule of certain quotient of $\underset{f\in\mathcal{T}^{(2)}}{\bigotimes}\mathbf{S}f$ that ensures the map $\underset{f\in\mathcal{T}^{(2)}}{\bigotimes}Tr_{Sf}$ can descends to $\underset{f\in\mathcal{T}^{(2)}}{\overline{\bigotimes}}\mathrm{Sk}(Sf)$.
Therefore we can only conclude that the map $\mu\colon\hat{\mathcal{G}}_{\mathcal{T}}\rightarrow\hat{\mathcal{G}}_{\mathcal{T}}^{[\text{PP}]}$
is surjective. We make the injectivity of $\mu$ a conjecture.
\begin{conjecture}
\label{conj:mu is injective}The surjective $R$-module homomorphism
$\mu\colon\hat{\mathcal{G}}_{\mathcal{T}}\rightarrow\hat{\mathcal{G}}_{\mathcal{T}}^{[\text{PP}]}$
is also injective, therefore it is an isomorphism of $R$-modules.
\end{conjecture}

\subsection{Proof of $Tr_{\mathcal{T}}^{[\text{PP}]}=\mu\circ Tr_{\mathcal{T}}$\label{subsec:Exact relation between our quantum trace and PP quantum trace}}

We are finally in a position to prove our second main comparison theorem, Theorem
\ref{thm: Main theorem 2, exact relation between our Tr and PP's Tr}.
We can decompose both ideal tetrahedra and face suspensions into face
cones, and recall that in both cases we have splitting homomorphisms of skein
modules coming from such decompositions (see the discussions in \ref{subsec:splitting homomorphisms for face cones and face suspensions}).
The key is to construct a further version of the quantum trace map defined on the partially
corner-reduced skein module $\overline{\mathrm{Sk}}^{pc}(Cf_{i})$ of
each face cone $Cf_{i}$, such that when the face cones are glued back into the face suspensions or ideal tetrahedra, the quantum trace maps on the face cones glue to compare with the quantum trace map $Tr_{Sf}$ on face suspensions and the quantum trace map $Tr_{T}^{c}$ on ideal tetrahedra.

We begin by describing the construction of the quantum trace map on $\overline{\mathrm{Sk}}(Cf_{i})$.
In \ref{subsec: partial corner reduction},
we have realized the skein module $\overline{\mathrm{Sk}}(Cf_{i})$
of the face cone of a bare face $f_{i}$ and its partial corner-reduction
$\overline{\mathrm{Sk}}^{pc}(Cf_{i})$ as a
\[
\left(\mathbb{T}_{f_{i}},\cdot\right)\otimes\tilde{\mathbb{T}}_{f_{i}}\text{-}\left(\left(\mathbb{B}^{\otimes3}\right)\cdot\right)=R[Cf_{i}]\otimes\tilde{\mathbb{T}}_{f_{i}}\text{-}\mathbb{T}\langle Cf_{i}\rangle
\]
-bimodule quotient of 
\[
\left(\mathbb{T}_{f_{i}},\cdot\right)\otimes\tilde{\mathbb{T}}_{f_{i}}\otimes\left(\left(\mathbb{B}^{\otimes3}\right)\cdot\right)=R[Cf_{i}]\otimes\tilde{\mathbb{T}}_{f_{i}}\otimes\mathbb{T}\langle Cf_{i}\rangle.
\]
Similar to what we did in \ref{subsec:quantum trace map on T}, we
first construct a map on the level of the algebra $R[Cf_{i}]\otimes\tilde{\mathbb{T}}_{f_{i}}\otimes\mathbb{T}\langle Cf_{i}\rangle$.

\begin{defn}
We define the following maps.
\begin{enumerate}
\item Let $
Q\colon R[Cf_{i}]\rightarrow R
$
 be the algebra homomorphism that maps each generator of $\Gamma_{ab}$ of $R[Cf_{i}]$
to the scalar $(-A^{2})^{-\frac{1}{2}}$. 
\item Let 
$
\eta_{f_{i}}\colon\mathbb{T}\langle Cf_{i}\rangle\rightarrow\mathbf{C}f_{i}
$
 be the isomorphism of quantum tori given by (assuming the labeling
of edges by $a$, $b$ and $c$)
\[
x_{a,f_{i}}\mapsto\{a,f_{i}\}
\]
\[
x_{b,f_{i}}\mapsto\{b,f_{i}\}
\]
\[
x_{c,f_{i}}\mapsto\{c,f_{i}\}
\]
(By the presentations of the quantum tori, see Lemma \ref{lem:presentation of (T,cdot) and (B^tensor 3, cdot)}
(ii) and Definition \ref{def:face cone module and face suspension module}
(i), $\eta_{f_{i}}$ is indeed an isomorphism of quantum tori.)
\item Let 
\[
P^{pc}\colon R[Cf_{i}]\otimes\tilde{\mathbb{T}}_{f_{i}}\otimes\mathbb{T}\langle Cf_{i}\rangle\rightarrow\mathbf{C}f_{i}
\]
be the $R$-module homomorphism given by
\[
P^{pc}\left(\Gamma\otimes\tilde{\Gamma}\otimes x\right)=A^{\frac{1}{2}\langle d_{f_{i}}(\Gamma),d_{f_{i}}(x)\rangle_{f_{i}}}Q(\Gamma)\rho_{f_{i}}(\tilde{\Gamma})\eta_{f_{i}}(x)
\]
for $\mathbb{Z}^{M_{f_{i}}}$-homogeneous elements and extend by
multilinearity. Recall the maps $\rho_{f_i}$ were introduced in Lemma \ref{lem: map of algebra that will descend to quantum trace on Sf}.
\end{enumerate}
\end{defn}

Similar to what we had in \ref{subsec:quantum trace map on T}, the
factor $A^{\frac{1}{2}\langle d_{f_{i}}(\Gamma),d_{f_{i}}(x)\rangle_{f_{i}}}$
guarantees that $P^{pc}$ is a $R[Cf_{i}]\otimes\tilde{\mathbb{T}}_{f_{i}}\text{-}\mathbb{T}\langle Cf_{i}\rangle$-bimodule
homomorphism.

\begin{lem}
\label{lem:P^pc is a bimodule homomorphism}
The $R$-module homomorphism $P^{pc}$ is in fact a $R[Cf_{i}]\otimes\tilde{\mathbb{T}}_{f_{i}}\text{-}\mathbb{T}\langle Cf_{i}\rangle$-bimodule
homomorphism\footnote{The left $R[Cf_{i}]$-module structure of $\mathbf{C}f_{i}$ is
given by the algebra homomorphism $Q$; the left $\tilde{\mathbb{T}}_{f_{i}}$-module
structure of $\mathbf{C}f_{i}$ is provided by the map $\rho_{f_{i}}$; the
right $\mathbb{T}\langle Cf_{i}\rangle$-module structure of $\mathbf{C}f_{i}$
comes from the algebra homomorphism $\eta_{f_{i}}$}. 
\end{lem}

\begin{proof}
The proof is a straightforward calculation similar to that in the
proof of Lemma \ref{lem:P is a bimodule homomorphism}. We leave the
details to the reader.
\end{proof}

\begin{prop}
\label{prop:P^pc descends}
The $R[Cf_{i}]\otimes\tilde{\mathbb{T}}_{f_{i}}\text{-}\mathbb{T}\langle Cf_{i}\rangle$-bimodule
homomorphism 
\[
P^{pc}\colon R[Cf_{i}]\otimes\tilde{\mathbb{T}}_{f_{i}}\otimes\mathbb{T}\langle Cf_{i}\rangle\rightarrow\mathbf{C}f_{i}
\]
descends to a well-defined $R[Cf_{i}]\otimes\tilde{\mathbb{T}}_{f_{i}}\text{-}\mathbb{T}\langle Cf_{i}\rangle$-bimodule
homomorphism 
\[
Tr_{Cf_{i}}\colon\overline{\mathrm{Sk}}(Cf_{i})\rightarrow\mathbf{C}f_{i},
\]
which further descends to a well-defined $R[Cf_{i}]\otimes\tilde{\mathbb{T}}_{f_{i}}\text{-}\mathbb{T}\langle Cf_{i}\rangle$-bimodule
homomorphism
\[
Tr_{Cf_{i}}^{pc}\colon\overline{\mathrm{Sk}}^{pc}(Cf_{i})\rightarrow\mathbf{C}f_{i}.
\]
\end{prop}

\begin{proof}
By Lemma \ref{lem: generators of Ann(emptyset) for Cf are homogeneous},
to show that $P^{pc}$ descends to $\overline{\mathrm{Sk}}(Cf_{i})$,
we only need to show that $P^{pc}$ maps the generators of $\mathrm{Ann}([\emptyset])$
as in Corollary \ref{cor:generators for skein modules of face cones and face suspensions}
(i) to $0$. This is achieved by a straightforward calculation similar
to that in the proof of Proposition \ref{prop:P descends to Tr}.
For example, one of the generators is
\[
(-A^{2})\Gamma_{ab}\tilde{\Gamma}_{ab}-x_{a,f_{i}}x_{b,f_{i}}=(-A^{2})\Gamma_{ab}\tilde{\Gamma}_{ab}-\left[x_{a,f_{i}}\cdot x_{b,f_{i}}\right].
\]
Thus by definition we have
\begin{eqnarray*}
\lefteqn{P^{pc}\left((-A^{2})\Gamma_{ab}\tilde{\Gamma}_{ab}-\left[x_{a,f_{i}}\cdot x_{b,f_{i}}\right]\right)}\\
&=&(-A^{2})Q(\Gamma_{ab})\rho_{f_{i}}(\tilde{\Gamma}_{ab})-\eta_{f_{i}}\left(\left[x_{a,f_{i}}\cdot x_{b,f_{i}}\right]\right)\\
&=&(-A^{2})(-A^{2})^{-\frac{1}{2}}(-A^{2})^{-\frac{1}{2}}\left[\{a,f_{i}\}\{b,f_{i}\}\right]-\left[\{a,f_{i}\}\{b,f_{i}\}\right]
=0.
\end{eqnarray*}

To show that $Tr_{Cf_{i}}$ descends to $\overline{\mathrm{Sk}}^{pc}(Cf_{i})$,
we need to show that it maps
$I_{f_{i}}^{c}\cdot\overline{\mathrm{Sk}}(Cf_{i})$ to $\{0\}$.
But this is obvious because $I_{f_{i}}^{c}$ is the ideal generated by
elements of the form 
\[
\Gamma_{ab}-(-A^{2})^{-\frac{1}{2}}
\]
and $Tr_{Cf_{i}}$ is left $R[Cf_{i}]$-linear (and recall that
the left $R[Cf_{i}]$-module structure on $\mathbf{C}f_{i}$ is
given by the algebra homomorphism $Q$, which sends $\Gamma_{ab}$
to $(-A^{2})^{-\frac{1}{2}}$.)
\end{proof}

Now we study the gluing of these quantum trace maps.

\begin{lem}
\label{lem:quantum trace map for face cones glue}
The tensor product
\[
\underset{f_{i}\in\mathbf{f}(\mathcal{T})}{\bigotimes}Tr_{Cf_{i}}^{pc}\colon\underset{f_{i}\in\mathbf{f}(\mathcal{T})}{\bigotimes}\overline{\mathrm{Sk}}^{pc}(Cf_{i})\rightarrow\underset{f_{i}\in\mathbf{f}(\mathcal{T})}{\bigotimes}\mathbf{C}f_{i}=\underset{f\in\mathcal{T}^{(2)}}{\bigotimes}\mathbf{S}f
\]
descends to a well-defined $R$-module homomorphism 
\[
\underset{f_{i}\in\mathbf{f}(\mathcal{T})}{\overline{\bigotimes}}Tr_{Cf_{i}}^{pc}\colon\underset{f_{i}\in\mathbf{f}(\mathcal{T})}{\overline{\bigotimes}}\overline{\mathrm{Sk}}^{pc}(Cf_{i})\rightarrow\underset{f\in\mathcal{T}^{(2)}}{\overline{\bigotimes}}\mathbf{S}f.
\]
\end{lem}

\begingroup
\allowdisplaybreaks
\begin{proof}
We need to show that $\underset{f_{i}\in\mathbf{f}(\mathcal{T})}{\bigotimes}Tr_{Cf_{i}}^{pc}$
maps the $R$--submodule 
\[
\overline{\mathfrak{R}}_{E}^{pc}+\sum_{T\in\mathcal{T}}\overline{\mathfrak{R}}_{T}
\]
 of $\underset{f_{i}\in\mathbf{f}(\mathcal{T})}{\bigotimes}\overline{\mathrm{Sk}}^{pc}(Cf_{i})$ (see Definition \ref{def:reduced tensor peoduct of partially corner-reduced modules of face cones})
into the $R$--submodule 
\[
\left(\underset{f\in\mathcal{T}^{(2)}}{\bigotimes}\mathbf{S}f\right)\left(\eta(\tilde{I}_{E})\right)+\sum_{T\in\mathcal{T}}\left(\rho(\tilde{I}_{T})\right)\left(\underset{f\in\mathcal{T}^{(2)}}{\bigotimes}\mathbf{S}f\right)
\]
of $\underset{f\in\mathcal{T}^{(2)}}{\bigotimes}\mathbf{S}f$.

The $R$-submodule $\overline{\mathfrak{R}}_{E}^{pc}$ is the $R$-span
of the elements that admit representatives in $\underset{f_{i}\in\mathbf{f}(\mathcal{T})}{\bigotimes}\overline{\mathrm{Sk}}(Cf_{i})$
of the form (Definition \ref{def:reduced tensor peoduct of partially corner-reduced modules of face cones}
(ii), (iii))
\[
\left(\underset{f_{i}\in\mathbf{f}(\mathcal{T})}{\otimes}x_{Cf_{i}}\right)\cup\left(\overline{e}-(-A^{2})\right),
\]
where $\underset{f_{i}\in\mathbf{f}(\mathcal{T})}{\otimes}x_{Cf_{i}}$
is partially balanced (Definition \ref{def:partially balanced elements})
and $\overline{e}$ is the element of $\underset{f_{i}\in\mathbf{f}(\mathcal{T})}{\bigotimes}\left(\mathbb{B}^{\otimes3}\right)_{Cf_{i}}$
associated with an edge $e\in\mathcal{T}^{(1)}$ given by (assume
the sequence of face cones around $e$ is given as in Figure \ref{fig:face cones around an edge})
\[
\overline{e}=x_{e_{k},f_{1}^{(1)}}x_{e_{1},f_{2}^{(1)}}x_{e_{1},f_{1}^{(2)}}x_{e_{2};f_{2}^{(2)}}\dots x_{e_{k-1},f_{1}^{(k)}}x_{e_{k};f_{2}^{(k)}}.
\]
Note that $\overline{e}$ is also partially balanced, therefore by
the discussion in Remark \ref{rem:matching face, opposite skew-symmetric form},
\[
\left(\underset{f_{i}\in\mathbf{f}(\mathcal{T})}{\otimes}x_{Cf_{i}}\right)\cup\left(\overline{e}-(-A^{2})\right)=\left(\underset{f_{i}\in\mathbf{f}(\mathcal{T})}{\otimes}x_{Cf_{i}}\right)\cdot\left(\overline{e}-(-A^{2})\right).
\]
Now because each $Tr_{Cf_{i}}^{pc}$ is right $\left(\left(\mathbb{B}^{\otimes3}\right)_{Cf_{i}},\cdot\right)=\mathbb{T}\langle Cf_{i}\rangle$-linear,
we have
\begin{eqnarray*}
\lefteqn{\underset{f_{i}\in\mathbf{f}(\mathcal{T})}{\overline{\bigotimes}}Tr_{Cf_{i}}^{pc}\left(\left(\underset{f_{i}\in\mathbf{f}(\mathcal{T})}{\otimes}x_{Cf_{i}}\right)\cdot\left(\overline{e}-(-A^{2})\right)\right)}\\
&=&\left(\underset{f_{i}\in\mathbf{f}(\mathcal{T})}{\otimes}Tr_{Cf_{i}}^{pc}\left(x_{Cf_{i}}\right)\right)\left(\underset{f_{i}\in\mathbf{f}(\mathcal{T})}{\bigotimes}\eta_{f_{i}}\left(\overline{e}\right)-(-A^{2})\right)\\
&=&\left(\underset{f_{i}\in\mathbf{f}(\mathcal{T})}{\otimes}Tr_{Cf_{i}}^{pc}\left(x_{Cf_{i}}\right)\right)\left(\{e_{k},f_{1}^{(1)}\}\{e_{1},f_{2}^{(1)}\}\dots\{e_{k-1},f_{1}^{(k)}\}\{e_{k},f_{2}^{(k)}\}-(-A^{2})\right)\\
&=&\left(\underset{f_{i}\in\mathbf{f}(\mathcal{T})}{\otimes}Tr_{Cf_{i}}^{pc}\left(x_{Cf_{i}}\right)\right)\left(\eta(\tilde{e})-(-A^{2})\right)
\in\left(\underset{f\in\mathcal{T}^{(2)}}{\bigotimes}\mathbf{S}f\right)\left(\eta(\tilde{I}_{E})\right).
\end{eqnarray*}

The $R$-submodule $\overline{\mathfrak{R}}_{T}$ is given by (Definition
\ref{def:reduced tensor peoduct of partially corner-reduced modules of face cones}
(i))
\[
\tilde{I}_{T}\left(\underset{f_{i}\in\mathbf{f}(\mathcal{T})}{\bigotimes}\overline{\mathrm{Sk}}^{pc}(Cf_{i})\right).
\]
Because each $Tr_{Cf_{i}}^{pc}$ is left $\tilde{\mathbb{T}}_{f_{i}}$-linear and is induced by the map $P^{pc}$,
we only need to show that $\underset{f_{i}\in\mathbf{f}(\mathcal{T})}{\bigotimes}P^{pc}$
maps the right ideal $\tilde{I}_{T}$ of $\underset{f_{i}\in\mathbf{f}(\mathcal{T})}{\bigotimes}\tilde{\mathbb{T}}_{f_{i}}$
into $\rho(\tilde{I}_{T})$, but this is trivial as, by definition,
the action of $P^{pc}$ on the factor $\tilde{\mathbb{T}}_{f_{i}}$
is given by the map $\rho_{f_{i}}$ (and the map $\rho$ is nothing
but the tensor product of the $\rho_{f_{i}}$'s over all bare faces).
\end{proof}
\endgroup

Next, we are going to compare the three ``total'' $R$-module homomorphisms we obtain by gluing up the quantum trace maps respectively over all tetrahedra, over all face suspensions, and over all face cones:
\begin{enumerate}
    \item (Tetrahedra.) The $R$-homomorphism $\underset{T\in\mathcal{T}}{\overline{\bigotimes}}Tr_{T}^{c}\colon\underset{T\in\mathcal{T}}{\overline{\bigotimes}}\overline{\mathrm{Sk}}^{c}(T)\rightarrow\hat{\mathcal{G}}_{\mathcal{T}}
$  given in Proposition \ref{prop:quantum trace for T glue}.
\item (Face suspensions.) The  $R$-homomorphism $\underset{f\in\mathcal{T}^{(2)}}{\overline{\bigotimes}}Tr_{Sf}\colon\underset{f\in\mathcal{T}^{(2)}}{\overline{\bigotimes}}\overline{\mathrm{Sk}}(Sf)\rightarrow\underset{f\in\mathcal{T}^{(2)}}{\overline{\bigotimes}}\mathbf{S}f$
given just before Definition \ref{def: quantum trace map of PP}.
\item (Face cones.) 
The $R$-homomorphism $\underset{f_{i}\in\mathbf{f}(\mathcal{T})}{\overline{\bigotimes}}Tr_{Cf_{i}}^{pc}\colon\underset{f_{i}\in\mathbf{f}(\mathcal{T})}{\overline{\bigotimes}}\overline{\mathrm{Sk}}^{pc}(Cf_{i})\rightarrow\underset{f\in\mathcal{T}^{(2)}}{\overline{\bigotimes}}\mathbf{S}f$
given in Lemma \ref{lem:quantum trace map for face cones glue}.
\end{enumerate}


\begin{lem}
\label{lem:comparing 2,3}The following diagram commutes
\[
\begin{tikzcd}
\underset{f\in\mathcal{T}^{(2)}}{\overline{\bigotimes}}\overline{\mathrm{Sk}}(Sf) \arrow[rrd, "\underset{f\in\mathcal{T}^{(2)}}{\overline{\bigotimes}}Tr_{Sf}"] \arrow[dd, "\underset{f\in\mathcal{T}^{(2)}}{\overline{\bigotimes}}\sigma_{Sf}"'] &  &                                                                    \\
&  & \underset{f\in\mathcal{T}^{(2)}}{\overline{\bigotimes}}\mathbf{S}f \\
\underset{f_{i}\in\mathbf{f}(\mathcal{T})}{\overline{\bigotimes}}\overline{\mathrm{Sk}}^{pc}(Cf_{i}) \arrow[rru, "\underset{f_{i}\in\mathbf{f}(\mathcal{T})}{\overline{\bigotimes}}Tr_{Cf_{i}}^{pc}"']        &  &                                                                   
\end{tikzcd}
\]
(Here, the map $\underset{f\in\mathcal{T}^{(2)}}{\overline{\bigotimes}}\sigma_{Sf}$
is the reduced tensor product of splitting homomorphism coming from
splitting each face suspension into two face cones, see Proposition
\ref{prop:gluing splitting homomorphism coming from splitting every single face suspension}.)
\end{lem}

\begingroup
\allowdisplaybreaks
\begin{proof}
Let $f\in\mathcal{T}^{(2)}$ and $f_{i}$,
$f_{j}$ be bare faces identified to $f$, so the edges of $f_{i}$
and $f_{j}$ are labeled as in Corollary \ref{cor:generators for skein modules of face cones and face suspensions}
(ii). 

To begin we will check commutativity for skeins localized at the source vertices on the edges of $f$. Let 
\begin{equation}
x=x_{a_{i},f_{i};a_{j},f_{j}}^{m}x_{b_{i},f_{i};b_{j},f_{j}}^{n}x_{c_{i},f_{i};c_{j},f_{j}}^{\ell}\in\left(\mathbb{B}^{\otimes 3}\right)_{Sf}.\label{eq:a monimial x}
\end{equation}
On the one hand we have
\[
\eta_{f}(x)=\{a_{i},f_{i}\}^{m}\{a_{j},f_{j}\}^{m}\{b_{i},f_{i}\}^{n}\{b_{j},f_{j}\}^{n}\{c_{i},f_{i}\}^{\ell}\{c_{j},f_{j}\}^{\ell};
\]
on the other hand, $x$ represents an element of $\overline{\mathrm{Sk}}(Sf)$,
under the splitting homomorphism $\sigma_{Sf}$, it becomes the element
of $\overline{\mathrm{Sk}}^{pc}(Cf_{i})\otimes\overline{\mathrm{Sk}}^{pc}(Cf_{j})$
represented by
\begin{align*}
x^{\prime} & =x_{a_{i},f_{i}}^{m}x_{b_{i},f_{i}}^{n}x_{c_{i},f_{i}}^{\ell}x_{a_{j},f_{j}}^{m}x_{b_{j},f_{j}}^{n}x_{c_{j},f_{j}}^{\ell}\\
 & =\left[x_{a_{i},f_{i}}^{m}\cdot x_{b_{i},f_{i}}^{n}\cdot x_{c_{i},f_{i}}^{\ell}\cdot x_{a_{j},f_{j}}^{m}\cdot x_{b_{j},f_{j}}^{n}\cdot x_{c_{j},f_{j}}^{\ell}\right]\in\mathbb{T}\langle Cf_{i}\rangle\otimes\mathbb{T}\langle Cf_{j}\rangle
\end{align*}
and we have
\begin{align*}
\left(\eta_{f_{i}}\otimes\eta_{f_{j}}\right)(x^{\prime}) & =\left[\{a_{i},f_{i}\}^{m}\{b_{i},f_{i}\}^{n}\{c_{i},f_{i}\}^{\ell}\{a_{j},f_{j}\}^{m}\{b_{j},f_{j}\}^{n}\{c_{j},f_{j}\}^{\ell}\right]\\
 & =\{a_{i},f_{i}\}^{m}\{b_{i},f_{i}\}^{n}\{c_{i},f_{i}\}^{\ell}\{a_{j},f_{j}\}^{m}\{b_{j},f_{j}\}^{n}\{c_{j},f_{j}\}^{\ell}\\
 & =\{a_{i},f_{i}\}^{m}\{a_{j},f_{j}\}^{m}\{b_{i},f_{i}\}^{n}\{b_{j},f_{j}\}^{n}\{c_{i},f_{i}\}^{\ell}\{c_{j},f_{j}\}^{\ell}\\
 & =\eta_{f}(x)
\end{align*}
(Here, the $A$ factors one gains when expending the Weyl-ordering
cancel each other because the corresponding generators of the two
face cone modules associated with the two face cones gluing into a single
face suspension commute up to $A$ factors of opposite powers, see the presentations of quantum tori in Definition \ref{def:face cone module and face suspension module}.) Now, $\underset{f\in\mathcal{T}^{(2)}}{\overline{\bigotimes}}\overline{\mathrm{Sk}}(Sf)$
is the $R$-span of elements of the form 
$
w=\underset{f\in\mathcal{T}^{(2)}}{\otimes}\left(\tilde{\Gamma}_{f_{i}}\otimes\tilde{\Gamma}_{f_{j}}\otimes x\right)$,
where $\tilde{\Gamma}_{f_{i}}\in\widetilde{\mathbb{T}}_{f_{i}}$,
$\tilde{\Gamma}_{f_{j}}\in\widetilde{\mathbb{T}}_{f_{j}}$ and $x\in\left(\mathbb{B}^{\otimes 3}\right)_{Sf}$ is a monomial
as given by (\ref{eq:a monimial x}). (Here, we use $f_{i}$ and $f_{j}$
to denote the two bare faces that are identified to $f$ for all $f\in\mathcal{T}^{(2)}$.)
We have
\[
\left(\underset{f\in\mathcal{T}^{(2)}}{\overline{\bigotimes}}Tr_{Sf}\right)(w)=\underset{f\in\mathcal{T}^{(2)}}{\otimes}\left(\rho_{f_{i}}\left(\tilde{\Gamma}_{f_{i}}\right)\otimes\rho_{f_{j}}\left(\tilde{\Gamma}_{f_{j}}\right)\otimes\eta_{f}(x)\right);
\]
on the other hand, we also have
\begin{align*}
\left(\underset{f_{i}\in\mathbf{f}(\mathcal{T})}{\overline{\bigotimes}}Tr_{Cf_{i}}^{pc}\right)\left(\underset{f\in\mathcal{T}^{(2)}}{\overline{\bigotimes}}\sigma_{Sf}\right)(w) & =\left(\underset{f_{i}\in\mathbf{f}(\mathcal{T})}{\overline{\bigotimes}}Tr_{Cf_{i}}^{pc}\right)\left(\underset{f\in\mathcal{T}^{(2)}}{\otimes}\left(\tilde{\Gamma}_{f_{i}}\otimes\tilde{\Gamma}_{f_{j}}\otimes x^{\prime}\right)\right)\\
 & =\underset{f\in\mathcal{T}^{(2)}}{\otimes}\left(\rho_{f_{i}}\left(\tilde{\Gamma}_{f_{i}}\right)\otimes\rho_{f_{j}}\left(\tilde{\Gamma}_{f_{j}}\right)\otimes\left(\eta_{f_{i}}\otimes\eta_{f_{j}}\right)(x^{\prime})\right)\\
 & =\underset{f\in\mathcal{T}^{(2)}}{\otimes}\left(\rho_{f_{i}}\left(\tilde{\Gamma}_{f_{i}}\right)\otimes\rho_{f_{j}}\left(\tilde{\Gamma}_{f_{j}}\right)\otimes\eta_{f}(x)\right)\\
 & =\left(\underset{f\in\mathcal{T}^{(2)}}{\overline{\bigotimes}}Tr_{Sf}\right)(w).\\
\end{align*}
\end{proof}
\endgroup

\begin{lem}
\label{lem:comparing 1,3}The following diagram commutes
\[
\begin{tikzcd} \underset{f_{i}\in\mathbf{f}(\mathcal{T})}{\overline{\bigotimes}}\overline{\mathrm{Sk}}^{pc}(Cf_{i}) \arrow[rr, "\underset{f_{i}\in\mathbf{f}(\mathcal{T})}{\overline{\bigotimes}}Tr_{Cf_{i}}^{pc}"] &  & \underset{f\in\mathcal{T}^{(2)}}{\overline{\bigotimes}}\mathbf{S}f \\ &  & {\hat{\mathcal{G}}_{\mathcal{T}}^{[\text{PP}]}} \arrow[u, hook]    \\ \underset{T\in\mathcal{T}}{\overline{\bigotimes}}\overline{\mathrm{Sk}}^{c}(T) \arrow[uu, "\underset{T\in\mathcal{T}}{\overline{\bigotimes}}\sigma_{T}"] \arrow[rr, "\underset{T\in\mathcal{T}}{\overline{\bigotimes}}Tr_{T}^{c}"']      &  & \hat{\mathcal{G}}_{\mathcal{T}} \arrow[u, "\mu"']                  \end{tikzcd}
\]
(Here, the map $\underset{T\in\mathcal{T}}{\overline{\bigotimes}}\sigma_{T}$
is the reduced tensor product of the splitting homomorphisms coming from
splitting each ideal tetrahedron into four face cones, as in Proposition
\ref{prop:gluing splitting homomorphism coming from splitting every ideal tetrahedron}.)
\end{lem}

\begingroup
\allowdisplaybreaks
\begin{proof}
Again, let us start with a preliminary observation. Let $T$ be an
ideal tetrahedron with its edges and faces labeled as in Figure
\ref{fig:ideal T with edges and faces all labeled} and let
\begin{align}
\hat{x}_{T} & =\hat{z}^{m}\hat{z}^{\prime n}\hat{z}^{\prime\prime\ell}\hat{y}^{p}\hat{y}^{\prime q}\hat{y}^{\prime\prime r}\label{eq:a monomial xhat}\\
 & =\left[\hat{z}^{m}\cdot\hat{z}^{\prime n}\cdot\hat{z}^{\prime\prime\ell}\cdot\hat{y}^{p}\cdot\hat{y}^{\prime q}\cdot\hat{y}^{\prime\prime r}\right]\in\mathbb{T}\langle T\rangle.
\end{align}
We have
\[
\mu_{T}(\hat{x}_{T})=\left[\tilde{z}^{m}\tilde{z}^{\prime n}\tilde{z}^{\prime\prime\ell}\tilde{y}^{p}\tilde{y}^{\prime q}\tilde{y}^{\prime\prime r}\right]\in\tilde{\mathbb{T}}\langle T\rangle.
\]
On the other hand, $\hat{x}_T$ represents an element of $\overline{\mathrm{Sk}}^{c}(T)$
and if we apply the splitting homomorphism $\sigma_{T}$ to it, the result
is the element of $\underset{f_{i}\in\mathbf{f}(T)}{\overline{\bigotimes}}\overline{\mathrm{Sk}}^{pc}(Cf_{i})$
represented by
\begin{eqnarray*}
\hat{x}_{T}^{\prime} & = &x_{z,f_{1}}^{m}x_{z,f_{3}}^{m}x_{z^{\prime},f_{1}}^{n}x_{z^{\prime},f_{2}}^{n}x_{z^{\prime\prime},f_{2}}^{\ell}x_{z^{\prime\prime},f_{3}}^{\ell}
x_{y,f_{2}}^{p}x_{y,f_{4}}^{p}x_{y^{\prime},f_{3}}^{q}x_{y^{\prime},f_{4}}^{q}x_{y^{\prime\prime},f_{1}}^{r}x_{y^{\prime\prime},f_{4}}^{r}\\
 &=& \left[x_{z,f_{1}}^{m}\cdot x_{z^{\prime},f_{1}}^{n}\cdot x_{y^{\prime\prime},f_{1}}^{r}\right]\otimes\left[x_{z^{\prime},f_{2}}^{n}\cdot x_{z^{\prime\prime},f_{2}}^{\ell}\cdot x_{y,f_{2}}^{p}\right]\\
 &&\otimes\left[x_{z,f_{3}}^{m}\cdot x_{z^{\prime\prime},f_{3}}^{\ell}\cdot x_{y^{\prime},f_{3}}^{q}\right]\otimes\left[x_{y,f_{4}}^{p}\cdot x_{y^{\prime},f_{4}}^{q}\cdot x_{y^{\prime\prime},f_{4}}^{r}\right].
\end{eqnarray*}
Therefore we have
\begin{eqnarray*}
\left(\underset{f_{i}\in\mathbf{f}(T)}{\bigotimes}\eta_{f_{i}}\right)(\hat{x}_{T}^{\prime})
&=&\left[\{z,f_{1}\}^{m}\{z^{\prime},f_{1}\}^{n}\{y^{\prime\prime},f_{1}\}^{r}\right]\otimes\left[\{z^{\prime},f_{2}\}^{n}\{z^{\prime\prime},f_{2}\}^{\ell}\{y,f_{2}\}^{p}\right]\\[-12pt]
& &\otimes\left[\{z,f_{3}\}^{m}\{z^{\prime\prime},f_{3}\}^{\ell}\{y^{\prime},f_{3}\}^{q}\right]\otimes\left[\{y,f_{4}\}^{p}\{y^{\prime},f_{4}\}^{q}\{y^{\prime\prime},f_{4}\}^{r}\right]\\
&=&[\{z,f_{1}\}^{m}\{z,f_{3}\}^{m}\{z^{\prime},f_{1}\}^{n}\{z^{\prime},f_{2}\}^{n}\{z^{\prime\prime},f_{2}\}^{\ell}\{z^{\prime\prime},f_{3}\}^{\ell}\\
 && \{y,f_{2}\}^{p}\{y,f_{4}\}^{p}\{y^{\prime},f_{3}\}^{q}\{y^{\prime},f_{4}\}^{q}\{y^{\prime\prime},f_{1}\}^{r}\{y^{\prime\prime},f_{4}\}^{r}]\\
 &=&\left[\tilde{z}^{m}\tilde{z}^{\prime n}\tilde{z}^{\prime\prime\ell}\tilde{y}^{p}\tilde{y}^{\prime q}\tilde{y}^{\prime\prime r}\right]\\
 &=&\mu_{T}(\hat{x}_{T}).\\
\end{eqnarray*}
Now $\underset{T\in\mathcal{T}}{\overline{\bigotimes}}\overline{\text{Sk}}^{c}(T)$
is the $R$-span of elements of the form 
$
w=\underset{T\in\mathcal{T}}{\otimes}\hat{x}_{T},
$
where $\hat{x}_{T}\in\mathbb{T}\langle T\rangle$ is a monomial as
given by (\ref{eq:a monomial xhat}). We have
\begin{align*}
\mu\circ\left(\underset{T\in\mathcal{T}}{\overline{\bigotimes}}Tr_{T}^{c}\right)(w) & =\underset{T\in\mathcal{T}}{\otimes}\mu_{T}(\hat{x}_{T}) =\underset{T\in\mathcal{T}}{\otimes}\left(\underset{f_{i}\in\mathbf{f}(T)}{\bigotimes}\eta_{f_{i}}\right)(\hat{x}_{T}^{\prime})\\
 & =\left(\underset{f_{i}\in\mathbf{f}(T)}{\underset{T\in\mathcal{T}}{\overline{\bigotimes}}}Tr_{Cf_{i}}^{pc}\right)\left(\underset{T\in\mathcal{T}}{\otimes}\hat{x}_{T}^{\prime}\right)
 =\left(\underset{f_{i}\in\mathbf{f}(T)}{\underset{T\in\mathcal{T}}{\overline{\bigotimes}}}Tr_{Cf_{i}}^{pc}\right)\left(\underset{T\in\mathcal{T}}{\overline{\bigotimes}}\sigma_{T}\right)(w).
\end{align*}
\end{proof}
\endgroup

Finally, we can prove our second main comparison result. The following is a restatement
of Theorem \ref{thm: Main theorem 2, exact relation between our Tr and PP's Tr}.
\begin{thm}
$Tr_{\mathcal{T}}^{[\text{PP}]}=\mu\circ Tr_{\mathcal{T}}$.
\end{thm}

\begin{proof}
Putting Proposition \ref{prop:commutative diagram of various splitting homomorphisms},
Lemma \ref{lem:comparing 2,3} and Lemma \ref{lem:comparing 1,3}
together, we obtain the following commutative diagram
\[
\begin{tikzcd}
 &  & \underset{f\in\mathcal{T}^{(2)}}{\overline{\bigotimes}}\overline{\mathrm{Sk}}(Sf) \arrow[dd] \arrow[rrdd, "\underset{f\in\mathcal{T}^{(2)}}{\overline{\bigotimes}}Tr_{Sf}"] &  &                                                                    \\
 &  &                                                                                &  &                                                                    \\
\mathrm{Sk}(Y) \arrow[rruu, "\tilde{\sigma}"] \arrow[rrdd, "\sigma"'] &  & \underset{f_{i}\in\mathbf{f}(\mathcal{T})}{\overline{\bigotimes}}\overline{\mathrm{Sk}}^{pc}(Cf_{i}) \arrow[rr]                                           &  & \underset{f\in\mathcal{T}^{(2)}}{\overline{\bigotimes}}\mathbf{S}f \\
 &  &                                                                                &  & {\hat{\mathcal{G}}_{\mathcal{T}}^{[\text{PP}]}} \arrow[u, hook]    \\
&  & \underset{T\in\mathcal{T}}{\overline{\bigotimes}}\overline{\mathrm{Sk}}^{c}(T) \arrow[uu] \arrow[rr, "\underset{T\in\mathcal{T}}{\overline{\bigotimes}}Tr_{T}^{c}"]         &  & \hat{\mathcal{G}}_{\mathcal{T}} \arrow[u, "\mu"']                 
\end{tikzcd}
\]
Taking the boundary of the above diagram, together with the fact that
the image of $Tr_{\mathcal{T}}^{[\text{PP}]}$ is in $\hat{\mathcal{G}}_{\mathcal{T}}^{[\text{PP}]}$,
we conclude that $Tr_{\mathcal{T}}^{[\text{PP}]}=\mu\circ Tr_{\mathcal{T}}$.
\end{proof}

\end{document}